\documentclass[11pt,reqno]{amsart}
\usepackage{amsmath,amssymb,amsthm,amsfonts,latexsym, amsthm, amscd, mathrsfs, stmaryrd}
\usepackage[colorlinks=true,linkcolor=blue,citecolor=purple]{hyperref}
\usepackage{graphicx,color,colortbl}
\usepackage{upgreek}
\usepackage[mathscr]{euscript}
\usepackage{hhline}
\numberwithin{equation}{section}
\usepackage{dynkin-diagrams}
\usepackage{epic}
\allowdisplaybreaks
\numberwithin{equation}{section}
\usepackage{float}
\usepackage{tikz}
\usetikzlibrary{
    cd,
	shapes,
	arrows,
	positioning,
	decorations.markings,
	decorations.pathmorphing,
	circuits.logic.US,
	circuits.logic.IEC,
	fit,
	calc,
	plotmarks,
	matrix
}
\tikzset{
>=stealth',
help lines/.style={dashed, thick},
axis/.style={<->},
important line/.style={thick},
connection/.style={thick, dotted},
punkt/.style={
rectan\mathrm{GL}e,
rounded corners,
draw=black, thick,
text width=4.5em,
minimum height=2em,
text centered,
},
pil/.style={
->,
thick,
gray,
shorten <=2pt,
shorten >=2pt,}
}

\usepackage{array}
\usepackage{adjustbox}
\usepackage{enumerate}

\newtheorem{proposition}{Proposition}[section]
\newtheorem{lemma}[proposition]{Lemma}
\newtheorem{corollary}[proposition]{Corollary}
\newtheorem{theorem}[proposition]{Theorem}
\theoremstyle{definition}
\newtheorem{definition}[proposition]{Definition}
\newtheorem{remark}[proposition]{Remark}

\newtheorem{conjecture}[proposition]{Conjecture}

\newtheorem{innercustomthm}{{\bf Theorem}}
\newenvironment{customthm}[1]
  {\renewcommand\theinnercustomthm{#1}\innercustomthm}
  {\endinnercustomthm}

\newcommand{\arxiv}[1]{\href{http://arxiv.org/abs/#1}{\tt arXiv:\nolinkurl{#1}}}

\newcommand{\Rmnum}[1]{\expandafter\@slowromancap\romannumeral #1@}
\def \g{\mathfrak{g}}

\def \dm{\diamond}

\def \N{\mathbb{N}}
\def \Q{\mathbb{Q}}
\def \C{\mathbb{C}}
\def \Z{\mathbb{Z}}
\def \I{\mathbb{I}}
\def \Br{\mathrm{Br}}

\def \cJ{\mathcal{J} }

\def \bcG{\mathcal{G}_{\bullet}}
\def \tbcG{\widetilde{\mathcal{G}}_{\bullet}}
\def \ck{\mathcal{K}}
\def \fX{\Upsilon}
\def \wI{\I_{\circ}}
\def \wItau{\I_{\circ,\tau}}
\def \bI{\I_{\bullet}}
\def \bIi{\I_{\bullet,i}}

\def \tfX{\widetilde{\Upsilon}}
\def \cR{\mathcal{R}}
\def \bs{\mathbf{r}} 
\def \bF{\mathbb{F}}
\def \bw{w_\bullet}
\def \bwi{w_{\bullet,i}}
\def \bW{W_{\bullet}}
\def \reW{W^{\circ}}
\def \bbw{{\boldsymbol{w}}_\circ}
\def \tbU{\tU_{\bullet}}
\def \bU{\U_{\bullet}}
\def \ba{\mathbf{a}}
\def \tk{\widetilde{k}}
\def \tT{\widetilde{\mathscr T}}
\def \tTD{\widetilde{T}}

\def \cT{\mathcal{T}}
\def \tTT{\widetilde{\mathbf{T}}}
\def \sT{\mathscr T}

\newcommand{\ta}[1]{T'_{#1,-1}}
\newcommand{\tb}[1]{T''_{#1,+1}}
\newcommand{\Ta}[1]{\TT'_{#1,-1}}
\newcommand{\Tb}[1]{\TT''_{#1,+1}}
\newcommand{\tTa}[1]{\tTT'_{#1,-1}}
\newcommand{\tTb}[1]{\tTT''_{#1,+1}}

\newcommand{\Tae}[1]{\TT'_{#1,e}}
\newcommand{\Tbe}[1]{\TT''_{#1, e}}
\newcommand{\tTae}[1]{\tTT'_{#1,e}}
\newcommand{\tTbe}[1]{\tTT''_{#1, e}}

\def \T{\mathrm{T}}

\def \cL{\DH} 
\def \Id{\mathrm{Id}}
\def \tpsi{\psi} 
\def \tPsi{\widetilde{\Psi}}

\newcommand{\U}{\mathbf{U}}
\newcommand{\Ui}{(\U)^\imath}

\newcommand{\tU}{\widetilde{{\mathbf U}} }
\newcommand{\tUi}{\widetilde{{\mathbf U}}^\imath}

\newcommand{\Aut}{\operatorname{Aut}\nolimits}

\newcommand{\qbinom}[2]{\begin{bmatrix} #1\\#2 \end{bmatrix} }

\def \ad{\text{Ad}}

\newcommand{\nc}{\newcommand}
\nc{\greentext}[1]{\textcolor{green}{#1}}
\nc{\redtext}[1]{\textcolor{red}{#1}}
\nc{\bluetext}[1]{\textcolor{blue}{#1}}
\nc{\brown}[1]{\browntext{ #1}}
\nc{\green}[1]{\greentext{ #1}}
\nc{\red}[1]{\redtext{ #1}}
\nc{\blue}[1]{\bluetext{ #1}}

\def \Q {\mathbb Q}
\def \TT{\mathbf T}

\newcommand{\wt}{\text{wt}}

\def \bvs{{\boldsymbol{\varsigma}}}
\def \obvs{\overline{\bvs}}
\def \balpha{\overline{\alpha}} 

\def \lc{l_\circ}
\def \vs{\varsigma}
\def \U{\mathbf U}
\def \Ui{\mathbf{U}^\imath}

\begin{document}
\title[Relative braid group symmetries on $\imath$quantum groups]{An intrinsic approach to relative braid group symmetries on $\imath$quantum groups}

\author[Weiqiang Wang]{Weiqiang Wang}

\author[Weinan Zhang]{Weinan Zhang}
\address{Department of Mathematics, University of Virginia, Charlottesville, VA 22904, USA}
\email{ww9c@virginia.edu (Wang)}
\email{wz3nz@virginia.edu (Zhang)}



\begin{abstract}
We initiate a general approach to the relative braid group symmetries on (universal) $\imath$quantum groups, arising from quantum symmetric pairs of arbitrary finite types, and their modules. Our approach is built on new intertwining properties of quasi $K$-matrices which we develop and braid group symmetries on (Drinfeld double) quantum groups. Explicit formulas for these new symmetries on $\imath$quantum groups are obtained.  We establish a number of fundamental properties for these symmetries on $\imath$quantum groups, strikingly parallel to their well-known quantum group counterparts. We apply these symmetries to fully establish rank one factorizations of quasi $K$-matrices, and this factorization property in turn helps to show that the new symmetries satisfy relative braid relations. As a consequence, conjectures of Kolb-Pellegrini and Dobson-Kolb are settled affirmatively. Finally, the above approach allows us to construct compatible relative braid group actions on modules over quantum groups for the first time.
\end{abstract}

\maketitle

\tableofcontents

\section{Introduction}

\subsection{Background}

Braid group symmetries have played an essential role in understanding the structures of Drinfeld-Jimbo quantum groups $\U$ and have found applications in geometric representation theory and categorification among others. These symmetries were constructed by Lusztig and used in first constructions of PBW bases and canonical bases in ADE type \cite{Lus90a}. They have further been generalized to non-simply-laced types and beyond \cite{Lus90b, Lus93}. Another crucial property is that there exists a compatible braid group action on integrable $\U$-modules. A systematic exposition on the braid group actions on quantum groups and their modules forms a significant portion of Lusztig's book  \cite[Ch.~ 5, Part~ VI]{Lus93}.

Let $\tU =\langle E_i, F_i, K_i, K_i' \mid i\in \I \rangle$ be the Drinfeld double quantum group, where $K_i K_i'$ are central. The quantum group $\U =\langle E_i, F_i, K_i^{\pm 1} \mid i\in \I \rangle$ is recovered from $\tU$ by a central reduction: \[
\U =\tU/ ( K_i K_i' -1 \mid i\in \I).
\]
The Drinfeld doubles naturally arise from the Hall algebra construction of Bridgeland \cite{Br13}, and it is shown in \cite{LW21b}  that reflection functors provide braid group actions on the Drinfeld doubles; see Proposition~\ref{prop:braid1}. As a straightforward generalization for Lusztig's symmetries on $\U$ \cite[37.2.4]{Lus93}, there are four variants of braid group operators $\tTD_{i,e}', \tTD_{i,e}''$ on $\tU$, 
for $e \in \{\pm 1\}$ and $i\in \I$, which are related to each other by conjugations of certain (anti-) involutions \cite{LW21b}; see \eqref{eq:sTs}:
\begin{equation}
  \label{eq:4braid}
  \tTD'_{i,-e} =\sigma \circ \tTD''_{i,+e} \circ \sigma,
  \qquad
  \tTD''_{i,-e}:= \tpsi \circ \tTD''_{i,+e} \circ \tpsi,
  \qquad
  \tTD'_{i,+e}:= \tpsi \circ \tTD'_{i,-e} \circ \tpsi.
\end{equation}
Here $\tpsi$ is the bar involution and $\sigma$ is an anti-involution on $\tU$; see Proposition~\ref{prop:QG4}.


%

Associated with any Satake diagram $(\I =\bI \cup \wI, \tau)$, a quantum symmetric pair $(\U, \Ui_\bvs)$ was introduced by Gail Letzter in finite type \cite{Let99, Let02} as a $q$-deformation of the usual symmetric pair; here, $\Ui_\bvs$ is a coideal subalgebra of $\U$ depending on parameters $\bvs =(\bvs_i)_{i\in \wI}$. {\em Universal} quantum symmetric pairs $(\tU, \tUi)$ (of quasi-split type) were formulated in \cite{LW22}, where the parameters are replaced by suitable central elements in $\tUi$, and $\Ui_\bvs$ is recovered from $\tUi$ by a central reduction. ($\Ui_\bvs, \tUi$ will be referred to as $\imath$quantum groups, and they are called {\em quasi-split} if $\bI =\emptyset$ and {\em split} if in addition $\tau =\Id$.) Several fundamental constructions on quantum groups, including (quasi) $R$-matrix, canonical bases, and Hall algebra realization have been generalized to the setting of quantum symmetric pairs in recent years; see \cite{BW18a, BW18b, BK19, LW22}.

Lusztig's braid group actions on $\U$ do not preserve the subalgebra $\Ui_\bvs$ in general. Kolb-Pellegrini \cite{KP11} proposed that there should be relative braid group symmetries on $\imath$quantum groups corresponding to the relative (or restricted) Weyl groups for the underlying symmetric pairs. For a class of $\imath$quantum groups of finite type (including all quasi-split types and type AII) with some specific parameters, formulas for such braid group actions were found and verified {\em loc. cit.} via computer computation. The relative braid group action for type AI appeared earlier in \cite{Ch07} and \cite{MR08}.

There has been some limited progress on relative braid group action on $\Ui_\bvs$ in the last decade; for type AIII see Dobson \cite{Dob20}. An $\imath$Hall algebra approach has been developed to realize the universal {\em quasi-split} $\imath$quantum groups $\tUi$ \cite{LW22}. As a generalization of Ringel's construction \cite{Rin96}, reflection functors \cite{LW21a, LW21b} are used to construct relative braid group actions on $\tUi$ of quasi-split type, where the braid group operators act on the central elements in $\tUi$ non-trivially.  For $\tUi$ or $\Ui_\bvs$ in general beyond quasi-split type, no conjectural formulas or conceptual explanations for relative braid group actions were available.

There are  braid group actions on $\U$-modules which are compatible with braid group actions on quantum groups, cf. \cite{Lus93}. In contrast, no relative braid group action on $\Ui_\bvs$-modules has been known to date. The Hall algebra approach does not help providing any clue on such action at the module level.

\subsection{Goal}

Our goal is to develop a conceptual and general approach to relative braid group actions on $\imath$quantum groups, arising from (universal) quantum symmetric pairs of arbitrary finite type, and on their modules for the first time. This in particular settles the longstanding conjecture of Kolb and Pellegrini \cite{KP11} in a constructive manner.

It is crucial for us to work with universal $\imath$quantum groups. We shall formulate relative braid group symmetries $\tTT_{i,e}', \tTT_{i,e}''$ on $\tUi$, for $e \in \{\pm 1\}$ and $i\in \wItau$, which are related to each other via conjugations by a bar involution $\tpsi^\imath$ and an anti-involution $\sigma^\imath$ on $\tUi$; compare \eqref{eq:4braid}:
\begin{equation*}
  \tTT'_{i,-e} =\sigma^\imath \circ \tTT''_{i,+e} \circ \sigma^\imath,
  \qquad
  \tTT''_{i,-e}:= \tpsi^\imath \circ \tTT''_{i,+e} \circ \tpsi^\imath,
  \qquad
  \tTT'_{i,+e}:= \tpsi^\imath \circ \tTT'_{i,-e} \circ \tpsi^\imath.
\end{equation*}

By central reductions and rescaling automorphisms, these symmetries descend to relative braid group actions on $\imath$quantum groups with parameters $\Ui_\bvs$. Moreover, we are able to formulate compatible relative braid group actions on integrable $\U$-modules. We further establish a number of basic properties of these new symmetries which are natural $\imath$-counterparts of well-known properties for Lusztig's braid group symmetries.

\subsection{The basic idea}

Various constructions for quantum groups can be regarded as constructions for quantum symmetric pairs of diagonal type $(\U \otimes \U, \U)$, and hence $\imath$quantum groups can be viewed as a vast generalization of quantum groups. This simple observation can be instrumental on determining what form a suitable $\imath$-generalization should take; for example, this view was applied successfully in the developments of $\imath$canonical bases arising from quantum symmetric pairs in \cite{BW18b} and $\imath$Hall algebras which realize universal $\imath$quantum groups \cite{LW22}; see also the recent development of $\imath$crystal bases by Watanabe \cite{W21b}.

Denote by ${\bf L}_i''$ the rank 1 quasi $R$-matrix associated to $i\in \I$, and let ${\bf L}_i'$ be its inverse. The following formula in \cite[37.3.2]{Lus93}:
\begin{align}
  \label{eq:TTL}
(\ta{i} \otimes  \ta{i}) \Delta (\tb{i} u) = {\bf L}_i' \Delta(u) {\bf L}_i''
\end{align}
provides a relation between braid group actions on $\U$ and $\U\otimes \U$; a formula similar to \eqref{eq:TTL} via a different formulation of braid operators appeared in \cite{LS90} and \cite{KR90}. A starting point of this paper is to view a variant of the identity \eqref{eq:TTL} as a formula in the setting of (universal) quantum symmetric pairs of diagonal type $(\tU \otimes \tU, \tU)$; see \S\ref{sec:diag}.

Now, let $(\tU, \tUi)$ be a general universal quantum symmetric pair. Inspired by the relation \eqref{eq:TTL}, we aim at formulating a relation between braid group action on the Drinfeld double $\tU$ and the desired relative braid group action on the universal $\imath$quantum group $\tUi$ through conjugations of rank 1 quasi $K$-matrices $\tfX_i$, for $i\in \wI$.

Quasi $K$-matrices were originally formulated in \cite{BW18a} as an intertwiner between the embedding $\imath: \Ui_\bvs \rightarrow \U$ and a bar-involution conjugated embedding (for parameters $\bvs$ satisfying strong constraints); a proof in greater generality was given in \cite{BK19} under a technical assumption (which was removed later in \cite{BW21}). A reformulation by Appel and Vlaar \cite{AV22} (also see \cite{KY20}) bypassed a direct use of the bar maps, allowing more general parameters $\bvs$. In this paper, we upgrade these constructions by formulating the quasi $K$-matrices $\tfX$ for universal quantum symmetric pairs, and in particular, the rank 1 quasi $K$-matrices $\tfX_i$, for $i\in \wI$.

Dobson and Kolb \cite{DK19} proposed (conjectural) factorizations of quasi $K$-matrices in finite types into products of rank 1 quasi $K$-matrices, analogous to factorizations of quasi $R$-matrices \cite{LS90, KR90}. In their formulation, a certain scaling twist shows up. In this paper, we upgrade the formulation of the factorization together with the corresponding scaling twist to quasi $K$-matrices $\tfX$ in the universal setting.

Examples indicate that our basic idea of constructing the desired relative braid group action on $\tUi$ via quasi $K$-matrix and braid group action on $\tU$ (viewed as a generalization of \eqref{eq:TTL}) basically works --- up to a simple twist: it is necessary to use {\em suitably rescaled} braid group operators on $\tU$. Remarkably, this scaling turns out to coincide with the aforementioned scaling which appears in the factorizations of a quasi $K$-matrix $\tfX$. We are able to explore this compatibility to draw strong consequences on the seemingly unrelated topics: relative braid group actions and the factorization of quasi $K$-matrices.

\subsection{Main results}

\subsubsection{New intertwining properties of quasi $K$-matrices}

We formulate universal quantum symmetric pairs $(\tU, \tUi)$ associated to arbitrary  Satake diagrams and their basic properties in Section~\ref{sub:iQG}, following and  generalizing the quasi-split setting in \cite{LW22}. The algebra $\tUi$ contains $\tU^{\imath 0}$ and $\tbU$ naturally as subalgebras, where $\tbU$ is the Drinfeld double associated to $\bI$ and $\tU^{\imath 0}$ is a Cartan subalgebra generated by $\tk_i =K_iK_{\tau i}'$, for $i\in \wI$.

We recall the recent somewhat technical formulation of a quasi $K$-matrix $\fX_\bvs$ for $(\U, \Ui_\bvs)$ from \cite{AV22} (cf. \cite{BW18a, BK19, BW18b} for earlier constructions) in Theorem~\ref{thm:qK} and upgrade it to a universal version $\tfX$ for $(\tU, \tUi)$ in Theorem~\ref{thm:qK2}. It turns out that $\tfX$ admits a more conceptual and simpler characterization in terms of the anti-involution $\sigma$ on $\tU$ as follows.

 \begin{customthm} {\bf A}
  [Theorem~\ref{thm:fX1}]
  \label{thm:A}
 The quasi $K$-matrix $\tfX =\sum_{\mu \in \N \I} \tfX^\mu$, for $\tfX^\mu\in \tU_\mu^+$, is uniquely characterized by $\tfX^0=1$ and the following intertwining relations:
 \begin{align*}
 B_i  \tfX &=  \tfX B_i^{\sigma}    \quad (i\in \wI),
 \qquad\quad
 x  \tfX = \tfX x    \quad (x\in \tU^{\imath 0}\tbU).
 \end{align*}
\end{customthm}
This characterization of $\tfX$ plays a basic role in producing explicit formulas for relative braid group actions on $\tUi$; see the proof of Theorem~\ref{thm:rktwo1} in \S\ref{subsec:proofTiBj}.
There is a similar simple characterization of $\fX_\bvs$ for $\Ui_\bvs$ in terms of the anti-involution $\sigma\tau$ on $\U$; see Theorem~\ref{thm:quasiKUi}. (It is tempting to regard this as a new definition of $\fX_\bvs$.)

We use a distinguished scaling automorphism $\tPsi_{\bvs_\star}$ to define a rescaled bar involution $\tpsi_\star$ on $\tU$ (by twisting the bar involution $\tpsi$ on $\tU$). By exploring further intertwining properties via $\tfX$ as in \cite{Ko21}, we establish in Kac-Moody generality a bar involution $\tpsi^\imath$ (see Proposition~\ref{prop:newb3}) and an anti-involution $\sigma^\imath$ (see Proposition~\ref{prop:newb1}) from $\tpsi_\star$ and $\sigma$, respectively. These (anti-)involutions $\tpsi^\imath$ and $\sigma^\imath$ were known in some quasi-split cases; see \cite{CLW23}.

Denote by $\tfX_i$, for $i\in \wI$, the quasi $K$-matrix associated to the rank one Satake subdiagram $(\bI\cup \{i,\tau i\},\tau)$.

\subsubsection{New symmetries $\tTT'_{i,e}, \tTT''_{i,e}$}

Associated to a Satake diagram $(\I =\bI \cup \wI, \tau)$, one has the (absolute) Weyl group $W$ generated by the simple reflections $s_i$, for $i\in \I$, and a finite parabolic subgroup $W_\bullet =\langle s_i \mid i \in \bI \rangle$ with the longest element $w_\bullet$. Given $i\in \wI$, one has a rank 1 Satake subdiagram $(\I_{\bullet,i} =\bI \cup \{i, \tau i\}, \tau)$, and define $\bs_i \in W$ as in \eqref{def:bsi}. As $\bs_i =\bs_{\tau i}$, it suffices to restrict to $\bs_i$, for $i\in \wItau$ (here $\wItau$ is a set of fixed representatives of $\tau$-orbits on $\I$). The relative Weyl group $\reW$ is a subgroup of $W$ generated by $\bs_i$, for $i\in \wItau$; abstractly, $\reW$ is a Weyl group with $\bs_i$ $(i \in \wItau$) as simple reflections \cite{Lus76}; also see \cite{OV90, Lus03, DK19}.

Let $\tTD''_{i,+1}$ and $\tTD'_{i,-1}$, for $i\in \I$, be the braid group operators on $\tU$ \cite{LW21b}; see Proposition~\ref{prop:braid1}. Let $\tT''_{i,+1}$ and $\tT'_{i,-1}$ be the rescaled version of $\tTD''_{i,+1}$ and $\tTD'_{i,-1}$ via conjugation by a scaling automorphism $ \tPsi_{\bvs_\diamond}$; see \eqref{def:tT}--\eqref{def:tT-1}. As $\tT'_{j,-1}$, for $j\in \I$, satisfy the braid relations, we can make sense of
$\tT_{w,-1}'$, for $w\in W$, and in particular $\tT_{\bs_i,-1}'$, for $i\in \wI$, as automorphisms of $\tU.$

\begin{customthm} {\bf B}
  [Theorem~\ref{thm:newb0}, Proposition~\ref{prop:Cartanblack}, Theorem~\ref{thm:rkone1}, Theorem~\ref{thm:rktwo1}]
\label{thm:B}
Let $i\in \wI$. There exists a unique automorphism $\tTa{i}$ of $\tUi$ such that the following intertwining relation holds:
\begin{align}
   \label{eq:braid}
\tTa{i}(x) \tfX_i
=\tfX_i \tT_{\bs_i,-1}' (x), \qquad \text{ for all } x\in \tUi.
\end{align}
More precisely, the action of $\tTa{i}$ on $\tUi$ is given as follows:
\begin{enumerate}
\item
$\tTa{i}(x) =(\widehat{\tau}_{\bullet,i} \circ \widehat{\tau})(x)$, and
$\tTa{i}(\tk_{j,\dm}) =\tk_{\bs_i \alpha_j,\dm}$, for all $x\in \tbU$, $j\in \wI$.

\item
$\tTa{i}(B_i) =-q^{-(\alpha_i,\bw\alpha_{\tau i}) } \tT_{w_\bullet}^2 ( B_{\tau_{\bullet,i} \tau i})\ck_{\tau_{\bullet,i} \tau i}^{ -1}.$

\item
The formulas for $\tTa{i}(B_j)$ ($i \neq j \in \wItau$) are listed in Table \ref{table:rktwoSatake}.
\end{enumerate}
\end{customthm}
See \eqref{def:taui} and \eqref{eq:kla} for notation $\tau_{\bullet,i}$ and $\tk_{\lambda,\dm}$; also see \eqref{eq:Tw} and Remark~\ref{rem:sameT} for the braid group operator $\tT_{w_\bullet}$. By definition, we have $\bs_i =\bs_{\tau i}$, $\tfX_i =\tfX_{\tau i}$, and $\tTa{i} =\tTa{\tau i}$; thus, we only need to consider $\tTT_{i,-1}'$, for $i\in\I_{\circ,\tau}$.

In the same spirit of \eqref{eq:braid} in Theorem~\ref{thm:B}, the identity \eqref{eq:TTL} for the Drinfeld double quantum group $\tU$ can be reformulated as the intertwining relation \eqref{eq:diag4} for quantum symmetric pair $(\tU \otimes \tU, \tU)$ of diagonal type.

Another symmetry $\tTb{i}$ on $\tUi$, for $i\in \wI$, is formulated in Theorem~\ref{thm:Tb} which satisfies the following intertwining relation in \eqref{eq:newb0-1}, similar to \eqref{eq:braid}:
\begin{align*}
\tTb{i}(x)\, \tT''_{\bs_i,+1}(\tfX_i^{-1}) = \tT''_{\bs_i,+1}(\tfX_i^{-1})\, \tT''_{\bs_i,+1}(x),
\quad \text{ for all } x\in \tUi.
\end{align*}

We further define 2 more symmetries $\tTT'_{i,+1}$ and $\tTT''_{i,-1}$ on $\tUi$ by conjugating $\tTa{i}$ and $\tTb{i}$ via the involution $\tpsi^\imath$; see \eqref{def:tTT2}.
These symmetries are related to each other as follows; compare \cite[Chap.~37]{Lus93}.

\begin{customthm} {\bf C}
  [Theorem~\ref{thm:newb1}]
  \label{thm:C}
  Let $e =\pm 1$ and $i \in \wI$.
The symmetries $\tTT_{i,e}'$ and $\tTT''_{i,-e}$ are mutual inverses. Moreover, we have
$\tTT_{i,e}' =\sigma^\imath \circ \tTT''_{i,-e} \circ \sigma^\imath.$
\end{customthm}
Actually, part of the proof of Theorem~\ref{thm:B} (i.e., the invertibility of $\tTa{i}$) is completed only when it is established in Theorem~\ref{thm:C} that $\tTa{i}$ and $\tTb{i}$ are mutual inverses. This is one main reason why we have formulated $\tTb{i}$ separately in spite of its many similarities with the properties for $\tTa{i}$ which we already established.

Here is an outline of proofs of Theorems~\ref{thm:B}--\ref{thm:C}. We first establish the existence of an endomorphism $\tTa{i}$ on $\tUi$ which satisfies the intertwining relation \eqref{eq:braid}, by proving Properties (1)-(3) in Theorem~\ref{thm:B} one-by-one. Properties (1)-(2) are established uniformly in Proposition~\ref{prop:Cartanblack} and  Theorem~\ref{thm:rkone1}.
We formulate a structural result in Proposition~\ref{prop:rktwoRij} as a main step toward a uniform proof of the rank 2 formulas in (3) (see Theorem~\ref{thm:rktwo1}); Proposition~\ref{prop:rktwoRij} is then verified by a type-by-type computation in Appendix~\ref{app1}.
In order to prove the invertibility of $\tTa{i}$, we establish another endomorphism $\tTb{i}$ on $\tUi$ which satisfies the intertwining relation \eqref{eq:newb0-1} in Theorem~\ref{thm:Tb}; the existence for $\tTb{i}$ is proved by a strategy similar to the one for $\tTa{i}$.
Finally, we show in Theorem~\ref{thm:newb1} that $\tTa{i}$ and $\tTb{i}$ are mutual inverses by invoking the uniqueness of elements satisfying an intertwining relation.

The formulas for actions of $\tTa{i}$ and $\tTb{i}$ on generators of $\tUi$ are mostly new. In quasi-split types, up to some twistings, we recover the formulas obtained by Hall algebra computation in \cite{LW21a}, and by central reductions to $\Ui_\bvs$, we recover formulas obtained by computer computation in \cite{KP11}.

\subsubsection{A basic property of braid symmetries}

The following theorem is a generalization of a well-known basic property of braid group action on quantum groups; see \cite{Lus93}.

\begin{customthm} {\bf D}
  [see Theorem~\ref{thm:fact1}]
  \label{thm:D}
Suppose that $wi\in \wI$, for $w\in \reW$ and $i\in \wI$. Then we have $\tTb{w}(B_i)=B_{wi}$.
\end{customthm}
The dependence in the formulation of Theorem~\ref{thm:fact1} on reduced expressions $\underline{w}$ of $w$ can be removed, once Theorem~\ref{thm:F} on braid relations for $\tTb{j}$ is established. We reduce the proof of Theorem~\ref{thm:D} to the rank 2 cases. The proofs in rank 2 cases are largely uniform (avoiding type-by-type computation), based on the counterpart results in quantum group setting, the defining intertwining property of $\tTb{w}$, and some weight arguments.

\subsubsection{Factorizations of a quasi $K$-matrix}

It is well known that a quasi $R$-matrix admits a factorization into a product of rank 1 $R$-matrices parametrized by positive roots; see \cite{KR90, LS90}; also cf. \cite{Ja95}.

Dobson and Kolb \cite{DK19} proposed a conjecture on an analogous factorization of a quasi $K$-matrix into a product, denoted by $\tfX_{\bbw}$, of rank 1 factors parametrized by restricted positive roots; see \eqref{eq:Upk} for notation. They established a reduction from a general finite type to the rank 2 Satake diagrams. In addition, they established the rank 2 cases of {\em split} types and type AII/AIII, via a type-by-type lengthy computation based on several explicit formulas for rank 1 quasi $K$-matrices which they computed.

Exploring (the rank 2 cases of) Theorem~\ref{thm:D} and some of its consequences, we provide a uniform and concise proof that $\tfX_{\bbw}$ satisfies the same defining intertwining relations for $\tfX$. Then the factorization property for arbitrary finite types follows by the uniqueness of $\tfX$.

\begin{customthm} {\bf E}
  [Dobson-Kolb Conjecture, Theorem~\ref{thm:factor}]
  \label{thm:E}
The quasi $K$-matrix $\tfX$ for $\tUi$ of finite type admits a factorization $\tfX = \tfX_{\bbw}$.
\end{customthm}

\subsubsection{Relative braid group relations}

Recall Lusztig's symmetries $T_{i,e}', T_{i,e}''$ on a quantum group $\U$ satisfy braid group relations associated to the (absolute) Weyl group $W$ \cite{Lus93}; see \cite{LW21b} for analogous statements on a Drinfeld double $\tU$. We have the following generalization in the setting of $\imath$quantum groups. Denote by $\Br(\reW)$ the braid group associated to $\reW.$

\begin{customthm} {\bf F}
  [Theorem~\ref{thm:newb2}]
  \label{thm:F}
Fix $e \in \{\pm 1\}$.
The symmetries $\tTae{i}$ (and respectively, $\tTbe{i}$) of $\tUi$, for $i \in \wItau$, satisfy the relative braid group relations in $\Br(\reW)$.
\end{customthm}

With the help of the intertwining relation \eqref{eq:braid}, the proof of Theorem~\ref{thm:F} is built on the braid group relations for $\tT_i$ ($i\in \I$) and the factorization properties of rank 2 quasi $K$-matrices established in Theorem~\ref{thm:E}.

It was shown in \cite{BW18b} that Lusztig's symmetries $T_{i,e}'$ and $T_{i,e}''$ on $\U$, for $i\in \bI$, preserve the subalgebra $\Ui_\bvs$ (under some constraints on $\bvs$). We easily upgrade this statement to the universal quantum symmetric pair $(\tU, \tUi)$, providing a braid group action of $\Br(\bW)$ on $\tUi$; see Proposition~\ref{prop:Tjblack}. Actually, we obtain 4 variants of actions of $\Br(\bW)$ on $\tUi$ generated by $\tT'_{j,e}$ or $\tT''_{j,e}$, for $j\in \bI$, respectively.

It is further established that the two (``black and white") braid group actions on $\tUi$ combine neatly into an action of a semi-direct product $\Br(\bW) \rtimes \Br(\reW)$ on $\tUi$.

\begin{customthm} {\bf G}
  [Theorem~\ref{thm:newb3}, Corollary~\ref{cor:newb1}]
  \label{thm:G}
  Let $e=\pm 1$.
\begin{enumerate}
    \item
There exists a braid group action of $\Br(\bW) \rtimes \Br(\reW)$ on $\tUi$ as automorphisms of algebras generated by $\tT'_{j,e} \; (j\in \bI)$ and $\tTT'_{i,e} \;(i\in \wItau)$.
    \item
There exists a braid group action of $\Br(\bW) \rtimes \Br(\reW)$ on $\tUi$ as automorphisms of algebras generated by $\tT''_{j,e} \; (j\in \bI)$ and $\tTT''_{i,e} \;(i\in \wItau)$.
\end{enumerate}
\end{customthm}

Theorem~\ref{thm:G} (or more precisely, its $\Ui_\bvs$-counterpart in Theorem~\ref{thm:braid6}; see \S\ref{subsec:BrUi} below) confirms an old conjecture of Kolb and Pellegrini \cite[Conjecture 1.2]{KP11} in full generality, and moreover, we have provided precise formulas for the braid group actions.

\subsubsection{Relative braid group symmetries on $\Ui_\bvs$}
 \label{subsec:BrUi}

By central reductions, the symmetries $\tTa{i}, \tTb{i}$ on the universal $\imath$quantum group $\tUi$, for $i\in \wI$, descend naturally to the $\imath$quantum group $\Ui_{\bvs_\dm}$ with the distinguished parameter $\bvs_\dm$. On the other hand, the symmetries $\tTT'_{i,+1}, \tTT''_{i,-1}$ naturally descend to $\Ui_{\overline{\bvs}_{\star\dm}}$; see the commutative diagrams in \S\ref{braid:Uibvs}.
We then transport the relative braid group symmetries from $\Ui_{\bvs_\dm}$ and  $\Ui_{\overline{\bvs}_{\star\dm}}$ to the $\imath$quantum groups $\Ui_\bvs$ (see Theorems~\ref{thm:braid5}--\ref{thm:braid6}), for an arbitrary parameter $\bvs$, thanks to the isomorphism $\Ui_{\bvs_\dm} \cong \Ui_{\bvs}$ given in Proposition~\ref{prop:QG3}.

\subsubsection{Relative braid group actions on $\U$-modules}
Let $i\in \wI$, $e =\pm 1$, and $\bvs$ be a balanced parameter (see the line below \eqref{def:par}).
We show that the symmetries  $\Tae{i}, \Tbe{i}$ on the $\imath$quantum group $\Ui_\bvs$ (defined by central reductions) satisfy natural intertwining relations with the usual braid group symmetries on $\U$. These intertwining properties allow us to formulate automorphisms (denoted again by the same notations $\Tae{i}, \Tbe{i}$) on an arbitrary finite-dimensional $\U$-module $M$ of type {\bf 1}; see \eqref{eq:mod4}. These operators on $M$ admit favorable properties parallel to those satisfied by Lusztig's braid group actions on modules.

\begin{customthm} {\bf H}
  [Theorem~\ref{thm:braidM}, Theorem~\ref{thm:mod2}]
  \label{thm:H}
Let $i\in \wI$ and $e =\pm 1$, and let $M$ be any finite-dimensional $\U$-module of type {\bf 1}. The automorphisms $\Tae{i},  \Tbe{i}$ on $M$ are compatible with the corresponding automorphisms on $\Ui_\bvs$, i.e.,
\begin{equation*}
\Tae{i}(x v)=\Tae{i}(x) \Tae{i}(v),\qquad
\Tbe{i}(x v)=\Tbe{i}(x) \Tbe{i}(v),
\end{equation*}
 for any $x\in \Ui_\bvs, v\in M.$
 Moreover, the operators $\Tae{i}$ (respectively, $\Tbe{i}$) on $M$, for $i\in \wI$, satisfy the relative braid group relations in $\Br(\reW)$.
\end{customthm}

%

In this paper we have assumed that a ground field $\mathbb F$ is the algebraic closure of $\Q(q)$ partly due to uses of rescaling automorphisms, though often it suffices to work with the field $\Q(q^{\frac12})$ if we choose the parameters $\bvs$ suitably. There is a $\Q(q)$-form ${}_{\Q}\tUi$ of $\tUi$ such that $\tUi =\bF\otimes_{\Q(q)} {}_{\Q}\tUi$; see \eqref{eq:UiQ}. The symmetries $\tTae{i}, \tTbe{i}$ indeed preserve the $\Q(q)$-subalgebra ${}_{\Q}\tUi$; see Proposition~\ref{prop:Q}.  Theorems~\ref{thm:A}--\ref{thm:G} remain valid for ${}_{\Q}\tUi$.

\subsection{Future works and applications}

The formulations of the main results (Theorems~\ref{thm:A}--\ref{thm:H}), up to some reasonable rephrasing, make sense for universal quantum symmetric pairs of arbitrary Kac-Moody type (cf. \cite{Ko14}), and we conjecture they are valid in this great generality. For example, the symmetries $\tTa{i}$, for $i\in \wI$, for $\tUi$ of Kac-Moody type will follow once Conjecture~\ref{conj:TiBjKM} is confirmed. The main reason on the restriction to finite types in this paper is that we rely on the classification of Satake diagrams to explicitly compute the rank 2 formulas for $\tTa{i}(B_j)$ and $\tTb{i}(B_j)$, which in particular verify that they lie in $\tUi$. Section~\ref{sec:quasi K} is valid in Kac-Moody generality. Steps (1)--(2) in Theorem~\ref{thm:B} (which occupy most of Section~\ref{sec:symmetry}) are also valid in the Kac-Moody setting.

Some further developments will be carried out in future works. We shall extend the constructions of relative braid group actions to (universal) $\imath$quantum groups of affine type. We plan to use the new tools developed in this paper to attack the conjectures in \cite{CLW21a, CLW23} on relative braid group actions on quasi-split universal $\imath$quantum groups of Kac-Moody type. We also plan to understand the relative braid group action on $\Ui_\bvs$-modules more explicitly, and this may serve as a starting point for a new approach toward relative braid group action on $\imath$quantum groups; compare \cite{Lus93}.

The relative braid group symmetries of this paper (and their affine generalization) will be used crucially in the Drinfeld type presentation of quasi-split affine $\imath$quantum groups in an upcoming work joint with Ming Lu. It is expected that they will continue to play a key role for Drinfeld type presentations of general affine $\imath$quantum groups.

One may hope that these new braid group symmetries preserve the integral $\Z[q,q^{-1}]$-form on (modified) $\imath$quantum groups in \cite{BW18b, BW21}. (This will be highly nontrivial to verify, as the $\imath$divided powers are much more sophisticated than the divided powers.) It will be interesting to develop further connections among relative braid group actions, PBW bases and $\imath$canonical bases; compare \cite{Lus93}. They may help to stimulate further KLR type categorification of $\imath$quantum groups as well as $\imath$Hall algebra realization of $\imath$quantum groups beyond quasi-split type.

Kolb and Yakimov \cite{KY20} extended the construction of quantum symmetric pairs to the setting of Nichols algebras of diagonal type.
The new intertwining properties of quasi $K$-matrices and the relative braid group actions established in this paper seem well suited for generalizations in this direction.

The notion of relative Coxeter groups, which is valid in a more general setting than symmetric pairs, admits a geometric interpretation \cite{Lus76, Lus03}. It will be exciting to realize relative braid group action in geometric and categorical frameworks, and develop possible connections to the representation theory of real groups (cf. \cite{BV21} and references therein). It will be very interesting to explore more general braid group actions associated to relative Coxeter groups.

\subsection{Organization}

The paper is organized into Sections~\ref{sec:QG}--\ref{sec:modules} and Appendix~\ref{app1}.
Below we provide a detailed description section by section.

In Section~\ref{sec:QG}, we review and set up the basics and notations on quantum groups $\U$ and Drinfeld doubles $\tU$, including several (anti-) involutions and a rescaling automorphism $\tPsi_{\ba}$ on $\tU$. We recall explicit formulas for braid group actions on $\tU$. Associated to a  Satake diagram $(\I =\bI \cup \wI, \tau)$, we form a relative Weyl group $\reW =\langle \bs_i \mid i\in \wI \rangle$. Then we formulate (universal) quantum symmetric pairs $(\tU, \tUi)$ and $(\U, \Ui_\bvs)$.

In Section~\ref{sec:quasi K}, we formulate quasi $K$-matrix $\tfX$ in the universal quantum symmetric pair setting, and establish a new intertwining property via the anti-involution $\sigma$ on $\tU$. We establish an anti-involution $\sigma^\imath$ on $\tUi$ via $\sigma$ and an intertwining property of $\tfX$. We formulate a rescaled bar involution $\tpsi_\star$ on $\tU$, and then establish a bar involution $\tpsi^\imath$ on $\tUi$ via $\tpsi_\star$ and an intertwining property of $\tfX$. An anti-involution $\sigma_\tau$ on $\Ui_\bvs$ for an arbitrary parameter $\bvs$ is also established.

In Section~\ref{sec:symmetry}, we formulate rescaled braid group symmetries $\tT'_{w,-1}$, for $w\in W$, on $\tU$ via a rescaling automorphism $\Psi_{\bvs_\diamond}$. We define $\tTa{i}$ in terms of an intertwining property involving $\tfX$ and the rescaled braid group symmetries $\tT_{\bs_i}^{-1} \equiv \tT'_{\bs_i,-1}$; see Theorem~\ref{thm:newb0}. We then formulate additional symmetries $\tTT_{i,-1}'$ and $\tTT_{i,\pm 1}''$ on $\tUi$ via conjugations of $\tTa{i}$ by an anti-involution $\sigma^\imath$ and a bar involution $\tpsi^\imath$. We obtain explicit formulas for the actions of $\tTa{i}$ on $\tU^{\imath 0} \tbU$ in Proposition~\ref{prop:Cartanblack} and on $B_i$ in Theorem~\ref{thm:rkone1}.

In Section~ \ref{sec:rktwo}, we formulate a general structural result which relates formulas for $\tTa{i}( B_j)$ and $\tT_{\bs_i}^{-1} (F_j)$; see Proposition~\ref{prop:rktwoRij}. The explicit formulas for $\tTa{i}( B_j)$ in each rank 2 universal $\imath$quantum group are collected in Table~ \ref{table:rktwoSatake}. The type-by-type verification of these formulas is postponed to Appendix~\ref{app1}. 

In Section~ \ref{sec:Tb}, we formulate another symmetry $\tTb{i}$ on $\tUi$ using a different intertwining property. Then we formulate the counterparts of the results in Sections~\ref{sec:symmetry}--\ref{sec:rktwo}. We collect all rank 2 formulas for $\tTb{i}(B_j)$ in Table~\ref{table:rktwoSatake2}, whose proofs similar to Appendix~\ref{app1} will be skipped (the detail can be found in Appendix~B in an arXiv version).

We then show that $\tTa{i}$ and $\tTb{i}$ are mutual inverses, completing the proofs that $\tTa{i}$ and $\tTb{i}$ are automorphisms of $\tUi$. The property $\tTT_{i,e}' =\sigma^\imath \circ \tTT''_{i,-e} \circ \sigma^\imath$ follows by inspection from the explicit formulas for the actions of $\tTa{i}$ and $\tTb{i}$.

In Section~\ref{sec:BiBj}, we establish a basic formula $\tTT_{\underline{w}} (B_i ) =B_j$, for $i,j \in \wI$ and $w\in \reW$ such that $w \alpha_i =\alpha_j$, generalizing a well-known formula in quantum groups. We reduce the proof of the formula to the rank 2 Satake diagrams. We then provide uniform proofs in the rank 2 cases.

In Section~\ref{sec:factor}, we prove uniformly the factorization property of quasi $K$-matrices in all rank 2 quantum symmetric pairs, completing the proof of Dobson-Kolb's conjecture in arbitrary finite types. This is an application of the formula established in Section~\ref{sec:BiBj}.

In Section~\ref{sec:braid}, we verify that the symmetries $\tTae{i},\tTbe{i}$ satisfy the braid group relations in $\Br(\reW)$. Together with the braid group action given by $\tT'_{j,e},\tT''_{j,e}$, for $j\in \bI$, we obtain four braid group actions of $\Br(W_\bullet) \rtimes \Br(\reW)$ on $\tUi$. By taking central reductions and using the isomorphism $\phi_\bvs:\Ui_{\bvs_\dm}\cong\Ui_\bvs$, we construct relative braid group symmetries $\Tae{i},\Tbe{i}$ on $\tUi_\bvs$ for general parameters $\bvs$, confirming the main conjecture in \cite{KP11}.

In Section~\ref{sec:modules}, we formulate linear operators $\Tae{i}, \Tbe{i}$ on any finite-dimensional $\U$-module. We show that they are compatible with corresponding automorphisms on $\tUi$, and that they satisfy the relative braid group relations.

\subsection{Notations}

We list the notations which are often used throughout the paper.
\smallskip

$\triangleright$ $\N,\Z,\Q$, $\C$ -- sets of nonnegative integers, integers, rational  and complex numbers
\medskip

$\triangleright$ $\cR, \cR^\vee$ -- systems of roots and coroots with simple systems $\Pi=\{\alpha_i|i\in \I\}$ and $\Pi^\vee=\{\alpha_i^\vee|i\in \I\}$, respectively
\medskip

$\triangleright$ $W, \ell(\cdot)$ -- the Weyl group and its length function
\medskip

$\triangleright$ $w_0,\tau_0$ -- the longest element in $W$ and its associated diagram involution
\medskip

$\triangleright$ $\tTD'_{i,e}, \tTD''_{i,e}$ -- braid group symmetries on $\tU$
\medskip

$\triangleright$ $(\I =\bI \cup \wI, \tau)$ -- admissible pairs (aka Satake diagrams)
\medskip

$\triangleright$ $W_\bullet,\cR_\bullet$ -- the Weyl group and root system associated to the subdiagram $\bI$
\medskip

$\triangleright$ $w_\bullet$ -- the longest element in $W_\bullet$
\medskip


$\triangleright$  $W_{\bullet,i}$ -- the parabolic subgroup of $W$ generated by $s_k$, for $k \in \I_{\bullet,i}:=\bI\cup \{i,\tau i\}$
\medskip

$\triangleright$  $w_{\bullet,i}$, $\tau_{\bullet,i}$ -- the longest element of $W_{\bullet,i}$ and its associated diagram involution
\medskip

$\triangleright$ $\reW, \ell_\circ(\cdot)$ -- the relative Weyl group generated by $\bs_i :=w_{\bullet,i} w_{\bullet}$, for $i\in \wI$, and its length function such that $\ell_\circ(\bs_i)=1$
\medskip

$\triangleright$ $\bbw$ -- the longest element in $\reW$
\medskip

$\triangleright$ $\U, \tU$ -- quantum group and Drinfeld double
\medskip

$\triangleright$ $\widehat{\tau},\widehat{\tau}_0$ -- involutions on $\tU$ induced by the diagram involutions $\tau,\tau_0$
\medskip

$\triangleright$ $\tUi,\Ui_\bvs$ -- universal $\imath$quantum group and  $\imath$quantum group with parameter $\bvs$
\medskip

$\triangleright$ $\tfX$ -- quasi $K$-matrix for universal quantum symmetric pair $(\tU, \tUi)$
\medskip

$\triangleright$ $\bvs_\diamond,\bvs_\star$ -- two distinguished parameters; see \eqref{def:vsi} and \eqref{eq:bvs star}
\medskip

$\triangleright$ $\tPsi_{\ba}$ -- a rescaling automorphism of $\tU$; see \eqref{tPsi}
\medskip

$\triangleright$ $\Phi_{\ba}$ -- a rescaling automorphism of $\U$; see \eqref{def:Phi}
\medskip

$\triangleright$ $\pi_{\bvs}$ -- a central reduction from $\tU$ to $\U$; see \eqref{pibvs}
\medskip

$\triangleright$ $\pi^\imath_{\bvs}$ -- a central reduction from $\tUi$ to $\Ui_\bvs$; see Proposition~\ref{prop:QG2}
\medskip

$\triangleright$ $\tpsi^\imath$ -- a bar involution on $\tUi$; see \eqref{eq:newb9}
\medskip

$\triangleright$ $\sigma^\imath$ -- an anti-involution on $\tUi$; see \eqref{eq:newb1-2}
\medskip

$\triangleright$ $\sigma_\tau$ -- an anti-involution on $\Ui_\bvs$; see \eqref{eq:fX8}
\medskip

$\triangleright$ $\tT'_{i,e}, \tT''_{i,e}$ -- rescaled (via $\tPsi_{\ba}$) braid group symmetries on $\tU$; see \eqref{def:tT}--\eqref{def:tT-1}
\medskip

$\triangleright$ $\tTT_{i,e}', \tTT_{i,e}''$ -- braid group symmetries on $\tUi$
\medskip

$\triangleright$ $\tT_i, \tT_i^{-1}, \tTT_i, \tTT_i^{-1}$ -- shorthand notations for $\tT''_{i,+1},\tT'_{i,-1}, \tTT''_{i,+1},\tTT'_{i,-1}$
\medskip

$\triangleright$ $\sT'_{i,e;\bvs}, \sT''_{i,e;\bvs}$ -- rescaled braid group symmetries on $\U$; see \eqref{def:sT1}, \eqref{def:sT2}



\vspace{4mm}

{\bf Acknowledgment.}
We thank Stefan Kolb and Ming Lu for helpful comments and suggestions. We thank an anonymous referee for a careful reading and helpful comments. WW is partially supported by the NSF grant DMS-2001351. WZ is supported by a GSAS fellowship at University of Virginia and WW's NSF Graduate Research Assistantship.

\section{Drinfeld doubles and quantum symmetric pairs}
  \label{sec:QG}

In this section, we set up notations for quantum groups, Drinfeld doubles, and quantum symmetric pairs. We review the relative Weyl and braid groups associated to Satake diagrams. Several basic properties of (universal) $\imath$quantum groups are presented.

\subsection{Quantum groups and Drinfeld doubles}

We set up notations for a quantum group $\U$ of finite type and its Drinfeld double $\tU$.

Let $\g$ be a semisimple Lie algebra over $\C$ with a symmetrizable Cartan matrix $C=(c_{ij})_{i,j \in \I}$. Let $D =\text{diag}(\epsilon_i\mid \epsilon_i\in \Z_{\ge 1},\; i\in \I)$ be a symmetrizer, i.e., $DC$ is symmetric, such that $\gcd \{\epsilon_i\mid i\in \I \}=1$. Fix a simple system $\Pi=\{\alpha_i|i\in\I\}$ of $\g$ and a set of simple coroots $\Pi^\vee=\{\alpha_i^\vee|i\in \I\}$. Let $\cR$ and $ \cR^\vee $ be the corresponding root and coroot systems. Denote the root lattice by $\Z\I:= \oplus_{i\in \I} \Z\alpha_i$. Let $(\cdot ,\cdot)$ be the normalized Killing form on $\Z\I$ so that the short roots have squared length~ 2. The Weyl group $W$ is generated by the  simple reflections $s_i:\Z { \I}\rightarrow\Z { \I}$, for $i \in  \I$, such that $s_i(\alpha_j)=\alpha_j-c_{ij}\alpha_i$. Set $w_0$ to be the longest element of $W$.

Let $q$ be an indeterminate and $\Q(q)$ be the field of rational functions in $q$ with coefficients in $\Q$, the field of rational numbers. Set $\bF$ to be the algebraic closure of $\Q(q)$ and $\bF^\times:= \bF \setminus\{0\}$. We denote
\[
q_i:=q^{\epsilon_i}, \qquad \forall i\in \I.
\]

Denote, for $r,m \in \N$,
\[
 [r]_t =\frac{t^r-t^{-r}}{t-t^{-1}},
 \quad
 [r]_t!=\prod_{i=1}^r [i]_t, \quad \qbinom{m}{r}_t =\frac{[m]_t [m-1]_t \ldots [m-r+1]_t}{[r]_t!}.
\]
We mainly take $t=q, q_i$.

Then $\tU := \tU_q(\g)$ is defined to be the $\bF$-algebra generated by $E_i,F_i, K_i,K_i'$, $i\in \I$, where $K_i, K_i'$ are invertible, subject to the following relations: $K_i, K_j'$ commute with each other, for all $i,j \in \I$,
\begin{align}
[E_i,F_j]= \delta_{ij} \frac{K_i-K_i'}{q-q^{-1}},   \quad
K_i E_j & =q_i^{c_{ij}} E_j K_i,  \quad K_i F_j=q_i^{-c_{ij}} F_j K_i,
\label{eq:EK}
\\
K_i' E_j=q_i^{-c_{ij}} E_j K_i', & \qquad K_i' F_j=q_i^{c_{ij}} F_j K_i',
 \label{eq:K2}
\end{align}
 and the quantum Serre relations, for $i\neq j \in \I$,
\begin{align}
& \sum_{s=0}^{1-c_{ij}} (-1)^s \qbinom{1-c_{ij} }{s}_{q_i} E_i^s E_j  E_i^{1-c_{ij}-s}=0,
  \label{eq:serre1} \\
& \sum_{s=0}^{1-c_{ij}} (-1)^s \qbinom{1-c_{ij} }{s}_{q_i} F_i^s F_j  F_i^{1-c_{ij}-s}=0.
  \label{eq:serre2}
\end{align}
Note that $K_i K_i'$ are central in $\tU$, for all $i \in \I$.

The comultiplication $\Delta: \widetilde{\U} \rightarrow \widetilde{\U} \otimes \widetilde{\U}$ is defined as follows:
\begin{align}  \label{eq:Delta}
\begin{split}
\Delta(E_i)  = E_i \otimes 1 + K_i \otimes E_i, & \quad \Delta(F_i) = 1 \otimes F_i + F_i \otimes K_{i}', \\
 \Delta(K_{i}) = K_{i} \otimes K_{i}, & \quad \Delta(K_{i}') = K_{i}' \otimes K_{i}'.
 \end{split}
\end{align}

Let $\U =\U_q(\g)$ be the Drinfeld-Jimbo quantum group associated to $\g$ over $\bF$ with Chevalley generators $\{ E_i,F_i,K_i^{\pm 1}|i\in \I\}$, whose relations can be obtained from $\tU$ above by simply replacing $K_i'$ by $K_i^{-1}$, for all $i$; that is, one identifies $\U =\tU/ ( K_i K_i' -1\mid i\in \I)$. Both $\tU$ and $\U$ admit standard triangular decompositions, $\tU =\tU^-\tU^0 \tU^+$ and $\U =\U^-\U^0 \U^+$; we identify $\tU^+ =\U^+ =\langle E_i \mid i\in \I \rangle$ and $\tU^- =\U^-$.

For any scalars $\ba=( a_i)_{i\in \I}\in  \bF^{\times,\I}$, one has a isomorphism
\[
\tU/ ( K_i K_i'- a_i \mid i\in \I ) \stackrel{\cong}{\longrightarrow} \U
\]
through the central reduction
\begin{align}
  \label{pibvs}
  \begin{split}
\pi_{\ba} &: \tU \longrightarrow \U,
\\
F_i \mapsto F_i, \quad & E_i \mapsto \sqrt{a_i}E_i, \quad K_i \mapsto \sqrt{a_i}K_i,  \quad K'_i \mapsto \sqrt{a_i}K_i^{-1}.
\end{split}
\end{align}
The canonical identification uses $\pi_{\bf 1}$, for ${\bf 1} =\{1 \}_{i\in \I}$.

\begin{proposition}
  \label{prop:QG1}
Let $\ba=(a_i)_{i\in \I} \in  ( \bF^\times )^\I$. We have an automorphism $\tPsi_{\ba}$ on the $\bF$-algebra $\tU$ such that
\begin{align}
\label{tPsi}
\tPsi_{\ba}&: K_i\mapsto a_{i}^{1/2}K_i,\quad  K'_i\mapsto a_{i }^{1/2} K'_i,\quad  E_i\mapsto a_{i}^{1/2} E_i, \quad  F_i\mapsto F_i.
\end{align}
We have an automorphism $\Phi_\ba$ on the $\bF$-algebra $\U$ such that
\begin{align}
\label{def:Phi}
\Phi_\ba&: K_i \mapsto K_i, \quad E_i \mapsto a_i^{1/2} E_i,\quad F_i \mapsto a_i^{-1/2} F_i.
\end{align}
\end{proposition}
We have
\begin{equation}\label{eq:braid12}
  \pi_{\ba}=\pi_{\mathbf{1}} \circ  \tPsi_{\ba}.
\end{equation}

A $\Q$-linear operator on a $\bF$-algebra  is {\em anti-linear} if it sends $q^m \mapsto q^{-m}$, for $m\in\Z$.

\begin{proposition}
  \label{prop:QG4}
  {\quad}
\begin{enumerate}
\item
There exists an anti-linear involution $\tpsi$ on $\tU$, which fixes $E_i,F_i$ and swaps $K_i \leftrightarrow K_i'$, for $i\in \I $;
\item
There exists an anti-linear involution on $\U$, also denoted by $\psi$, which fixes $E_i,F_i$ and swaps $K_i \leftrightarrow K_i^{-1}$, for $i\in \I $;
\item
There exists an anti-involution $\sigma$ on $\tU$ which fixes $E_i,F_i$ and swaps $K_i \leftrightarrow K_i'$, for $i\in \I $;
\item
There exists a Chevalley involution $\omega$ on $\tU$ which swaps $E_i$ and $F_i$ and swaps $K_i \leftrightarrow K_i'$, for $i\in \I$.
\end{enumerate}
\end{proposition}

Let $\tU=\bigoplus_{\nu \in \Z\I} \tU_{\nu}$ be the weight decomposition of $\tU$ such that $E_i\in \tU_{\alpha_i},F_i\in \tU_{-\alpha_i}, K_i,K_i'\in \tU_0$. Write $\tU_\nu^+:=\tU_\nu \cap \tU^+.$

%
%
\subsection{Braid group action on the Drinfeld double $\tU$}
  \label{subsec:BrDouble}

Lusztig introduced braid group symmetries $T'_{i,e},T''_{i,e}$, for $i\in \I$ and $e =\pm 1$, on a quantum group $\U$ \cite[\S 37.1.3]{Lus93}.
Analogous braid group symmetries $\tTD'_{i,e},\tTD''_{i,e}$, for $i\in \I$ and $e =\pm 1$,  exist on the Drinfeld double $\tU$; see \cite[Propositions 6.20--6.21]{LW21b}. (Our notations $\tTD'_{i,e}$, $\tTD''_{i,e}$ here correspond to $\TT'_{i,e}$, $\TT''_{i,e}$ therein.) We recall the formulation of $\tTD''_{i,+1}$ below.

\begin{proposition}
  \cite[Proposition 6.21]{LW21b}
   \label{prop:braid1}
Set $r=-c_{ij}$. There exist an automorphism $\tTD''_{i,+1}$, for $i\in \I$, on $\tU$ such that
\begin{align*}
&\tTD''_{i,+1}(K_j)= K_j K_i^{-c_{ij}}, \qquad \tTD''_{i,+1}(K_j')=  K'_j {K'_i}^{-c_{ij}},\\
&\tTD''_{i,+1}(E_i)=- F_i {K_i'}^{-1},\qquad \tTD''_{i,+1}(F_i)=- K_i^{-1}E_i,\\
&\tTD''_{i,+1}(E_j)= \sum_{s=0}^r (-1)^s q_i^{-s} E_i^{(r-s)} E_j E_i^{(s)},\qquad  j\neq i,\\
&\tTD''_{i,+1}(F_j)= \sum_{s=0}^r (-1)^s q_i^{s} F_i^{(s)} F_j F_i^{(r-s)}, \qquad  j\neq i.
\end{align*}
Moreover, the $\tTD''_{i,+1}$, for $i\in \I$, satisfy the braid relations.
\end{proposition}

We sometimes use the following conventional short notations
\begin{align*}
    \tTD_{i}:= \tTD''_{i,+1},\quad \tTD_{i}^{-1}:= \tTD'_{i,-1},
    \quad
    T_{i}:= T''_{i,+1},\quad T_{i}^{-1}:= T'_{i,-1}.
\end{align*}

Hence, we can define
\begin{align*}  
\tTD_w \equiv \tTD_{w,+1}'' :=\tTD_{i_1}\cdots \tTD_{i_r} \in \Aut (\tU),
\end{align*}
where $w=s_{i_1} \cdots s_{i_r}$ is any reduced expression of $w\in W$.
Similarly, one defines $T_{w}$ for $w\in W$.

The symmetries $\tTD'_{i,e}$ and $\tTD''_{i,e}$, for $i\in \I$, satisfy the following identities in $\tU$ \cite{LW21b} (analogous to  \cite[37.2.4]{Lus93} in $\U$)
\begin{align}
  \label{eq:sTs}
  \begin{split}
  \tTD'_{i,-1} &=\sigma \circ \tTD''_{i,+1} \circ \sigma,
  \\
  \tTD''_{i,-e} = \tpsi \circ \tTD''_{i,+e} \circ \tpsi,
  & \qquad
  \tTD'_{i,+e} = \tpsi \circ \tTD'_{i,-e} \circ \tpsi.
  \end{split}
\end{align}

The automorphism $\tTD''_{i,+1}$ descends to Lusztig's automorphisms $T''_{i,+1}$ on $\U$:
\begin{equation}  \label{eq:braid11}
\pi_{\mathbf{1}}\circ \tTD''_{i,+1} =T''_{i,+1}\circ \pi_{\mathbf{1}}.
\end{equation}


%
\subsection{Satake diagrams and relative Weyl/braid groups}

Given a subset $\bI\subset \I$, denote by $W_{\bullet}$ the parabolic subgroup of $W$ generated by $s_i,i\in \bI.$ Set $\bw$ to be the longest element of $W_\bullet$. Let $\cR_{\bullet}$ be the set of roots which lie in the span of $\alpha_i,i\in \bI.$ Similarly, $\cR_{\bullet}^\vee$ is the set of coroots which lie in the span of $\alpha_i^\vee,i\in \bI.$
Let $\rho_\bullet$ be the half sum of positive roots in the root system $\cR_{\bullet}$, and $\rho_\bullet^\vee$ be the half sum of positive coroots in $\cR_\bullet^\vee$.

An {\em admissible pair} $(\I=\bI \cup \wI,\tau)$ (cf. \cite{BBMR, Ko14}) consists of a partition $\bI\cup \wI$ of $\I$, and a Dynkin diagram involution $\tau$ of $\g$ (where $\tau=\Id$ is allowed) such that
\begin{itemize}
\item[(1)]
 $\bw(\alpha_j) = - \alpha_{\tau j}$ for $j\in \bI$,
\item[(2)]
 If $j\in \wI$ and $\tau j =j$, then $\alpha_j(\rho_{\bullet}^\vee)\in \Z$.
\end{itemize}
The diagrams associated to admissible pairs are known as  Satake diagrams.
We shall use the terms between admissible pairs and  Satake diagrams interchangeably.
Throughout the paper, we shall always work with admissible pairs $(\I=\bI \cup \wI,\tau)$. A symmetric pair $(\g,\theta)$ (of finite type) consists of a semisimple Lie algebra $\g$ and an involution $\theta$ on $\g$; the irreducible symmetric pairs (of finite type) are classified by Satake diagrams.

Given an admissible pair $(\I=\bI \cup \wI,\tau)$, the corresponding involution $\theta$ (acting on the weight lattice) is recovered as
\begin{align}
  \label{eq:theta}
\theta=-\bw \circ \tau.
\end{align}
Set $\wItau$ to be a (fixed) set of representatives of $\tau$-orbits in $\wI$. The (real) rank of a Satake diagram is the cardinality of $\wItau$. We call a Satake diagram $(\I^1=\bI^1 \cup \wI^1,\tau^1)$ a subdiagram of another Satake diagram $(\I=\bI \cup \wI,\tau)$, if
$\wI^1\subset \wI$,
$\bI^1\subset \bI,$
 $\tau^1 |_{\wI^1} =\tau |_{\wI^1}$,
and $\bI^1$ contains all black nodes in $\I$ which lie in the connected components of $\wI^1$ in $\bI \cup \wI^1$.

%

Given an admissible pair $(\I=\bI \cup \wI,\tau)$ and $i\in \wI$, we set
\begin{align}
  \label{eq:Iib}
 \I_{\bullet,i}:=\bI\cup \{i,\tau i\}.
 \end{align}
 Let $W_{\bullet,i}$ be the parabolic subgroup of $W$ generated by $s_i,i\in \I_{\bullet,i}$. Let $w_{\bullet,i}$ the longest element of $W_{\bullet,i}$. The following constructions are a special case of those by Lusztig \cite{Lus76}; also cf. \cite{Lus03, DK19}. Define $\bs_i\in W_{\bullet,i}$ such that
\begin{align}\label{def:bsi}
\bwi= \bs_i w_\bullet \, (=w_\bullet \bs_i),
\qquad \text{ where }
\ell(\bwi) = \ell(\bs_i) + \ell(w_\bullet).
\end{align}
(It follows from the admissible pair requirement that $\bwi,\bs_i,$ and $w_\bullet$ commute with each other.) Then the subgroup of $W$,
\[
\reW :=\langle \bs_i| i\in \wItau \rangle,
\]
is a Weyl group by itself with its simple reflections identified with $\{\bs_i \mid i\in \wItau\}$. Denote by $\ell_{\circ}$ the length function of the Coxeter systerm $(\reW, \wItau)$ and by $\bbw$ its longest element.

\begin{proposition}  \cite[Theorem 5.9]{Lus76} 
 \label{prop:LL}
Let $w_1, w_2 \in \reW$. Then $\ell(w_1 w_2) =\ell(w_1) +\ell(w_2)$ if and only if $\ell_\circ(w_1 w_2) =\ell_\circ (w_1) +\ell_\circ (w_2)$.
\end{proposition}
Hence there is no ambiguity to refer to the Coxeter system $\reW$ or $W$ when we talk about reduced expressions of an element $w\in \reW \subset W$.
By definition, we have identifications $\I_{\bullet,i} =\I_{\bullet,\tau i}, W_{\bullet,i} =W_{\bullet,\tau i}, w_{\bullet,i} =w_{\bullet,\tau i}$, and $\bs_i=\bs_{\tau i} $. Denote by $\tau_{\bullet,i}$ the diagram involution on $\bIi$ such that
\begin{align} \label{def:taui}
\bwi(\alpha_j) =-\alpha_{\tau_{\bullet,i} j}, \qquad \forall j\in \bIi.
\end{align}

The {\em relative Weyl group} associated to the Satake diagram $(\I=\bI \cup \wI,\tau)$ can be identified with $\reW$. Let $\{\balpha_i|i\in \wItau\}$  be the simple system of the relative (or restricted) root system, where $\balpha_i$ is identified with the following element (cf. \cite[\S 2.3]{DK19})
\begin{align}
 \label{def:balpha}
\balpha_i:=\frac{\alpha_i - \theta(\alpha_i)}{2},\qquad (i\in \wI).
\end{align}
Note that $\balpha_i=\balpha_{\tau i}$.

We introduce a subgroup of $W$:
\[
W^\theta = \{w\in W \mid w \theta =\theta w\}.
\]
It is well known that (see, e.g., \cite[\S 2.2]{DK19})
\[
W_\bullet \rtimes \reW \cong W^\theta.
\]

We shall refer to the braid group associated to the relative Weyl group $\reW$ as the {\em relative braid group} and denote it by $\Br(\reW)$. Accordingly, we denote the braid group associated to $\bW$ by $\Br(\bW)$.

\subsection{Universal $\imath$quantum groups}
  \label{sub:iQG}

We set up some basics for the universal quantum symmetric pair $(\tU,\tUi)$, following and somewhat generalizing \cite{LW22}.

Let $(\I=\bI\cup \wI, \tau)$ be a  Satake diagram. Define $\tbU$ to be the subalgebra of $\tU$ with the set of Chevalley generators
\[
\tbcG := \{E_j,F_j,K_j,K_j' \mid j\in \bI\}.
\]
The universal $\imath$quantum group associated to the Satake diagram $(\I=\bI\cup \wI,\tau)$ is defined to be the $\bF$-subalgebra of $\tU$
\[
\tUi=\langle B_i,\tk_i, g \mid i\in \wI, g \in\tbcG\rangle
\]
via the embedding $\imath:\tUi\rightarrow\tU$, $u \mapsto u^\imath$, with
\begin{align}\label{def:gen}
B_i&\mapsto  F_i + \tTD_{\bw}(E_{\tau i}) K'_i,\quad \tk_i \mapsto K_i K'_{\tau i}, \quad g \mapsto g,\quad \text{for }\;  i\in \wI, \; g\in \tbcG.
\end{align}
(The notation $u^\imath$ for $u\in \tUi$, e.g., $B_i^\imath$, is mainly used when we need to apply braid group operators on $\tU$ to $u^\imath$.)
By definition, $\tUi$ contains the Drinfeld double $\tbU$ associated to $\bI$ as a subalgebra.

Let $\tU^{\imath 0}$ denote the subalgebra of $\tUi$ generated by $\tk_i, K_j,K_j',$ for $i\in \wI, j\in \bI$. The following lemma is clear.

\begin{lemma}
If $i=\tau i,i\in \wI$, then $\tk_i$ is central in $\tUi$. If $\tau i \neq i\in \wI$, then $\tk_i \tk_{\tau i}$ is central in $\tUi$.
\end{lemma}

Following \cite{Let99} and \cite[\S 6.2]{Ko14}, we formulate a monomial basis for $\tUi$.  Denote $B_j =F_j$, for $j\in \bI$. For a multi-index $J=(j_1,j_2,\cdots,j_n)\in \I^n$, we define $F_J:=F_{j_1}F_{j_2}\cdots F_{j_n}$ and $B_J:=B_{j_1}B_{j_2}\cdots B_{j_n}$. Let $\cJ$ be a fixed subset of $\bigcup_{n\geq 0} \I^n$ such that $\{F_J| J\in \cJ\}$ forms a basis of $\tU $ as a $\tU^+ \tU^0$-module.

\begin{proposition}  (cf. \cite[Proposition 6.2]{Ko14})
\label{prop:QG5}
The set $\{B_J| J\in \cJ \}$ is a basis of the left (or right) $\tbU^+ \tU^{\imath 0}$-modules $\tUi$.
\end{proposition}


\subsection{$\imath$Quantum group $\Ui_\bvs$ via central reduction}

We recall some basics for quantum symmetric pairs $(\U, \Ui_\bvs)$, cf. \cite{Let99, Ko14}, where the parameter $\bvs=(\vs_i)_{i\in \wI}\in {\mathbb F}^{\times, \wI}$ is always assumed to satisfy the following conditions (cf. \cite{Let99} \cite[Section 5.1]{Ko14})
\begin{align}
   \label{def:par}
\vs_i&=\vs_{\tau i} , \qquad \text{ if } \tau i\neq i \text{ and } (\alpha_i, \bw \alpha_{\tau i})=0.
\end{align}
We call $\bvs$ a {\em balanced parameter}, if $\vs_i=\vs_{\tau i}$ for any $i\in \wI$. For an arbitrary parameter $\bvs$, we define an associated balanced parameter $\bvs^e$ such that
\begin{align}
  \label{def:bvse}
\vs^e_i=\vs^e_{\tau i}= \sqrt{\vs_i \vs_{\tau i}}.
\end{align}

 Define $\bU$ to be the subalgebra of $\U$ with the set of Chevalley generators
\[
\bcG := \{E_j,F_j,K_j^{\pm 1} \mid j\in \bI\}.
\]
The $\imath$quantum group associated to the Satake diagram $(\I=\bI\cup \wI,\tau)$ with parameter $\bvs$ is defined to be the $\bF$-subalgebra of $\U$
\[
\Ui_\bvs =\langle B_i,k_j , g \mid i\in \wI, j \in \I \setminus \wItau, g \in\bcG\rangle
\]
via the embedding $\imath:\Ui_\bvs \rightarrow \U$ with
\begin{align}
 \label{eq:Uibvs}
B_i & \mapsto F_i + \vs_i T_{\bw}(E_{\tau i}) K^{-1}_i,\qquad k_j \mapsto K_j K_{\tau j}^{-1}, \qquad \forall i\in \wI,j\in \I\setminus \wItau.
\end{align}
Note that $\Ui$ contains $\bU$ as a subalgebra. For $i\in \wItau$, we set $k_i=1$ if $i=\tau i$ and $k_i=k_{\tau i}^{-1}$ if $i\neq \tau i$.
Similarly, we denote by $\U^{\imath 0}$ the subalgebra of $\Ui$ generated by $k_i, K_j$, for $i\in \wI, j\in \bI$.


Recall from \eqref{def:balpha} that $\balpha_i =(\alpha_i+ \bw\alpha_{\tau i} )/{2}$. Define a distinguished balanced parameter $\bvs_{\diamond}=(\vs_{i,\diamond})_{i\in \wI}$ such that
\begin{align}
   \label{def:vsi}
 \vs_{i,\dm} 
 = - q^{-(\alpha_i , \alpha_i+\bw\alpha_{\tau i})/2} =-q^{- (\balpha_i, \balpha_i)},
 \quad \text{for }i \in \wI.
\end{align}
The parameter $\bvs_\dm$ will play a basic role in this paper; also cf. \cite{DK19}.

\begin{table}[h]
\caption{Rank 1 Satake diagrams and local datum}
\label{table:localSatake}
\resizebox{5.5 in}{!}{%
\begin{tabular}{| c | c | c | c |}
\hline
Type & Satake diagram & $\vs_{i,\dm}$ & $\bs_i$
\\
\hline
\begin{tikzpicture}[baseline=0]
\node at (0, -0.15) {AI$_1$};
\end{tikzpicture}

&	
    \begin{tikzpicture}[baseline=0]
		\node  at (0,0) {$\circ$};
		\node  at (0,-.3) {\small 1};
	\end{tikzpicture}
& $\vs_{1,\dm}=-q^{-2}$
& $\bs_1=s_1$
\\
\hline
\begin{tikzpicture}[baseline=0]
\node at (0, -0.15) {AII$_3$};
\end{tikzpicture}
&
   \begin{tikzpicture}[baseline=0]
		\node at (0,0) {$\bullet$};
		\draw (0.1, 0) to (0.4,0);
		\node  at (0.5,0) {$\circ$};
		\draw (0.6, 0) to (0.9,0);
		\node at (1,0) {$\bullet$};
		\node at (0,-.3) {\small 1};
		\node  at (0.5,-.3) {\small 2};
		\node at (1,-.3) {\small 3};
	\end{tikzpicture}
& $\vs_{2,\dm}=-q^{-1}$
& $\bs_2=s_{2132}$
\\
\hline
\begin{tikzpicture}[baseline=0]
\node at (0, -0.2) {AIII$_{11}$};
\end{tikzpicture}
 &
\begin{tikzpicture}[baseline = 6] 
		\node at (-0.5,0) {$\circ$};
		\node at (0.5,0) {$\circ$};
		\draw[bend left, <->] (-0.5, 0.2) to (0.5, 0.2);
		\node at (-0.5,-0.3) {\small 1};
		\node at (0.5,-0.3) {\small 2};
	\end{tikzpicture}
& $\vs_{1,\dm}=-q^{-1}$
& $\bs_1=s_1 s_2$
\\
\hline
\begin{tikzpicture}[baseline=0]
\node at (0, -0.2) {AIV, n$\geq$2};
\end{tikzpicture}
&
\begin{tikzpicture}	[baseline=6]
		\node at (-0.5,0) {$\circ$};
		\draw[-] (-0.4,0) to (-0.1, 0);
		\node  at (0,0) {$\bullet$};
		\node at (2,0) {$\bullet$};
		\node at (2.5,0) {$\circ$};
		\draw[-] (0.1, 0) to (0.5,0);
		\draw[dashed] (0.5,0) to (1.4,0);
		\draw[-] (1.6,0)  to (1.9,0);
		\draw[-] (2.1,0) to (2.4,0);
		\draw[bend left, <->] (-0.5, 0.2) to (2.5, 0.2);
		\node at (-0.5,-.3) {\small 1};
		\node  at (0,-.3) {\small 2};
		\node at (2.5,-.3) {\small n};
	\end{tikzpicture}
& $\vs_{1,\dm}=-q^{-1/2}$
& $\bs_1=s_{1 \cdots  n \cdots 1} $
\\
\hline
\begin{tikzpicture}[baseline=0]
\node at (0, -0.2) {BII, n$\ge$ 2};
\end{tikzpicture} &
    	\begin{tikzpicture}[baseline=0, scale=1.5]
		\node at (1.05,0) {$\circ$};
		\node at (1.5,0) {$\bullet$};
		\draw[-] (1.1,0)  to (1.4,0);
		\draw[-] (1.4,0) to (1.9, 0);
		\draw[dashed] (1.9,0) to (2.7,0);
		\draw[-] (2.7,0) to (2.9, 0);
		\node at (3,0) {$\bullet$};
		\draw[-implies, double equal sign distance]  (3.05, 0) to (3.55, 0);
		\node at (3.6,0) {$\bullet$};
		\node at (1,-.2) {\small 1};
		\node at (1.5,-.2) {\small 2};
		\node at (3.6,-.2) {\small n};
	\end{tikzpicture}	
& $\vs_{1,\dm}=-q_1^{-1}$
& $ \bs_1=s_{1\cdots n \cdots 1}$
\\
\hline
\begin{tikzpicture}[baseline=0]
\node at (0, -0.15) {CII, n$\ge$3};
\end{tikzpicture}
&
		\begin{tikzpicture}[baseline=6]
		\draw (0.6, 0.15) to (0.9, 0.15);
		\node  at (0.5,0.15) {$\bullet$};
		\node at (1,0.15) {$\circ$};
		\node at (1.5,0.15) {$\bullet$};
		\draw[-] (1.1,0.15)  to (1.4,0.15);
		\draw[-] (1.4,0.15) to (1.9, 0.15);
		\draw (1.9, 0.15) to (2.1, 0.15);
		\draw[dashed] (2.1,0.15) to (2.7,0.15);
		\draw[-] (2.7,0.15) to (2.9, 0.15);
		\node at (3,0.15) {$\bullet$};
		\draw[implies-, double equal sign distance]  (3.1, 0.15) to (3.7, 0.15);
		\node at (3.8,0.15) {$\bullet$};
		\node  at (0.5,-0.15) {\small 1};
		\node at (1,-0.15) {\small 2};
		\node at (3.8,-0.15) {\small n};
	\end{tikzpicture}
& $\vs_{2,\dm}=-q_2^{-1/2}$
&	$\bs_2=s_{2\cdots n\cdots 2 1 2\cdots n\cdots 2} $	
\\
\hline
\begin{tikzpicture}[baseline=0]
\node at (0, -0.05) {DII, n$\ge$4};
\end{tikzpicture}&
	\begin{tikzpicture}[baseline=0]
		\node at (1,0) {$\circ$};
		\node at (1.5,0) {$\bullet$};
		\draw[-] (1.1,0)  to (1.4,0);
		\draw[-] (1.4,0) to (1.9, 0);
		\draw[dashed] (1.9,0) to (2.7,0);
		\draw[-] (2.7,0) to (2.9, 0);
		\node at (3,0) {$\bullet$};
		\node at (3.5, 0.4) {$\bullet$};
		\node at (3.5, -0.4) {$\bullet$};
		\draw (3.05, 0.05) to (3.4, 0.39);
		\draw (3.05, -0.05) to (3.4, -0.39);
		\node at (1,-.3) {\small 1};
		\node at (1.5,-.3) {\small 2};
		\node at (3.6, 0.25) {\small n-1};
		\node at (3.5, -0.6) {\small n};
	\end{tikzpicture}		
& $\vs_{1,\dm}=-q^{-1}$
&
\begin{tikzpicture}[baseline=0]
\node at (0, 0.15) { $\bs_1=$};
\node at (0, -0.35) { $ s_{1 \cdots n-2 \cdot n-1 \cdot n \cdot n-2  \cdots 1}$};
\end{tikzpicture}
\\
\hline
\begin{tikzpicture}[baseline=0]
\node at (0, -0.2) {FII};
\end{tikzpicture}
&
\begin{tikzpicture}[baseline=0][scale=1.5]
	\node at (0,0) {$\bullet$};
	\draw (0.1, 0) to (0.4,0);
	\node at (0.5,0) {$\bullet$};
	\draw[-implies, double equal sign distance]  (0.6, 0) to (1.2,0);
	\node at (1.3,0) {$\bullet$};
	\draw (1.4, 0) to (1.7,0);
	\node at (1.8,0) {$\circ$};
	\node at (0,-.3) {\small 1};
	\node at (0.5,-.3) {\small 2};
	\node at (1.3,-.3) {\small 3};
	\node at (1.8,-.3) {\small 4};
\end{tikzpicture}
& $\vs_{4,\dm}= -q_4^{-1/2}$
& $\bs_4=s_{432312343231234}$
\\
\hline
\end{tabular}
}
\newline
\smallskip
\end{table}

Letzter \cite{Let02} and Kolb \cite[Proposition 9.2, Theorem 9.7]{Ko14} raised and addressed the question on when $\imath$quantum groups for different parameters are related by Hopf algebra automorphisms of $\tU$. Watanabe \cite[Lemma ~2.5.1]{W21a} showed that the $\imath$quantum groups for {\em arbitrarily} different parameters are all isomorphic (not necessarily by Hopf algebra automorphisms); we recall the following special case of Watanabe's result.

\begin{proposition}
\cite[Lemma 2.5.1]{W21a}
  \label{prop:QG3}
For any parameter $\bvs$, there exists an algebra isomorphism $\phi_\bvs:\Ui_{\bvs_\diamond} \rightarrow \Ui_{\bvs}$ which sends $B_i\mapsto \sqrt{\vs_{i,\dm}(\vs_i \vs_{\tau i})^{-1/2}}B_i, E_j\mapsto E_j, F_j \mapsto F_j, K_j\mapsto K_j,k_r\mapsto \sqrt{\vs_{r}^{-1}\vs_{\tau r}}  k_r,i\in \wI,j\in \bI,r\in \I\setminus\wItau$.
\end{proposition}
It follows that there is an algebra isomorphism \begin{align} \label{phiphi}
    \phi_{\bvs}\phi_{\bvs^e}^{-1}: \Ui_{\bvs^e} \longrightarrow \Ui_{\bvs}
\end{align}
which sends $B_i\mapsto B_i, E_j\mapsto E_j, F_j \mapsto F_j, K_j\mapsto K_j, k_r\mapsto \sqrt{\vs_{r}^{-1}\vs_{\tau r}}  k_r$, for $i\in \wI,j\in \bI,r\in \I\setminus\wItau$.

We have the following central reduction $\pi_{\bvs}^\imath: \tUi \rightarrow \Ui_\bvs $, generalizing \cite[Proposition 6.2]{LW22} in the quasi-split setting.

\begin{proposition}
  \label{prop:QG2}
There exists a quotient morphism $\pi_{\bvs}^\imath:\tUi\rightarrow \Ui_{\bvs}$ sending
\begin{align*}
B_i \mapsto B_i, \quad \tk_j \mapsto \vs_{\tau j} k_j, \quad \tk_{\tau j} \mapsto \vs_{j} k_{\tau j},\qquad (i\in \wI, j\in \wItau),
\end{align*}
and $\pi_{\bvs}^\imath|_{\tbU}=\pi_{\bf 1} |_{\tbU}$. The kernel of $\pi_{\bvs}^\imath$ is generated by
\begin{align*}
& \tk_i - \vs_i  \quad (i=\tau i,i\in \wI), \qquad \tk_i \tk_{\tau i}-\vs_i\vs_{\tau i} \quad (i\neq \tau i,i\in \wI), \quad K_j K_j'-1\quad(j\in \bI).
\end{align*}
\end{proposition}

\begin{proof}
By \eqref{pibvs}, the restriction of $\pi_{\bvs}$ on $\tUi$ sends 
\begin{align*}
&B_i \mapsto F_i + \sqrt{\vs_{i}\vs_{\tau i}} T_{\bw}(E_{\tau i}) K_i^{-1},
\qquad 
\tk_i \mapsto \sqrt{\vs_{i}\vs_{\tau i}} k_i,
\qquad i\in \wI,
\\
&K_j \mapsto K_j,
\qquad
E_j\mapsto E_j,
\qquad 
F_j \mapsto F_j,
\qquad 
j\in \bI.
\end{align*}
Since the images generate $\Ui_{\bvs^e}$ (see \eqref{def:bvse} for the definition of $\bvs^e$), $\pi_{\bvs}$ restricts to a surjective homomorphism $\tUi \rightarrow \Ui_{\bvs^e}$, and we denote it by $\pi_{\bvs^e}^\imath$. Moreover, we have $\ker \pi_{\bvs^e}^\imath=\ker \pi_{\bvs} \cap \tUi$. Since $\ker \pi_{\bvs}$ is generated by elements $K_i K_i'-\vs_{i}$ for $i\in \I$, we conclude that $\ker \pi_{\bvs^e}^\imath$ is generated by 
\begin{align*}
& \tk_i - \vs_i \quad (i=\tau i,i\in \wI), \qquad \tk_i \tk_{\tau i}-\vs_{i}\vs_{\tau i} \quad (i\neq \tau i,i\in \wI), \quad K_j K_j'-1\quad(j\in \bI).
\end{align*}

The composition with the isomorphism $\phi_{\bvs}\phi_{\bvs^e}^{-1}$ from \eqref{phiphi}, $\pi_{\bvs}^\imath:=\phi_{\bvs}\phi_{\bvs^e}^{-1}\circ \pi_{\bvs^e}^\imath$, defines a surjective homomorphism $\tUi \rightarrow \Ui_\bvs$. Finally, it is clear from Proposition~\ref{prop:QG3} that $\ker \pi_{\bvs}^\imath$ is generated by the desired elements.
\end{proof}

\begin{remark}
For a balanced parameter $\bvs$, $\pi_{\bvs}^\imath$ coincides with the restriction of $\pi_{\bvs}$ on $\tUi$. However, this is not the case for an unbalanced parameter.
\end{remark}

\section{Quasi $K$-matrix and intertwining properties}
\label{sec:quasi K}

In this section, we establish the quasi $K$-matrix $\tfX$ for the universal quantum symmetric pair $(\tU, \tUi)$, and a new characterization of $\tfX$ in terms of an anti-involution $\sigma$.
Then using suitable intertwining properties with the quasi $K$-matrix, we establish an anti-involution $\sigma^\imath$ and a bar involution $\tpsi^\imath$ on $\tUi$ from the anti-involution $\sigma$ and a rescaled bar involution $\tpsi_\star$ on $\tU$. We also establish an anti-involution $\sigma_\tau$ on $\Ui_\bvs$ for an arbitrary parameter $\bvs$.

\subsection{Quasi $K$-matrix}

 The quasi $K$-matrix was introduced in \cite[\S 2.3]{BW18a} as the intertwiner between the embedding $\imath: \Ui_{\bvs} \rightarrow \U$ and its bar-conjugated embedding (where some constraints on $\bvs$ are imposed); this was expected to be valid for general quantum symmetric pairs early on. A proof for the existence of the quasi $K$-matrix was given in \cite{BK19} in greater generality (modulo a technical assumption, which was later removed in \cite{BW21}).
Appel-Vlaar \cite[Theorem 7.4]{AV22} reformulated the definition of quasi $K$-matrix $\fX_{\bvs}$ associated to $(\U,\Ui_{\bvs})$ without reference to the bar involution on $\Ui_\bvs$; this somewhat technical (see \eqref{eq:fX1av}) reformulation removes constraints on the parameter $\bvs$ for quasi $K$-matrix. Recall the bar involution $\psi$ on $\U$. 

 \begin{theorem} (cf. \cite{AV22})
   \label{thm:qK}
 There exists a unique element $\fX_{\bvs}=\sum_{\mu \in \N \I} \fX_{\bvs}^\mu$, for $\fX_{\bvs}^\mu\in \U_\mu^+$, such that $\fX_{\bvs}^0=1$ and the following identities hold:
 \begin{align}\label{eq:fX1av}
 B_i   \fX_{\bvs} &=  \fX_{\bvs}  \Big(F_i+(-1)^{\alpha_i( 2\rho_\bullet^\vee)}q^{(\alpha_i, w_\bullet(\alpha_{\tau i})+2\rho_\bullet)}\vs_{\tau i}\psi\big(T_{w_\bullet}E_{\tau i} \big) K_i \Big),\\
 x  \fX_{\bvs}  &=  \fX_{\bvs}  x,
 \end{align}
 for $i\in \wI$ and $x \in \U^{\imath 0}  \bU$. Moreover, $\tfX^\mu =0$ unless $\theta (\mu) =-\mu$.
 \end{theorem}

Recall the bar involution $\tpsi$ on $\tU$ from Proposition~ \ref{prop:QG4}. The quasi $K$-matrix $\tfX$ associated to $(\tU,\tUi)$ is defined in a similar way as in Theorem~\ref{thm:qK}.

\begin{theorem}\label{thm:qK2}
 There exists a unique element $\tfX=\sum_{\mu \in \N \I} \tfX^{\mu}$ such that $\tfX^0=1, \tfX^{\mu}\in \tU_{\mu}^+$ and the following identities hold:
 \begin{align}
  B_i   \tfX &=  \tfX \Big(F_i+(-1)^{\alpha_i( 2\rho^\vee_\bullet)}q^{(\alpha_i, w_\bullet \alpha_{\tau i} +2\rho_\bullet)} \tpsi \big(\tTD_{w_\bullet}E_{\tau i} \big)K_i \Big),
 \label{eq:fX1} \\
 x  \tfX &= \tfX  x,
 \label{eq:fX1b}
\end{align}
 for $i\in \wI$ and $x \in \tU^{\imath 0}\tbU$. Moreover, $\tfX^\mu =0$ unless $\theta (\mu) =-\mu$.
\end{theorem}

 \begin{proof}
 Follows by a rerun of the proof of Theorem~\ref{thm:qK} as in \cite{AV22} or in \cite{Ko21}. (The strategy of the proof does not differ substantially from the one given in \cite{BW18a}.)
 \end{proof}

\begin{remark}
Applying the central reduction $\pi_\bvs$ in \eqref{pibvs} to \eqref{eq:fX1} gives us
\begin{align}\notag
\big( F_i + &\sqrt{\vs_i \vs_{\tau i}} T_{\bw}(E_{\tau i}) K_i^{-1}\big) \pi_\bvs(\tfX)\\\label{eq:fX5}
&= \pi_\bvs(\tfX) \big(F_i+(-1)^{\alpha_i( 2\rho_\bullet^\vee)}q^{(\alpha_i, w_\bullet(\alpha_{\tau i})+2\rho_\bullet)}\sqrt{\vs_i \vs_{\tau i}} \psi\big(T_{w_\bullet}(E_{\tau i}) \big) K_i\big),\\
 x \pi_\bvs(\tfX)  &= \pi_\bvs(\tfX)  x,
\end{align}
for $i\in \wI,x \in \U^{\imath 0} \bU$. Comparing \eqref{eq:fX5} with \eqref{eq:fX1av}, we obtain by the uniqueness of the quasi $K$-matrix that (see \eqref{def:bvse} for $\bvs^e$)
\begin{align}
  \label{eq:sameUp}
\pi_\bvs(\tfX)=\fX_{\bvs^e}.
\end{align}
 In particular, $\pi_\bvs(\tfX)=\fX_{\bvs}$ if and only if $\bvs$ is a balanced parameter.
\end{remark}

\subsection{A bar involution $\tpsi^\imath$ on $\tUi$}

Introduce a balanced parameter $\bvs_\star=(\vs_{i,\star})_{i\in \wI}$ by letting
\begin{align}
 \label{eq:bvs star}
\vs_{i,\star} =(-1)^{\alpha_i( 2\rho^\vee_\bullet)}q^{(\alpha_i, w_\bullet \alpha_{\tau i}+2\rho_\bullet)},\qquad (i\in \wI).
\end{align}
Note that $\vs_{i,\star}$ are exactly the scalars appearing on the RHS \eqref{eq:fX1}. We extend $\bvs_\star$ trivially to an $\I$-tuple, again denoted by $\bvs_\star$ by abuse of notation, by setting
\begin{align*}
\vs_{j,\star}=1 \qquad (j\in \bI).
\end{align*}

Recall the scaling automorphism $\tPsi_{\bvs_\star}$ from \eqref{tPsi} and the bar involution $\tpsi$ on $\tU$ from Proposition~\ref{prop:QG4}. The composition
\begin{align} \label{eq:psi star}
\tpsi_\star :=\tPsi_{\bvs_\star} \circ \tpsi
\end{align}
is an anti-linear involutive automorphism of $\tU$.

Let $\ad_{y}$ be the operator such that $\ad_{y}(u ):=y u y^{-1}$ for $y$ invertible.
\begin{proposition}
   \label{prop:newb3}
There exists a unique anti-linear involution $\tpsi^\imath$ of $\tUi$ such that
 \begin{align}\label{eq:newb9}
 \tpsi^\imath(B_i) =B_i, \qquad \tpsi^\imath(x) =\tpsi_\star(x),
\qquad
\text{ for } i\in \wI, x \in \tU^{\imath 0} \tbU.
 \end{align}
Moreover, $\tpsi^\imath$ satisfies the following intertwining relation,
 \begin{equation}\label{eq:newb10}
 \tpsi^\imath(x)   \tfX =\tfX  \tpsi_\star(x),
 \qquad
  \text{ for all }x\in \tUi.
 \end{equation}
($\tpsi^\imath$ is called a bar involution on $\tUi.$)
\end{proposition}

\begin{proof}
We follow the same strategy in \cite{Ko21} who established a bar involution on $\Ui_\bvs$ (for suitable $\bvs$) without using a Serre presentation.

By definition of $\tpsi_\star$, we have, for $i\in \wI,x\in \tU^{\imath 0}\tbU,$
\begin{align}
\begin{split}
\tpsi_\star(B_i)&=F_i+(-1)^{\alpha_i( 2\rho^\vee_\bullet)}q^{(\alpha_i, w_\bullet \alpha_{\tau i}+2\rho_\bullet)} \tpsi\big(\tTD_{w_\bullet} E_{\tau i} \big)K_i,\\
\tpsi_\star(x) &\in \tU^{\imath 0} \tbU.
\end{split}
\label{eq:newb4}
\end{align}
The composition $\ad_{\tfX} \circ \tpsi_\star$ is an anti-linear homomorphism from $\tU$ to a completion of $\tU$. Then the image of $\tUi$ under $\ad_{\tfX} \circ \tpsi_\star$ is a subalgebra generated by
 \begin{align*}
 (\ad_{\tfX} \circ \tpsi_\star)(B_i),\quad (\ad_{\tfX} \circ \tpsi_\star)(x), \qquad \text{ for } i\in \wI, x\in \tU^{\imath 0}\tbU.
 \end{align*}
 By Theorem~\ref{thm:qK2} and the identities \eqref{eq:newb4}, we have, for $i\in \wI,x\in \tU^{\imath 0}\tbU,$
 \begin{align}
   \label{eq:newb5}
 (\ad_{\tfX} \circ \tpsi_\star)(B_i)=B_i,\quad (\ad_{\tfX} \circ \tpsi_\star)(x)=\tpsi_\star(x).
 \end{align}
 Since each element in \eqref{eq:newb5} lies in $\tUi$, $\ad_{\tfX} \circ \tpsi_\star$ restricts to an anti-linear endomorphism on $\tUi$, which we shall denote by $\tpsi^\imath: \tUi \rightarrow \tUi$.

 By construction, $\tpsi^\imath$ satisfies \eqref{eq:newb9}--\eqref{eq:newb10}. Finally, $\tpsi^\imath$ is unique and is an involutive automorphism of $\tUi$ since it satisfies \eqref{eq:newb9}.
\end{proof}

\begin{proposition}
\label{prop:psi1}
We have
\begin{align}
\label{eq:psi1}
 \tpsi_\star(\tfX) \tfX=1.
\end{align}
\end{proposition}

\begin{proof}
Applying $\psi_\star$ to \eqref{eq:newb10} results the identity
$\tpsi_\star(y) \tpsi_\star(\tfX )= \tpsi_\star(\tfX)\tpsi^\imath(y)$, for $y\in \tUi$.
We rewrite this identity as
\begin{align}
\label{eq:psi2}
 \tpsi^\imath(y) \tpsi_\star(\tfX )^{-1}= \tpsi_\star(\tfX)^{-1} \tpsi_\star(y).
\end{align}
Using \eqref{eq:newb4} and Proposition~\ref{prop:newb3}, the above identity \eqref{eq:psi2} implies following relations
\begin{align}
\label{eq:psi3}
\begin{split}
 B_i \tpsi_\star(\tfX )^{-1}
 &=
 \tpsi_\star(\tfX)^{-1}\Big(F_i+(-1)^{\alpha_i( 2\rho^\vee_\bullet)}q^{(\alpha_i, w_\bullet \alpha_{\tau i}+2\rho_\bullet)} \tpsi\big(\tTD_{w_\bullet} E_{\tau i} \big)K_i\Big),\\
 x \tpsi_\star(\tfX )^{-1}&= \tpsi_\star(\tfX)^{-1} x,
 \end{split}
\end{align}
for $i\in \wI,x\in \tU^{\imath 0} \tbU$. Hence, $\tpsi_\star(\tfX )^{-1}$ satisfies \eqref{eq:fX1}--\eqref{eq:fX1b} as well. Clearly, $\tpsi_\star(\tfX )^{-1}$ has constant term 1. Thanks to the uniqueness of $\tfX$ in Theorem~\ref{thm:qK2}, we have $\tpsi_\star(\tfX )^{-1}=\tfX$.
\end{proof}

\subsection{Quasi $K$-matrix and anti-involution $\sigma$}

We provide a new characterization for $\tfX$ in terms of the anti-involution $\sigma$ (see Proposition~\ref{prop:QG4}), which turns out to be much cleaner than Theorem~\ref{thm:qK2}. Denote
\begin{align}
 \label{eq:Bsig}
 B_i^\sigma=\sigma(B_i)= F_i + K_i \tTD_{w_\bullet}^{-1}(E_{\tau i}),
\end{align}
where the second identity above follows by noting $\tTD_{w_\bullet}^{-1} =\sigma \tTD_{w_\bullet} \sigma$; see \eqref{eq:sTs}. The following characterization of a quasi $K$-matrix $\tfX$ is valid for $\tUi$ of arbitrary Kac-Moody type.

 \begin{theorem}
  \label{thm:fX1}
  There exists a unique element $\tfX=\sum_{\mu \in \N \I} \tfX^{\mu}$ such that $\tfX^0=1, \tfX^{\mu}\in \tU_{\mu}^+$ and the following intertwining relations hold:
 \begin{align}\label{eq:fX2}
 \begin{split}
 B_i  \tfX &=  \tfX B_i^{\sigma},\qquad (i\in \wI),\\
 x  \tfX &= \tfX x,\qquad (x\in \tU^{\imath 0}\tbU).
 \end{split}
 \end{align}
 Moreover, $\tfX^\mu =0$ unless $\theta (\mu) =-\mu$.
 \end{theorem}

 \begin{proof}
 We show that the identity \eqref{eq:fX2} is equivalent to \eqref{eq:fX1}, for any fixed $i\in \wI$. Since $\tpsi\big(\tTD_{w_\bullet}(E_{\tau i})\big)$ has weight $w_\bullet \alpha_{\tau i}$, the identity \eqref{eq:fX1} is equivalent to
 \begin{align}
   \label{eq:fX3}
 B_i \tfX &=  \tfX \Big(F_i+(-1)^{\alpha_i(2\rho^\vee_\bullet)}q^{(\alpha_i, 2\rho_\bullet)} K_i\tpsi\big(\tTD_{w_\bullet}(E_{\tau i})\big) \Big).
 \end{align}
 Moreover, by \cite[Lemma 4.17]{BW18b} and $\tU^+ =\U^+$, we have
 \begin{equation*}
 (-1)^{\alpha_i(2\rho_\bullet^\vee)} q^{(\alpha_i,2\rho^\vee_\bullet)}\tpsi\big(\tTD_{w_\bullet}(E_{\tau i})\big)=\tTD_{w_\bullet}^{-1} (E_{\tau i}),
 \end{equation*}
 and hence, the identity \eqref{eq:fX3} is equivalent to \eqref{eq:fX2} as desired.
 \end{proof}

 \begin{remark}
 By abuse of notation, we denote again by $\sigma$ the anti-involution on $\U$ which fixes $E_i,F_i$ and sends $K_i\mapsto K_i^{-1}$ for $i\in \I$. For a balanced parameter $\bvs$, we obtain the intertwining relation for $\Ui_\bvs$,
 $ B_i  \fX_{\bvs} =  \fX_{\bvs}  B_i^{\sigma}$ $(i\in \wI),$
 by applying the central reduction $\pi_\bvs$ to \eqref{eq:fX2}, thanks to \eqref{eq:sameUp}. Here $B_i^{\sigma}=\sigma(B_i)=F_i + \vs_i K_i T_{w_\bullet}^{-1}(E_{\tau i})$.

 On the other hand, for (not necessarily balanced) parameter $\bvs$, we have
 \begin{align}\label{eq:fX6}
 B_i  \fX_{\bvs} &=  \fX_{\bvs}  B_{\tau i}^{\sigma \tau}.
 \end{align}
 \end{remark}

Note that the involution $\tau$ induces an involution $\widehat{\tau} \in \Aut (\tU)$ which preserves $\tUi$. For $i\in \wI$, the rank one quasi $K$-matrix
\[
\tfX_i \in \tU^+_{\bIi} (\subset \tU^+)
\]
is defined to be the quasi $K$-matrix associated to the rank one Satake subdiagram $(\bI\cup \{i,\tau i\}, \tau)$; cf. \eqref{eq:Iib}. Clearly, we have $\tfX_i =\tfX_{\tau i}$.

 \begin{proposition}
   \label{prop:inv}
We have $\sigma(\tfX)=\tfX$ and $\widehat{\tau} (\tfX)=\tfX$, In addition, for $i\in \wI$, we have
\begin{align*}
    \sigma(\tfX_i)=\tfX_i,\qquad
    \widehat{\tau} (\tfX_i)=\tfX_i.
\end{align*}
In addition, $\widehat{\tau}_{\bullet,i}(\tfX_i)=\tfX_i.$
\end{proposition}

\begin{proof}
 By applying the anti-involution $\sigma$ to the identities in Theorem~\ref{thm:fX1}, we have
  \begin{align}
 \sigma (\tfX)  B_i^{\sigma} &=  B_i \sigma (\tfX),\qquad (i\in \wI),\\
\sigma (\tfX) y &= y \sigma (\tfX),\qquad (x\in \tU^{\imath 0}\tbU),
 \end{align}
 where $y =\sigma (x) \in \tU^{\imath 0}\tbU$. This means that $\sigma (\tfX)$ satisfies the same characterization in Theorem~\ref{thm:fX1} as $\tfX$, and hence by uniqueness, we have $\sigma(\tfX)=\tfX$.

 Noting that $\sigma \widehat{\tau} =\widehat{\tau} \sigma$ and $\widehat{\tau}$ preserves $\tU^{\imath 0}\tbU$, then the identity $\widehat{\tau} (\tfX)=\tfX$ follows by the same type argument as above.

 The identities $\sigma(\tfX_i)=\tfX_i$ and $\widehat{\tau} (\tfX_i)=\tfX_i$ are immediate by restricting $\sigma$ and $\widehat{\tau}$ to the Drinfeld double associated to rank 1 Satake subdiagram $(\bIi,\bI,\tau|_{\bIi})$.

 According to the rank one Table~\ref{table:localSatake}, $\tau_{\bullet,i} =1$ except in type AIV when $\tau_{\bullet,i}$ coincides with the restriction of $\tau$ to the rank 1 Satake diagram. In either case, we have $\widehat{\tau}_{\bullet,i}(\tfX_i)=\tfX_i$.
\end{proof}

 \begin{remark}
 For balanced parameters $\bvs$, by taking a central reduction $\pi_\bvs$, the property $\tau (\fX_{i,\bvs})=\fX_{i,\bvs}$ remains valid. However, for unbalanced parameters $\bvs$, we do not necessarily have $\tau (\fX_{i,\bvs})=\fX_{i,\bvs}$; instead, we have $\tau (\fX_{i,\bvs})=\fX_{i,\tau\bvs}$, which can be proved by Theorem~\ref{thm:qK}. The property $\fX_{i,\bvs} =\fX_{\tau i,\bvs}$ is true, regardless of balanced or unbalanced parameters.
 \end{remark}

\begin{remark}
\label{rmk:fX2}
It follows by Theorem~\ref{thm:qK2} that the rank one quasi $K$-matrix $\tfX_i$ has the form $\tfX_i=\sum_{m\geq 0} \tfX_{i,m}$, for $\tfX_{i,m}\in \tU_{m(\alpha_i+\bw\alpha_{\tau i})}$.
\end{remark}

 \subsection{An anti-involution $\sigma^\imath$ on $\tUi$}


 Define $\ck_i\in \tUi$ by
\begin{align}
  \label{def:Ki}
\ck_i = K_i K'_{\bw \alpha_{\tau i}},\qquad  \text{ for } i\in \wI.
\end{align}

\begin{lemma}
  \label{lem:rkone1}
Let $i \in \wI$. We have $\ck_i\in \tU^{\imath 0}$.
 \end{lemma}

 \begin{proof}
By definition, the element $\ck_i$ is a product of $\tk_i=K_i K'_{\tau i} \in \tU^{\imath 0}$ and an element in $\tbU^0$, and hence $\ck_i\in \tU^{\imath 0}$. 
\end{proof}

Recall the anti-involution $\sigma$ on $\tU$ from Proposition~\ref{prop:QG4}.

 \begin{proposition}
   \label{prop:newb1}
 There exists a unique anti-involution $\sigma^\imath$ of $\tUi$ such that
 \begin{align}\label{eq:newb1-2}
 \sigma^\imath(B_i) =B_i, \qquad \sigma^\imath(x) =\sigma(x),
 \quad \text{ for } i\in \wI,x\in \tU^{\imath 0}\tbU.
 \end{align}
Moreover, $\sigma^\imath$ satisfies the following intertwining relation:
 \begin{equation}
   \label{eq:newb1}
 \sigma^\imath(x)  \tfX =\tfX  \sigma(x), \qquad \text{ for all } x\in \tUi.
 \end{equation}
 \end{proposition}

 \begin{proof}
Given $x\in \tUi$, an element $\widehat x \in \tUi$ (if it exists) such that $\widehat{x}  \tfX =\tfX  \sigma(x)$ must be unique due to the invertibility of $\tfX$.

{\bf Claim $(*)$.}
Suppose that there exist  $\widehat x, \widehat y \in \tUi$ that $\widehat{x}  \tfX =\tfX  \sigma(x)$ and $\widehat{y}  \tfX =\tfX  \sigma(y)$, for given $x, y \in \tUi$. Then we have
\[
\widehat{y} \widehat{x} \tfX =\tfX  \sigma(xy).
\]
Indeed, the Claim holds since $\widehat{y} \widehat{x} \tfX =\widehat{y} \tfX  \sigma(x) = \tfX  \sigma(y) \sigma(x) =\tfX  \sigma(xy)$.

 Observe that $\sigma$ preserves the subalgebra $\tU^{\imath 0}\tbU$ of $\tUi$. Hence by Theorem~\ref{thm:fX1}, we have $\sigma(x)  \tfX =\tfX  \sigma(x)$, for all $x\in \tU^{\imath 0}\tbU$. By Theorem~\ref{thm:fX1} again, we have $B_i \tfX =\tfX  \sigma(B_i)$, for all $i\in \wI$. Since the assumption for Claim ($*$) holds for a generating set $\tU^{\imath 0}\tbU \cup \{B_i|i\in \wI\}$ of $\tUi$, we conclude by Claim ($*$) that there exists a (unique) $\widehat x\in \tUi$ such that $\widehat{x}  \tfX =\tfX  \sigma(x)$, for any $x\in \tUi$, and moreover, sending $x\mapsto \widehat{x}$ defines an anti-endmorphism of $\tUi$ (which will be denoted by $\sigma^\imath$).

Clearly, by construction $\sigma^\imath$ satisfies \eqref{eq:newb1-2} and the identity \eqref{eq:newb1}. Finally, $\sigma^\imath$ is an involutive anti-automorphism of $\tUi$ since it satisfies \eqref{eq:newb1-2}.
\end{proof}

\begin{remark}
The strategy in establishing a bar involution on $\Ui_\bvs$ without use of a Serre presentations appeared first in \cite{Ko21}. For quasi-split $\imath$quantum groups, i.e., $\bI=\varnothing,$ our $\tpsi^\imath$ coincides with the bar involution in \cite[Lemma 2.4(a)]{CLW23} (see also \cite[Lemma 6.9]{LW21b}). Unlike the proof {\em loc. cit.}, our proofs of Propositions~\ref{prop:newb3} and  \ref{prop:newb1} do not use a Serre presentation of $\tUi$. Hence, the (anti-) involutions $\sigma^\imath$ and $\tpsi^\imath$ are valid for $\tUi$ of arbitrary Kac-Moody type.
\end{remark}

 \subsection{An anti-involution $\sigma_\tau$ on $\Ui_\bvs$}

The anti-involution $\sigma^\imath$ on $\tUi$ in Proposition~\ref{prop:newb1} can descend to an $\imath$quantum group $\Ui_\bvs$, only for any {\em balanced} parameter $\bvs$. It turns out that the anti-involution $\sigma^\imath \tau$ on $\tUi$ can descend to an $\imath$quantum group $\Ui_\bvs$, for an {\em arbitrary} parameter $\bvs$.

\begin{proposition}
 \label{prop:sigmatau}
 Let $\bvs$ be an arbitrary parameter. There exists a unique anti-involution $\sigma_\tau$ of $\Ui_\bvs$ such that
 \begin{align}
 \label{eq:fX8}
 \sigma_\tau(B_i) =B_{\tau i}, \qquad \sigma_\tau(x) =\sigma\tau(x),
 \quad \text{ for } i\in \wI,x\in \U^{\imath 0}\bU.
 \end{align}
Moreover, $\sigma_\tau$ satisfies the following intertwining relation:
 \begin{equation}
 \label{eq:fX9}
 \sigma_\tau(x)  \fX_\bvs =\fX_\bvs  \sigma\tau(x), \qquad \text{ for all } x\in \Ui_\bvs.
 \end{equation}
 \end{proposition}

\begin{proof}
A proof similar to the one for Proposition~\ref{prop:newb1} works here, and we outline it.

 We claim that, for any $x\in \Ui_\bvs$, there exists $\widehat{x}\in \Ui_\bvs$ such that
 \begin{align}
 \label{eq:fX7}
    \widehat{x}\fX_\bvs =\fX_\bvs  \sigma\tau(x).
 \end{align}
As argued in the proof of Proposition~\ref{prop:newb1}, it suffices to show that \eqref{eq:fX7} holds for $x$ in a generating set $\{ B_i|i\in \wI\} \cup\U^{\imath 0}\bU$ of $\Ui_\bvs$. Indeed, by \eqref{eq:fX6}, we have $B_{\tau i}\fX_\bvs =\fX_\bvs  \sigma\tau(B_i)$.
 For $x\in \U^{\imath 0}\bU$, note that $\sigma\tau(x)\in\U^{\imath 0}\bU$, and then by Theorem~\ref{thm:qK}, we have $\sigma\tau(x)\fX_\bvs =\fX_\bvs  \sigma\tau(x)$. This proves \eqref{eq:fX7}.

Now sending $x\mapsto \widehat{x}$ defines an anti-endomorphism $\sigma_\tau$, which satisfies \eqref{eq:fX8} and \eqref{eq:fX9} by construction above. Finally, $\sigma_\tau$ is involutive since it satisfies \eqref{eq:fX8}.
 \end{proof}

\begin{remark}
Our construction of $\sigma_\tau$ generalizes the $\sigma_\imath$ in \cite[Proposition 3.13]{BW21}, which is constructed via bar involutions under certain restrictions on parameters.
\end{remark}

Thanks to Proposition~\ref{prop:sigmatau}, we have a conceptual formulation of the quasi K-matrix $\fX_\bvs$ for $\Ui_\bvs$ below, which is a variant of Theorem~\ref{thm:fX1}; compare Theorem~\ref{thm:qK} (see \cite{AV22}). This new formulation can also be proved directly.

\begin{theorem}
 \label{thm:quasiKUi}
 Let $\bvs$ be an arbitrary parameter. There exists a unique element $\fX_\bvs=\sum_{\mu \in \N \I} \fX^{\mu}$ such that $\fX^0=1, \fX^{\mu}\in \tU_{\mu}^+$ and the following intertwining relations hold:
\begin{align*}
 B_{\tau i}  \tfX &=  \tfX \sigma\tau (B_i),\qquad (i\in \wI),\\
 x  \tfX &= \tfX x,\qquad (x\in \U^{\imath 0}\bU).
\end{align*}
\end{theorem}

\section{New symmetries $\tTa{i}$ on $\tUi$}
   \label{sec:symmetry}

In this section, we define explicitly certain rescaled braid group actions $\tT'_{j,-1}$ on a Drinfeld double $\tU$. We then formulate the new symmetries $\tTa{i}$ on $\tUi$, for $i\in \wI$,  via an intertwining property using the quasi $K$-matrix $\tfX$ and a rescaled braid automorphism $\tT'_{\bs_i,-1}$; the proof will be completed in the coming sections.
We show that $\tT'_{\bs_i,-1}$ on $\tU$ preserves the subalgebra $\tU^{\imath 0} \tbU$, and that the actions of $\tTa{i}$ and $\tT'_{\bs_i,-1}$ on $\tU^{\imath 0} \tbU$ coincide. Explicit formulas for the action of $\tTa{i}$ on $\tU^{\imath 0} \tbU$ are presented. Then we obtain a compact close rank one formula for $\tTa{i}(B_i)$.

\subsection{Rescaled braid group action on $\tU$}
  \label{Double}

Recall the distinguished parameter $\bvs_{\diamond}$ from \eqref{def:vsi}. Extend $\bvs_\diamond$ trivially to an $\I$-tuple of scalars $(\vs_{i,\diamond})_{i\in \I} $ by setting
\begin{align}  \label{def:vsi1}
\vs_{j,\diamond}=1, \qquad \text{ for } j\in \bI.
\end{align}
Then we have the scaling automorphism $ \tPsi_{\bvs_\diamond}$ on $\tU$ by Proposition~\ref{prop:QG1}. We define symmetries $\tT''_{i,+1}$ and $\tT'_{i,-1}$ on $\tU$ by rescaling $ \tTD''_{i,+1}$ and $\tTD'_{i,-1}$ in Proposition~\ref{prop:braid1} and \eqref{eq:sTs} via the rescaling automorphism $ \tPsi_{\bvs_\diamond}$:
\begin{align}
  \label{def:tT}
& \tT_{i,+1}'' := \tPsi_{\bvs_\diamond}^{-1} \circ \tTD''_{i,+1} \circ \tPsi_{\bvs_\diamond}
\\
\label{def:tT-1}
 & \tT_{i,-1}' := \tPsi_{\bvs_\diamond}^{-1} \circ \tTD'_{i,-1} \circ \tPsi_{\bvs_\diamond}.
\end{align}

 Since $\tTD''_{i,+1},\tTD'_{i,-1}$ are mutually inverses, $\tT''_{i,+1},\tT'_{i,-1}$ are also mutually inverses. We shall often use the shorthand notation
 \begin{align} \label{TTshort}
 \tT_i =\tT_{i,+1}'',
 \qquad
 \tT_i^{-1} =\tT_{i,-1}'.
 \end{align}

\begin{remark}
These rescaled symmetries $\tT_{i}^{-1}$ will play a central role in our construction of symmetries on $\tUi$; see Theorem~\ref{thm:newb0}. Our rescaling twist using $ \tPsi_{\bvs_\diamond}$ is compatible with the rescaling twist in \cite[(3.45), Remark 3.16]{DK19}.
\end{remark}

We write down the explicit actions for $\tT_i$ and $\tT_i^{-1}$ for later use.

\begin{proposition}
  \label{prop:braid0}
Set $r=-c_{ij}$, for $i, j \in \I$. The automorphism $\tT_i \in \Aut (\tU)$  defined in \eqref{def:tT} is given by
\begin{align*}
&\tT_i(K_j)= \vs_{i,\diamond}^{c_{ij}/2} K_j K_i^{-c_{ij}}, \qquad \tT_i(K_j')=\vs_{i,\diamond}^{c_{ij}/2}  K'_j {K'_i}^{-c_{ij}},\\
 &\tT_i(E_i)=-\vs_{i,\diamond}F_i {K_i'}^{-1},\quad \tT_i(F_i)=- K_i^{-1}E_i,\\
&\tT_i(E_j)=\vs_{i,\diamond}^{-r/2}\sum_{s=0}^r (-1)^s q_i^{-s} E_i^{(r-s)} E_j E_i^{(s)},\qquad j\neq i,\\
&\tT_i(F_j)= \sum_{s=0}^r (-1)^s q_i^{s} F_i^{(s)} F_j F_i^{(r-s)}, \qquad j\neq i.
\end{align*}
The inverse of $\tT_i$ (see \eqref{def:tT-1}) is given by
\begin{align*}
&\tT_i^{-1}(K_j)=\vs_{i,\diamond}^{c_{ij}/2} K_j K_i^{-c_{ij}}, \qquad \tT_i^{-1}(K'_j)=\vs_{i,\diamond}^{c_{ij}/2} K'_j {K_i'}^{-c_{ij}},\\
&\tT_i^{-1}(E_i)=-\vs_{i,\diamond}K_i^{-1}F_i, \qquad \tT_i^{-1}(F_i)=- E_i {K'_i}^{-1},\\
&\tT_i^{-1}(E_j)=\vs_{i,\diamond}^{-r/2}\sum_{s=0}^r (-1)^s q_i^{-s} E_i^{(s)} E_j E_i^{(r-s)},\qquad j\neq i,\\
&\tT_i^{-1}(F_j)= \sum_{s=0}^r (-1)^s q_i^{s} F_i^{(r-s)} F_j F_i^{(s)}, \qquad j\neq i.
\end{align*}
Moreover, $\tT_i$, for $i\in \I$, satisfy the braid group relations.
\end{proposition}

Hence, we obtain
\begin{align}
  \label{eq:Tw}
\tT_w = \tT_{w,+1}'' :=\tT_{i_1}\cdots \tT_{i_r} \in \Aut (\tU), \quad \text{ for } w\in W,
\end{align}
where $w=s_{i_1} \cdots s_{i_r}$ is any reduced expression. Similarly, we have $\tT_{w,-1}'  \in \Aut (\tU)$.

\begin{remark}
 \label{rem:sameT}
Let $i \in \bI$. The rescaling for $\tT_i^{\pm 1}$ is trivial, thanks to $\vs_{\diamond,i}=1$; that is, $\tT_i =\tTD_i$. In particular, $\tT_{w_\bullet} =\tTD_{w_\bullet}$. Moreover, $\tTD_{\bw}(E_{\tau i}) =\tT_{\bw}(E_{\tau i}) =T_{\bw}(E_{\tau i})$ in $\tU^+ =\U^+$; cf. the formula for $B_i$ in \eqref{def:gen}.
\end{remark}

Let $\tau_0$ be the diagram automorphism associated to the longest element $w_0$ of the Weyl group $W$. The following fact is well known (up to the rescaling via $\bvs_{\diamond}$); cf., e.g., \cite[Lemma 3.4]{Ko14}.

\begin{lemma}
  \label{lem:braid1}
We have, for $j\in \I,$
\begin{align*}
\tT_{w_0}(F_j)&=-K_{\tau_0 j}^{-1} E_{\tau_0 j}, \qquad \qquad \tT_{w_0}(E_j)=-\vs_{j,\diamond} F_{\tau_0 j} K_{\tau_0 j}'^{-1},\\
 \tT_{w_0}^{-1}(E_j)&= -\vs_{j,\diamond}  K_{\tau_0 j}^{-1} F_{\tau_0 j},\qquad\;\;\; \tT_{w_0}^{-1}(F_j)=-E_{\tau_0 j} K_{\tau_0 j}^{'-1}.
\end{align*}
\end{lemma}

\subsection{Symmetries $\tT''_{j,+1}$, for $j\in \bI$}

It is known \cite{BW18b} that Lusztig's operators $T'_{j,\pm 1},T''_{j,\pm 1}$ on $\U$, for $j\in \bI$, restrict to automorphisms of $\Ui_{\bvs}$ (where the $\bvs$ satisfies certain constraints); moreover, these operators fix $\Upsilon$. In this subsection, we formulate analogous statements for the universal quantum symmetric pair $(\tU,\tUi)$ while skipping the identical proofs.

Recall the automorphisms $\tTD''_{i,+1}$ on the Drinfeld double $\tU$, for $i\in \I$, from Proposition~\ref{prop:braid1}, and recall Remark~\ref{rem:sameT}.

\begin{proposition}[\text{cf. \cite[Theorem 4.2]{BW18b}}]
  \label{prop:Tjblack}
Let $j\in \bI$. The automorphism $\tT''_{j,+1} =\tTD''_{j,+1}$ on $\tU$ restricts to an automorphism of $\tUi$. Moreover, the action of $\tT''_{j,+1}$ on $B_i \; (i\in \wI)$ is given by
\begin{align}\label{eq:newb-1}
\begin{split}
\tT''_{j,+1}(B_i)=\sum_{s=0}^r (-1)^s q_j^s F_j^{(s)} B_i F_j^{(r-s)},\qquad \text{ for } r=-c_{ij}.
\end{split}
\end{align}
\end{proposition}

\begin{proposition}[\text{cf. \cite[Proposition 4.13]{BW18b}}]
  \label{prop:newb-1}
Let $j\in \bI$. Then  $\tT''_{j,+1}(\tfX)=\tfX$, and $\tT''_{j,+1}(\tfX_i)=\tfX_i$, for $i\in \wI$.
\end{proposition}

\subsection{Characterization of $\tTa{i}$}

Let $(\tU,\tUi)$ be the quantum symmetric pair associated to an arbitrary Satake diagram $(\I=\bI \cup \wI,\tau)$. Recall that  $\tfX_i$, for $i\in \wI$, are the quasi $K$-matrix associated to the rank one Satake subdiagram $(\bI\cup \{i,\tau i\},\tau|_{\bI\cup \{i,\tau i\}})$. Recall $\bs_i \in W$ from \eqref{def:bsi} and $\tT_{\bs_i,-1}' \in \Aut(\tU)$ from \eqref{eq:Tw} whose definition uses \eqref{def:tT}. We now formulate our first main result.

\begin{theorem}
 \label{thm:newb0}
Let $i\in \wI$.
\begin{enumerate}
\item
For any $x \in \tUi$, there is a unique element $\hat x \in \tUi$ such that $\hat x \tfX_i =\tfX_i \tT'_{\bs_i,-1}(x^\imath)$.
\item
The map $x \mapsto \hat x$ is an automorphism of the algebra $\tUi$, denoted by $\tTa{i}$.
\end{enumerate}
\end{theorem}

Therefore, we have
\begin{align}
   \label{eq:newb0}
\tTa{i}(x) \tfX_i
=\tfX_i \tT_{\bs_i,-1}'(x^\imath), \qquad \text{ for all } x\in \tUi.
\end{align}

\begin{proof}
A complete proof of this theorem requires the developments in the coming Sections~\ref{sec:symmetry}--\ref{sec:Tb}. Let us outline the main steps below.

For a given $x\in \tUi$, the element $\hat{x} \in \tUi$ satisfying the identity in (1) is clearly unique (if it exists) since $\tfX_i$ is invertible.

The explicit formulas of $\hat{x}$ associated to generators $x$ of $\tUi$, for each of (rank one and two) Satake diagrams, are given in the forthcoming Sections~\ref{sec:symmetry}--\ref{sec:rktwo}. The formulas therein show manifestly that $\hat{x} \in \tUi$; see Proposition~\ref{prop:Cartanblack} on $\tU^{\imath 0} \tbU$, Theorem~\ref{thm:rkone1} for rank 1, and
Theorem~\ref{thm:rktwo1} for rank 2.

Assume that $\hat{x}, \hat{y} \in \tUi$ satisfy (1), for $x,y \in \tUi$; that is, $\hat{x} \tfX_i =\tfX_i \tT_{\bs_i,-1}'(x^\imath)$, and $y' \tfX_i =\tfX_i \tT_{\bs_i,-1}'(y^\imath)$. Then it follows readily that $\hat{x}\hat{y} \in \tUi$ satisfies the identity in (1) for $xy$; that is, $\hat{x}\hat{y} \tfX_i
=\tfX_i \tT_{\bs_i,-1}'((xy)^\imath)$. Hence we have obtained a well-defined endomorphism $\tTa{i}$ on $\tUi$ which sends $x \mapsto \hat{x}$.

To complete the proof of the theorem, it remains to show that $\tTT_{i,-1}'$ is surjective. To this end, we introduce and study in depth a variant of $\tTa{i}$, a second endomorphism $\tTb{i}$ on $\tUi$ in Section~\ref{sec:Tb}. The bijectivity of $\tTa{i}$ follows by Theorem~\ref{thm:newb1} which shows that $\tTa{i}$ and $\tTb{i}$ are mutual inverses.
\end{proof}

\begin{remark}
 \label{rmk:newb0}
By Proposition~\ref{prop:inv} and the definition \eqref{def:bsi} of $\bs_i$, we have $\tfX_i =\tfX_{\tau i}$, $\bs_i=\bs_{\tau i}$, and hence $\tTa{i}= \tTa{\tau i}$. Thus, we may label $\tTa{i}$ by $\wItau$ instead of $\wI.$
\end{remark}



In this and later sections, we shall construct 4 variants of symmetries of $\tUi$ (denoted by $\tTae{i}$, $\tTbe{i}$) via \eqref{eq:newb0} and 3 additional intertwining relations and the rescaled braid group symmetries $\tT'_{\bs_i,\pm 1},\tT''_{\bs_i,\pm 1}$ of $\tUi$. We choose to start with the (simplest) intertwining relation \eqref{eq:newb0} for $\tT'_{\bs_i,-1}$. From now on, following \eqref{TTshort}, we often write
\[
\tT_{\bs_i}^{-1} =\tT'_{\bs_i,-1},
\qquad
\tT_{\bs_i} =\tT''_{\bs_i,+1}.
\]

\subsection{Quantum symmetric pairs of diagonal type}
\label{sec:diag}

Recall from Proposition~\ref{prop:QG4} the Chevalley involution $\omega$  and the comultiplication $\Delta$ \eqref{eq:Delta} on $\tU$. Denote ${}^\omega{\bf L}''_i:=(\omega\otimes 1){\bf L}''_i $ for $i\in \I$, where ${\bf L}''_i,i\in \I$ is the rank one quasi $R$-matrix for $\tU$ (same as for $\U$); see \cite{Lus93}. We regard $\tU$ as a coideal subalgebra of $\tU \otimes \tU$ via the embedding ${}^\omega\Delta:=(\omega\otimes 1)\Delta$ and then $(\tU\otimes \tU,\tU)$ is a universal quantum symmetric pair of diagonal type; cf. \cite[Remark 4.10]{BW18b}.
In this way, the rank one quasi $K$-matrices for quantum symmetric pairs of diagonal type are given by $ {}^\omega{\bf L}''_i $.

In this subsection, we shall reformulate the identity \cite[37.3.2]{Lus93} (= \eqref{eq:TTL})
as an intertwining relation in the framework of quantum symmetric pairs of diagonal type.



\begin{proposition}
For the quantum symmetric pair of diagonal type $(\tU\otimes \tU,\tU)$, the following intertwining relation holds:
\begin{align}
  \label{eq:diag4}
{}^\omega\Delta(\tTD'_{i,-1} u)\;{}^\omega{\bf L}_i''&={}^\omega{\bf L}_i''\;  (\tT''_{i,-1} \otimes  \tT_{i,-1}') \; {}^\omega\Delta (u),\qquad \forall u\in \tU.
\end{align}
\end{proposition}

\begin{proof}
Recall from \cite[37.2.4]{Lus93} that \begin{align}  \label{oTo}
    \omega\circ \tTD_{i,-1}'\circ\omega=\tTD_{i,-1}''. 
\end{align}
The identity \eqref{eq:TTL} for $\U$ admits an identical version for $\tU$. Applying $\omega\otimes 1$ to this identity, we obtain
\begin{align}  \label{eq:TTLomega}
{}^\omega\Delta(\tTD_{i,-1}' u)\; {}^\omega{\bf L}_i''={}^\omega{\bf L}_i''\; (\tTD''_{i,-1} \otimes  \tTD_{i,-1}') \; {}^\omega\Delta (u),\qquad \forall u\in \tU.
\end{align}
To prove \eqref{eq:diag4}, it suffices to prove the following identity
\begin{align}
\label{eq:diag6}
    (\tT''_{j,-1} \otimes \tT'_{j,-1} )\; {}^\omega\Delta(u)
    =(\tTD''_{j,-1} \otimes \tTD'_{j,-1} ) \; {}^\omega\Delta(u),
    \qquad \forall u\in \tU.
\end{align}
Clearly, it suffices to prove \eqref{eq:diag6} when $u$ is the generator of $\tU$. We have the following formulas:
\begin{align}
\label{eq:diag5}
    \begin{split}
        {}^\omega\Delta(E_j)=F_j \otimes 1 + K_j'\otimes E_j,
        &\qquad
        {}^\omega\Delta(F_j)=1\otimes F_j  + E_j\otimes K_j',\\
        {}^\omega\Delta(K_j)=K_j' \otimes K_j,
        &\qquad {}^\omega\Delta(K_j')=K_j \otimes K_j'.
    \end{split}
\end{align}

Recall $\tT'_{j,-1}=\tPsi_{\bvs_{\dm}}^{-1}\tTD'_{j,-1}\tPsi_{\bvs_{\dm}}$ from \eqref{def:tT-1}.
By Lemma \ref{lem:rescale2} and noting that $\bvs_{\star\dm} =\bvs_{\dm}$ in our case, the twisting for $\tT''_{j,-1}$ is opposite to the one on $\tT'_{j,-1}$, i.e.,
$\tT''_{j,-1} =\tPsi_{\bvs_{\dm}}\tTD''_{j,-1}\tPsi_{\bvs_{\dm}}^{-1}.$ By Proposition~\ref{prop:QG1}, we see that the RHS of each formula in \eqref{eq:diag5} is fixed by $\tPsi_{\bvs_{\dm}}^{-1}\otimes\tPsi_{\bvs_{\dm}}$. The formulas for $\tTD''_{i,+1}$ is given in Proposition~\ref{prop:braid1}, and the formulas for $\tTD'_{i,-1},\tTD''_{i,-1}$ can be obtained from there by suitable twisting; using these formulas, we observe that $(\tTD''_{j,-1} \otimes \tTD'_{j,-1} ) \; {}^\omega\Delta(u)$ is fixed by $\tPsi_{\bvs_{\dm}}\otimes\tPsi_{\bvs_{\dm}}^{-1}$ for $u=E_j,F_j,K_j,K_j'$. Hence, for $u=E_j,F_j,K_j,K_j',j\in \I$,
\begin{align*}
(\tT''_{j,-1} \otimes \tT'_{j,-1} )\; {}^\omega\Delta(u)
    &= (\tPsi_{\bvs_{\dm}}\otimes\tPsi_{\bvs_{\dm}}^{-1})
    (\tTD''_{j,-1} \otimes \tTD'_{j,-1})
    (\tPsi_{\bvs_{\dm}}^{-1}\otimes\tPsi_{\bvs_{\dm}}){}^\omega\Delta(u)\\
    &= (\tTD''_{j,-1} \otimes \tTD'_{j,-1} ) \; {}^\omega\Delta(u),
\end{align*}
which implies the desired identity \eqref{eq:diag6}.
\end{proof}

In this way, the intertwining relation \eqref{eq:diag4} (reformulated from \eqref{eq:TTLomega} via \eqref{oTo}) can be viewed as a variant of the intertwining relation \eqref{eq:newb0} in the setting of quantum symmetric pair $(\tU\otimes \tU, \tU)$, where the coideal subalgebra is identified with the image of the embedding ${}^\omega\Delta: \tU \rightarrow \tU \otimes \tU$.




%
%
\subsection{Action of $\tTa{i}$ on $\tU^{\imath 0} \tbU$}
  \label{tTtk}

We formulate $\tTa{i}(x)$, for $i\in \wI,x\in \tU^{\imath 0}\tbU$ in this subsection. We will show that $\tT_{\bs_i}^{-1}$ preserves both $\tU^{\imath 0}$ and $\tbU$; hence, by Theorem~\ref{thm:fX1}, the element $\tTa{i}(x):=\tT_{\bs_i}^{-1}(x)$ satisfies \eqref{eq:newb0} for $x \in \tU^{\imath 0} \tbU$.

Recall that the diagram involution associated to $\bwi$ is denoted by $\tau_{\bullet,i}$. By definition of admissible pairs, the diagram involution associated to $\bw$ is $\tau |_{\bI}$.  Both $\tau_{\bullet,i}$ and $\tau$ induce (commuting) involutive automorphisms, denoted by $\widehat{\tau}_{\bullet,i}$ and $\widehat{\tau}$, on $\tbU$.

We first calculated $\tT_{\bs_i}^{-1}(x)$ for $x\in \tbU$. By applying Lemma~\ref{lem:braid1} twice, we obtain
\begin{align*}
\tT_{\bw}^{-1}\widehat{\tau}(x) = \tT_{\bwi}^{-1}\widehat{\tau}_{\bullet,i}(x)
=\tT_{\bw}^{-1}\tT_{\bs_i}^{-1}\widehat{\tau}_{\bullet,i}(x);
\end{align*}
note that the second identity above holds since $\tT_{\bwi}^{-1} =\tT_{\bw}^{-1} \tT_{\bs_i}^{-1}$ by \eqref{def:bsi}.
Hence, we have $\widehat{\tau}(x) = \tT_{\bs_i}^{-1}\widehat{\tau}_{\bullet,i}(x)$, which implies that
\begin{align}\label{eq:newb8}
\tT_{\bs_i}^{-1}(x)=\widehat{\tau}_{\bullet,i} \circ \widehat{\tau}(x) \in \tbU,
\qquad \text{ for all }x\in \tbU.
\end{align}

We next formulate the actions of $\tT_{\bs_i}^{-1}$ on $\tU^{\imath 0}$, for $i\in \wI$. Recall $\bvs_\diamond$ from \eqref{def:vsi} and \eqref{def:vsi1}, and $\tPsi_{\bvs_\diamond}$ from Proposition~\ref{prop:QG1}. Denote
\begin{align}
   \label{def:tkdm}
\tk_{i,\dm} := \tPsi_{\bvs_\diamond}^{-1}(\tk_i )=\vs_{i,\dm}^{-1} K_i K'_{\tau i} \in \tU^{\imath 0}.
\end{align}
Note that $\tk_{j,\dm}=\tk_j=K_j K_{\tau j}'$, for $j\in \bI$. We shall denote
\begin{align}
  \label{eq:kla}
\tk_{\lambda,\dm}:= \prod_{i\in \I}\tk_{i,\dm}^{m_i}  \in \tU^{\imath 0},
\qquad \text{ for } \lambda=\sum_{i\in \I} m_i \alpha_i \in \Z \I.
\end{align}

\begin{lemma}
  \label{lem:newb1}
Let $ w\in W$ be such that $w \tau= \tau w$. Then $\tT'_{w,-1}(\tk_{j,\dm}) =\tk_{w \alpha_j,\dm}$, for $j\in \wI$.
\end{lemma}

\begin{proof}
By Proposition~\ref{prop:braid1}, we have
\begin{align*}
\tTD'_{w,-1}(\tk_j)=\tTD'_{w,-1}(K_j K'_{ \tau j }) = K_{w\alpha_j} K'_{w\alpha_{\tau j}}= K_{w\alpha_j}K'_{\tau w\alpha_{j}}=\tk_{w\alpha_j}.
\end{align*}
By \eqref{def:tkdm}--\eqref{eq:kla}, we have
$\tk_{\lambda,\dm}= \tPsi_{\bvs_\diamond}^{-1}(\tk_\lambda)$,  for $\lambda\in \Z \I$. By \eqref{def:tT} and \eqref{eq:Tw}, we have $\tT'_{w,-1}=\tPsi_{\bvs_\diamond}^{-1} \circ \tTD'_{w,-1}\circ\tPsi_{\bvs_\diamond}$, and hence
\begin{align*}
\tT'_{w,-1}(\tk_{j,\dm}) = (\tPsi_{\bvs_\diamond}^{-1} \circ \tTD'_{w,-1}) (\tk_j)=\tPsi_{\bvs_\diamond}^{-1}(\tk_{w\alpha_j})=\tk_{w\alpha_j,\dm}.
\end{align*}
The lemma is proved.
\end{proof}

In particular, setting $w=\bs_i$ ($i\in \wI$) in Lemma \ref{lem:newb1} gives us
\[
\tT_{\bs_i}^{-1}(\tk_{j,\dm})=\tk_{\bs_i \alpha_j,\dm}.
\]

Summarizing the above discussion, we have obtained the following.

\begin{proposition}
  \label{prop:Cartanblack}
Let $i \in \wI$. 
There exists element $\tTa{i}(x):=\tT_{\bs_i}^{-1}(x)$, which satisfies the intertwining relation \eqref{eq:newb0}, for $x \in \tU^{\imath 0} \tbU$. More explicitly, we have
\begin{align}
\tTa{i}(u)& =(\widehat{\tau}_{\bullet,i} \circ \widehat{\tau} )(u),
\qquad
\text{ for } u\in \tbU,
  \label{Tiu}\\
\tTa{i}(\tk_{j,\dm})&=\tk_{\bs_i \alpha_j,\dm},
\qquad\qquad
\text{ for }  j\in \wI.
\label{Tik}
\end{align}
\end{proposition}

\subsection{Integrality of $\tTa{i}$}

The formula \eqref{Tiu} clearly preserves the Lusztig integral $\Z[q,q^{-1}]$-form on $\tbU$.
We shall explain below that our braid group action is also integral on the Cartan part, even though the definition \eqref{def:tkdm} of $\tk_{j,\dm}$ may involve $q^{1/2}$.

\begin{lemma} \label{lem:Z}
We have
\begin{align}
\label{eq:cartan1}
\tTa{i}(\tk_{j }) =\vs_{\bs_i\alpha_j-\alpha_j,\dm}^{-1}\tk_{\bs_i \alpha_j },
\end{align}
where $\vs_{\bs_i\alpha_j-\alpha_j,\dm}^{-1} \in \Z[q,q^{-1}]$, for all $i,j\in \wI$.
\end{lemma}

\begin{proof}
Formula \eqref{eq:cartan1} follows from \eqref{Tik} by unraveling the notation $\tk_{j,\dm}, \tk_{\bs_i \alpha_j,\dm}$ in \eqref{def:tkdm}--\eqref{eq:kla}.

It remains to show that $\vs_{\bs_i\alpha_j-\alpha_j,\dm}^{-1} \in \Z[q,q^{-1}]$.
Recall from the definition \eqref{def:vsi}, we have $\vs_{j,\diamond} \in - q^{\Z/2}$, for all $j \in \wI.$

For $j=i$, since $\bs_i(\alpha_i)=-\alpha_i + \alpha_\bullet$ for some $\alpha_\bullet\in \Z\bI$, we have
\begin{align*}
\tTa{i}(\tk_{i}) =\vs_{i,\dm}^2 \tk_{\bs_i \alpha_i }.
\end{align*}
where $\vs_{i,\dm}^2 \in q^\Z$.
The integrality for $\tTa{i}(\tk_{\tau i})$ can be then obtained by applying $\widehat{\tau}$ to the above formula.

For $j\neq i,\tau i$, we only need to consider the case $\vs_{i,\dm}=-q^{-1/2}$. In this case, by \eqref{def:vsi}, $\balpha_i$ is a short root.
Moreover, due to the classification of Satake diagrams and the corresponding restricted root systems \cite{Ar62}, we have $\frac{(\balpha_j,\balpha_i)}{(\balpha_i,\balpha_i)}=-2 $ or $0$.
It remains to consider the nontrivial case $\frac{(\balpha_j,\balpha_i)}{(\balpha_i,\balpha_i)}=-2 $.
It follows that $\bs_i\balpha_j-\balpha_j=2\balpha_i$, which implies
$\bs_i\alpha_j\in \alpha_j + k\alpha_i +l\alpha_{\tau i}+\Z\bI$, for some $k,l\geq 0,k+l=2.$
Since $\vs_{i,\dm}=\vs_{\tau i,\dm}$, the formula \eqref{eq:cartan1} is unraveled as the following integral formula
$\tTa{i}(\tk_{j}) =\vs_{i,\dm}^{-2} \tk_{\bs_i \alpha_j }.$

Therefore, the integrality of \eqref{eq:cartan1} holds in all cases.
\end{proof}

\subsection{A uniform formula for $\tTa{i}(B_i)$}
  \label{rkone1}

In this subsection, we introduce a uniform method to calculate $\tTa{i}(B_i)$. Note that $\tTT_i=\tTT_{\tau i}$ and this takes care of $\tTa{i}(B_{\tau i})$. To that end, without loss of generality, we can restrict ourselves to a Satake diagram $(\I=\bI \cup \wI,\tau)$ of real rank one; that is, $\wI =\{i,\tau i\}$ for some $i\in \wI$.

 Recall the diagram involution $\tau_{\bullet,i}$ associated to the longest element $\bwi$ in the Weyl group $W_{\bI\cup \{i,\tau i\}}$. By definition of admissible pairs, the diagram involution associated to $\bw$ is $\tau$.  Observe that $\tau_{\bullet,i} \tau i \in \{i,\tau i\}$, by Table~\ref{table:localSatake} on rank one Satake diagrams.

Recall $\ck_i, \ck_{\tau i} \in \tU^{\imath 0}$  from \eqref{def:Ki}.

\begin{lemma}
  \label{lem:bsiBi}
We have
\begin{align}
  \label{eq:TBi}
 \tT_{\bs_i}^{-1}(B_i) =-q^{-(\alpha_i,\bw\alpha_{\tau i}) } \tT_{w_\bullet}^2 ( B_{\tau_{\bullet,i} \tau i}^{\sigma})\ck_{\tau_{\bullet,i} \tau i}^{-1},
\end{align}
where $B_i^{\sigma}$ is given in \eqref{eq:Bsig}.
\end{lemma}

\begin{proof}
Recall from \eqref{def:vsi} and \eqref{def:vsi1} that $ \vs_{i,\diamond}= - q^{-(\alpha_i , \alpha_i+\bw\alpha_{\tau i})/2},$ for $i\in \wI,$ and $\vs_{j,\diamond}=1$, for $j\in \bI$.
By \eqref{def:bsi}, we have $\tT_{\bwi}= \tT_{\bs_i} \tT_{w_\bullet}$. By Lemma \ref{lem:braid1}, we compute
 \begin{align*}
 \tT_{\bs_i}^{-1}(B_i)
 &= \tT_{\bs_i}^{-1}\big(F_i+\tT_{w_\bullet}(E_{\tau i}) K_i'\big)
 \\
 &= \tT_{w_\bullet} \tT_{\bwi}^{-1}\big(F_i+\tT_{w_\bullet}(E_{\tau i}) K_i'\big)
 \\
 &= \tT_{w_\bullet}^2 \big(\tT_{w_\bullet}^{-1}\tT_{\bwi}^{-1}(F_i)+\tT_{\bwi}^{-1}(E_{\tau i}) \tT_{w_\bullet}^{-1}\tT_{\bwi}^{-1}(K_i')\big)
 \\
 &= \tT_{w_\bullet}^2 \big(-\tT_{w_\bullet}^{-1}( E_{\tau_{\bullet,i} i} K_{\tau_{\bullet,i} i}'^{-1}) - q^{-(\alpha_i,\alpha_i + \bw\alpha_{\tau i})}K_{\tau_{\bullet,i} \tau i}^{-1} F_{\tau_{\bullet,i} \tau i}\tT_{w_\bullet}^{-1}(K_{\tau_{\bullet,i} i}'^{-1})\big)
 \\
 &= -\tT_{w_\bullet}^2 \big( \tT_{w_\bullet}^{-1}(E_{\tau_{\bullet,i} i}) K_{\tau_{\bullet,i} \tau i}  + q^{ -(\alpha_i,\bw\alpha_{\tau i})}F_{\tau_{\bullet,i} \tau i}\big)K_{\tau_{\bullet,i} \tau i}^{-1} \tT_{w_\bullet}(K_{\tau_{\bullet,i} i}')^{-1}
 \\
 &= -q^{-(\alpha_i,\bw\alpha_{\tau i}) } \tT_{w_\bullet}^2 \big( K_{\tau_{\bullet,i} \tau i} \tT_{w_\bullet}^{-1}(E_{\tau_{\bullet,i} i}) +  F_{\tau_{\bullet,i} \tau i}\big)\ck_{\tau_{\bullet,i} \tau i}^{-1}
 \\
 &= -q^{-(\alpha_i,\bw\alpha_{\tau i}) } \tT_{w_\bullet}^2 ( B_{\tau_{\bullet,i} \tau i}^{\sigma})\ck_{\tau_{\bullet,i} \tau i}^{-1}.
 \end{align*}
 This proves the lemma. 
\end{proof}

 \begin{theorem}
  \label{thm:rkone1}
Let $i\in \wI$. There exists a unique element $\tTa{i}(B_i)\in \tUi$ which satisfies the following intertwining relation (see \eqref{eq:newb0})
\[
\tTa{i}(B_i) \tfX_i =\tfX_i \tT_{\bs_i}^{-1}(B_i).
\]
More explicitly, we have
 \begin{equation}
   \label{eq:rkone2}
 \tTa{i}(B_i)=-q^{-(\alpha_i,\bw\alpha_{\tau i}) } \tT_{w_\bullet}^2 ( B_{\tau_{\bullet,i} \tau i})\ck_{\tau_{\bullet,i} \tau i}^{ -1}.
 \end{equation}
 \end{theorem}

\begin{proof}
Recall $\tau_{\bullet,i} \tau i \in \{i,\tau i\}$; see Table~\ref{table:localSatake}. By Theorem~\ref{thm:fX1}, we have $\tfX_i B_{\tau_{\bullet,i} \tau i}^{\sigma} = B_{\tau_{\bullet,i} \tau i} \tfX_i$. By Proposition~\ref{prop:newb-1}, we have $\tT_{w_\bullet} (\tfX_i) =\tfX_i$, and hence $\tfX_i \tT_{w_\bullet}^2 (B_{\tau_{\bullet,i} \tau i}^{\sigma}) = \tT_{w_\bullet}^2 (B_{\tau_{\bullet,i} \tau i})\tfX_i$. By Lemma~\ref{lem:rkone1}, we have $\ck_{\tau_{\bullet,i} \tau i} \in \tU^{\imath 0}$, 
and hence $\ck_{\tau_{\bullet,i} \tau i}$ commutes with $\tfX_i$. Putting these together with \eqref{eq:TBi}, we have
 \begin{equation}
    \label{eq:rkone1}
  -q^{-(\alpha_i,\bw(\alpha_{\tau i})) }\tT_{w_\bullet}^2 ( B_{\tau_{\bullet,i} \tau i})\ck_{\tau_{\bullet,i} \tau i}^{-1}   \tfX_i
  =\tfX_i  \tT_{\bs_i}^{-1}(B_i).
 \end{equation}
It follows by Proposition~\ref{prop:Tjblack} that
$-q^{-(\alpha_i,\bw(\alpha_{\tau i})) }\tT_{w_\bullet}^2 ( B_{\tau_{\bullet,i} \tau i})\ck_{\tau_{\bullet,i} \tau i}^{-1} \in \tUi.$  Hence, setting $\tTa{i}(B_i)=-q^{-(\alpha_i,\bw\alpha_{\tau i}) } \tT_{w_\bullet}^2 ( B_{\tau_{\bullet,i} \tau i})\ck_{\tau_{\bullet,i} \tau i}^{ -1}$, we have proved the theorem.
 \end{proof}

\section{Rank 2 formulas for  $\tTa{i}(B_j)$}
  \label{sec:rktwo}

Let $(\I=\bI \cup \wI,\tau)$ be a rank 2 irreducible Satake diagram. Fix $i,j\in \wItau$ such that $i\neq j$, such that $\wI =\{i, \tau i, j, \tau j\}$. A complete list of formulas for $\tTa{i}(B_j)$ is formulated in Table~\ref{table:rktwoSatake} (listed after \S\ref{subsec:braidmodule}). We show that the formulas for $\tTa{i}(B_j)$ in Table~\ref{table:rktwoSatake} satisfy the intertwining relation \eqref{eq:newb0}; see  Theorem~\ref{thm:rktwo1}. Together with the formulas in the previous section, we have established the existence of an endomorphism $\tTa{i}$ on $\tUi$ satisfying  \eqref{eq:newb0}.

%
%
\subsection{Some commutator relations with $\tfX$}

For $w\in W,$ let $\U^+[w]$ be the well-known subalgebra of $\U^+$ spanned by PBW basis elements generated by certain $q$-root vectors so that $\U^+[w_0] =\U^+$; see \cite[8.24]{Ja95}. As we identify $\tU^+ =\U^+$, we denote by $\tU^+[w]$ the subalgebra of $\tU^+$ corresponding to $\U^+[w]$.
The next lemma is valid for all  Satake diagrams.

\begin{lemma}
  \label{lem:rktwo1}
For $i\neq j\in \wItau$, we have
\begin{align}
 F_j  \tfX_i &=\tfX_i  F_j,
 \label{eq:FjUp}\\
\tT_{\bs_i}^{-1}\big(\tT_{\bw}(E_{\tau j})K_j'\big)  \cdot \tfX_i
&= \tfX_i \cdot  \tT_{\bs_i}^{-1}\big(\tT_{\bw}(E_{\tau j}) K_j'\big).
 \label{eq:TEjUp}
\end{align}
\end{lemma}

\begin{proof}
Write $\tfX_i=\sum_{m\geq 0} \tfX_{i,m},$ where $ \tfX_{i,m}\in \tU^+_{m (\alpha_i+ w_\bullet \alpha_{\tau i})}$. By \cite[Proposition~4.5]{BW18b}, we have $\tfX_{i,m}\in  \tU^+[\bs_i]$, for $m\geq 0$. Since the simple reflection $s_j$ does not appear in any reduced expression of $\bs_i$,  $F_j$ commutes with any element in $\tU^+[\bs_i]$; in particular, $F_j$ commutes with $\tfX_i.$ This proves the identity \eqref{eq:FjUp}.

By Proposition~\ref{prop:inv}, $\tfX$ is fixed by $\widehat{\tau}_{\bullet,i}$ (which is equal to either $\text{Id}$ or $\widehat{\tau}$). Hence, by Lemma~ \ref{lem:braid1} and the fact that $\tfX_{i,m}\in \tU^+_{m (\alpha_i+ w_\bullet \alpha_{\tau i})}$, we have
\begin{align*}
\tT_{w_{\bullet,i}}(\tfX_{i,m})
=\tT_{w_{\bullet,i}}\widehat{\tau}_{\bullet,i}(\tfX_{i,m}) \in \tU^-_{-m (\alpha_i +\bw\alpha_{\tau i})} K_{\alpha_i+w_\bullet\alpha_{\tau i}}'^{-m},
\end{align*}
or equivalently,
\begin{align}
 \label{eq:TUpK}
\mathcal{F}:= \tT_{\bwi}(\tfX_{i,m}) K_{\alpha_i+w_\bullet\alpha_{\tau i}}'^{m}
\in \tU^-_{-m (\alpha_i +\bw\alpha_{\tau i})}.
\end{align}
Since $\bw\alpha_{\tau i}= \alpha_{\tau i}+\sum_{r\in \bI} a_r \alpha_r$ for some $a_r\in \N$, the eigenspace $\tU^-_{-m (\alpha_i +\bw\alpha_{\tau i})}$ lies in the subalgebra of $\tU^-$ generated by $F_i,F_{\tau i},F_r,r\in \bI$; clearly, $E_{\tau j}$ commutes with any of these elements and hence we have by \eqref{eq:TUpK} that
$[E_{\tau j}, \mathcal{F}]=0$. For each $m$, we compute
\begin{align*}
& \big[E_{\tau j} \tT_{\bw}(K_j'), \tT_{w_{\bullet,i}}(\tfX_{i,m}) \big]
\\
&= \big[E_{\tau j} \tT_{\bw}(K_j'), \mathcal{F} K_{\alpha_i+w_\bullet(\alpha_{\tau i})}'^{-m} \big]
\\
&= q^{m(\bw\alpha_{ j},\alpha_i+w_\bullet\alpha_{\tau i} )} E_{\tau j}\mathcal{F}\tT_{\bw}(K_j')K_{\alpha_i+w_\bullet\alpha_i}'^{-m}
  -q^{m(\alpha_{\tau j},\alpha_i+w_\bullet\alpha_{\tau i} )}\mathcal{F} E_{\tau j}\tT_{\bw}(K_j')K_{\alpha_i+w_\bullet\alpha_i}'^{-m}
  \\
&= q^{m(\alpha_{\tau j},\alpha_i+w_\bullet\alpha_{\tau i} )}[E_{\tau j},  \mathcal{F}] \cdot \tT_{\bw}(K_j')K_{\alpha_i+w_\bullet\alpha_i}'^{-m}= 0.
\end{align*} 
Hence we obtain an identity
\begin{align}
  \label{eq:rktwo2}
E_{\tau j}  \tT_{\bw}(K_j')  \cdot \tT_{w_{\bullet,i }}(\tfX_i)= \tT_{w_{ \bullet,i}}(\tfX_i) \cdot E_{\tau j}  \tT_{\bw}(K_j') .
\end{align}
The desired identity \eqref{eq:TEjUp} now follows by applying $\tT_{\bs_i}^{-1} \tT_{\bw}$ to \eqref{eq:rktwo2}. Indeed, we have $\tT_{\bs_i}^{-1} \tT_{\bw} \tT_{w_{\bullet,i }} (\tfX_i) =\tT_{\bw}^2 (\tfX_i) =\tfX_i$
since $\tT_{w_{\bullet,i }} =\tT_{\bs_i}\tT_{\bw} =\tT_{\bw}\tT_{\bs_i}$ by \eqref{def:bsi} and $\tT_{\bw} (\tfX_i) =\tfX_i$ by Proposition~\ref{prop:newb-1}. Also we clearly  have $\tT_{\bw}^2(K_j') =K_j'.$
\end{proof}

\subsection{Motivating examples: types BI, DI, DIII$_4$}
  \label{subsec:example}

We provide examples in this subsection to motivate how we obtain the general rank two formulas $\tTa{i}(B_j)$ in Theorem~\ref{thm:rktwo1} below. The three examples are of types BI$_n$ ($n\ge 3$), DI$_n$ ($n\ge 5$), DIII$_4$, and they will be treated uniformly.

The Satake diagrams of these types are listed below. For each type, we define elements $t_j\in W_\bullet$ for $j\in \bI$ following each diagram; these notations $t_j$ allow a uniform proof of Lemma~\ref{lem:rktwo2-1} thanks to the properties \eqref{rsts} below.
 \begin{center}
\begin{tikzpicture}[baseline=0, scale=1.5]
		\node at (0.55,0) {$\circ$};
		\node at (1.05,0) {$\circ$};
		\node at (1.5,0) {$\bullet$};
		\draw[-] (0.6,0)  to (1.0,0);
		\draw[-] (1.1,0)  to (1.4,0);
		\draw[-] (1.4,0) to (1.9, 0);
		\draw[dashed] (1.9,0) to (2.7,0);
		\draw[-] (2.7,0) to (2.9, 0);
		\node at (3,0) {$\bullet$};
		\draw[-implies, double equal sign distance]  (3.1, 0) to (3.7, 0);
		\node at (3.8,0) {$\bullet$};
		\node at (0.5,-.2) {\small 1};
		\node at (1,-.2) {\small 2};
		\node at (1.5,-.2) {\small 3};
		\node at (3, -.2) {\small n-1};
		\node at (3.8,-.2) {\small n};
	\end{tikzpicture}\\
BI$_n,n\geq 3$\\
$t_a =s_a \cdots s_n \cdots s_a,\qquad (3\leq a \leq n).$
\end{center}
\vspace{1em}
\begin{center}
\begin{tikzpicture}[baseline=0, scale=1.5]
		\node at (0.55,0) {$\circ$};
		\node at (1.05,0) {$\circ$};
		\node at (1.5,0) {$\bullet$};
		\draw[-] (0.6,0)  to (1.0,0);
		\draw[-] (1.1,0)  to (1.4,0);
		\draw[-] (1.4,0) to (1.9, 0);
		\draw[dashed] (1.9,0) to (2.7,0);
		\draw[-] (2.7,0) to (2.9, 0);
		\node at (3,0) {$\bullet$};
		\node at (3.8,0.5) {$\bullet$};
		\node at (3.8,-0.5) {$\bullet$};
        \draw (3,0) to (3.8,0.5);
        \draw (3,0) to (3.8,-0.5);
		\node at (0.5,-.2) {\small 1};
		\node at (1,-.2) {\small 2};
		\node at (1.5,-.2) {\small 3};
		\node at (3, -.2) {\small n-2};
		\node at (3.8,0.3) {\small n-1};
		\node at (3.8,-0.7) {\small n};
	\end{tikzpicture}\\
DI$_n,n\geq 5$\\
$t_{a}=s_{a\cdots n-2}s_{n-1}s_n s_{n-2\cdots a}, \qquad (3\leq a \leq n-2),$\qquad
$t_{n-1}=t_{n}=s_{n-1} s_n.$
\end{center}

\vspace{1em}
\begin{center}
\begin{tikzpicture}[baseline=0, scale=1.5]
		\node at (1.05,0) {$\circ$};
		\node at (1.5,0) {$\circ$};
		\draw[-] (1.1,0)  to (1.45,0);
		\node at (2.3,0.5) {$\bullet$};
		\node at (2.3,-0.5) {$\bullet$};
        \draw (1.55,0) to (2.3,0.5);
        \draw (1.55,0) to (2.3,-0.5);
		\node at (1,-.2) {\small 1};
		\node at (1.5,-.2) {\small 2};
		\node at (2.3,0.3) {\small 3};
		\node at (2.3,-0.7) {\small 4};
	\end{tikzpicture}\\
DIII$_4$\\
$t_3=t_4=s_3 s_4.$
\end{center}
Note that, for each of the three types, we always have
\begin{align}  \label{rsts}
\bs_2=s_2 t_3 s_2,
\qquad \ell(\bs_2) =  \ell( t_3) + 2,
\qquad
B_1=F_1+E_1 K_1'.
\end{align}
Recall the notation $B_i^\sigma$ from \eqref{eq:Bsig}.

\begin{lemma}
  \label{lem:rktwo2-1}
We have
\begin{align}
  \label{eq:rktwo6-1}
\tT_{\bs_2}^{-1}(F_1)&=\big[ \tT_{\bw} ( B_2^\sigma), [B_2^\sigma,F_1]_{q_2}\big]_{q_2}-{q_2} F_1 \tT_{\bw}(K_2) K_2'.
\end{align}
\end{lemma}

\begin{proof}
By Lemma \ref{lem:app1}, $[B_2^\sigma,F_1]_{q_2} =[F_2,F_1]_{q_2}$, and RHS \eqref{eq:rktwo6-1} is simplified as follows:
\begin{align}\notag
\big[ \tT_{\bw} ( B_2^\sigma), [B_2^\sigma,F_1]_{q_2}\big]_{q_2}&= \big[ \tT_{\bw} ( B_2^\sigma), [F_2 ,F_1]_{q_2}\big]_{q_2}\\\notag
&=\big[ \tT_{\bw} (F_2) , [F_2,F_1]_{q_2}\big]_{q_2} + \big[ E_2 \tT_{\bw}(K_2), [F_2 ,F_1]_{q_2}\big]_{q_2}\\\label{eq:rktwo7}
&=\big[ \tT_{\bw} (F_2) , [F_2,F_1]_{q_2}\big]_{q_2} +{q_2} F_1 \tT_{\bw}(K_2) K_2'.
\end{align}
On the other hand, by a direct computation using \eqref{rsts} and Proposition~\ref{prop:braid0}, we have
\begin{align}\notag
\tT_{\bs_2}^{-1}(F_1)
&=  \tT_2^{-1}\tT_{t_3}^{-1}([F_2,F_1]_{q_2} )
\\
&= \big[\tT_2^{-1}\tT_{t_3}^{-1}(F_2) , [F_2,F_1]_{q_2}\big]_{q_2}
\notag \\
&= \big[ \tT_{t_3} (F_2) , [F_2,F_1]_{q_2}\big]_{q_2}.
  \label{eq:rktwo8}
\end{align}
The desired formula \eqref{eq:rktwo6-1} follows from \eqref{eq:rktwo7}--\eqref{eq:rktwo8} by noting that $\tT_{\bw} ( F_2 )=\tT_{t_3}(F_2)$.
\end{proof}

Note that $q_1=q_2$ in all three types.

\begin{lemma}\label{lem:rktwo2-2}
We have
\begin{align}
\tT_{\bs_2}^{-1} ( E_1 K_1' )&=\big[ \tT_{\bw} ( B_2), [B_2, E_1 K_1']_{q_2}\big]_{q_2}-{q_2} E_1 K_1' \tT_{\bw}(K_2) K_2'.
\label{eq:rktwo6-2}
\end{align}
\end{lemma}

\begin{proof}
We shall establish the identity \eqref{eq:rktwo6-2} by applying the operator $\cL:=\tT_{\bw}\tT_{w_0}$ to \eqref{eq:rktwo6-1} as follows.

Recall $\ck_i \in \tU^{\imath 0}$  from \eqref{def:Ki}.
By \eqref{eq:TBi} 
and noting $(\alpha_2, w_\bullet \alpha_{\tau 2}) =0$ in each of the three types, we have $\tT_{\bs_2}^{-1}(B_2^\imath) = -\tT_{\bw}^2(B_2^{\sigma}) \ck_2^{-1},$ or equivalently,
\begin{align}
  \label{eq:rktwo3}
\tT_{\bs_2}^{-1}(B_2^\imath)  \ck_2 =-\tT_{\bw}^2(B_2^{\sigma}).
\end{align}
By Lemma~\ref{lem:braid1}, we have $\tT_{w_0}(B_2^\sigma)=\tT_{w_{\bullet,2}}(B_2^\sigma)$. Hence, applying $\tT_{\bs_2}$ to both sides of \eqref{eq:rktwo3} we obtain
\begin{align}\label{eq:rktwo3-2}
B_2 \tT_{\bs_2}(\ck_2^\imath)=- \tT_{\bw}\tT_{w_{\bullet,2}}(B_2^\sigma) =- \tT_{\bw}\tT_{w_0}(B_2^\sigma).
\end{align}
Moreover, by Lemma~\ref{lem:braid1}, we have $\cL(F_1)=-K_1^{-1} E_1=-{q_2}^{-2}E_1K_1' \tk_{1}^{-1}$. Note also that $\cL$ commutes with both $\tT_{\bw}$ and $\tT_{\bs_2}$. Hence, by applying $\cL$ to \eqref{eq:rktwo6-1} and then using \eqref{eq:rktwo3-2}, we have
\begin{align}
 \notag
 \tT_{\bs_2}^{-1}(E_1 K_1') \tT_{\bs_2}(\tk_1^{-1})
 &=\big[ \tT_{\bw} ( B_2)\tT_{w_{\bullet,2}}(\ck_2), [B_2 \tT_{\bs_2}(\ck_2),E_1K_1'\tk_{1}^{-1} ]_{q_2}\big]_{q_2}\\
 &-{q_2} E_1 K_1' \tk_1^{-1} \tT_{\bw}\cL(\ck_2).
 \label{eq:rktwo4-1}
\end{align}
For weight reason, \eqref{eq:rktwo4-1} is simplified as
\begin{align}\notag
 \tT_{\bs_2}^{-1}(E_1 K_1') \tT_{\bs_2}(\tk_1^{-1})
 &={q_2}^2 \big[ \tT_{\bw} ( B_2), [B_2, E_1K_1']_{q_2}\big]_{q_2}\tT_{w_{\bullet,2}}(\ck_2)\tT_{\bs_2}(\ck_2)\tk_{1}^{-1}
  \\
 &-{q_2} E_1 K_1' \tk_1^{-1} \tT_{\bs_2}(\tk_2) K_{\bw\alpha_2 -\alpha_2}.
 \label{eq:rktwo4-2}
\end{align}
By definition \eqref{def:Ki}, we have $\ck_2=\tk_2 K'_{\bw\alpha_2 -\alpha_2}$; in addition, by \eqref{eq:newb8}, $K'_{\bw\alpha_2 -\alpha_2}$ is fixed by $\tT_{\bs_2}$. We also have $\tT_{w_{\bullet,2}}(\ck_2)={q_2}^{-2} \ck_2^{-1}$. Hence, \eqref{eq:rktwo4-2} is further simplified as
\begin{align}\notag
 \tT_{\bs_2}^{-1}(E_1 K_1') \tT_{\bs_2}(\tk_1^{-1})
 &=\big[ \tT_{\bw} ( B_2), [B_2, E_1K_1' ]_{q_2}\big]_{q_2} \tk_2^{-1} \tT_{\bs_2}(\tk_2)\tk_{1}^{-1}\\
 &-{q_2} E_1 K_1' K_{\bw(\alpha_2)}K_2'\tk_1^{-1} \tT_{\bs_2}(\tk_2) \tk_2^{-1}.\label{eq:rktwo4-3}
\end{align}
Finally, by Lemma~\ref{lem:newb1}, we have $\tT_{\bs_2}(\tk_1^{-1})=\tk_1^{-1} \tT_{\bs_2}(\tk_2) \tk_2^{-1}$, and then the identity \eqref{eq:rktwo4-3} can be transformed into an equivalent form \eqref{eq:rktwo6-2}.
\end{proof}

\begin{proposition}
The following element
\begin{align}
   \label{eq:rktwo4-4}
\tTa{2}(B_1) :=\big[ \tT_{\bw} ( B_2), [B_2 ,B_1 ]_{q_2}\big]_{q_2} - {q_2}  B_1 \tT_{\bw}( \ck_2) \in\tUi
\end{align}
satisfies the intertwining relation $\tTa{2}(B_1) \tfX_2
=\tfX_2 \tT_{\bs_2}^{-1}(B_1^\imath)$ (i.e., \eqref{eq:newb0}, for $i=2,x=B_1$).
\end{proposition}

\begin{proof}
The intertwining relation follows by the following computation:
\begin{align*}
\tfX_2 & \tT_{\bs_2}^{-1}(B_1^\imath) \tfX_2^{-1}
\\
&=\tfX_2 \big(\tT_{\bs_2}^{-1}(F_1)+ \tT_{\bs_2}^{-1}(E_1 K_1')\big) \tfX_2^{-1}\\
&\overset{\eqref{eq:TEjUp}}{=}
\tfX_2\tT^{-1}_{\bs_2}(F_1)\tfX_2^{-1} + \tT^{-1}_{\bs_2}(E_1 K_1')
\\
& \overset{\eqref{eq:rktwo6-1}}{=}
\tfX_2\Big(\big[ \tT_{\bw} ( B_2^\sigma), [B_2^\sigma,F_1]_{q_2}\big]_{q_2}-{q_2} F_1 \tT_{\bw}(K_2) K_2' \Big)\tfX_2^{-1} + \tT_{\bs_2}^{-1}(E_1 K_1')\\
& \overset{(*)}{=} \big[ \tT_{\bw} ( B_2), [B_2, F_1]_{q_2}\big]_{q_2}-{q_2} F_1 \tT_{\bw}(K_2) K_2' + \tT_{\bs_2}^{-1}(E_1 K_1') \\
& \overset{\eqref{eq:rktwo6-2}}{=}
\big[ \tT_{\bw} ( B_2), [B_2, F_1]_{q_2}\big]_{q_2}-{q_2} F_1 \tT_{\bw}(K_2) K_2'
\\
&\qquad
+\big[ \tT_{\bw} ( B_2), [B_2, E_1 K_1']_{q_2}\big]_{q_2}-{q_2} E_1 K_1' \tT_{\bw}(K_2) K_2'\\
& =\big[ \tT_{\bw} ( B_2), [B_2, B_1]_{q_2}\big]_{q_2}-{q_2} B_1 \tT_{\bw}(K_2) K_2'
= \tTa{2}(B_1),
\end{align*}
where the equality (*) follows from Theorem~\ref{thm:fX1} and Lemma~\ref{lem:rktwo1}.
%
\end{proof}


%
%
\subsection{Formulation for $\tTa{i}(B_j)$}

 \begin{table}[H]
\caption{Rank 2 Satake diagrams}
     \label{table:rktwodatum}
 \resizebox{5.4 in}{!}{%
\begin{tabular}{|c| c | c|| c|c|c|}
\hline
\begin{tikzpicture}[baseline=0]
\node at (0, -0.15) {SP};
\end{tikzpicture}
&
\begin{tikzpicture}[baseline=0]
\node at (0, -0.15) {Satake diagrams};
\end{tikzpicture}
&

\begin{tikzpicture}[baseline=0]
\node at (0, -0.15) {RS};
\end{tikzpicture}
&
\begin{tikzpicture}[baseline=0]
\node at (0, -0.15) {SP};
\end{tikzpicture}
&
\begin{tikzpicture}[baseline=0]
\node at (0, -0.15) {Satake diagrams};
\end{tikzpicture}
&

\begin{tikzpicture}[baseline=0]
\node at (0, -0.15) {RS};
\end{tikzpicture}
\\
\hline
AI$_2$
&
\begin{tikzpicture}[baseline=0, scale=1.2]
		\node at (-0.5,0.2) {$\circ$};
		\node at (0.5,0.2) {$\circ$};
       \draw[-]  (0.45, 0.2) to (-0.45, 0.2);
		\node at (-0.5, 0) {\small 1};
		\node at (0.5, 0) {\small 2};
	\end{tikzpicture}
&
A$_2$
&
CII$_n$
&
\begin{tikzpicture}[baseline=6,scale=1.1]
		\node  at (0,0.2) {$\bullet$};
		\node  at (0,0) {1};
		\draw (0.05, 0.2) to (0.45, 0.2);
		\node  at (0.5,0.2) {$\circ$};
		\node  at (0.5,0) {2};
		\draw (0.55, 0.2) to (0.95, 0.2);
		\node at (1,0.2) {$\bullet$};
		\node at (1,0) {3};
		\node at (1.5,0.2) {$\circ$};
		\node at (1.5,0) {4};
		\draw[-] (1.05,0.2)  to (1.45,0.2);
		\draw[-] (1.55,0.2) to (1.95, 0.2);
		\node at (2,0.2) {$\bullet$};
		\node at (2,0) {5};
		\draw (1.9, 0.2) to (2.1, 0.2);
		\draw[dashed] (2.1,0.2) to (2.7,0.2);
		\draw[-] (2.7,0.2) to (2.9, 0.2);
		\node at (3,0.2) {$\bullet$};
		\draw[implies-, double equal sign distance]  (3.1,0.2) to (3.7, 0.2);
		\node at (3.8,0.2) {$\bullet$};
		\node at (3.8,0) {$n$};
	\end{tikzpicture}
&
BC$_2$
\\
\hline
CI$_2$
&
\begin{tikzpicture}[baseline=0, scale=1.2]
		\node at (-0.5,0.2) {$\circ$};
		\node at (0.5,0.2) {$\circ$};
		\draw[-implies, double equal sign distance]  (0.4, 0.2) to (-0.4, 0.2);
		\node at (-0.5,0) {\small 1};
		\node at (0.5,0) {\small 2};
	\end{tikzpicture}
&
C$_2$
&
CII$_4$
&
\begin{tikzpicture}[baseline=6,scale=1.5]
        \node at (-1, 0.2) {$\bullet$};
        \node at (-1,0) {1};
		\draw[-] (-0.95,0.2) to (-0.55, 0.2);
        \node at (-0.5,0.2) {$\circ$};
        \node at (-0.5,0) {2};
		\draw[-] (-.45,0.2) to (-0.05, 0.2);
		\node at (0,0.2) {$\bullet$};
		\node at (0,0) {3};
		\draw[implies-, double equal sign distance]  (0.05, 0.2) to (0.75, 0.2);
		\node at (0.8,0.2) {$\circ$};
		\node at (0.8,0) {4};
	\end{tikzpicture}
&
C$_2$
\\
\hline
G$_2$
&
\begin{tikzpicture}[baseline=0, scale=1.5]
		\node at (-0.5,0) {$\circ$};
		\node at (0.5,0) {$\circ$};
		\draw[->]  (0.4, 0.05) to (-0.4, 0.05);
		\draw[->]  (0.4, -0.05) to (-0.4, -0.05);
		\draw[->]  (0.4, 0) to (-0.4, 0);
		\node at (-0.5, -.2) {\small 1};
		\node at (0.5,-.2) {\small 2};
	\end{tikzpicture}
&
G$_2$
&
EIV
&
\begin{tikzpicture}[baseline = 0, scale =1.5]
		\node at (-1,0.2) {$\circ$};
        \node at (-1,0) {1};
		\draw (-0.95,0.2) to (-0.55,0.2);
		\node at (-0.5,0.2) {$\bullet$};
        \node at (-0.5,0) {2};
		\draw (-0.45,0.2) to (-0.05,0.2);
		\node at (0,0.2) {$\bullet$};
        \node at (0.1,0) {3};
		\draw (0.05,0.2) to (0.45,0.2);
		\node at (0.5,0.2) {$\bullet$};
		\node at (0.5,0) {4};
		\draw (0.55,0.2) to (0.95,0.2);
		\node at (1,0.2) {$\circ$};
		\node at (1,0) {5};
		\draw (0, 0.15) to (0,-0.25);
		\node at (0,-0.2) {$\bullet$};
		\node at (-.15,-0.15) {6};
\end{tikzpicture}
&
A$_2$
\\
\hline
BI$_n$
&
 \begin{tikzpicture}[baseline=0, scale=1.2]
		\node at (0.5,0) {$\circ$};
		\node at (1.0,0) {$\circ$};
		\node at (1.5,0) {$\bullet$};
		\draw[-] (0.55,0)  to (0.95,0);
		\draw[-] (1.05,0)  to (1.45,0);
		\draw[-] (1.55,0) to (1.8, 0);
		\draw[dashed] (1.8,0) to (2.5,0);
		\draw[-] (2.5,0) to (2.75, 0);
		\node at (2.8,0) {$\bullet$};
		\draw[-implies, double equal sign distance]  (2.85, 0) to (3.45, 0);
		\node at (3.5,0) {$\bullet$};
		\node at (0.5,-.2) {\small 1};
		\node at (1,-.2) {\small 2};
		\node at (1.5,-.2) {\small 3};
		\node at (3.5,-.2) {$n$};
	\end{tikzpicture}
&
B$_2$
&
AIII$_3$
&
\begin{tikzpicture}[baseline=0,scale=1.5]
		\node  at (-0.65,0) {$\circ$};
		\node  at (0,0) {$\circ$};
		\node  at (0.65,0) {$\circ$};
		\draw[-] (-0.6,0) to (-0.05, 0);
		\draw[-] (0.05, 0) to (0.6,0);
		\node at (-0.65,-0.15) {1};
		\node at (0,-0.15) {2};
		\node at (0.65,-0.15) {3};
        \draw[bend left,<->,red] (-0.65,0.1) to (0.65,0.1);
        \node at (0,0.2) {$\textcolor{red}{\tau}$};
	\end{tikzpicture}
&
C$_2$
\\
\hline
DI$_n$
&
  \begin{tikzpicture}[baseline=0, scale=1.2]
		\node at (0.55,0) {$\circ$};
		\node at (1.05,0) {$\circ$};
		\node at (1.5,0) {$\bullet$};
		\draw[-] (0.6,0)  to (1.0,0);
		\draw[-] (1.1,0)  to (1.4,0);
		\draw[-] (1.4,0) to (1.9, 0);
		\draw[dashed] (1.9,0) to (2.7,0);
		\draw[-] (2.7,0) to (2.9, 0);
		\node at (3,0) {$\bullet$};
		\node at (3.8,0.35) {$\bullet$};
		\node at (3.8,-0.35) {$\bullet$};
        \draw (3,0) to (3.8,0.35);
        \draw (3,0) to (3.8,-0.35);
		\node at (0.5,-.2) {\small 1};
		\node at (1,-.2) {\small 2};
		\node at (1.5,-.2) {\small 3};
	\end{tikzpicture}
&
B$_2$
&
AIII$_n$
&
 \begin{tikzpicture}[baseline=0,scale=1.0]
		\node  at (-2.1,0) {$\circ$};
		\node  at (-1.3,0) {$\circ$};
		\node  at (-0.5,0) {$\bullet$};
		\node  at (0.5,0) {$\bullet$};
		\node  at (1.3,0) {$\circ$};
		\node  at (2.1,0) {$\circ$};
		\draw[-] (-2.05,0) to (-1.35, 0);
		\draw[-] (-1.25,0) to (-0.55, 0);
		\draw[-] (0.55,0) to (1.25, 0);
		\draw[-] (1.35, 0) to (2.05,0);
		\node at (-2.1,-0.2) {1};
		\node at (-1.3,-0.2) {2};
		\node at (1.3,-0.2) {\small$n-1$};
		\node at (2.1,-0.2) {$n$ };
        \draw[dashed] (-0.5,0) to (0.5,0);
        \draw[bend left,<->,red] (-1.3,0.1) to (1.3,0.1);
        \draw[bend left,<->,red] (-2.1,0.1) to (2.1,0.1);
        \node at (0,0.6) {$\textcolor{red}{\tau} $};
	\end{tikzpicture}
&
BC$_2$
\\
\hline
DIII$_4$
&
\begin{tikzpicture}[baseline=0]
\end{tikzpicture}
  \begin{tikzpicture}[baseline=0, scale=1.2]
		\node at (0.65,0) {$\circ$};
		\node at (1.5,0) {$\circ$};
		\draw[-] (0.7,0)  to (1.45,0);
		\node at (2.3,0.4) {$\bullet$};
		\node at (2.3,-0.4) {$\bullet$};
        \draw (1.55,0) to (2.3,0.4);
        \draw (1.55,0) to (2.3,-0.4);
		\node at (0.65,-.2) {\small 1};
		\node at (1.5,-.2) {\small 2};
		\node at (2.35,0.25) {\small 3};
		\node at (2.35,-0.25) {\small 4};
	\end{tikzpicture}
&
C$_2$
&
DIII$_5$
&
\begin{tikzpicture}[baseline=0,scale=1.5]
		\node at (-0.5,0) {$\bullet$};
		\node at (0,0) {$\circ$};
		\node at (0.5,0) {$\bullet$};
        \node at (1,0.3) {$\circ$};
        \node at (1,-0.3) {$\circ$};
		\draw[-] (-0.45,0) to (-0.05, 0);
		\draw[-] (0.05, 0) to (0.45,0);
		\draw[-] (0.5,0) to (0.965, 0.285);
		\draw[-] (0.5, 0) to (0.965,-0.285);
		\node at (-0.5,-0.2) {1};
		\node at (0,-0.2) {2};
		\node at (0.5,-0.2) {3};
        \node at (1,0.15) {4};
        \node at (1,-0.45) {5};
        \draw[bend left,<->,red] (1.1,0.3) to (1.1,-0.3);
        \node at (1.4,0) {$\textcolor{red}{\tau} $};
	\end{tikzpicture}
&
BC$_2$
\\
\hline
AII$_5$
&
 \begin{tikzpicture}[baseline=0,scale=1.2]
		\node at (-0.5,0) {$\bullet$};
		\node  at (0,0) {$\circ$};
		\node at (0.5,0) {$\bullet$};
		\node at (1,0) {$\circ$};
		\node at (1.5,0) {$\bullet$};
		\draw[-] (-0.5,0) to (-0.05, 0);
		\draw[-] (0.05, 0) to (0.5,0);
		\draw[-] (0.5,0) to (0.95,0);
		\draw[-] (1.05,0)  to (1.5,0);
		\node at (-0.5,-0.2) {1};
		\node  at (0,-0.2) {2};
		\node at (0.5,-0.2) {3};
		\node at (1,-0.2) {4};
		\node at (1.5,-0.2) {5};
	\end{tikzpicture}
&
A$_2$
&
EIII
&
\begin{tikzpicture}[baseline = 0, scale =1.5]
		\node at (-1,0) {$\circ$};
        \node at (-1,-0.2) {1};
		\draw (-0.95,0) to (-0.55,0);
		\node at (-0.5,0) {$\bullet$};
        \node at (-0.5,-0.2) {2};
		\draw (-0.45,0) to (-0.05,0);
		\node at (0,0) {$\bullet$};
        \node at (0.1,-0.2) {3};
		\draw (0.05,0) to (0.45,0);
		\node at (0.5,0) {$\bullet$};
		\node at (0.5,-0.2) {4};
		\draw (0.55,0) to (0.95,0);
		\node at (1,0) {$\circ$};
		\node at (1,-0.2) {5};
		\draw (0,-0.05) to (0,-0.35);
		\node at (0,-0.4) {$\circ$};
		\node at (-.15,-0.35) {6};
        \draw[bend left, <->, red] (-0.9,0.1) to (0.9,0.1);
        \node at (0,0.2) {$\color{red} \tau $};
	\end{tikzpicture}
&
BC$_2$
\\
\hline
\end{tabular}
}%
\newline
(SP=symmetric pair, RS=relative root system)
\end{table}

\begin{theorem}
   \label{thm:rktwo1}
The elements $\tTa{i}(B_j) \in \tUi$ listed in Table \ref{table:rktwoSatake} satisfy the following intertwining relation (see \eqref{eq:newb0}):
\begin{align}
  \label{eq:twinij}
\tTa{i}(B_j)  \tfX_i =\tfX_i  \tT'_{\bs_i,-1}(B_j).
\end{align}
\end{theorem}

We clarify a few points regarding Table~\ref{table:rktwoSatake} in the following remarks.

\begin{remark}
Recall that $\tT_s$ ($s\in \bI$) restrict to automorphisms on $\tUi$ by Proposition~\ref{prop:Tjblack}; hence, the use of $\tT_s$ ($s\in \bI$) in the formulas of $\tTa{i}(B_j)$ is legitimate; see \eqref{eq:newb-1}.
\end{remark}

\begin{remark}
Let $\rho$ be a diagram involution on the underlying Dynkin diagram ($\rho$ is not necessarily equal to $\tau$). By the intertwining relation \eqref{eq:newb0}, the formula of $\tTa{\rho i} (B_{\rho j})$ can be obtained from $\tTa{i}(B_j)$ via
\begin{align*}
\tTa{\rho i} (B_{\rho j})=\rho \big( \tTa{i}(B_j)\big).
\end{align*}
In particular, when $\rho=\tau,$ we have $\tTa{i}(B_{\tau j})=\widehat{\tau} \big( \tTa{i}(B_j)\big)$ by Remark~\ref{rmk:newb0}.
Accordingly, only one formula of $\tTa{\rho i}(B_{\rho j})$ and $\tTa{i}(B_ j)$ is included in the table; see types AII$_5$, EIV, and all types with $\tau \neq \Id$.
\end{remark}

\begin{remark}
The formulas of $\tTa{i}(B_j)$ only depend on the subdiagram generated by nodes $i,\tau i,j$ and the component of black nodes which is connected to either $i$ or $\tau i$. For example, the formula for $\tTa{2}(B_4)$ in type DIII$_5$ is formally identical to the formula for $\tTa{2}(B_4)$ in type AII$_5$. (Note that such a subdiagram may not be a Satake subdiagram as the vertex $\tau j$ is not included.)
\end{remark}

Recall $\tUi$ is defined over an extension field $\bF$ of $\Q(q)$. Denote
 \begin{align}
   \label{eq:UiQ}
 _{\Q}\tUi:= \text{ $\Q(q)$-subalgebra of $\tUi$ generated by $B_i,\tk_i,x$ for $i\in \wI,x\in \tbcG$.}
 \end{align}

 \begin{proposition}
    \label{prop:Q}
The symmetries $\tTa{i}$ ($i\in \wI$) preserve the $\Q(q)$-algebra $_{\Q}\tUi$.
 \end{proposition}

\begin{proof}
This follows by the formula for $\tTa{i}$ acting on the Cartan part in Proposition~\ref{prop:Cartanblack} (see Lemma~ \ref{lem:Z}), the rank one formulas in \eqref{eq:rkone2}, and the rank 2 formulas in Table~\ref{table:rktwoSatake}.
\end{proof}

\begin{remark}
It would cause no difficulty if we have replaced $\tUi$ (over $\mathbb F$) by $_{\Q}\tUi$ over $\Q(q)$ throughout the paper. We need to work with $\tU$ over $\Q(q^{\frac12})$ in several places. The results for $\Ui_{\bvs_\dm}$ will be valid over $\Q(q)$, while some results over $\Ui_\bvs$, for
$\bvs$ over $\Q(q)$, are valid over $\Q(q^{\frac12})$.
\end{remark}

\subsection{Proof of Theorem~\ref{thm:rktwo1}}
 \label{subsec:proofTiBj}


\begin{proposition}
  \label{prop:rktwoRij}
Let $i,j\in \wItau$ be such that $j \not \in \{i, \tau i\}$. Then there exists a non-commutative polynomial $R_{ij}(x_i,x_{\tau i}, y_i, y_{\tau i}, z ;\tbcG)$, which is linear in $z$, such that
\begin{enumerate}
\item
$\tT_{\bs_i}^{-1}(F_j)= R_{ij}(B_i^\sigma,B_{\tau i}^\sigma, \ck_i, \ck_{\tau i}, F_j;\tbcG)$;
\item
$\tT_{\bs_i}^{-1}(\tT_{\bw}\big(E_{\tau j})K_j'\big)=R_{ij}(B_i, B_{\tau i}, \ck_i, \ck_{\tau i}, \tT_{\bw}(E_{\tau j}) K_j';\tbcG)$.
\end{enumerate}
\end{proposition}

\begin{remark}
 \label{rem:fixed}
In case $\tau i =i$, the polynomials $R_{ij}$  depend only on $x_i, y_i, z$ and $\tbcG$. In this case, it is understood in Proposition~\ref{prop:rktwoRij} that $R_{ij}(x_i,x_{\tau i}, y_i, y_{\tau i}, z; \tbcG)$ is replaced by $R_{ij}(x_i, y_i, z; \tbcG)$ (which is linear in $z$), and    $R_{ij}(B_i^\sigma,B_{\tau i}^\sigma, \ck_i, \ck_{\tau i}, F_j;\tbcG)$ is replaced by $R_{ij}(B_i^\sigma, \ck_i, F_j;\tbcG)$, and so on.
\end{remark}

The proof of Proposition~\ref{prop:rktwoRij} will be carried out through type-by-type computation in Appendix~\ref{app1}.

We define
\begin{align}
  \label{eq:TaiBj}
\tTa{i}(B_j) :=
\begin{cases}
R_{ij}(B_i,\ck_i, B_j;\tbcG),
& \text{ if } i=\tau i,
  \\
R_{ij}(B_i,B_{\tau i}, \ck_i,\ck_{\tau i}, B_j;\tbcG)
& \text{ if }i \neq \tau i.
\end{cases}
\end{align}
Clearly, we have $\tTa{i}(B_j) \in \tUi$; see Table \ref{table:rktwoSatake}.

The polynomials $R_{ij}$ in all types can be read off from Table \ref{table:rktwoSatake}.
For instance, in type AII$_5$ it reads as follows:
\begin{align*}
R_{ij}(x,y,z; \tbcG) =\big[[x,F_3],z\big]_q.
\end{align*}
In order to read $R_{ij}$ off from Table \ref{table:rktwoSatake}, one first needs to unravel $\tT_w$, for $w\in \bW$,  appearing in those formulas in terms of $E_j,F_j,K_j,K_j',j\in \bI$.

\begin{proof} [Proof of Theorem~\ref{thm:rktwo1}]
We start with a general comment. Originally, we computed the explicit formulas in Table ~\ref{table:rktwoSatake} type by type; see \S\ref{subsec:example} for examples in types BI, DI, and DIII$_4$. In the process, we observed that parts of the arguments can be streamlined a uniform formulation in  Proposition~\ref{prop:rktwoRij}, even though its proof requires quite some computations. We hope this uniform formulation helps to conceptualize the structures of the formulas for $\tT_{\bs_i}^{-1}(B_j)$.

We now prove Theorem~\ref{thm:rktwo1} using Proposition~\ref{prop:rktwoRij}.
Recall by Theorem~\ref{thm:fX1} that $\tfX_i B_i^\sigma \tfX_i^{-1}=B_i$ and $\tfX_i x \tfX_i^{-1}=x$ for $x\in \tU^{\imath 0} \tbU.$

For definiteness, let us assume that $i\neq\tau i$. (The case when $i=\tau i$ is similar using the interpretation of notation in Remark~\ref{rem:fixed}.) By Lemma~\ref{lem:rktwo1}, Proposition~\ref{prop:rktwoRij} and definition of $\tTa{i}(B_j)$ in \eqref{eq:TaiBj}, we have
\begin{align*}
\tfX_i \tT_{\bs_i}^{-1}(B_j)
&= \tfX_i\tT_{\bs_i}^{-1}(F_j)
+ \tfX_i \tT_{\bs_i}^{-1}(\tT_{\bw}\big(E_{\tau j})K_j'\big)
 \\
&= \tfX_i\tT_{\bs_i}^{-1}(F_j)
+   \tT_{\bs_i}^{-1}(\tT_{\bw}\big(E_{\tau j})K_j'\big) \tfX_i
 \\
&= \tfX_iR_{ij}(B_i^\sigma,B_{\tau i}^\sigma, \ck_i,\ck_{\tau i}, F_j;\tbcG)
 +  R_{ij}(B_i,B_{\tau i}, \ck_i,\ck_{\tau i}, \tT_{\bw}(E_{\tau j}) K_j';\tbcG) \tfX_i
 \\
&= R_{ij}(B_i,B_{\tau i}, \ck_i,\ck_{\tau i}, F_j;\tbcG) \tfX_i
 + R_{ij}(B_i,B_{\tau i}, \ck_i,\ck_{\tau i}, \tT_{\bw}(E_{\tau j}) K_j';\tbcG) \tfX_i
 \\
&= R_{ij}(B_i,B_{\tau i}, \ck_i,\ck_{\tau i}, B_j;\tbcG) \tfX_i
= \tTa{i}(B_j) \tfX_i,
\end{align*}
where the second last step follows form the linearity of $R_{ij}$ in its fifth component. This proves the desired identity \eqref{eq:twinij}, whence the theorem.
\end{proof}

\begin{conjecture}
  \label{conj:TiBjKM}
For $\tUi$ of Kac-Moody type, Proposition~\ref{prop:rktwoRij} remains valid.
\end{conjecture}
Assume Conjecture~\ref{conj:TiBjKM} holds. Then
$\tTa{i}(B_j) \in \tUi$ defined in \eqref{eq:TaiBj} satisfies the intertwining relation \eqref{eq:twinij}, and hence, $\tTa{i}$ is a symmetry of $\tUi$ of Kac-Moody type.

 \subsection{A comparison with earlier results}
   \label{rkone2}

We compare our formulas with some special cases obtained in the literature.

By choosing a reduced expression of $\bw,$ we can write out the formula \eqref{eq:rkone2} explicitly for rank one Satake diagrams in Table~\ref{table:localSatake}. We list some explicit formulas of $\tTa{i}(B_i)$ and compare them with braid group actions obtained earlier in \cite{LW21a},\cite{Dob20},\cite{KP11}. (The index $i$ is specified in each case.) In some rank 2 cases, our formulas  differ from those in \cite{LW21a} and they can be matched by some twisting. As noted in \cite[Remark~7.4]{LW21a}, the formulas for braid operators in \cite{KP11} may involve $\sqrt{v}$ and are related to those in \cite{LW21a} by some other twisting.

\subsubsection{Type AI$_1$}

We shall label the single white node in rank 1 type AI by $1$.
In this case, the formula \eqref{eq:rkone2} reads as follows:
\begin{align}
  \label{eq:rkone21}
\tTa{1}(B_1)=-q^{-2} B_1 \ck_1^{-1}=-q^{-2} B_1 \tk_1^{-1}.
\end{align}
Note also that $\vs_{1,\dm}=-q^{-2}$.
Applying the central reduction $\pi_{\bvs_\dm}^\imath$ to \eqref{eq:rkone21}, we have
$\TT_{1,\dm}^{-1}(B_1)= B_1 \in \Ui_{\bvs_\dm}.$
Our formula \eqref{eq:rkone21} of $\tTa{1}(B_1)$ coincides with the formula $\T_i^{-1}(B_i)$ in \cite[Lemma 5.1]{LW21a}. Our formulation of $\TT_{1,\dm}^{-1}(B_1)$ coincides with the formula $\tau_i^{-1}(B_i)$ given in \cite[(3.1)]{KP11} for $(\U,\Ui_{q^{-2}})$.

\subsubsection{Type AII$_3$}

The rank 1 Satake diagram of type AII is given by
 
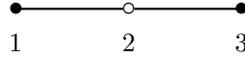
\begin{center}
\begin{tikzpicture}[scale=1.5,thick]
\node (1) at (-1,-0.2) {$\bullet $};
\node (2) at (0,-0.2) {$\circ $};
\node (3) at (1,-0.2) {$\bullet $};
\node at (-1,-0.5) {\small 1};
\node at (0,-0.5) {\small 2};
\node at (1,-0.5) {\small 3};
\draw[-] (-1,-0.2) edge (-0.05,-0.2);
\draw[-]  (1,-0.2) edge (0.05,-0.2);
\end{tikzpicture}
\end{center}
By Table~\ref{table:localSatake}, $\bs_2 = s_{2132}$, and the formula \eqref{eq:rkone2} reads as follows
\begin{align*}
\tTa{2}(B_2) &=  -q^{-2} (q-q^{-1})^2 \big[[B_{2},F_{3}]_q, F_{1}\big]_q E_{3} E_{ 1}  \tk_2^{-1} \\
& +(q-q^{-1})\big( [B_2,F_{3}]_q K_{ 1} E_{3} +  [B_2,F_{1}]_q K_{3} E_{ 1}  \big) \tk_2^{-1}
 -q^2 B_2 K_{3}K_{ 1} \tk_2^{-1}.
\end{align*}

\subsubsection{Type AIII$_{11}$}

The AIII$_{11}$ Satake diagram is given by
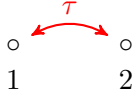
\begin{center}
\begin{tikzpicture}[scale=1.5,thick]
\node (1) at (-0.5,-0.2) {$\circ$};
\node (2) at (0.5,-0.2) {$\circ$};
\node  at (-0.5,-.5) {$ 1$};
\node  at (0.5,-.5) {$ 2$};
\path[bend left,<->,red] (1) edge node[above]{$\tau $} (2);
\end{tikzpicture}
\end{center}
In this case, the formula \eqref{eq:rkone2} reads as $\tTa{1}(B_1)=- B_2 \ck_2^{-1}=- B_2 \tk_2^{-1}$.

\subsubsection{Type AIII$_{11}$}

The rank 1 AIV Satake diagram is given by
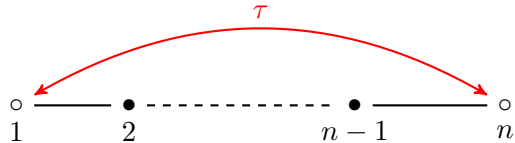
\begin{center}
\begin{tikzpicture}[scale=.5,thick]
\node (1) at (-6,-0.3) {$\circ$};
\node (2) at (-3,-0.3) {$\bullet$};
\node (3) at (3,-0.3) {$\bullet$};
\node (4) at (7,-0.3) {$\circ$};
\node  at (-6,-1) {$ 1$};
\node  at (-3,-1) {$ 2$};
\node  at (3,-1) {$ n-1 $};
\node  at (7,-1) {$ n$};
\path (1) edge (2);
\path (3) edge (4);
\path[bend left,<->,red] (1) edge node[above]{$\tau $} (4);
\begin{scope}[dashed]
\draw (2) -- (3);
\end{scope}
\end{tikzpicture}
\end{center}
In this case, the formula \eqref{eq:rkone2} reads as  $\tTa{1}(B_1)= -q \tT_{\bw}^2(B_1)  \ck_1^{-1} \prod_{j\in \bI} {K_j'}^{-1}$.

 \begin{remark}
 For type AIV, Dobson \cite[Theorem 3.4]{Dob20} obtained a different automorphism $\cT_1$ on $\Ui_{\bvs}$ such that
$\cT_{1}^{-1}(B_1)= q B_1 k_{n} K_{\varpi_{n-1}-\varpi_2}.$
Here $\varpi_j$ are the fundamental weights and $k_i$ is denoted by $L_i$ {\em loc. cit.}
 \end{remark}

\subsubsection{Split type}

The formulas of $\tTa{i}(B_j)$ in the split types AI$_2$, CI$_2$ and G$_2$ are identical to the braid group operators obtained using the $\imath$Hall algebra approach, cf. \cite[Lemma 5.1]{LW21a}.

\subsubsection{Formulas on $\Ui_{\bvs_\dm}$}

Applying central reductions and isomorphisms $\phi_\bvs:\Ui_{\bvs_\dm}\cong \Ui_\bvs$ (see \S\ref{braid:Uibvs} below) to our formulas, we recover various formulas obtained for $\Ui_{\bvs}$  in \cite{KP11} in split types and type AII. 

\section{New symmetries $\tTb{i}$ on $\tUi$}
  \label{sec:Tb}

In this section, we introduce new symmetries $\tTb{i}$ on $\tUi$, for $i\in \wI$, via a new intertwining property using the quasi $K$-matrix, and establish explicit formulas of $\tTb{i}$ acting on the generators of $\tUi$. Then we
show that $\tTa{i}$ and $\tTb{i}$ are mutual inverses. (This in particular completes the proof of Theorem~\ref{thm:newb0} that $\tTa{i}$ is an automorphism.)

 \subsection{Characterization of $\tTb{i}$}
 \label{subsec:Tb}

We formulate $\tTb{i}$ below, as a variant of $\tTa{i}$ introduced in Theorem~\ref{thm:newb0}.

\begin{theorem}
 \label{thm:Tb}
Let $i\in \wI$.
\begin{enumerate}
\item
For any $x \in \tUi$, there is a unique element $x'' \in \tUi$ such that $x'' \tT_{\bs_i}(\tfX_i^{-1})= \tT_{\bs_i}(\tfX_i^{-1}) \tT_{\bs_i}(x)$.
\item
The map $x \mapsto x''$ defines an automorphism of the algebra $\tUi$, denoted by $\tTb{i}$.
\end{enumerate}
\end{theorem}

The strategy of proving Theorem~\ref{thm:Tb} is largely parallel to that of Theorem~\ref{thm:newb0} given in the previous sections.
We shall prove Theorem~\ref{thm:Tb}(1) and a weaker version of Part (2) that $x \mapsto x''$ defines an endomorphism $\tTb{i}$ of the algebra $\tUi$, by combining Proposition~\ref{prop:Cartanblack2}, Proposition~\ref{prop:TiBi}, and Theorem~\ref{thm:TiBj}. Finally, we show that $\tTb{i}$ is an automorphism of $\tUi$ in Theorem~\ref{thm:newb1}.

Hence $\tTb{i}$ satisfies the following intertwining relation:
\begin{align}
  \label{eq:newb0-1}
\tTb{i}(x) \tT_{\bs_i}(\tfX_i^{-1}) = \tT_{\bs_i}(\tfX_i^{-1}) \tT_{\bs_i}(x),
\quad \text{ for all } x\in \tUi.
\end{align}


%
%
\subsection{Action of $\tTb{i}$ on $\tU^{\imath 0} \tbU$}
  \label{sub:TbCartan}

Just as for Proposition~\ref{prop:Cartanblack}, we can prove the following.
\begin{proposition}
  \label{prop:Cartanblack2}
Let $i \in \wI$. For each $x \in \tU^{\imath 0} \tbU$, there is a unique element $\tTb{i}(x)\in \tU^{\imath 0} \tbU$ such that the intertwining relation
$\tTb{i}(x) \tT_{\bs_i}(\tfX_i^{-1})= \tT_{\bs_i}(\tfX_i^{-1}) \tT_{\bs_i}(x)$ holds; see \eqref{eq:newb0-1}. More explicitly, 
\begin{align*}
\tTb{i}(u)& =(\widehat{\tau}_{\bullet,i} \circ \widehat{\tau} )(u),\quad
\tTb{i}(\tk_{j,\dm}) =\tk_{\bs_i \alpha_j,\dm},
\qquad
\text{ for } u\in \tbU \text{ and } j\in \wI.
\end{align*}
\end{proposition}
It follows by Proposition~\ref{prop:Cartanblack} and Proposition~\ref{prop:Cartanblack2} that $\tTa{i}, \tTb{i},$ and $\tT_{\bs_i}^{\pm 1}$ coincide on $\tU^{\imath 0} \tbU$. In particular, we have
\begin{align}
 \label{eq:TTCartan}
\tTb{i}(x)=(\sigma^\imath \circ \tTa{i} \circ \sigma^\imath)(x),
\qquad
\text{ for }
x\in \tU^{\imath 0} \tbU.
\end{align}

\subsection{Rank 1 formula for $\tTb{i}(B_i)$}
  \label{subsec:TbRank1} 

We shall establish a uniform formula for $\tTb{i}(B_i), $ for $i\in \wI$, a counterpart of Theorem~\ref{thm:rkone1}. Recall the anti-involution $\sigma^\imath$ of $\tUi$ from Proposition~\ref{prop:newb1}.

\begin{proposition}
  \label{prop:TiBi}
Let $i\in \wI$. There exists a unique element $\tTb{i}(B_i)\in \tUi$ which satisfies the following intertwining relation (see \eqref{eq:newb0-1})
\begin{align}
   \label{eq:rkone16}
\tTb{i}(B_i) \, \tT_{\bs_i}(\tfX_i)^{-1} = \tT_{\bs_i}(\tfX_i)^{-1} \, \tT_{\bs_i}(B_i).
\end{align}
More explicitly, we have
\begin{align}
  \label{eq:rkone20}
\tTb{i}(B_i)=-q^{-(\alpha_i, \alpha_{i})}  \tT_{w_\bullet}^{-2}( B_{\tau_{\bullet,i} \tau i})\tT_{\bw}(\ck_{\tau_{\bullet,i}  i}^{-1}).
\end{align}
In particular, we have $\tTb{i}(B_i)=(\sigma^\imath \circ \tTa{i})(B_i)$.
\end{proposition}

\begin{proof}
By Theorem~\ref{thm:fX1} (applied to the rank 1 setting), we have $B_i  \tfX_i = \tfX_i B_i^{\sigma}$, which can be rewritten as
\begin{align}
  \label{eq:rkone17}
\tT_{\bs_i}(\tfX_i) \tT_{\bs_i}(B_i^\sigma)=\tT_{\bs_i}(B_i) \tT_{\bs_i}(\tfX_i).
\end{align}
Hence, by comparing \eqref{eq:rkone16} and \eqref{eq:rkone17} and then applying \eqref{eq:sTs}, we obtain that
\begin{align}\label{eq:rkone18}
\tTb{i}(B_i) =\tT_{\bs_i}(B_i^\sigma)
=(\sigma\circ\tT_{\bs_i}^{-1})(B_i).
\end{align}

We now convert the formula \eqref{eq:rkone18} to the desired formula \eqref{eq:rkone20} for $\tTb{i}(B_i)$, which particularly shows that $\tTb{i}(B_i) \in \tUi$. To that end, note that $\sigma(\ck_{\tau_{\bullet,i}\tau i})= \tT_{\bw}(\ck_{\tau_{\bullet,i}  i})$, by Proposition~\ref{prop:QG4} and definition \eqref{def:Ki} of $\ck_i$. Applying $\sigma$ to the identity $\tT_{\bs_i}^{-1}(B_i) =-q^{-(\alpha_i,\bw\alpha_{\tau i}) } \tT_{w_\bullet}^2 ( B_{\tau_{\bullet,i} \tau i}^{\sigma})\ck_{\tau_{\bullet,i} \tau i}^{-1}$ in \eqref{eq:TBi} and using \eqref{eq:rkone18}, we have established the formula \eqref{eq:rkone20} for $\tTb{i}(B_i)$.

It remains to show that $\tTb{i}(B_i)=(\sigma^\imath \circ \tTa{i})(B_i)$. Recall $(\sigma^\imath)^2 =1$. 
Indeed, we have
\begin{align*}
\tTb{i}(B_i)
&\overset{\eqref{eq:rkone18}}{=}(\sigma\circ\tT_{\bs_i}^{-1})(B_i)
\\
&\overset{(*)}{=}(\sigma\circ \ad_{\tfX_i^{-1}}\circ \tTa{i})(B_i)
\\
&\overset{(\dag)}{=}\big(\sigma^\imath \circ \tTa{i}\big)(B_i),
\end{align*}
where (*) follows by Theorem~\ref{thm:rkone1}, and (\dag) follows by applying \eqref{eq:newb1} to the rank 1 Satake subdiagram associated with $i$.
\end{proof}

\subsection{Rank 2 formulas for $\tTb{i}(B_j)$}

The following lemma is a reformulation of Lemma~\ref{lem:rktwo1}.

\begin{lemma}
\label{lem:rktwo4b}
We have
\begin{itemize}
\item[(1)] $\tT_{\bs_i}(F_j)$ commutes with $\tT_{\bs_i}(\tfX_i)$.
\item[(2)] $\tT_{\bw}(E_{\tau j}) K_j'$ commutes with $\tT_{\bs_i}(\tfX_i)$.
\end{itemize}
\end{lemma}

Introduce a shorthand notation
\begin{align}
 \label{eq:Bhat}
\widehat{B}_i :=\tT_{\bs_i}\big(\tTa{i}(B_i) \big).
\end{align}
We reformulate the intertwining relation \eqref{eq:twinij} as
\begin{align}
 \label{eq:appB1}
\widehat{B}_i \cdot \tT_{\bs_i}(\tfX_i) = \tT_{\bs_i}(\tfX_i) \cdot B_i.
\end{align}

\begin{proposition}
  \label{prop:rktwoRijb}
Let $i\neq j\in \wItau$ be such that $j \not \in \{i, \tau i\}$. Then there exists a non-commutative polynomial $P_{ij}(x_i,x_{\tau i}, y_i, y_{\tau i}, z;\tbcG)$, which is linear in $z$, such that
\begin{enumerate}
\item
$\tT_{\bs_i} (F_j) = P_{ij}(B_i, B_{\tau i}, \tk_i, \tk_{\tau i}, F_j; \tbcG)$,
\item
$\tT_{\bs_i} \big(\tT_{\bw}(E_{\tau j})K_j'\big) = P_{ij}(\widehat{B}_i,\widehat{B}_{\tau i},\tk_i, \tk_{\tau i},\tT_{\bw}(E_{\tau j})K_j'; \tbcG)$.
\end{enumerate}
\end{proposition}

The proof of Proposition~\ref{prop:rktwoRijb} is carried out through a type-by-type computation similar to Appendix~\ref{app1} (the detail can be found in Appendix~B in an arXiv version).

We set
\begin{align}\label{eq:appB2}
\tTb{i}(B_j) := P_{ij}(B_i,B_{\tau i},\tk_i,\tk_{\tau i}, B_j; \tbcG).
\end{align}
Clearly, we have $\tTb{i}(B_j)  \in \tUi$.

\begin{theorem}
 \label{thm:TiBj}
Let $i\neq j \in \wItau$. The elements $\tTb{i}(B_j)$ listed in Table~ \ref{table:rktwoSatake2} satisfy the following intertwining relation (see \eqref{eq:newb0-1})
\begin{align}
 \label{eq:TbiBj}
 \tTb{i}(B_j)  \tT_{\bs_i}(\tfX_i)^{-1} = \tT_{\bs_i}(\tfX_i)^{-1} \tT_{\bs_i}(B_j).
\end{align}
\end{theorem}

\begin{proof}
Recall $B_j =F_j +\tT_{\bw}(E_{\tau j})K_j'$.
By Lemma~\ref{lem:rktwo4b}, \eqref{eq:appB1} and \eqref{eq:appB2}, we have
\begin{align*}
&\tT_{\bs_i}(\tfX_i)^{-1} \cdot \tT_{\bs_i}(B_j)\cdot \tT_{\bs_i}(\tfX_i)
\notag \\
&=\tT_{\bs_i}(\tfX_i)^{-1} \Big(\tT_{\bs_i}(F_j)+\tT_{\bs_i} \big(\tT_{\bw}(E_{\tau j})K_j'\big) \Big) \tT_{\bs_i}(\tfX_i)
\\
&= \tT_{\bs_i}(\tfX_i)^{-1}\tT_{\bs_i}(F_j)\tT_{\bs_i}(\tfX_i)
+ \tT_{\bs_i}(\tfX_i)^{-1}\tT_{\bs_i}\big(\tT_{\bw}(E_{\tau j})K_j'\big)\tT_{\bs_i}(\tfX_i)
\\
&=P_{ij}(B_i,B_{\tau i},\tk_i, \tk_{\tau i},F_j; \tbcG)
+ P_{ij}(B_i,B_{\tau i},\tk_i, \tk_{\tau i},\tT_{\bw}(E_j)K_j' ; \tbcG)
\\
&=P_{ij}(B_i,B_{\tau i},\tk_i, \tk_{\tau i}, B_j; \tbcG)
\\
& =\tTb{i}(B_j),
\end{align*}
where the linearity of the polynomial $P_{ij}$ with respect to the fifth variable is used in the last step.
This proves the desired intertwining property \eqref{eq:TbiBj} and whence the theorem.
\end{proof}

\subsection{$\tTT_{i,e}'$ and $\tTT''_{i,-e}$ as inverses }

Recall the automorphisms $\tTa{i} \in \Aut (\tUi)$ by Theorem~\ref{thm:newb0}. Recalling the bar involution $\tpsi^\imath$ on $\tUi$ from Proposition~\ref{prop:newb3}, we define two more automorphisms $ \tTT''_{i,-1},\tTT'_{i,+1} \in \Aut (\tUi)$ via
\begin{align}
  \label{def:tTT2}
\tTT''_{i,-1}:= \tpsi^\imath \circ \tTT''_{i,+1} \circ \tpsi^\imath,\qquad \tTT'_{i,+1}:= \tpsi^\imath \circ \tTT'_{i,-1} \circ \tpsi^\imath.
\end{align}

Recall that Lusztig's symmetries $\tTD'_{i,e} $ and $\tTD''_{i,-e}$ are mutually inverses, for $i\in \I,e=\pm 1$; see \cite[37.1.2]{Lus93}. They in addition satisfy the relation $\tTD'_{i,-1} =\sigma \circ \tTD''_{i,+1} \circ \sigma$; see \eqref{eq:sTs}. We prove the following $\imath$-analog of Lusztig's symmetries.

\begin{theorem}
 \label{thm:newb1}
$\tTT_{i,e}'$ and $\tTT''_{i,-e}$ are mutually inverse automorphisms on $\tUi$, for $e=\pm 1,i\in \wI$.
Moreover, we have
\begin{align}
  \label{eq:Tsigma}
\tTT_{i,e}' =\sigma^\imath \circ \tTT''_{i,-e} \circ \sigma^\imath.
\end{align}
 \end{theorem}

\begin{proof}
By definition \eqref{def:tTT2},
 $\tTT''_{i,-1}= \tpsi^\imath \tTT''_{i,+1} \tpsi^\imath,$ and $\tTT'_{i,+1}= \tpsi^\imath \tTT'_{i,-1} \tpsi^\imath$. Hence,
 it suffices to show that $\tTT_{i,-1}'$ and $\tTT''_{i,+1}$ are mutually inverses.

We already knew that $\tTa{i}:\tUi \rightarrow \tUi$ is an injective endomorphism.
Let us now prove that this endomorphism $\tTa{i}$ is surjective.
More precisely, 
we shall show the following.

{\bf Claim.}  For any $z \in \tUi$, set  $y :=\tTb{i}(z)$. Then we have $z=\tTa{i}(y)$.

Let us prove the claim. The identity \eqref{eq:newb0-1} reads in our setting as $ y^\imath \,\tT_{\bs_i}(\tfX_i^{-1})= \tT_{\bs_i}(\tfX_i^{-1}) \tT_{\bs_i}(z)$. Applying $\tT_{\bs_i}^{-1}$ to both sides of this identity, we obtain $\tT_{\bs_i}^{-1}(y^\imath) \tfX_i^{-1} = \tfX_i^{-1} z$, which can be rewritten as $z \tfX_i =  \tfX_i \tT_{\bs_i}^{-1}(y^\imath)$. By  \eqref{eq:newb0} and the uniqueness in Theorem~\ref{thm:newb0}(1), we conclude that $z=\tTa{i}(y)$.

By an entirely similar argument as above (switching the role of $\tTa{i}$ and $\tTb{i}$) and using the uniqueness in Theorem~\ref{thm:Tb}(1), we show that, for any $y_1 \in \tUi$, we have $y_1=\tTb{i}(z_1)$, where $z_1 :=\tTa{i}(y_1)$.

Hence $\tTa{i}$ and $\tTb{i}$ are mutually inverses. As $\tTa{i}$ is an endomorphism, we see that both $\tTa{i}$ and $\tTb{i}$ are automorphisms of $\tUi$.

Recall the anti-involuton $\sigma^\imath$ on $\tUi$ from Proposition~\ref{prop:newb1}. It remains to prove that $\tTb{i} =\sigma^\imath \circ \tTa{i} \circ \sigma^\imath$. This follows from the identity \eqref{eq:TTCartan}, the identity $\tTb{i}(B_i)=(\sigma^\imath \circ \tTa{i})(B_i)$ from Proposition~\ref{prop:TiBi}, and $\tTb{i}(B_j)=(\sigma^\imath \circ \tTa{i})(B_j)$, for $i\neq j \in \wItau$; the last identity follows by comparing the rank 2 formulas for $\tTa{i}(B_j)$ in Table~ \ref{table:rktwoSatake} and for $\tTb{i}(B_j)$ in Table~ \ref{table:rktwoSatake2}.
\end{proof}

In particular, Theorem~\ref{thm:newb1} above completes the proof of Theorem~\ref{thm:newb0} that $\tTa{i}$ are automorphisms of $\tUi$.
From now on, thanks to Theorem~\ref{thm:newb1}, we shall denote
\[
\tTT_i :=\tTT_{i,+1}'',
\qquad
\tTT_i^{-1} :=\tTT_{i,-1}'.
\]

\section{A basic property of new symmetries}
  \label{sec:BiBj}

In this section, we establish a basic property that $\tTT_{\underline{w}}$, for $w\in \reW$, sends $B_i$ to $B_j$, if $w \alpha_i =\alpha_j$; see Theorem~\ref{thm:fact1}. This is a generalization of a well-known property of braid group action on Chevalley generators in the setting of quantum groups.

We shall first study the rank 2 cases separately, depending on whether $\ell_\circ (\bbw) =3$, $4$, or $6$. Then we deal with the general cases.

\subsection{Rank 2 cases with $\ell_\circ (\bbw) =3$}

Assume that $\wItau=\{i,j\}$ such that $\ell_\circ (\bbw) =3$; in this case, according to Table~\ref{table:rktwodatum}, we must have $\tau =\mathrm{Id}$ and hence we identify $\wI=\{i,j\}$ as well.

\begin{lemma}
\label{lem:simple}
We have $\tT_{\bs_i} \tT_{\bs_j}(B_i )=B_j$.
\end{lemma}

\begin{proof}
Noting that $\ell(\bs_i \bs_j) = \ell(\bs_i)  + \ell(\bs_j)$, we have $\tT_{\bs_i} \tT_{\bs_j} =\tT_{\bs_i \bs_j}$.
Noting that
$ \bs_i \bs_j (\alpha_i) =\alpha_j$, we have that  $\tT_{\bs_i \bs_j}(X_i)=X_j$, for $X=F, E$ or $K'$;  cf. \cite[39.2]{Lus93} or \cite[Proposition 8.20]{Ja95}.

Recall $\tau =\mathrm{Id}$, and $B_i =F_i + \tT_{w_\bullet} (E_i) K_i'$. Thanks to \eqref{def:bsi}, $\tT_{w_\bullet}$ commutes with both $\tT_{\bs_i}$ and $\tT_{\bs_j}$. Therefore, we have
\[
\tT_{\bs_i} \tT_{\bs_j}(B_i )
= \tT_{\bs_i \bs_j}(F_i + \tT_{w_\bullet} (E_i) K_i')
= F_j + \tT_{w_\bullet} (E_j) K_j'
= B_j.
\]
The lemma is proved.
\end{proof}

 \begin{proposition}
   \label{prop:fact1}
 We have $\tTT_i^{-1} \tTT_j^{-1} (B_i)=B_j$; or equivalently, $\tTT_j \tTT_i (B_j)=B_i$.
 \end{proposition}

\begin{proof}
Since $\tTT_i^{-1}$ and $ \tTT_j^{-1}$ are automorphism of $\tUi$, we have $\tTT_i^{-1} \tTT_j^{-1} (B_i)-B_j\in \tUi$. Then we can write this element in terms of monomial basis of $\tUi$ (see Proposition~\ref{prop:QG5}):
\begin{align}\label{eq:fact-5}
\tTT_i^{-1} \tTT_j^{-1} (B_i)-B_j =\sum_{J\in \cJ} A_J B_J,\qquad \text{ for some } A_J\in \tbU^+ \tU^{\imath 0}.
\end{align}

On the other hand, using the intertwining relation \eqref{eq:newb0} twice, we have
\begin{align*}
\tTT_i^{-1} \tTT_j^{-1} (B_i)
&= \tfX_i \tT_{\bs_i}^{-1}(\tfX_j)\cdot \tT_{\bs_i}^{-1} \tT_{\bs_j}^{-1}(B_i)\cdot \tT_{\bs_i}^{-1}(\tfX_j^{-1}) \tfX_i^{-1}
\end{align*}
By Lemma~\ref{lem:simple}, we rewrite the above identity as
\begin{align}
\tTT_i^{-1} \tTT_j^{-1} (B_i)&=\tfX_i \tT_{\bs_i}^{-1}(\tfX_j)\cdot  B_j \cdot \tT_{\bs_i}^{-1}(\tfX_j^{-1}) \tfX_i^{-1}.\label{eq:fact-4}
\end{align}
By the equality \eqref{eq:fact-4}, we rewrite \eqref{eq:fact-5} in the following form
\begin{align}\label{eq:fact-6}
\tfX_i \tT_{\bs_i}^{-1}(\tfX_j)\cdot  B_j \cdot \tT_{\bs_i}^{-1}(\tfX_j^{-1}) \tfX_i^{-1}-B_j =\sum_{J\in \cJ} A_J B_J.
\end{align}

Now we claim $A_J B_J=0$, for each $J\in \cJ$, by comparing the weights in $\Z \I$. Recall from Remark \ref{rmk:fX2} that $\tfX_i=\sum_{m\geq 0} \tfX_i^m$ where $\wt(\tfX_i^m)=m(\alpha_i+\bw \alpha_{\tau i} )$ and then weights of $\tT_{\bs_i}^{-1}(\tfX_j)$ lie in $\N (\bs_i\alpha_j + \bs_i \bw\alpha_{\tau j})$. Hence, the weights appearing on LHS \eqref{eq:fact-6} must belong to the set $Q_{ij}$, where
\begin{align*}
& Q_{ij}=Q_{ij}^- \cup Q_{ij}^+,\\
& Q_{ij}^-:=  - \alpha_j+\N(\alpha_i+\bw \alpha_{\tau i}) + \N (\bs_i\alpha_j + \bs_i \bw\alpha_{\tau j}), \\
& Q_{ij}^+:=  \bw( \alpha_j ) + \N(\alpha_i+\bw \alpha_{\tau i}) + \N (\bs_i\alpha_j + \bs_i \bw\alpha_{\tau j} ) .
\end{align*}

On the other hand, note that the weight of the lowest weight component of $A_J B_J$ lies in $Q_J := -\wt (J)+\N \bI $. Then $A_J B_J\neq 0$ only if $Q_J\cap Q_{ij}\neq \varnothing$. It immediately follows that $A_J B_J=0$ unless $\wt(J)\in \alpha_j+ \N\bI$. Moreover, when $\wt(J)\in \alpha_j+ \N\bI$, the only possible element in the intersection $Q_J\cap Q_{ij} $ is $-\alpha_j$.

However, since $\tfX_i\tT_{\bs_i}^{-1}(\tfX_j)$ has constant term $1$, the weight $(-\alpha_j)$ component for LHS \eqref{eq:fact-6} is $0$. This implies that $A_J B_J=0$, for each $J\in \cJ$, and then the desired identity follows by \eqref{eq:fact-5}.
\end{proof}

\begin{corollary}
We have
\begin{align}\label{eq:fact-7}
\tfX_i \tT_{\bs_i}^{-1}(\tfX_j) B_j =B_j \tfX_i\tT_{\bs_i}^{-1} (\tfX_j).
\end{align}
\end{corollary}

\begin{proof}
One reads off from the proof of Proposition~\ref{prop:fact1} that
$A_J B_J=0$, for $J\in \cJ$, and hence the corollary follows from the relation \eqref{eq:fact-6}.
\end{proof}

\begin{corollary}
\label{cor:fact1}
We have
 \begin{align}
   B_i  \tfX_j \tT_{\bs_i}(\tfX_j) & = \tfX_j \tT_{\bs_i}(\tfX_j) B_i,
   \label{eq:cor-1} \\
   B_j^\sigma\tT_{\bs_i}(\tfX_j) \tfX_i & = \tT_{\bs_i}(\tfX_j) \tfX_i B_j^\sigma.
     \label{eq:cor-2}
 \end{align}
\end{corollary}

\begin{proof}
Switching $i,j$ in \eqref{eq:fact-7},  we obtain
 \begin{align}\label{eq:fact-8}
\tfX_j \tT_{\bs_j}^{-1}(\tfX_i) B_i =B_i \tfX_j\tT_{\bs_j}^{-1} (\tfX_i).
 \end{align}
By Proposition~\ref{prop:tau0}, we have $\tT_{\bs_j}^{-1} (\tfX_i) =\tT_{\bs_i}(\tfX_j)$. Hence, \eqref{eq:fact-8} implies the desired identity \eqref{eq:cor-1}.

Recall from Proposition~\ref{prop:inv} that $\tfX_i,\tfX_j$ are both fixed by the anti-involution $\sigma$. Recall also that $\sigma \tT_{\bs_i}^{-1} \sigma=\tT_{\bs_i} $. Applying the anti-involution $\sigma$ to the identity \eqref{eq:fact-7}, we have proved \eqref{eq:cor-2}.
\end{proof}

\subsection{Rank 2 cases with $\ell_\circ (\bbw) =4$}

In this subsection, we assume that $\wItau=\{i,j\}$ such that $\ell_\circ (\bbw) =4$. Let $\{i,\tau i\}$ and $\{j,\tau j\}$ be the corresponding two distinct $\tau$-orbits of $\wI$.

\begin{lemma}
 \label{lem:rho4}
Denote the diagram involution $\varrho := \tau_0 \tau_{\bullet,i}$. Then we have
\[
\bs_j \bs_i\bs_j (\alpha_i) =\alpha_{\varrho i},
\quad\text{ and }
\quad
\tT_{\bs_j}\tT_{\bs_i} \tT_{\bs_j}(B_{i}) =B_{\varrho i}.
\]
(Moreover, a nontrivial $\varrho$ can occur only in type AIII, and in this case, $\varrho =\tau$.)
\end{lemma}

\begin{proof}

 As before, set $w_0$ to be the longest element of the Weyl group $W$ and $w_{\bullet,i}=\bs_i w_\bullet$; set $\tau_0$ and $\tau_{\bullet,i}$ to be the diagram automorphisms corresponding to $w_0$ and $w_{\bullet,i}$, respectively. In this case, $w_0$ satisfies the relation $w_0=\bbw w_\bullet=\bs_j \bs_i \bs_j \bs_i w_\bullet=\bs_j \bs_i \bs_j w_{\bullet,i}$. Then we have
\begin{align*}
    \tau_0 (\alpha_i)
    = -w_0(\alpha_i)
    = -\bs_j \bs_i \bs_j w_{\bullet,i} (\alpha_i)
    = \bs_j \bs_i \bs_j \tau_{\bullet,i} (\alpha_i).
\end{align*}
Setting $\varrho:= \tau_0 \tau_{\bullet,i}$, we have obtained $\bs_j \bs_i\bs_j (\alpha_i) =\alpha_{\varrho i}$. (We thank Stefan Kolb for providing the above conceptual argument which replaces our earlier case-by-case proof of the existence of $\varrho$; moreover, his argument produces a precise formula for $\varrho$.)

In particular, we observe that a nontrivial $\varrho$ occurs only in type AIII (for some particular $i$), and in this case, $\varrho =\tau$.

Recalling $\bs_i =\bs_{\tau i}$, we also have $\bs_j \bs_i\bs_j (\alpha_{\tau i}) =\alpha_{\varrho \tau i}$.

We have $\ell (\bs_j \bs_i\bs_j) = \ell (\bs_j) +\ell(\bs_i) +\ell(\bs_j)$, by Proposition~\ref{prop:LL}. Therefore, it follows from $\bs_j \bs_i\bs_j (\alpha_i) =\alpha_{\varrho i}$ that $\tT_{\bs_j}\tT_{\bs_i} \tT_{\bs_j}(X_i) =X_{\varrho i}$, for $X= F, K'$;  cf. \cite[39.2]{Lus93} or \cite[Proposition 8.20]{Ja95}. Similarly, we have $\tT_{\bs_j}\tT_{\bs_i} \tT_{\bs_j}(E_{\tau i}) =E_{\varrho \tau i}$.

Recall $B_i =F_i + \tT_{w_\bullet} (E_{\tau i}) K_i'$. Thanks to \eqref{def:bsi}, $\tT_{w_\bullet}$ commutes with both $\tT_{\bs_i}$ and $\tT_{\bs_j}$. Therefore, we have
\[
 \tT_{\bs_j}\tT_{\bs_i} \tT_{\bs_j}(B_i )
= \tT_{\bs_j \bs_i \bs_j}(F_i + \tT_{w_\bullet} (E_{\tau i}) K_i')
= F_{\varrho i} + \tT_{w_\bullet} (E_{\varrho \tau i}) K_{\varrho i}'
= B_{\varrho i}.
\]
The lemma is proved.
\end{proof}

\begin{proposition}
  \label{prop:fact2}
Retain the notation in Lemma~\ref{lem:rho4}. Then $\tTT_{j}^{-1} \tTT_i^{-1} \tTT_{j}^{-1} (B_{i}) =B_{\varrho i}$; or equivalently,  $\tTT_{ j} \tTT_i \tTT_{j}(B_{i})=B_{\varrho i}$.
\end{proposition}

\begin{proof}
Since $\tTT_i^{-1}$ and $ \tTT_j^{-1}$ are automorphism of $\tUi$, we have $\tTT_j^{-1} \tTT_i^{-1} \tTT_j^{-1} (B_i)-B_{\varrho i}\in \tUi$. Then we can write this element in terms of monomial basis of $\tUi$ (see Proposition~\ref{prop:QG5}):
\begin{align}\label{eq:fact-10}
\tTT_j^{-1} \tTT_i^{-1} \tTT_j^{-1} (B_i)-B_{\varrho i} =\sum_{J\in \cJ} A_J B_J,\qquad \text{ for some } A_J\in \tbU^+ \tU^{\imath 0}.
\end{align}

On the other hand, using the intertwining relation \eqref{eq:newb0} of $\tTT_i^{-1}$, we have
\begin{align*}
 &\tTT_j^{-1} \tTT_i^{-1} \tTT_j^{-1} (B_i)\\
&= \tfX_j  \, \tT_{\bs_j}^{-1}(\tfX_i) \, \tT_{\bs_j}^{-1}\tT_{\bs_i}^{-1}(\tfX_j)\cdot \tT_{\bs_j}^{-1}\tT_{\bs_i}^{-1}  \tT_{\bs_j}^{-1}(B_i)\cdot \tT_{\bs_j}^{-1}\tT_{\bs_i}^{-1}(\tfX_j^{-1}) \,  \tT_{\bs_j}^{-1}(\tfX_i^{-1}) \, \tfX_j^{-1}.
\end{align*}
Since $\tT_{\bs_j}^{-1}\tT_{\bs_i}^{-1} \tT_{\bs_j}^{-1}(B_i)=B_{\varrho i}$ by Lemma~\ref{lem:rho4}, we rewrite the above identity as
\begin{align}
\tTT_j^{-1} \tTT_i^{-1} \tTT_j^{-1} (B_i)
= \tfX_j  \, \tT_{\bs_j}^{-1}(\tfX_i)  \, \tT_{\bs_j}^{-1}\tT_{\bs_i}^{-1}(\tfX_j)\cdot B_{\varrho i} \cdot \tT_{\bs_j}^{-1}\tT_{\bs_i}^{-1}(\tfX_j^{-1})  \, \tT_{\bs_j}^{-1}(\tfX_i^{-1}) \,  \tfX_j^{-1}.
\label{eq:fact-11}
\end{align}
By the identity \eqref{eq:fact-11}, we rewrite \eqref{eq:fact-10} in the following form
\begin{align}\label{eq:fact-12}
 \tfX_j  \, \tT_{\bs_j}^{-1}(\tfX_i)  \, \tT_{\bs_j}^{-1}\tT_{\bs_i}^{-1}(\tfX_j)\cdot B_{\varrho i} \cdot \tT_{\bs_j}^{-1}\tT_{\bs_i}^{-1}(\tfX_j^{-1})  \, \tT_{\bs_j}^{-1}(\tfX_i^{-1}) \,  \tfX_j^{-1} - B_{\varrho i} =\sum_{J\in \cJ} A_J B_J.
\end{align}

By a weight argument entirely similar to the proof of Proposition~\ref{prop:fact1}, we obtain $\sum_{J\in \cJ} A_J B_J=0$. Thus, the proposition follows by \eqref{eq:fact-10}.
\end{proof}

\begin{corollary}
We have
\begin{align}\label{eq:fact-13}
B_{\varrho i} \tfX_j \, \tT_{\bs_j}^{-1}(\tfX_i) \, \tT_{\bs_j}^{-1}\tT_{\bs_i}^{-1}(\tfX_j) = \tfX_j \,  \tT_{\bs_j}^{-1}(\tfX_i) \,  \tT_{\bs_j}^{-1}\tT_{\bs_i}^{-1}(\tfX_j) B_{\varrho i}.
\end{align}
\end{corollary}

\begin{proof}
Since $\sum_{J\in \cJ} A_J B_J=0$, as shown in the proof of Proposition~\ref{prop:fact2}, the corollary follows from the relation \eqref{eq:fact-12}.
\end{proof}

\begin{corollary}
\label{cor:fact2}
We have
\begin{align}
 B_i \tfX_j  \, \tT_{\bs_j}^{-1}(\tfX_i)  \,  \tT_{\bs_i}(\tfX_j) &=  \tfX_j \,  \tT_{\bs_j}^{-1}(\tfX_i)  \, \tT_{\bs_i}(\tfX_j) B_i,
  \label{eq:fact-14}  \\
 B_j^\sigma \tT_{\bs_j}^{-1}(\tfX_i)  \, \tT_{\bs_i}(\tfX_j)  \, \tfX_i &= \tT_{\bs_j}^{-1}(\tfX_i)  \, \tT_{\bs_i}(\tfX_j)  \, \tfX_i B_j^\sigma.\label{eq:fact-15}
\end{align}
\end{corollary}

\begin{proof}

We prove \eqref{eq:fact-14}. Noting that $\varrho$ equals either $\Id$ or $\tau$, we have by Remark~\ref{rmk:newb0} that $\varrho$ commutes with $\tT_{\bs_i},\tT_{\bs_j} $, and by Proposition~\ref{prop:inv} that $\varrho$ fixes $\tfX_i,\tfX_j$. Hence, applying $\varrho$ to both sides of \eqref{eq:fact-13}, we have

\begin{align}\label{eq:fact-16}
B_{  i} \tfX_j \, \tT_{\bs_j}^{-1}(\tfX_i) \, \tT_{\bs_j}^{-1}\tT_{\bs_i}^{-1}(\tfX_j) = \tfX_j  \, \tT_{\bs_j}^{-1}(\tfX_i) \,  \tT_{\bs_j}^{-1}\tT_{\bs_i}^{-1}(\tfX_j) B_{  i}.
\end{align}
 By Proposition~\ref{prop:tau0}, we have $\tT_{\bs_j}^{-1}\tT_{\bs_i}^{-1}(\tfX_j) =\tT_{\bs_i}(\tfX_j)$. Hence, the desired relation \eqref{eq:fact-14} follows by \eqref{eq:fact-16}.

We next show \eqref{eq:fact-15}. Recall from Proposition~\ref{prop:inv} that $\tfX_i,\tfX_j$ are both fixed by the anti-involution $\sigma$. Recall also that $\sigma \tT_{\bs_i}^{-1} \sigma=\tT_{\bs_i} $. Switching $i,j$ in \eqref{eq:fact-14} and then applying $\sigma$ to it, we obtain \eqref{eq:fact-15}.
\end{proof}

 \subsection{Rank 2 case with $\ell_\circ(\bbw) =6$}
   \label{sub:split}

The rank 2 case with $\ell_\circ(\bbw) =6$ occurs only in split $G_2$ type.
Let $(\I=\wI,\mathrm{Id})$ be a Satake diagram of split type $G_2$. In this case, the relative Weyl group $\reW$ is identified with $W$ and $\bs_a=s_a$ for $a\in \I =\wI =\{i, j\}$. We do not specify which root $i$ or $j$ is long.

Set $\underline{w_i}= s_j s_i s_j s_i s_j$ and $\tT_{\underline{w_i}} = \tT_{j}\tT_{i}\tT_{j}\tT_{i}\tT_{j}$. Then we have $\underline{w_i}(\alpha_i)=\alpha_i$.

\begin{lemma}
\label{lem:rho5}
We have $\tT_{\underline{w_i}}^{-1}(B_i)=B_{i}.$
\end{lemma}

\begin{proof}
Follows by \cite[39.2]{Lus93} and the same type of arguments as for Lemmas~\ref{lem:simple} and  \ref{lem:rho4}.
\end{proof}

\begin{proposition}
  \label{prop:fact-G2}
    We have
$\tTT_{\underline{w_i}}^{-1} (B_i) =B_i$; or equivalently, $\tTT_{\underline{w_i}} (B_i) =B_i$.
 \end{proposition}

\begin{proof}
Since $\tTT_i^{-1}$ and $ \tTT_j^{-1}$ are automorphism of $\tUi$, we have $ \tTT_{\underline{w_i}}^{-1}  (B_i)-B_{ i}\in \tUi$. Then we can write this element in terms of monomial basis of $\tUi$ (see Proposition~\ref{prop:QG5}):
\begin{align}\label{eq:fact-17}
 \tTT_{\underline{w_i}}^{-1}  (B_i)-B_{i} =\sum_{J\in \cJ} A_J B_J,\qquad \text{ for some } A_J\in \tbU^+ \tU^{\imath 0}.
\end{align}

On the other hand, using the intertwining relation \eqref{eq:newb0} of $\tTT_i^{-1}$, we have
\begin{align}
\label{eq:O1}
 &\quad \tTT_{\underline{w_i}}^{-1}  (B_i)= \Omega_i \tT_{\underline{w_i}}^{-1}(B_i) \Omega_i^{-1},
\end{align}
where
\begin{align}
\label{def:Omegai}
\Omega_i= \tfX_j  \, \tT_{ j}^{-1}(\tfX_i)  \, \tT_{ j}^{-1}\tT_{i}^{-1}(\tfX_j) \,  \tT_{j}^{-1}\tT_{i}^{-1}\tT_{j}^{-1}(\tfX_i) \, \tT_{ j}^{-1}\tT_{i}^{-1}\tT_{ j}^{-1}\tT_{ i}^{-1}(\tfX_j).
\end{align}

By Lemma~\ref{lem:rho5}, we rewrite the identity \eqref{eq:O1} as
\begin{align}
&\quad \tTT_{\underline{w_i}}^{-1} (B_i)= \Omega_i B_i \Omega_i^{-1}.
\label{eq:fact-18}
\end{align}
By the identity \eqref{eq:fact-18}, we rewrite \eqref{eq:fact-17} in the following form
\begin{align}\label{eq:fact-19}
  \Omega_i B_i \Omega_i^{-1}  - B_{ i} =\sum_{J\in \cJ} A_J B_J.
\end{align}

By a weight argument entirely similar to the proof of Proposition~\ref{prop:fact1}, we obtain $\sum_{J\in \cJ} A_J B_J=0$. Thus, the proposition follows by \eqref{eq:fact-17}.
\end{proof}

\begin{corollary}
Let $\Omega_i$ be as in \eqref{def:Omegai}. We have
\begin{align}  \label{eq:fact-20}
B_{i} \Omega_i = \Omega_i B_{i}.
\end{align}
\end{corollary}

\begin{proof}
Since $\sum_{J\in \cJ} A_J B_J=0$, as shown in the proof of Proposition~\ref{prop:fact-G2}, the corollary follows from the formula \eqref{eq:fact-19}.
\end{proof}

\begin{corollary}
\label{cor:G2}
We have the following intertwining relations:
\begin{align}
\notag
& B_i \tfX_j \tT_{s_i s_j s_i s_j}(\tfX_i) \tT_{s_i s_j s_i}(\tfX_j) \tT_{s_i s_j}(\tfX_i)\tT_{i}(\tfX_j)\\
& \quad = \tfX_j \tT_{s_i s_j s_i s_j}(\tfX_i) \tT_{s_i s_j s_i}(\tfX_j) \tT_{s_i s_j}(\tfX_i)\tT_{i}(\tfX_j) B_i, \label{eq:fact-21}
\\\notag
& B_j^\sigma\tT_{ s_i s_j s_i s_j}(\tfX_i) \tT_{s_i s_j s_i}(\tfX_j) \tT_{s_i s_j}(\tfX_i) \tT_{i} (\tfX_j) \tfX_i \\
&\quad = \tT_{ s_i s_j s_i s_j}(\tfX_i) \tT_{s_i s_j s_i}(\tfX_j) \tT_{s_i s_j}(\tfX_i) \tT_{i} (\tfX_j) \tfX_i B_j^\sigma.\label{eq:fact-22}
\end{align}
\end{corollary}

\begin{proof}
By Proposition~\ref{prop:tau0}, we have $\tT_{s_j s_i s_j s_i s_j}(\tfX_i)=\tfX_i $ and $\tT_{s_i s_j s_i s_j s_i}(\tfX_j) =\tfX_j$. Then we have
\begin{align*}
\Omega_i=\tfX_j \tT_{s_i s_j s_i s_j}(\tfX_i) \tT_{s_i s_j s_i}(\tfX_j) \tT_{s_i s_j}(\tfX_i)\tT_{i}(\tfX_j).
\end{align*}
Hence, the desired identity \eqref{eq:fact-21} follows by \eqref{eq:fact-20}.

We next prove \eqref{eq:fact-22}. Switching $i,j$ in \eqref{eq:fact-20}, we have
\begin{align}
\label{eq:fact-23}
B_{j} \Omega_j = \Omega_j B_{j},
\end{align}
where $\Omega_j$ is defined by switching $i,j$ in \eqref{def:Omegai}.

Recall from Proposition~\ref{prop:inv} that $\tfX_i,\tfX_j$ are both fixed by $\sigma$. Then by the definition of $\Omega_j$, we have
\begin{align*}
\sigma(\Omega_j)=\tT_{ s_i s_j s_i s_j}(\tfX_i) \tT_{s_i s_j s_i}(\tfX_j) \tT_{s_i s_j}(\tfX_i) \tT_{i} (\tfX_j) \tfX_i.
\end{align*}
Hence, applying $\sigma$ to \eqref{eq:fact-23} and then using this formula of $\sigma(\Omega_j)$, we obtain \eqref{eq:fact-22}.
\end{proof}

\subsection{The general identity $\tTT_{w}(B_i)=B_{wi}$}

Let $w\in \reW$. Given a reduced expression $\underline{w} =\bs_{i_1} \bs_{i_2} \ldots \bs_{i_k}$ for $w$, we shall denote $\tTT_{\underline{w}} =\tTT_{i_1} \tTT_{i_2} \ldots \tTT_{i_k}$.

\begin{theorem}
  \label{thm:fact1}
Suppose that $wi\in \wI$, for $w\in \reW$ and $i\in \wI$. Then $\tTT_{\underline{w}}(B_i)=B_{wi}$, for {\em some} reduced expression $\underline{w}$ of $w$.
\end{theorem}
(Once Theorem~\ref{thm:newb2} on braid relation for $\tTT_i$ is proved, we can replace $\tTT_{\underline{w}}$ in Theorem~\ref{thm:fact1} by $\tTT_w$, which depends only on $w$, not on a reduced expression $\underline{w}$ of $w$.)

\begin{proof}
The strategy of the proof is modified from a well-known quantum group counterpart, cf. \cite[Lemma 8.20]{Ja95}. We shall reduce the proof to the rank 2 cases which were established earlier and finish the proof by induction on $\ell_\circ (w)$.

The statement holds for arbitrary rank 2 Satake (sub) diagrams $(\bI \cup \{i, \tau i, j, \tau j\}, \tau)$.
 Indeed in case when $\ell(\bbw) =2$, the claim is trivial. In case when $\ell(\bbw) =3,4$ or $6$, the claim has been established in Propositions~\ref{prop:fact1},  \ref{prop:fact2}, and  \ref{prop:fact-G2}, respectively.

In general, we use an induction on $l_\circ(w)$, for $w\in \reW$, where $\lc$ is the length function for the relative Weyl group $\reW$. Recall the simple system $\{\balpha_i|i\in \wItau\}$ for the relative root system from \eqref{def:balpha}. Since $w \theta = \theta w$ and $w i \in \wI$ by assumption, we have $w(\balpha_i)=\balpha_{wi}$. We denote a positive (and negative) root in the relative root system by $\beta>0$ (and respectively, $\beta<0$).

Suppose that $l_\circ(w)>0$. Then there exists $j\in \wItau$ such that $w(\balpha_j)<0$; clearly $j\neq i$ since $w(\balpha_i) >0$. Consider the minimal length representatives of $\reW$ with respect to the rank 2 parabolic subgroup $\langle \bs_i, \bs_j \rangle$. We have a decomposition $w=w' w''$ in $\reW$ such that $w'(\balpha_i)>0,w'(\balpha_j)>0$ and $w''$ lies in the subgroup $\langle \bs_i, \bs_j \rangle$; moreover, $\lc(w)=\lc(w')+\lc(w'')$. Now $w(\balpha_i)>0$ and $w(\balpha_j)<0$ implies that $w''(\balpha_i)>0$ and $w''(\balpha_j)<0$ (since $w'$ preserves the signs of the roots $w''(\balpha_i)$ and $w''(\balpha_j)$). It follows that
\begin{align*}
w''(\alpha_i)>0,\qquad w''(\alpha_j)<0,\qquad
w'(\alpha_i)>0,\qquad w'(\alpha_j)>0.
\end{align*}
(The positive system of the restricted root system is compatible with the positive system of $\cR$.) Moreover since $\bs_{s}$, for any $s\in \wI$, acts on $\bI$ as the involution $\tau_{\bullet,s} \tau$, we must have $w'(\alpha_a)>0$, for any $a \in \bI$; see also Proposition~\ref{prop:Cartanblack}.

We show that $w''i\in \wI $. Since $w''(\alpha_i)>0 $ and $w''(\alpha_i)\in \cR \cap (\Z\alpha_i + \Z \alpha_j + \Z\bI)$, we can write $w''(\alpha_i) \in \cR$ in the following form
\begin{align*}
w''(\alpha_i)= r \alpha_i + s \alpha_j +\alpha_\bullet
\end{align*}
for some $r,s\geq 0,\alpha_\bullet\in \N\bI$. We consider the following cases:
\begin{enumerate}
\item
At least two of $r,s,\alpha_\bullet$ are nonzero. Then $w'w''(\alpha_i)=r w'(\alpha_i) + s w'(\alpha_j) +w'(\alpha_\bullet)$ cannot be simple for $w'(\alpha_i)>0,w'(\alpha_j)>0,w'(\alpha_\bullet)>0$; this contradicts that $w (\alpha_i)=w'w''(\alpha_i)$ is simple.
\item
$r=0,\alpha_\bullet=0$ and $s>0$. Then $s=1$ and $w''(\alpha_i) =\alpha_j$ is simple. A similar argument applying to the case $s=0,\alpha_\bullet=0$ and $r>0$ shows that $w''(\alpha_i) =\alpha_i$ is simple.
\item
$r=s=0,\alpha_\bullet\neq 0$. We show this case cannot occur. Indeed, we have
$
\theta w''(\alpha_i) = \theta(\alpha_\bullet)=\alpha_\bullet=w''(\alpha_i).
$ 
Since $w'' \theta=\theta w''$, the above identity implies that $\alpha_i$ is fixed by $\theta$, which is impossible for $i\in \wI$.
\end{enumerate}

Therefore, we have shown $w''i\in \wI $ and $w''(\alpha_i)=\alpha_{w''i}$. By the rank two results in Propositions~\ref{prop:fact1} and  \ref{prop:fact2}, we have $\tTT_{\underline{w''}}(B_i) =B_{w''i}$, for any reduced expression $\underline{w''}$ of $w''$. Now using the induction hypothesis, there exists a reduced expression $\underline{w'}$ such that $\underline{w} = \underline{w'} \cdot \underline{w''}$ is a reduced expression for $w$ and
\begin{align*}
\tTT_{\underline{w}}(B_i)
=\tTT_{\underline{w'}} \tTT_{\underline{w''}} (B_i)
=\tTT_{\underline{w'}} (B_{w''i})=B_{wi}.
\end{align*}
The theorem is proved.
\end{proof}

\section{Factorization of quasi $K$-matrices}
  \label{sec:factor}

It is conjectured by Dobson and Kolb \cite{DK19} that quasi $K$-matrices admit factorization into products of rank 1 quasi $K$-matrices analogous to the factorization properties of quasi $R$-matrices. They showed that the factorization of quasi $K$-matrices for arbitrary finite types reduces to the rank two cases.
In this section, using (the rank 2 cases of) Theorem~\ref{thm:fact1} we provide a uniform proof of the factorization of quasi $K$-matrices for all rank two Satake diagrams, hence completing the proof of Dobson-Kolb conjecture in all finite types.

\subsection{Factorization of $\tfX$}

Let $(\I=\bI \cup \wI,\tau)$ be a Satake diagram of arbitrary finite type. Let $w$ be any element in the relative Weyl group $\reW$ with a reduced expression
\[
\underline{w} =\bs_{i_1}\bs_{i_2}\cdots \bs_{i_m};
\]
here $m =\ell_\circ(w)$, the length of $w \in \reW$ (not to be confused as the length $\ell(w)$ in $W$).

Following \cite{DK19} (who worked in the setting of $\Ui_\bvs$), we define, for $1\leq k\leq m,$
\begin{align}
  \label{eq:Upk}
\begin{split}
\tfX^{[k]} 
&= \tT_{\bs_{i_1}} \tT_{\bs_{i_2}} \cdots\tT_{\bs_{i_{k-1}}}(\tfX_{i_k}),
\\
\tfX_{\underline{w}} &=\tfX^{[m]}\tfX^{[m-1]}\cdots\tfX^{[1]}.
\end{split}
\end{align}
(In the notation $\tfX^{[k]}$ above, we have suppressed the dependence on $\underline{w}$.)

The goal of this section is to establish Theorem~\ref{thm:factor}, which is a $\tUi$-variant of (and implies) \cite[Conjecture~3.22]{DK19} for $\Ui_\bvs$ with general parameters $\bvs$. The restriction on parameters $\bvs$ in \cite{DK19} can be removed in light of the development in \cite{AV22, KY20} which allows more general parameters in quasi $K$-matrices. Recall that $\bbw$ is the longest element in the relative Weyl group $\reW$.

\begin{theorem}
  \label{thm:factor}
  \quad
  \begin{enumerate}
  \item
For any $w \in \reW$, the partial quasi $K$-matrix $\tfX_{\underline{w}}$ is independent of the choice of reduced expressions of $w$ (and hence can be denoted by $\tfX_w$).
\item
The quasi $K$-matrix $\tfX$ for $\tUi$ of any finite type admits a factorization $\tfX = \tfX_{\bbw}$.
\end{enumerate}
\end{theorem}
\subsection{Reduction to rank 2}

Let us recall some partial results from \cite{DK19} in this direction (which can be adapted from $\Ui_\bvs$ to $\tUi$ without difficulties).

\begin{theorem}
 \cite[Theorems~3.17 and 3.20]{DK19}
  \label{thm:BK}
Theorem~\ref{thm:factor} holds for $\tUi$ of a given finite type if it holds for all its rank 2 Satake subdiagrams.
\end{theorem}

The arguments for Theorem~\ref{thm:BK} are largely formal once the following crucial result (see \cite[Proposition 3.18]{DK19}) is in place. We provide a short new proof below. Recall $\bbw$ is the longest element in $\reW$. Recall also the diagram involution $\tau_0$ such that $w_0 (\alpha_i) = -\alpha_{\tau_0 \alpha_i}$, for all $i$, where $w_0$ is the longest element in $W$.

\begin{proposition}
 \cite[Proposition 3.18]{DK19}
   \label{prop:tau0}
 Let $\bbw =\bs_{i_1}\bs_{i_2}\cdots \bs_{i_m}$ be a reduced expression of $\bbw$.  Then we have
$
 \tT_{\bs_{i_1}} \tT_{\bs_{i_2}} \cdots\tT_{\bs_{i_{m-1}}}(\tfX_{i_m})
 =\tfX_{\tau_0 i_m}.
$
\end{proposition}

\begin{proof}
We have $w_0=\bbw \bw$, 
and hence, $\tT_{ w_0} =\tT_{\bbw} \tT_{w_{\bullet}}.$
It follows by Lemma~\ref{lem:braid1} that $\tT_{ w_0}^{-1} \widehat{\tau}_0 =\tT_{ w_{\bullet,i_m}}^{-1}  \widehat{\tau}_{\bullet,i_m}$ when acting on $\tU_{\I_{\bullet,i_m}}$. Thus,
\[
\tT_{ w_0}^{-1} \widehat{\tau}_0 (\tfX_{i_m})
=\tT_{ w_{\bullet,i_m}}^{-1}\widehat{\tau}_{\bullet,i_m} (\tfX_{i_m})
=\tT_{ w_{\bullet,i_m}}^{-1} (\tfX_{i_m}),
\]
since the quasi $K$-matrix $\tfX_{i_m}$ lies in a completion of $\tU^+_{\I_{\bullet,i_m}}$ and $\widehat{\tau}_{\bullet,i_m} (\tfX_{i_m}) =\tfX_{i_m}$ (see Proposition~\ref{prop:inv}). Then we obtain
\begin{align*}
\tT_{ w_0}^{-1} (\tfX_{\tau_0 i_m})
&=\tT_{ w_0}^{-1} \widehat{\tau}_0 (\tfX_{i_m})
=\tT_{ w_{\bullet,i_m}}^{-1} (\tfX_{i_m})
= \tT_{ \bs_{i_m}}^{-1} \tT_{w_{\bullet}}^{-1} (\tfX_{i_m})
= \tT_{ \bs_{i_m}}^{-1} (\tfX_{i_m}),
\end{align*}
where the last equality follows by Proposition~\ref{prop:newb-1}. By Proposition~\ref{prop:newb-1} again we have
\begin{align*}
\tT_{\bbw}^{-1} (\tfX_{\tau_0 i_m})
&=\tT_{ w_0}^{-1} \tT_{w_{\bullet}} (\tfX_{\tau_0 i_m})
=\tT_{ w_0}^{-1} (\tfX_{\tau_0 i_m})
= \tT_{ \bs_{i_m}}^{-1} (\tfX_{i_m}).
\end{align*}
Hence,
$ \tT_{\bs_{i_1}} \tT_{\bs_{i_2}} \cdots\tT_{\bs_{i_{m-1}}}(\tfX_{i_m})
 =\tT_{\bbw}\tT_{ \bs_{i_m}}^{-1}(\tfX_{i_m})
 =\tT_{\bbw} \tT_{\bbw}^{-1} (\tfX_{\tau_0 i_m})
= \tfX_{\tau_0 i_m}.$
\end{proof}

\begin{remark}
It was verified in \cite{DK19} that Theorem~\ref{thm:factor} holds in all type A rank 2 and all split rank 2 cases. The long computational proof therein is carried out  case-by-case based on several explicit rank 1 formulas which they also computed.
\end{remark}

We note that in the rank 2 setting the first statement in Theorem~\ref{thm:factor} is nontrivial only when $w=\bbw$, the longest element in $\reW$. Hence, in the remainder of this section, to prove Theorem~\ref{thm:factor} we can and shall assume that
\[
\text{ $(\I=\bI \cup \wI,\tau)$ is any {\em rank two} Satake diagram of finite type, and $w = \bbw$}.
\]
Moreover, we denote $\wI =\{i, \tau i, j, \tau j\}$.

Let $\bbw =\bs_{i_1}\bs_{i_2}\cdots \bs_{i_m}$ be a reduced expression.
Theorem~\ref{thm:factor} in the case for $\ell_\circ (\bbw) =2$, i.e., $\bbw = \bs_i \bs_j =\bs_j \bs_i$, trivially holds.
The next proposition reduces the proof of Theorem~\ref{thm:factor} in the remaining nontrivial cases into verifying its assumption.

\begin{proposition}
 \label{prop:UU}
Assume that $B_p \tfX_{\bbw} =\tfX_{\bbw} B_p^\sigma$, for $p =i,j.$
Then we have $\tfX =\tfX_{\bbw}$, for any reduced expression of ${\bbw}$.
\end{proposition}

\begin{proof}
The identity $x  \tfX_{\bbw} = \tfX_{\bbw} x,$ for $x\in \tU^{\imath 0}\tbU$, holds by \eqref{eq:fX1b}, Proposition~\ref{prop:Cartanblack}, and \eqref{eq:Upk}.
Together with the assumption that $B_p \tfX_{\bbw} =\tfX_{\bbw} B_p^\sigma$ $(p =i,j)$, we conclude that $\tfX_{\bbw}$ satisfies the same intertwining relations in Theorem~\ref{thm:fX1} as for $\tfX$. Note also that clearly we have the constant term  $(\tfX_{\bbw})^0=1$.
Therefore, the desired identity  $\tfX =\tfX_{\bbw}$ follows by the uniqueness in Theorem~\ref{thm:fX1}.
\end{proof}

\subsection{Factorizations in rank 2}

The verification that $B_p \tfX_{\bbw} =\tfX_{\bbw} B_p^\sigma$ in the three cases $\ell_\circ (\bbw) =3, 4$, or $6$, is based on the same idea, though the notations are a little different. In the subsections below, we shall consider the three cases separately.

\subsubsection{Factorization for $\ell_\circ (\bbw) =3$}
  \label{cij'=-1}

In this subsection, we deal with the rank 2 cases for $\ell_\circ (\bbw) =3$, with the help of Proposition~\ref{prop:fact1} and Corollary~\ref{cor:fact1}.

Assume that $\wItau=\{i,j\}$ such that $\ell_\circ (\bbw) =3$; in this case only $\tau =\mathrm{Id}$ and hence we identify $\wI=\{i,j\}$ as well. The longest element $\bbw$ of the relative Weyl group has a reduced expression
\begin{align}
\bbw=\bs_i \bs_j \bs_i.
\end{align}

By definition \eqref{eq:Upk} of $\tfX^{[k]}$ and $\tfX_{\bbw}$, we have
\begin{align}
 \label{Up321}
\tfX_{\bbw}= \tfX^{[3]} \tfX^{[2]} \tfX^{[1]}.
\end{align}
 where by Proposition~\ref{prop:tau0}, $\tT_{\bs_i} \tT_{\bs_j} (\tfX_i) =\tfX_j$ and $\tT_{\bs_j} \tT_{\bs_i} (\tfX_j) =\tfX_i$, and hence,
 \begin{align}
  \label{eq:DK}
 \tfX^{[3]}&=\tfX_j,\qquad  \tfX^{[2]}=\tT_{\bs_i}(\tfX_j), \qquad  \tfX^{[1]}=\tfX_i.
 \end{align}

  By Corollary~\ref{cor:fact1}, we have
 \begin{align}
   B_i  \tfX^{[3]} \tfX^{[2]} & = \tfX^{[3]} \tfX^{[2]} B_i,
   \label{eq:fact-1} \\
   B_j^\sigma \tfX^{[2]} \tfX^{[1]} & = \tfX^{[2]} \tfX^{[1]} B_j^\sigma.
     \label{eq:fact-2}
 \end{align}

It follows by Theorem~\ref{thm:fX1} that, for $p =i,j,$
 \begin{align}
  \label{eq:Bp}
 B_p \tfX_p =\tfX_p B_p^\sigma.
 \end{align}

Now we show that $\tfX_{\bbw}$ satisfies the following intertwining relations
 \begin{align*}
 B_p \tfX_{\bbw} =\tfX_{\bbw} B_p^\sigma ,\qquad (p =i,j).
 \end{align*}
 Indeed, $B_i \tfX_{\bbw} = B_i \tfX^{[3]} \tfX^{[2]} \tfX^{[1]} = \tfX^{[3]} \tfX^{[2]} B_i \tfX^{[1]} = \tfX^{[3]} \tfX^{[2]} \tfX^{[1]} B_i^\sigma$, by \eqref{Up321}, \eqref{eq:fact-1} and \eqref{eq:Bp}. Also,
 $B_j \tfX_{\bbw} = B_j \tfX^{[3]} \tfX^{[2]} \tfX^{[1]} = \tfX^{[3]} B_j^\sigma \tfX^{[2]}  \tfX^{[1]} = \tfX^{[3]} \tfX^{[2]} \tfX^{[1]} B_j^\sigma$, by  \eqref{eq:DK}, \eqref{eq:Bp},  and \eqref{eq:fact-2}.

It follows by Proposition~\ref{prop:UU} (whose assumption is verified above), we have $\tfX =\tfX_{\bbw}$. Using the other reduced expression for $\bbw$ amounts to switching notations $i,j$ above. Hence, $\tfX =\tfX_{\bbw}$ is independent of the choice of a reduced expression for $\bbw$.

\subsubsection{Factorization for $\ell_\circ (\bbw) =4$}
  \label{cij'=-2}

In this subsection, we deal with the rank 2 cases for $\ell_\circ (\bbw) =4$, with the help of Proposition~\ref{prop:fact2} and Corollary~\ref{cor:fact2}.

 Assume that $\wItau=\{i,j\}$ such that $\ell_\circ (\bbw) =4$. Let $\{i,\tau i\}$ and $\{j,\tau j\}$ be the corresponding two distinct $\tau$-orbits of $\wI$.
 The longest element $\bbw$ of the relative Weyl group has a reduced expression
\begin{align}
\bbw=\bs_i \bs_j \bs_i \bs_j.
\end{align}
By definition \eqref{eq:Upk} of $\tfX^{[k]}$ and $\tfX_{\bbw}$,  we have
\begin{align}
\tfX_{\bbw}= \tfX^{[4]} \tfX^{[3]} \tfX^{[2]} \tfX^{[1]},
\end{align}
 where by Proposition~\ref{prop:tau0}, $\tT_{\bs_i} \tT_{\bs_j} \tT_{\bs_i} (\tfX_j) =\tfX_j$ and $\tT_{\bs_j} \tT_{\bs_i} \tT_{\bs_j} (\tfX_i) =\tfX_i$, and hence
 \begin{align}
  \label{eq:DK2}
 \tfX^{[4]}&=\tfX_j,\quad
 \tfX^{[3]}=\tT_{\bs_i} \tT_{\bs_j} (\tfX_i)=\tT_{\bs_j}^{-1}(\tfX_i),\quad
 \tfX^{[2]}=\tT_{\bs_i}(\tfX_j), \quad \tfX^{[1]}=\tfX_i.
 \end{align}

By Corollary~\ref{cor:fact2}, we have
\begin{align}
 B_i \tfX^{[4]} \tfX^{[3]} \tfX^{[2]} &=  \tfX^{[4]} \tfX^{[3]} \tfX^{[2]} B_i,
  \label{eq:fact1}  \\
 B_j^\sigma \tfX^{[3]}\tfX^{[2]} \tfX^{[1]}  &= \tfX^{[3]} \tfX^{[2]} \tfX^{[1]} B_j^\sigma.\label{eq:fact2}
\end{align}

Just as in \S\ref{cij'=-1}, using the identities \eqref{eq:fact1}--\eqref{eq:fact2} we can show that $\tfX_{\bbw}$ satisfies the following intertwining relations
$B_p \tfX_{\bbw} =\tfX_{\bbw} B_p^\sigma$, for $p =i,j$.
It follows by Proposition~\ref{prop:UU} (whose assumption is verified above), we have $\tfX =\tfX_{\bbw}$, which is independent of the choice of a reduced expression for $\bbw$.

\subsubsection{Factorization for $\ell_\circ (\bbw) =6$}
  \label{cij'=-3}

The case for $\ell_\circ (\bbw) =6$ occurs only in split $G_2$ type. We shall prove this using Proposition~\ref{prop:fact-G2} and Corollary~\ref{cor:G2}.

Let $(\I=\wI,\tau=\Id)$ be the Satake diagram of split type $G_2$. In this case, $W^\circ=W$ and $\bs_a=s_a$.
 Assume that $\I =\{i,j\}$ such that $\ell_\circ (\bbw) =6$.
 The longest element $\bbw$ of the relative Weyl group has a reduced expression
\begin{align}
\bbw=s_i s_j s_i s_j s_i s_j.
\end{align}
By definition \eqref{eq:Upk} of $\tfX^{[k]}$ and $\tfX_{\bbw}$,  we have
\begin{align}
\tfX_{\bbw}= \tfX^{[6]}\tfX^{[5]}\tfX^{[4]} \tfX^{[3]} \tfX^{[2]} \tfX^{[1]}.
\end{align}
 where by Proposition~\ref{prop:tau0},
 $\tT_{s_i s_j s_i s_j s_i} (\tfX_j) =\tfX_j,$
 and hence
 \begin{align}\notag
 \tfX^{[6]}&=\tfX_j,\quad
 \tfX^{[5]} =\tT_{ s_i s_j s_i s_j}(\tfX_i),\quad
 \tfX^{[4]} =\tT_{ s_i s_j s_i}(\tfX_j),\\
 \tfX^{[3]}&=\tT_{ s_i  s_j} (\tfX_i) ,\quad
 \tfX^{[2]}=\tT_{ s_i}(\tfX_j), \quad \tfX^{[1]}=\tfX_i.
  \label{eq:fact-G2}
 \end{align}

By Corollary~\ref{cor:G2}, we have
\begin{align}
 B_i \tfX^{[6]}\tfX^{[5]}\tfX^{[4]} \tfX^{[3]} \tfX^{[2]} &=  \tfX^{[6]}\tfX^{[5]}\tfX^{[4]} \tfX^{[3]} \tfX^{[2]} B_i,
  \label{eq:fact4}  \\
 B_j^\sigma \tfX^{[5]}\tfX^{[4]} \tfX^{[3]}\tfX^{[2]} \tfX^{[1]}  &= \tfX^{[5]}\tfX^{[4]} \tfX^{[3]} \tfX^{[2]} \tfX^{[1]} B_j^\sigma.\label{eq:fact5}
\end{align}

Just as in \S\ref{cij'=-1}, using the identities \eqref{eq:fact4}--\eqref{eq:fact5} we can show that $\tfX_{\bbw}$ satisfies the following intertwining relations
$B_p \tfX_{\bbw} =\tfX_{\bbw} B_p^\sigma$, for $p =i,j$.
It follows by Proposition~\ref{prop:UU} (whose assumption is verified above), we have $\tfX =\tfX_{\bbw}$, which is independent of the choice of a reduced expression for $\bbw$.

\begin{remark}
A different and more computational proof of the factorization of the quasi $K$-matrix in split type $G_2$ was given earlier in Dobson's thesis \cite{Dob19}.
\end{remark}

\section{Relative braid group actions on $\imath$quantum groups}
   \label{sec:braid}

In this section, we show that $\tTae{i}, \tTbe{i}$, where $e =\pm 1$ and $i\in \wItau$, satisfy the relative braid group relations in $\Br(\reW)$.
An action of $\Br(\bW) \rtimes \Br(\reW)$ on $\tUi$ is then established.
Moreover we show that, by central reductions and isomorphisms among $\imath$quantum groups with different parameters, the symmetries $\tTae{i}, \tTbe{i}$ on $\tUi$ descend to $\TT'_{i,e}, \TT''_{i,e}$ on the $\imath$quantum groups $\Ui_\bvs$, inducing relative braid group actions on $\Ui_\bvs$, for an arbitrary parameter $\bvs$.

\subsection{Braid group relations among $\tTT_i$}

For $i\neq j\in \wItau$, let $m_{ij}$ be the order of $\bs_i \bs_j$ in $\reW$, with $m_{ij}\in \{2,3,4,6\}$. Then the following braid relation is satisfied in $\Br(\reW)$:
\begin{align}
\underbrace{\bs_i\bs_j \bs_i \cdots }_{m_{ij}}=\underbrace{\bs_j\bs_i\bs_j \cdots }_{m_{ij}}.
\end{align}


\begin{theorem}
  \label{thm:newb2}
For $i\neq j\in \wItau,e=\pm1$, we have
\begin{align}
\label{eq:braid1}
\begin{split}
\underbrace{\tTT'_{i,e} \tTT'_{j,e} \tTT'_{i,e}  \cdots }_{m_{ij}}
&=\underbrace{\tTT'_{j,e} \tTT'_{i,e} \tTT'_{j,e}  \cdots}_{m_{ij}},
\\
\underbrace{\tTT''_{i,e} \tTT''_{j,e} \tTT''_{i,e}  \cdots }_{m_{ij}}
&=\underbrace{\tTT''_{j,e} \tTT''_{i,e} \tTT''_{j,e}  \cdots}_{m_{ij}}.
\end{split}
\end{align}
\end{theorem}

\begin{proof}
By Theorem~\ref{thm:newb1}, $\tTb{i}$ is the inverse of $\tTa{i}$. Moreover, by definition \eqref{def:tTT2}, $\tTT'_{i,+1},\tTT''_{i,-1}$ are conjugations of $\tTa{i},\tTb{i}$ respectively. Hence. it suffices to prove the identity \eqref{eq:braid1} for $\tTa{i}$. 

Set $m =m_{ij}$. Let $\bbw = \bs_i \bs_j \bs_i \cdots$ be a reduced expression of length $m$. Define $\boldsymbol{w}_k$, for $1\leq k \leq m$, to be
\begin{align*}
\boldsymbol{w}_1 = \bs_i,\quad \boldsymbol{w}_2 =\bs_i \bs_j,\quad \boldsymbol{w}_3= \bs_i \bs_j \bs_i,\quad \ldots, \quad\boldsymbol{w}_{m}=\bbw.
\end{align*}
Write $\bbw'$ for the other reduced expression $\bs_j\bs_i\bs_j\cdots$, and define $\boldsymbol{w}_k'$, for $1\le k \le m$, accordingly. Let $r$ denote the last index in the reduced expression of $\bbw$; that is, $r=j$ if $m=2,4,6$ and $r=i$ if $m=3.$ Similarly, we define $r'$ for $\bbw'.$

Applying the intertwining property \eqref{eq:newb0} for $m$ times, we obtain the following 2 identities:
\begin{align}
\underbrace{\tTa{i} \tTa{j} \tTa{i} \cdots (u)}_m \cdot 
    & \tfX_i \cdot \tT'_{\boldsymbol{w}_1,-1}(\tfX_j)\cdots \tT'_{\boldsymbol{w}_{m-1},-1}(\tfX_r)
    \notag
    \\
    =\, & \tfX_i \cdot \tT'_{\boldsymbol{w}_1,-1}(\tfX_j)\cdots \tT'_{\boldsymbol{w}_{m-1},-1}(\tfX_r) \cdot \underbrace{\tT'_{\bs_i,-1} \tT'_{\bs_j,-1}\cdots}_m (u^\imath),
    \label{m times} \\
   \underbrace{\tTa{j} \tTa{i} \tTa{j} \cdots (u)\cdot}_m 
    & \tfX_j \cdot \tT'_{\boldsymbol{w}'_1,-1}(\tfX_i)\cdots \tT'_{\boldsymbol{w}'_{m-1},-1}(\tfX_{r'})
    \notag
    \\
    =\, & \tfX_j \cdot \tT'_{\boldsymbol{w}'_1,-1}(\tfX_i)\cdots \tT'_{\boldsymbol{w}'_{m-1},-1}(\tfX_{r'}) \cdot \underbrace{\tT'_{\bs_j,-1} \tT'_{\bs_i,-1} \cdots}_m (u^\imath),
    \label{m times2} 
\end{align}
for all $u\in \tUi$.

By Proposition~\ref{prop:braid0}, the $\tT'_{k,-1}$'s satisfy braid relations. As $\ell(\bs_i\bs_j \bs_i \cdots) = \ell(\bbw) =\ell(\bs_j\bs_i \bs_j \cdots)$, we have
\begin{align}
\label{eq:braid4}
\underbrace{\tT'_{\bs_i,-1} \tT'_{\bs_j,-1} \tT'_{\bs_i,-1}\cdots }_{m}=\underbrace{\tT'_{\bs_j,-1} \tT'_{\bs_i,-1} \tT'_{\bs_j,-1} \cdots }_{m}.
\end{align}

Hence, by a comparison of \eqref{m times}--\eqref{m times2}, we reduce the proof of the desired identity \eqref{eq:braid1} to showing that
\begin{align}
\label{eq:braid6}
 \tfX_i \cdot \tT'_{\boldsymbol{w}_1,-1}(\tfX_j)\cdots \tT'_{\boldsymbol{w}_{m-1},-1}(\tfX_r) =\tfX_j \cdot \tT'_{\boldsymbol{w}'_1,-1}(\tfX_i) \cdots \tT'_{\boldsymbol{w}'_{m-1},-1}(\tfX_{r'}).
\end{align}
By definition~\eqref{eq:Upk}, $\tfX_{\bbw}=\tT_{\boldsymbol{w}_{m-1}}(\tfX_r)\cdots \tT_{\boldsymbol{w}_1}(\tfX_j)\tfX_i$. Applying $\sigma$ to this identity and then using Proposition~\ref{prop:inv}, we obtain
\begin{align}
\label{eq:braid5}
\sigma(\tfX_{\bbw})=\tfX_i \cdot \tT'_{\boldsymbol{w}_1,-1}(\tfX_j)\cdots \tT'_{\boldsymbol{w}_{m-1},-1}(\tfX_r).
\end{align}
We have a similar formula for $\sigma(\tfX_{\bbw'})$ as well. It follows by Theorem~\ref{thm:factor} that $\sigma(\tfX_{\bbw})=\sigma(\tfX_{\bbw'}).$ The identity \eqref{eq:braid6} now follows by the formula \eqref{eq:braid5} and its $\bbw'$-counterpart.

This completes the proof of the theorem. 
\end{proof}

For $w\in \reW$, take a reduced expression $w=\bs_{i_1}\bs_{i_2}\cdots\bs_{i_k}$ and define
\begin{align}\label{def:tTT3}
\tTT'_{w,e}:= \tTT'_{i_1,e} \tTT'_{i_2,e} \cdots \tTT'_{i_k,e},
\qquad
\tTT'_{w,e}:= \tTT''_{i_1,e} \tTT''_{i_2,e} \cdots \tTT''_{i_k,e}.
\end{align}
By Theorem~\ref{thm:newb2}, these are independent of the choice of reduced expressions for $w$.

\subsection{Action of the braid group $\Br(\bW) \rtimes \Br(\reW)$ on $\tUi$}

We first establish a commutator relation between $\tTT'_{i,-1}$ ($i\in \wI$) and $\tT_j^{-1}\equiv \tT'_{j,-1}$ ($j\in \bI$).

\begin{lemma}
  \label{lem:TbTw}
We have $\tT_{ j}^{-1} \tTa{i} (x) =\tTa{i} \tT_{\tau_{\bullet,i}\tau j}^{-1} (x)$, for $i\in \wItau,j\in \bI$, and $x \in \tUi$.
\end{lemma}

\begin{proof}
Note that $\tau (j), \tau_{\bullet,i} (j), \tau_{\bullet,i}\tau (j) \in \bI$,  for $j\in \bI$.
Since $\bw s_j =s_{\tau j} \bw$,  for $j\in \bI$, and $\bwi s_j =s_{\tau_{\bullet,i} j} \bwi$, for $i\in \wI$, we have
\begin{align}
  \label{eq:braid10}
\bs_i s_j = \bwi \bw^{-1} s_j =s_{\tau_{\bullet,i}\tau j}  \bwi \bw^{-1}  = s_{\tau_{\bullet,i}\tau j} \bs_i.
\end{align}

Since $\ell(\bs_i s_j) = \ell(\bs_i) +1$, it follows by \eqref{eq:braid10} that
\begin{align}
  \label{eq:TT}
  \tT_{\tau_{\bullet,i}\tau(j)} \tT_{\bs_i} =\tT_{\bs_i} \tT_j.
\end{align}
 By Proposition~\ref{prop:newb-1}, $\tfX_i$ is fixed by $\tT_j^{-1}$. Hence, applying $\tT_j^{-1}$ to the intertwining relation \eqref{eq:newb0} in Theorem~ \ref{thm:newb0} and then using \eqref{eq:TT}, we obtain, for $x\in \tUi$,
\begin{align}
\tT_j^{-1}\tTa{i} (x)  \tfX_i
&= \tfX_i \tT_j^{-1}\tT^{-1}_{\bs_i}(x)
  \notag \\
&= \tfX_i \tT^{-1}_{\bs_i}\tT_{\tau_{\bullet,i}\tau j}^{-1}(x)
 = \tTa{i} \tT_{\tau_{\bullet,i}\tau j}^{-1}(x) \tfX_i,
  \label{eq:fact0-2}
\end{align}
where the last step uses Theorem~ \ref{thm:newb0} and the fact that $\tT_{\tau_{\bullet,i}\tau j}^{-1}(x) \in \tUi$ by Proposition~\ref{prop:Tjblack}.
The identity \eqref{eq:fact0-2} clearly implies the identity in the lemma.
\end{proof}

Let $\Br(\bW)$ and $\Br(\reW)$ be the braid groups associated to $\bW$ and $\reW$ respectively.

\begin{theorem}
  \label{thm:newb3}
There exists a braid group action of $\Br(\bW) \rtimes \Br(\reW)$ on $\tUi$ as automorphisms of algebras generated by $\tT'_{j,-1} \; (j\in \bI)$ and $\tTT'_{i,-1} \;(i\in \wItau)$.
\end{theorem}

\begin{proof}
By Remark \ref{rmk:newb0}, $\tTT'_{i,-1}$ is independent of the choice of representatives in a $\tau$-orbit. The defining relations of $\Br(\bW) \rtimes \Br(\reW) $ consist of braid relations for $\Br(\bW)$, the braid relations for $\Br(\reW)$, and relations \eqref{eq:braid10}.
The braid relations for $\tT'_{j,-1},j\in \bI$ are verified in Proposition~\ref{prop:braid0}. The braid relations for $\tTa{i},i\in \wItau$ are verified in Theorem~\ref{thm:newb2}.
The commutator relation for $\tT'_{j,-1},\tTa{i}$ corresponding to \eqref{eq:braid10} is verified in Lemma~\ref{lem:TbTw}. 
\end{proof}

\begin{remark}
\label{rmk:newb1}
Since $\tTa{i},\tTb{i}$ are mutually inverses and $\tT'_{j,-1},\tT''_{j,+1}$ are mutually inverses,
there also exists a braid group action of $\Br(\bW) \rtimes \Br(\reW)$ on $\tUi$ as automorphisms of algebras generated by $\tT''_{j,+1} \; (j\in \bI)$ and $\tTb{i} \;(i\in \wItau)$.
\end{remark}

Recall the remaining two symmetries $\tTT'_{i,+1},\tTT''_{i,-1}$ from \eqref{def:tTT2}. We shall establish a variant of Theorem~\ref{thm:newb3} for $\tTae{i}$ and $\tT'_{j,e}$ (and respectively, $\tTbe{i}$ and $\tT''_{j,e}$).

Let $j\in \I$. Recall
$\tT_{j,+1}''$ and
$\tT_{j,-1}'$ from \eqref{def:tT}--\eqref{def:tT-1}. Recalling $\tpsi_\star =\tPsi_{\bvs_\star} \circ \tpsi$ from \eqref{eq:psi star}, we define
 \begin{align}
  \label{eq:Tjstar}
 \tT''_{j,-1}:= \psi_\star \circ \tT_{j,+1}'' \circ \psi_\star,\qquad
 \tT'_{j,+1}:= \psi_\star \circ \tT_{j,-1}' \circ \psi_\star.
 \end{align}

Let $\bvs_{\star\dm} :=(\vs_{j,\star}\vs_{j,\dm})_{j\in \wI}$ be the parameter obtained as the componentwise product of parameters $\bvs_\dm$ and $\bvs_\star$ from \eqref{def:vsi} and \eqref{eq:bvs star}.
\begin{lemma}
 \label{lem:rescale2}
The $\tT''_{j,-1},\tT'_{j,+1}$ are related to   $\tTD''_{j,-1},\tTD'_{j,+1}$ via a rescaling automorphism:
\begin{align*}
\tT''_{j,-1} =\tPsi_{\bvs_{\star\dm}}\tTD''_{j,-1}\tPsi_{\bvs_{\star\dm}}^{-1},\qquad
\tT'_{j,+1} =\tPsi_{\bvs_{\star\dm}}\tTD'_{j,+1}\tPsi_{\bvs_{\star\dm}}^{-1}.
\end{align*}
\end{lemma}

\begin{proof}
Recall
$\tT_{j,+1}'' = \tPsi_{\bvs_\diamond}^{-1} \circ \tTD''_{j,+1} \circ \tPsi_{\bvs_\diamond}$ and
$\tT_{j,-1}' = \tPsi_{\bvs_\diamond}^{-1} \circ \tTD'_{j,-1} \circ \tPsi_{\bvs_\diamond}$ from \eqref{def:tT}--\eqref{def:tT-1}.

Recall from \eqref{eq:sTs} that $\tTD''_{i,-1} = \tpsi \circ \tTD''_{i,+1} \circ \tpsi$ and $\tTD'_{i,+1} = \tpsi \circ \tTD'_{i,-1} \circ \tpsi$. Then we have
\begin{align*}
\tT''_{j,-1} &= \psi_\star \circ \tT_{j,+1}'' \circ \psi_\star
    \\
&= \tPsi_{\bvs_\star} \tpsi \circ \tPsi_{\bvs_\diamond}^{-1}  \tTD''_{j,+1}  \tPsi_{\bvs_\diamond} \circ  \tPsi_{\bvs_\star} \tpsi
= \tPsi_{\bvs_{\star\dm}}\tTD''_{j,-1}\tPsi_{\bvs_{\star\dm}}^{-1},
\end{align*}
where we used $\tpsi \circ \tPsi_{\bvs_\diamond}^{-1} = \tPsi_{\bvs_\diamond} \circ \tpsi$. The proof for the other formula is similar.
\end{proof}

By Proposition~\ref{prop:Tjblack}, the automorphisms $\tT''_{j,+1},\tT'_{j,-1}$ for $j\in \bI$ restrict to automorphisms on $\tUi$.

 \begin{lemma}
 \label{lem:newb2}
 The automorphisms $ \tT''_{j,e},\tT'_{j,e}$, for $j\in \bI$ and $e=\pm 1$, restrict to automorphisms on $\tUi$. Moreover, the following identities hold:
 \begin{align}
 \label{eq:tT2}
 \tT''_{j,-1}:= \psi^\imath \circ \tT''_{j,+1} \circ \psi^\imath,\qquad
 \tT'_{j,+1}:= \psi^\imath \circ \tT'_{j,-1} \circ \psi^\imath.
 \end{align}
\end{lemma}

\begin{proof}
As $\tT_j \equiv \tT''_{j,+1}$ restricts to an automorphism on $\tUi$ by Proposition~\ref{prop:Tjblack}, it suffices to prove \eqref{eq:tT2}.

By Proposition~\ref{prop:newb3}, we have $\psi_\star=\ad_{\tfX^{-1}}\circ\psi^\imath$ when acting on $\tUi $. By Proposition~\ref{prop:newb-1}, $\ad_{\tfX^{-1}}$ commutes with $\tT_j$. By Proposition~\ref{prop:psi1}, we have
$\psi_\star\circ \ad_{\tfX^{-1}}=\ad_{\tfX }\circ \psi_\star.$
Using these properties and \eqref{eq:Tjstar}, we have, for $x\in \tUi$,
\begin{align*}
\tT''_{j,-1}(x)
&=\psi_\star \circ \tT''_{j,+1} \circ \psi_\star(x)
= \psi_\star\circ \tT''_{j,+1} \circ \ad_{\tfX^{-1}}\circ  \psi^\imath(x)
\\
&= \psi_\star\circ \ad_{\tfX^{-1}}\circ \tT''_{j,+1} \circ  \psi^\imath(x)
= \ad_{\tfX }\circ \psi_\star \circ \tT''_{j,+1} \circ  \psi^\imath(x)
= \psi^\imath \circ \tT''_{j,+1} \circ  \psi^\imath(x),
\end{align*}
where the last equality uses \eqref{eq:fX2}.

 The proof of the other formula for $\tT'_{j,+1}$ is similar and hence skipped.
\end{proof}

The next result follows from \eqref{def:tTT2}, Theorem~\ref{thm:newb3}, Remark~\ref{rmk:newb1}, and Lemma~\ref{lem:newb2}.

\begin{corollary}
  \label{cor:newb1}
  Let $e=\pm 1$.
\begin{enumerate}
    \item
There exists a braid group action of $\Br(\bW) \rtimes \Br(\reW)$ on $\tUi$ as automorphisms of algebras generated by $\tT'_{j,e} \; (j\in \bI)$ and $\tTT'_{i,e} \;(i\in \wItau)$.

    \item
There exists a braid group action of $\Br(\bW) \rtimes \Br(\reW)$ on $\tUi$ as automorphisms of algebras generated by $\tT''_{j,e} \; (j\in \bI)$ and $\tTT''_{i,e} \;(i\in \wItau)$.
\end{enumerate}
\end{corollary}

\subsection{Intertwining properties of $\tTT'_{i,+1},\tTT''_{i,-1}$}

The automorphisms $\tTT'_{i,+1},\tTT''_{i,-1}$ on $\tUi$ also satisfy intertwining relations similar to those satisfied by $\tTa{i}$ in \eqref{eq:newb0} and $\tTb{i}$ in \eqref{eq:newb0-1}. These relations on $\tUi$ will descend to $\Ui_\bvs$ (see Proposition~\ref{prop:inter2}) and will then be used to define the relative braid operators on module level (see Definition~\ref{def:4T}).

\begin{proposition}
The automorphisms $\tTT'_{i,+1},\tTT''_{i,-1}$ satisfy the following intertwining relations
\begin{align}
\label{eq:braid14}
\tTT'_{i,+1}(x)  \tT'_{\bs_i,+1}(\tfX_i^{-1} )
&= \tT'_{\bs_i,+1}(\tfX_i^{-1} )  \tT'_{\bs_i,+1}(x),
 \\
\tTT''_{i,-1}(x)   \tfX_i   &=  \tfX_i   \tT''_{\bs_i,-1}(x).
\label{eq:braid15}
\end{align}
\end{proposition}

\begin{proof}
We prove the first identity \eqref{eq:braid14}; the second identity \eqref{eq:braid15} can be derived from the first one by noting that $\tTT'_{i,+1},\tTT''_{i,-1}$ are inverses and $\tT'_{\bs_i,+1}, \tT''_{\bs_i,-1}$ are inverses.

We claim the following identity holds:
\begin{align}
\label{eq:braid13}
\tTT'_{i,+1}(x) \cdot \tfX \psi_\star(\tfX_i) \tT'_{\bs_i,+1}(\tfX^{-1} )
= \tfX \psi_\star(\tfX_i) \tT'_{\bs_i,+1}(\tfX^{-1}) \cdot \tT'_{\bs_i,+1}(x).
\end{align}

Let us prove \eqref{eq:braid13}. Recall from \eqref{def:tTT2} that $\tTT'_{i,+1}= \tpsi^\imath  \tTT'_{i,-1} \tpsi^\imath$ and from \eqref{eq:newb10} that $\tfX^{-1} \tpsi^\imath(u)   \tfX = \tpsi_\star(u)$. Hence, \[
\tfX^{-1} \tTT'_{i,+1}(x) \tfX =\tfX^{-1} \tpsi^\imath (\tTa{i}(\tpsi^\imath x) ) \tfX =\tpsi_\star(\tTa{i}(\tpsi^\imath x)).
\]
By \eqref{eq:newb0}, $\tfX_i^{-1} \tTa{i}(\tpsi^\imath x) \tfX_i = \tT_{\bs_i,-1}'(\tpsi^\imath x)$. Hence
\[
\psi_\star(\tfX_i)^{-1} \tfX^{-1} \tTT'_{i,+1}(x) \tfX \psi_\star(\tfX_i) =\psi_\star(\tT_{\bs_i,-1}'(\tpsi^\imath (x))).
\]
This allows us to write \eqref{eq:braid13} as an equivalent identity \begin{align}
\label{eq:same2}
\psi_\star(\tT_{\bs_i,-1}'(\tpsi^\imath (x))) \tT'_{\bs_i,+1}(\tfX^{-1} )= \tT'_{\bs_i,+1}(\tfX^{-1}) \tT'_{\bs_i,+1}(x).
\end{align}
Recalling by \eqref{eq:Tjstar} that $\tT'_{\bs_i,+1} = \psi_\star \tT_{\bs_i,-1}' \psi_\star$, we reduce the proof of \eqref{eq:same2} to verifying that $\tpsi^\imath (x) \tpsi_\star (\tfX)^{-1} = \tpsi_\star (\tfX)^{-1} \tpsi_\star (x)$, which by Proposition~\ref{prop:psi1} is equivalent to
$\tpsi^\imath (x)  \tfX  =  \tfX  \tpsi_\star (x)$. This last identity holds by \eqref{eq:newb10}. Therefore,  \eqref{eq:braid13} is proved.

Observe that if we define $\tfX_{[w]}$ by replacing $\tT_{\bs_i} \equiv \tT''_{\bs_i,+1}$ in the definition \eqref{eq:Upk} of $\tfX_{w}$ by $\tT'_{\bs_i,+1}$, then we still have a factorization
$\tfX=\tfX_{[\bbw]},$ for any reduced expression of $\bbw$. Below we shall use this version of factorization.

Let $\bbw'$ be a reduced expression of $\bbw$ starting with $\bs_i$, and $\bbw'' (= w_0 \bbw' w_0)$ be a reduced expression of $\bbw$ ending with $\bs_{\tau_0 i}$. It follows by definition that
\begin{align}
\label{eq:fact7}
    \tfX=\tfX_{[\bbw']}=\tT'_{\bs_i,+1}(\tfX_{[\bs_i\bbw']}) \tfX_i.
\end{align}
Since $w_0 \bs_{\tau_0 i}=\bs_i w_0$ and $w_0 =\bbw \bw$, we have $\bbw \bs_{\tau_0 i}=\bs_i \bbw$. By definition and  Proposition~\ref{prop:tau0}, we obtain
\begin{align}
\label{eq:fact8}
    \tfX=\tfX_{[\bbw'']}=\tfX_i \tfX_{[\bbw\bs_{\tau_0 i}]} =\tfX_i \tfX_{[\bs_{ i}\bbw]}.
\end{align}

Now, using \eqref{eq:fact7}-\eqref{eq:fact8}, we can simplify a key component appearing in \eqref{eq:braid13} as follows:
\begin{align*}
\tfX \psi_\star(\tfX_i) \tT'_{\bs_i,+1}(\tfX^{-1} )
&=\tfX \tfX_i^{-1}\tT'_{\bs_i,+1}(\tfX^{-1} )\\
&=\tT'_{\bs_i,+1}(\tfX_{[\bs_i\bbw]} \tfX^{-1} ) =\tT'_{\bs_i,+1}(\tfX_i^{-1} ) .
\end{align*}
Hence, the identity \eqref{eq:braid14} follows from \eqref{eq:braid13}.
\end{proof}

\subsection{Braid group action on $\Ui_{\bvs}$}
\label{braid:Uibvs}

Recall from \eqref{eq:Uibvs} the $\imath$quantum group $\Ui_{\bvs}$ with parameter $\bvs$ satisfying \eqref{def:par} (\`a la Letzter), and recall a central reduction $\pi_{\bvs}^\imath:\tUi\rightarrow \Ui_{\bvs}$ from Proposition~\ref{prop:QG2}.

We first construct the braid group action on $\Ui_{\bvs_\diamond}$ for the distinguished parameter $\bvs_\diamond$ \eqref{def:vsi}. By the definition \eqref{def:tkdm} of $\tk_{j,\dm}$ and Proposition~\ref{prop:QG2}, the kernel $\ker \pi^\imath_{\bvs_\diamond}$ is generated by
\begin{align*}
\tk_{j,\dm}-1 \quad (\tau j=j\in \wI), \quad \tk_{j,\dm}\tk_{\tau j,\dm}-1 \quad (\tau j\neq j\in \wI),\quad K_j K_j'-1 \quad (j\in \bI).
\end{align*}
In addition, by Proposition \ref{prop:Cartanblack}, we have
$\tTb{i} (\tk_{j,\dm}) = \tk_{\bs_i\alpha_j,\dm}.$
Hence, the kernel of $\pi_{\bvs_\diamond}$ is preserved by $\tTb{i}$. Therefore, $\tTb{i}$ induces a automorphism $\TT''_{i,+1;\bvs_\diamond}$ on $\Ui_{\bvs_\diamond}$ such that the following diagram commutes:
\begin{center}
\begin{tikzpicture}[scale=.9,thick]
\node (tUi1) at (-2,1) {$\tUi $};
\node (tUi2) at (2,1) {$\tUi$};
\node (Ui1) at (-2,-1) {$\Ui_{\bvs_\diamond}$};
\node (Ui2) at (2,-1) {$\Ui_{\bvs_\diamond}$};
\path[->] (tUi1) edge node[above]{$\tTb{i}$} (tUi2)
          (Ui1) edge node[above]{$\TT''_{i,+1;\bvs_\diamond}$} (Ui2);
\path[->] (tUi1) edge node[right]{$\pi^\imath_{\bvs_\diamond}$} (Ui1) ;
\path[->] (tUi2) edge node[right]{$\pi_{\bvs^\imath_\diamond}$} (Ui2);
\end{tikzpicture}
\end{center}

It follows from Theorem~\ref{thm:newb2} that $\TT''_{i,+1;\bvs_\diamond}$ satisfy the braid relations. By definition, $\tT_j$ ($j\in \bI$) descends to Lusztig's automorphism $T_j$ under the central reduction $\pi^\imath_{\bvs_\diamond}$. It then follows by Theorem~\ref{thm:newb3} and Remark~\ref{rmk:newb1} that there exists an action of the braid group $\Br(\bW) \rtimes \Br(\reW)$ on $\Ui_{\bvs_\diamond}$ generated by $T_j, \TT''_{i,+1;\bvs_\diamond}$, for $j\in \bI, i\in\wItau.$

We now consider the symmetries on $\Ui_\bvs$, for an arbitrary parameter $\bvs$ satisfying \eqref{def:par}.

Via the isomorphism $\phi_\bvs:\Ui_{\bvs_\diamond} \rightarrow\Ui_{\bvs}$ constructed in Proposition \ref{prop:QG3}, we transport the relative braid group action on $\Ui_{\bvs_\diamond}$ to a relative braid group action on $\Ui_\bvs.$ More precisely, there exist automorphisms $\TT''_{i,+1;\bvs}$ on $\Ui_\bvs$ such that the following diagram commutes:
\begin{center}
\begin{tikzpicture}[scale=.9,thick]
\node (tUi1) at (-2,1) {$\Ui_{\bvs_\diamond} $};
\node (tUi2) at (2,1) {$\Ui_{\bvs_\diamond}$};
\node (Ui1) at (-2,-1) {$\Ui_{\bvs}$};
\node (Ui2) at (2,-1) {$\Ui_{\bvs}$};
\path[->] (tUi1) edge node[above]{$\TT''_{i,+1;\bvs_\diamond}$} (tUi2)
          (Ui1) edge node[above]{$\TT''_{i,+1}$} (Ui2);
\path[->] (tUi1) edge node[right]{$\phi_{\bvs}$} (Ui1) ;
\path[->] (tUi2) edge node[right]{$\phi_{\bvs}$} (Ui2);
\end{tikzpicture}
\end{center}
{\bf Our convention here and below is that we suppress the dependence on a general parameter $\bvs$ for the symmetries $\TT''_{i,+1}$ (and $\TT'_{i,-1}$, $\TT''_{i,-1}$ and $\TT'_{i,+1}$ below) on $\Ui_\bvs$.}

In addition, $T_j$ commutes with $\phi_\bvs$ for $j\in \bI$. Summarizing we have obtained the following braid group action on $\Ui_\bvs$ (from Theorem~\ref{thm:newb2}, Theorem~\ref{thm:newb3} and Remark~\ref{rmk:newb1}).

\begin{theorem}
  \label{thm:braid5}
For an arbitrary parameter $\bvs$ satisfying \eqref{def:par}, there exists a braid group action of $\Br(\bW) \rtimes \Br(\reW)$ on $\Ui_\bvs$ as automorphisms of algebras generated by $T_j \; (j\in \bI)$ and $\TT''_{i,+1} \;(i\in \wItau)$.
\end{theorem}

 We next construct $\TT'_{i,+1}$ on $\Ui_\bvs$ for general parameters $\bvs$. 
By a similar argument as in \S \ref{tTtk}, we have $\tTT'_{i,+1}=\tT'_{\bs_i,+1}$ on $\tU^{\imath 0}$ and both are given by
\begin{align}
\label{eq:newb11}
 \vs_{j,\star\dm}\tk_j \mapsto  \vs_{\bs_i \alpha_j,\star\dm} \tk_{\bs_i \alpha_j}.
\end{align}

Denote the parameter $\obvs_{\star\dm}:=(\obvs_{j,\star\dm})_{j\in \wI}$. Then by \eqref{eq:newb11}, $\tTT'_{i,+1}$ preserves the kernel of $\pi^\imath_{\obvs_{\star\dm}}$ and hence it induces an automorphism $\TT'_{i,+1;\obvs_{\star\dm}}$ on $\Ui_{\obvs_{\star \dm}}$ such that the following diagram commutes:
\begin{center}
\begin{tikzpicture}[scale=.9,thick]
\node (tUi1) at (-2,1) {$\tUi $};
\node (tUi2) at (2,1) {$\tUi$};
\node (Ui1) at (-2,-1) {$\Ui_{\obvs_{\star\dm}}$};
\node (Ui2) at (2,-1) {$\Ui_{\obvs_{\star\dm}}$};
\path[->] (tUi1) edge node[above]{$\tTT'_{i,+1}$} (tUi2)
          (Ui1) edge node[above]{$\TT'_{i,+1;\obvs_{\star\dm}}$} (Ui2);
\path[->] (tUi1) edge node[right]{$\pi^\imath_{\obvs_{\star\dm}}$} (Ui1) ;
\path[->] (tUi2) edge node[right]{$\pi^\imath_{\obvs_{\star\dm}}$} (Ui2);
\end{tikzpicture}
\end{center}
On the other hand, by Lemma~\ref{lem:rescale2}, $\tT'_{j,+1}$ descends to Lusztig's automorphism $T'_{j,+1}$ under the central reduction $\pi^\imath_{\obvs_{\star\dm}}$. Hence, by Corollary~\ref{cor:newb1}, there exits an action of the braid group $\Br(\bW) \rtimes \Br(\reW)$ on $\Ui_{\obvs_{\star \dm}}$ generated by $T'_{j,+1}\; (j\in \bI)$ and $\TT'_{i,+1;\obvs_{\star\dm}}\; (i\in \wItau).$

Now, for an arbitrary parameter $\bvs$, we can use the isomorphism $\phi_{\bvs}\phi_{\obvs_{\star\dm}}^{-1}$ to translate this action on $\Ui_{\obvs_{\star \dm}}$ to an action on $\Ui_\bvs$, i.e., there exists automorphisms $\TT'_{i,+1}$ on $\Ui_{\bvs}$ such that
\begin{align*}
\TT'_{i,+1}\circ \phi_{\bvs}\phi_{\obvs_{\star\dm}}^{-1}=\phi_{\bvs}\phi_{\obvs_{\star\dm}}^{-1} \circ \TT'_{i,+1;\bvs_{\star\dm}}.
\end{align*}
In addition, $\tT'_{j,+1}$ commutes with $\phi_{\bvs}\phi_{\obvs_{\star\dm}}^{-1}$.

Similarly, we can formulate the automorphisms $\TT'_{i,-1}, \TT''_{i,-1}$ on $\Ui_{\bvs}$, which are inverses to $\TT''_{i,+1}, \TT'_{i,+1}$; the detail is skipped. Summarizing, we have established the following theorem, which was conjectured in \cite[Conjecture 1.2]{KP11}.

\begin{theorem}
\label{thm:braid6}
Let $e=\pm1$, and $\bvs$ be an arbitrary parameter satisfying \eqref{def:par}.
\begin{enumerate}
    \item
There exists a braid group action of $\Br(\bW) \rtimes \Br(\reW)$ on $\Ui_\bvs$ as automorphisms of algebras generated by $T'_{j,e} \; (j\in \bI)$ and $\TT'_{i,e} \;(i\in \wItau)$.
   \item
There exists a braid group action of $\Br(\bW) \rtimes \Br(\reW)$ on $\Ui_\bvs$ as automorphisms of algebras generated by $T''_{j,e} \; (j\in \bI)$ and $\TT''_{i,e} \;(i\in \wItau)$.
\end{enumerate}
\end{theorem}

\section{Relative braid group actions on $\U$-modules}
 \label{sec:modules}

Let $(\I=\bI \cup \wI,\tau)$ be a Satake diagram of arbitrary type and $(\U,\U_{\bvs})$ be the associated quantum symmetric pair. We set $\bvs$ to be a balanced parameter throughout this section. Based on the intertwining properties of $\Tae{i},\Tbe{i}$ on $\Ui_\bvs$, we formulate the compatible action of corresponding operators on an arbitrary finite-dimensional $\U$-module $M$. We then show that these operators on $M$ satisfy relative braid group relations.

\subsection{Intertwining relations on $\Ui_\bvs$}

Recall that the symmetries $\TT'_{i,e}$ and $\TT''_{i,e}$ on $\Ui_\bvs$, for $e=\pm 1$, were defined in \S\ref{braid:Uibvs}. In this subsection we formulate the intertwining properties of these symmetries.

Recall $\phi_\bvs$ from Proposition~\ref{prop:QG3}. Since $\bvs$ is a balanced parameter, $\phi_\bvs$ is the restriction of $\Phi_{\overline{\bvs}_\dm \bvs}$, where $\overline{\bvs}_\dm \bvs$ is defined by componentwise multiplication with $\overline{\bvs}_\dm =(\overline{\vs_{j,\dm}})_{j\in \wI}$; cf. also Proposition~\ref{prop:QG3}. Define
\begin{align}
\label{def:sT1}
\sT''_{i,+1;\bvs}:= \Phi_{\overline{\bvs}_\dm \bvs} T''_{i,+1} \Phi_{\overline{\bvs}_\dm \bvs}^{-1},
\qquad
\sT'_{i,-1;\bvs}:= \Phi_{\overline{\bvs}_\dm \bvs} T'_{i,-1} \Phi_{\overline{\bvs}_\dm \bvs}^{-1}.
\end{align}

\begin{proposition}
\label{prop:inter1}
Let $\bvs$ be a balanced parameter. The automorphisms $\TT'_{i,-1}$ and $\TT''_{i,+1}$ on $\Ui_\bvs$ satisfy the following intertwining relations:
\begin{align}
\label{eq:inter4}
    \TT'_{i, -1}(x) \fX_{i,\bvs}&= \fX_{i,\bvs} \sT'_{\bs_i,-1;\bvs}(x),
    \\
    \TT''_{i,+1}(x) \, \sT''_{\bs_i,+1;\bvs}(\fX_{i,\bvs}^{-1} )&=  \sT''_{\bs_i,+1;\bvs}(\fX_{i,\bvs}^{-1} ) \, \sT''_{\bs_i,+1;\bvs}(x),
    \label{eq:inter5}
\end{align}
for $x\in \Ui_\bvs$.
\end{proposition}

\begin{proof}
By Theorem~\ref{thm:newb0} and Theorem~\ref{thm:Tb}, we have, for any $x\in \tUi$,
\begin{align}
  \label{eq:inter1}
\begin{split}
    \tTT'_{i,-1}(x) \; \tfX_{i }&= \tfX_{i} \; \tT'_{\bs_i,-1}(x),
    \\
    \tTT''_{i,+1} (x) \; \tT_{\bs_i}(\tfX_{i }^{-1})&= \tT_{\bs_i}(\tfX_{i}^{-1}) \; \tT''_{\bs_i,+1}(x).
\end{split}
\end{align}

Let $T'_{i,e},T''_{i,e}$ be Lusztig's automorphisms on $\U$. Recall the central reduction $\pi_{\bvs_\dm}: \tU \rightarrow \U$ from \eqref{pibvs}. By \eqref{eq:braid12} (with $\ba =\bvs_{\diamond}$) and \eqref{eq:braid11}, we have
\begin{equation*}
\pi_{\bvs_{\diamond}}\circ \tT''_{i,+1} =T''_{i,+1} \circ \pi_{\bvs_{\diamond}},
\qquad
\pi_{\bvs_{\diamond}}\circ \tT'_{i,-1} =T'_{i,-1} \circ \pi_{\bvs_{\diamond}}.
\end{equation*}
Hence, $\pi^\imath_{\bvs_\dm} \circ \tTT''_{i,+1} = \TT''_{i,+1;\bvs_\dm} \circ \pi^\imath_{\bvs_\dm} $. Since the parameter $\bvs_\dm$ is balanced, $\pi^{\imath}_{\bvs_\dm}$ is the restriction of $\pi_{\bvs_\dm}$ to $\Ui_{\bvs_\dm}$. Applying $\pi_{\bvs_\dm}$ to the intertwining relations \eqref{eq:inter1}, we obtain, for any $x\in \Ui_{\bvs_\dm}$,
\begin{align}
  \label{eq:inter2}
\begin{split}
\TT'_{i,-1;\bvs_\dm}(x) \; \fX_{i,\bvs_\dm}& = \fX_{i,\bvs_\dm} \; T'_{\bs_i,-1}(x),
\\
\TT''_{i,+1;\bvs_\dm}(x) \; T''_{\bs_i,+1}(\fX_{i,\bvs_\dm}^{-1})& = T''_{\bs_i,+1}(\fX_{i,\bvs_\dm}^{-1}) \; T''_{\bs_i,+1}(x).
\end{split}
\end{align}

Recall $\phi_\bvs$ from Proposition~\ref{prop:QG3}. As we have seen in \S \ref{braid:Uibvs}, we have $\phi_\bvs \circ \TT''_{i,+1;\bvs_\dm} = \TT''_{i,+1} \circ \phi_{\bvs}$, and $\phi_\bvs \circ \TT'_{i,-1;\bvs_\dm} = \TT'_{i,-1} \circ \phi_{\bvs}$. Therefore, applying $\phi_\bvs$ to the identities \eqref{eq:inter2} gives us the desired intertwining relations in the proposition.
\end{proof}

 We next formulate intertwining relations for the other two automorphisms $\TT'_{i,+1}$ and $\TT''_{i,-1}$.

 Recall the central reductions $\pi_\bvs: \tU \rightarrow \U$ from \eqref{pibvs} and $\pi^\imath_\bvs: \tUi \rightarrow \Ui_\bvs$ from Proposition~\ref{prop:QG2}.
 By Lemma~\ref{lem:rescale2}, we have $\pi_{\obvs_{\star\dm}} \circ \tT'_{i,+1} =T'_{i,+1} \circ \pi_{\obvs_{\star\dm}}$ and $\pi^\imath_{\obvs_{\star\dm}} \circ \tTT'_{i,+1} = \TT'_{i,+1;\obvs_{\star\dm}} \circ \pi^\imath_{\obvs_{\star\dm}} $. Since the parameter $\obvs_{\star\dm}$ is balanced, $\pi^{\imath}_{\obvs_{\star\dm}}$ is the restriction of $\pi_{\obvs_{\star\dm}}$ to $\tUi$. Applying $\pi_{\obvs_{\star\dm}}$ to \eqref{eq:braid14}-\eqref{eq:braid15}, we have, for any $x\in \Ui_{\obvs_{\star\dm}}$,
\begin{align}
\label{eq:inter6}
\begin{split}
    \TT'_{i,+1;\obvs_{\star\dm}}(x)\, T'_{\bs_i,+1}(\fX_{i,\obvs_{\star\dm}}^{-1})& = T'_{\bs_i,+1}(\fX_{i,\obvs_{\star\dm}}^{-1})\, T'_{\bs_i,+1}(x),
    \\
    \TT''_{i,-1;\obvs_{\star\dm}}(x) \fX_{i,\obvs_{\star\dm}}& = \fX_{i,\obvs_{\star\dm}} T''_{\bs_i,-1}(x).
 \end{split}
\end{align}

Since $\bvs$ is a balanced parameter, by the proof of Proposition~\ref{prop:QG3}, $\phi_\bvs \phi_{\obvs_{\star\dm}}^{-1}$ is the restriction of $\Phi_{\obvs_{\star\dm}^{-1} \bvs}=\Phi_{\bvs_{\star\dm} \bvs}$. Define
\begin{align}
\label{def:sT2}
\sT''_{i,-1;\bvs}:= \Phi_{\bvs_{\star\dm} \bvs} T''_{i,-1} \Phi_{\bvs_{\star\dm} \bvs}^{-1},
\qquad
\sT'_{i,+1;\bvs}:= \Phi_{\bvs_{\star\dm} \bvs} T'_{i,+1} \Phi_{\bvs_{\star\dm} \bvs}^{-1}.
\end{align}

Applying $\phi_\bvs \phi_{\obvs_{\star\dm}}^{-1}$ to \eqref{eq:inter6}, we have established the following.

\begin{proposition}
\label{prop:inter2}
Let $\bvs$ be a balanced parameter. The automorphisms $\TT'_{i,+1;\bvs}$ and $\TT''_{i,-1;\bvs}$ on $\Ui_\bvs$ satisfy the following intertwining relations, for all  $x\in \Ui_\bvs$:
\begin{align*}
\TT'_{i,+1}(x) \; \sT'_{\bs_i,+1;\bvs} (\fX_{i,\bvs}^{-1} )
&=  \sT'_{\bs_i,+1;\bvs} (\fX_{i,\bvs}^{-1} ) \; \sT'_{\bs_i,+1;\bvs}(x),
\\
\TT''_{i,-1}(x) \fX_{i,\bvs}  & =   \fX_{i,\bvs}  \sT''_{\bs_i,-1;\bvs} (x).
\end{align*}
\end{proposition}

\subsection{Compatible actions of $\Tae{i},\Tbe{i}$ on $\U$-modules}

Denote by $E_i^{(n)}, F_i^{(n)}$ the divided powers $\frac{E_i^n}{[n]_i!},\frac{F_i^n}{[n]_i!}$ in $\U$, for $n\in \N$.

Let $\mathcal{F}$ be the category of finite-dimensional $\U$-modules of type {\bf 1}. By definition, $M\in \mathcal{F}$ has a weight space decomposition (with respect a fixed $i\in \I$)
\[
M=\bigoplus_{n\in \Z} M_n,\qquad M_n=\{v\in M| K_i v = q_i^n v\}.
\]
Following \cite{Lus93}, we define linear operators $T'_{i,e},T''_{i,e},e=\pm1$ on $M$ by
\begin{align}
   \label{eq:mod1}
T'_{i,e}(v)&=\sum_{\substack{a,b,c \geq 0;\\a- b+c=m}}(-1)^b q_i^{e(b-ac)} F_i^{(a)}E_i^{(b)}F_i^{(c)}v,\qquad v\in M_m,\\
T''_{i,e}(v)&=\sum_{\substack{a,b,c \geq 0;\\-a+b-c=m}}(-1)^b q_i^{e(b-ac)} E_i^{(a)}F_i^{(b)}E_i^{(c)}v,\qquad v\in M_m.\label{eq:mod2}
\end{align}

\begin{proposition}
  \cite[39.4.3]{Lus93}
  Let $M \in \mathcal F$. Then, for any $u\in \U,v\in M,e=\pm 1$, we have
\begin{equation}
   \label{eq:mod3}
T'_{i,e}(uv)=T'_{i,e}(u)T'_{i,e}(v),\qquad T''_{i,e}(uv)=T''_{i,e}(u)T''_{i,e}(v).
\end{equation}
\end{proposition}

Recall $\sT'_{i,e;\bvs}, \sT''_{i,e;\bvs}$ are merely rescalings of $T'_{i,e},T''_{i,e}$ defined in \eqref{def:sT1} and \eqref{def:sT2}.
Applying exactly the same rescalings to the operators on modules \eqref{eq:mod1}--\eqref{eq:mod2}, we obtain operators $\sT'_{i,e;\bvs}, \sT''_{i,e;\bvs}$ on $M$ which satisfy
\begin{align}
\label{eq:mod3-2}
\sT'_{i,e;\bvs}(uv)&=\sT'_{i,e;\bvs} (u)\sT'_{i,e;\bvs} (v),\qquad \sT''_{i,e;\bvs}(uv)=\sT''_{i,e;\bvs}(u)\sT''_{i,e;\bvs}(v).
\end{align}
for any $u\in \U,v\in M$.

We regard the $\U$-module $M$ as an $\Ui$-module by restriction.
\begin{definition}
 \label{def:4T}
Define linear operators $\Tae{i},\Tbe{i}$ on $M$, for $i\in \wI$ and $e =\pm 1$, by
\begin{align}
  \label{eq:mod4}
  \begin{split}
\Ta{i}(v) & := \fX_{i,\bvs} \sT'_{\bs_i,-1;\bvs}(v),
\\
\Tb{i}(v) & := \sT''_{\bs_i,+1;\bvs}(\fX_{i,\bvs}^{-1}) \sT''_{\bs_i,+1;\bvs}(v),
\\
\TT'_{i,+1}(v) & :=
\sT'_{\bs_i,+1;\bvs}(\fX_{i,\bvs}^{-1}) \sT'_{\bs_i,+1;\bvs}(v),
\\
\TT''_{i,-1}(v) & := \fX_{i,\bvs} \sT''_{\bs_i,-1;\bvs}(v),
\end{split}
\end{align}
for any $v\in M$.

(In these notations, we have suppressed the dependence on $\bvs$ on these operators.)
\end{definition}

The automorphisms $\Tae{i},  \Tbe{i}$ on $M$ in \eqref{eq:mod4} are compatible with the corresponding automorphisms on $\Ui_\bvs$.

\begin{theorem}
 \label{thm:braidM}
Let $M \in\mathcal F$, $i\in \wI$ and $e=\pm 1$. Then we have
\begin{equation}
\Tae{i}(x v)=\Tae{i}(x) \Tae{i}(v),\qquad
\Tbe{i}(x v)=\Tbe{i}(x) \Tbe{i}(v),
\end{equation}
 for any $x\in \Ui_\bvs,v\in M.$
\end{theorem}

\begin{proof}
We prove the identity for $\Ta{i}$ ; the proofs for the remaining ones are similar. In the proof, we omit the subindex $\bvs$ for $\fX_{i,\bvs}$ and $\sT'_{\bs_i,-1;\bvs}$ as there is no confusion.

Since $ \sT'_{\bs_i,-1}(x v)= \sT'_{\bs_i,-1}(x)  \sT'_{\bs_i,-1}(v)$, we have
\begin{equation}\label{eq:mod5}
\fX_{i}  \sT'_{\bs_i,-1}(x v)
=\big(\fX_{i} \sT'_{\bs_i,-1}(x) \fX_{i}^{-1}\big) \fX_{i}  \sT'_{\bs_i,-1}(v),
\end{equation}
By Proposition~\ref{prop:inter1}, we have $\fX_{i} \sT_{\bs_i,-1}(x) \fX_{i}^{-1}= \Ta{i}(x).$ Hence, using the definition \eqref{eq:mod4}, the identity \eqref{eq:mod5} implies
that $\Ta{i}(x v)=\Ta{i}(x) \Ta{i}(v)$ as desired.
\end{proof}

\subsection{Relative braid relations on $\U$-modules}
\label{subsec:braidmodule}

Let $m_{ij}$ denotes the order of $\bs_i\bs_j$ in $\reW$.

\begin{theorem}
\label{thm:mod2}
Let $M \in \mathcal F$. Then the relative braid relations hold for the linear operators $\Tae{i}$ (and respectively, $\Tbe{i}$) on $M$; that is,
for any $i\neq j\in \wItau$ and for any $v\in M$, we have
\begin{align}
\label{eq:mod7}
    \underbrace{\Tae{i}\Tae{j}\Tae{i}\cdots}_{m_{ij}} (v)
    =
    \underbrace{\Tae{j}\Tae{i}\Tae{j} \cdots}_{m_{ij}}(v).
    \\
    \underbrace{\Tbe{i}\Tbe{j}\Tbe{i}\cdots}_{m_{ij}} (v)
    =
    \underbrace{\Tbe{j}\Tbe{i}\Tbe{j} \cdots}_{m_{ij}}(v).
\end{align}
\end{theorem}

\begin{proof}
We prove the first identity for $e=-1$ ; the proofs for the remaining ones are similar and skipped.

Set $m =m_{ij}$.
We keep the notations $\bbw,\bbw',\boldsymbol{w}_k,\boldsymbol{w}_k'$ for $1\leq k\leq m$ from the proof of Theorem~\ref{thm:newb2}.
We shall write $\sT_{\bs_i}^{-1}$ for $\sT'_{\bs_i,-1,\bvs}$ and omit the subindex $\bvs$ for $\fX_{i,\bvs}$ in the proof, since there is no confusion.

By definition \eqref{eq:mod4}, for any $v \in M$, we have
\begin{align}
 \label{eq:TM}
  \underbrace{\Ta{i}\Ta{j}\Ta{i}\cdots}_{m } (v)
   = (\fX_{i} \sT_{\boldsymbol{w}_1}^{-1}\fX_{j}\sT_{\boldsymbol{w}_2}^{-1} \fX_{i} \cdots)   \underbrace{\sT_{\bs_i}^{-1}\sT_{\bs_j}^{-1}\sT_{\bs_i}^{-1}\cdots}_{m }(v)
\end{align}
 By taking a central reduction to \eqref{eq:braid5}, the first factor on RHS \eqref{eq:TM} is $\sigma(\fX_{\bbw})$. Hence, we have
\begin{align}
\label{eq:mod8}
  \underbrace{\Ta{i}\Ta{j}\Ta{i}\cdots}_{m} (v)
  &= \sigma(\fX_{\bbw})\underbrace{\sT_{\bs_i}^{-1}\sT_{\bs_j}^{-1}\sT_{\bs_i}^{-1}\cdots}_{m } (v).
\end{align}
Similarly, by switching $i,j$ in \eqref{eq:mod8}, we obtain
\begin{align}
\label{eq:mod9}
    \underbrace{\Ta{j}\Ta{i}\Ta{j}\cdots}_{m} (v)
    &=
    \sigma(\fX_{\bbw'}) \underbrace{\sT_{\bs_j}^{-1}\sT_{\bs_i}^{-1}\sT_{\bs_j}^{-1}\cdots}_{m } (v).
\end{align}
Applying a central reduction to Theorem~\ref{thm:factor}, we have $\fX_{\bbw}=\fX_{\bbw'}$. Since $\sT_i$ are defined by rescaling $T''_{i,+1}$ in \eqref{def:sT1}, they satisfy the braid relations. Hence, we have
\begin{align}
\label{eq:mod8L}
    \underbrace{\sT_{\bs_i}^{-1}\sT_{\bs_j}^{-1}\sT_{\bs_i}^{-1}\cdots}_{m} (v) =\underbrace{\sT_{\bs_j}^{-1}\sT_{\bs_i}^{-1}\sT_{\bs_j}^{-1}\cdots}_{m} (v).
\end{align}
Combining \eqref{eq:mod8}--\eqref{eq:mod8L}, we have proved the first identity for $e=-1$.
\end{proof}

%

\newpage

 \begin{table}[H]
\caption{Rank 2 formulas for $\tTa{i}(B_j)$ \, ($i\neq j \in \wItau$)}
     \label{table:rktwoSatake}
     \resizebox{5.5 in}{!}{%
\begin{tabular}{| c  | c |}
\hline

\begin{tikzpicture}[baseline=0]
\node at (0, -0.15) {Rank 2 Satake diagrams};
\end{tikzpicture}
&

\begin{tikzpicture}[baseline=0]
\node at (0, -0.15) {Formulas for $\tTa{i}(B_j)$};
\end{tikzpicture}\\
\hline
\begin{tikzpicture}[baseline=0, scale=1.2]
		\node at (-0.5,0) {$\circ$};
		\node at (0.5,0) {$\circ$};
       \draw[-]  (0.45, 0) to (-0.45, 0);
		\node at (-0.5, -.2) {\small 1};
		\node at (0.5,-.2) {\small 2};
\node at (-1.5, -0.05) {AI$_2$};
	\end{tikzpicture}
&
\begin{tikzpicture}[baseline=0]
\node at (0, -0.15) {$\tTa{1}(B_2) = [B_1,B_2]_q$};
\end{tikzpicture}
\\
\hline
\begin{tikzpicture}[baseline=0, scale=1.2]
		\node at (-0.5,0) {$\circ$};
		\node at (0.5,0) {$\circ$};
		\draw[-implies, double equal sign distance]  (0.4, 0) to (-0.4, 0);
		\node at (-0.5, -.2) {\small 1};
		\node at (0.5,-.2) {\small 2};
\node at (-1.5,-0.05) {CI$_2$};
	\end{tikzpicture}
&
\begin{tikzpicture}[baseline=0]
\node at (0, -0.15) {$\tTa{1}(B_2) = \frac{1}{[2]_{q_1}} \big[B_1,[B_1,B_2]_{q_1^2}\big]-q_1^2 B_2 \ck_1$};
\end{tikzpicture}
\\
\hline
\begin{tikzpicture}[baseline=0, scale=1.5]
		\node at (-0.5,0) {$\circ$};
		\node at (0.5,0) {$\circ$};
		\draw[->]  (0.4, 0) to (-0.4, 0);
		\draw[->]  (0.4, 0.05) to (-0.4, 0.05);
		\draw[->]  (0.4, -0.05) to (-0.4, -0.05);
		\node at (-0.5, -.2) {\small 1};
		\node at (0.5,-.2) {\small 2};
\node at (-1 , -0.05) {G$_2$};
	\end{tikzpicture}
&
\begin{tikzpicture}[baseline=0]
\node at (0,  0.15) {$\tTa{1}(B_2) = \frac{1}{[3]! } \Big[B_1,\big[B_1, [B_1,B_2]_{q^3} \big]_q \Big]_{q^{-1}}$};
\node at (0,  -0.55) {$\quad -\frac{1}{[3]!} \big( q(1+[3])[B_1, B_2]_{q^3} + q^3[3] [B_1,B_2]_{q^{-1}}\big)\tk_1$};
\end{tikzpicture}
\\
\hline
\begin{tikzpicture}[baseline=0]
\end{tikzpicture}
    \begin{tikzpicture}[baseline=0, scale=1.2]
		\node at (0.5,0) {$\circ$};
		\node at (1.0,0) {$\circ$};
		\node at (1.5,0) {$\bullet$};
		\draw[-] (0.55,0)  to (0.95,0);
		\draw[-] (1.05,0)  to (1.45,0);
		\draw[-] (1.55,0) to (1.8, 0);
		\draw[dashed] (1.8,0) to (2.5,0);
		\draw[-] (2.5,0) to (2.75, 0);
		\node at (2.8,0) {$\bullet$};
		\draw[-implies, double equal sign distance]  (2.85, 0) to (3.45, 0);
		\node at (3.5,0) {$\bullet$};
		\node at (0.5,-.2) {\small 1};
		\node at (1,-.2) {\small 2};
		\node at (1.5,-.2) {\small 3};
		\node at (3.5,-.2) {$n$};
 \node at (2.5, -0.35) {BI$_n,n\geq 3$};
	\end{tikzpicture}
&
\begin{tikzpicture}[baseline=0]
\node at (0, -0.15) {$\tTa{2}(B_1)=\big[ \tT_{\bw} ( B_2), [B_2 ,B_1 ]_{q_2}\big]_{q_2} - q_2  B_1 \tT_{\bw}( \ck_2)$};
\end{tikzpicture}\\
\hline
\begin{tikzpicture}[baseline=0]
\end{tikzpicture}
    \begin{tikzpicture}[baseline=0, scale=1.2]
		\node at (0.55,0) {$\circ$};
		\node at (1.05,0) {$\circ$};
		\node at (1.5,0) {$\bullet$};
		\draw[-] (0.6,0)  to (1.0,0);
		\draw[-] (1.1,0)  to (1.4,0);
		\draw[-] (1.4,0) to (1.9, 0);
		\draw[dashed] (1.9,0) to (2.7,0);
		\draw[-] (2.7,0) to (2.9, 0);
		\node at (3,0) {$\bullet$};
		\node at (3.8,0.5) {$\bullet$};
		\node at (3.8,-0.5) {$\bullet$};
        \draw (3,0) to (3.8,0.5);
        \draw (3,0) to (3.8,-0.5);
		\node at (0.5,-.2) {\small 1};
		\node at (1,-.2) {\small 2};
		\node at (1.5,-.2) {\small 3};
 \node at (2.5, -0.45) {DI$_n,n\geq 5$};
	\end{tikzpicture}
&
\begin{tikzpicture}[baseline=0]
\node at (0, -0.15) {$\tTa{2}(B_1)=\big[ \tT_{\bw} ( B_2), [B_2 ,B_1 ]_q\big]_q -q B_1 \tT_{\bw}( \ck_2)$};
\end{tikzpicture}
\\
\hline
\begin{tikzpicture}[baseline=0]
\end{tikzpicture}
  \begin{tikzpicture}[baseline=0, scale=1.2]
		\node at (0.65,0) {$\circ$};
		\node at (1.5,0) {$\circ$};
		\draw[-] (0.7,0)  to (1.45,0);
		\node at (2.3,0.5) {$\bullet$};
		\node at (2.3,-0.5) {$\bullet$};
        \draw (1.55,0) to (2.3,0.5);
        \draw (1.55,0) to (2.3,-0.5);
		\node at (0.65,-.2) {\small 1};
		\node at (1.5,-.2) {\small 2};
		\node at (2.3,0.3) {\small 3};
		\node at (2.3,-0.3) {\small 4};
 \node at (-.5, 0) {DIII$_4$};
	\end{tikzpicture}
&
\begin{tikzpicture}[baseline=0]
\node at (0, -0.15) {$\tTa{2}(B_1)=\big[ \tT_{\bw} ( B_2), [B_2 ,B_1 ]_q\big]_q -q  B_1 \tT_{\bw}( \ck_2)$};
\end{tikzpicture}
\\
\hline

\begin{tikzpicture}[baseline=0]
\end{tikzpicture}
 \begin{tikzpicture}[baseline=0,scale=1.2]
		\node at (-0.5,0) {$\bullet$};
		\node  at (0,0) {$\circ$};
		\node at (0.5,0) {$\bullet$};
		\node at (1,0) {$\circ$};
		\node at (1.5,0) {$\bullet$};
		\draw[-] (-0.5,0) to (-0.05, 0);
		\draw[-] (0.05, 0) to (0.5,0);
		\draw[-] (0.5,0) to (0.95,0);
		\draw[-] (1.05,0)  to (1.5,0);
		\node at (-0.5,-0.2) {1};
		\node  at (0,-0.2) {2};
		\node at (0.5,-0.2) {3};
		\node at (1,-0.2) {4};
		\node at (1.5,-0.2) {5};
\node at (-1.5, -0.05) {AII$_5$};
	\end{tikzpicture}
&
\begin{tikzpicture}[baseline=0]
\node at (0, -0.15) {$\tTa{4}(B_2)=[\tT_3(B_4),B_2]_q$};
\end{tikzpicture}
\\
\hline

\begin{tikzpicture}[baseline=6,scale=1.2]
		\node  at (0,0) {$\bullet$};
		\node  at (0,-0.2) {1};
		\draw (0.05, 0) to (0.45, 0);
		\node  at (0.5,0) {$\circ$};
		\node  at (0.5,-0.2) {2};
		\draw (0.55, 0) to (0.95, 0);
		\node at (1,0) {$\bullet$};
		\node at (1,-.2) {3};
		\node at (1.5,0) {$\circ$};
		\node at (1.5,-0.2) {4};
		\draw[-] (1.05,0)  to (1.45,0);
		\draw[-] (1.55,0) to (1.95, 0);
		\node at (2,0) {$\bullet$};
		\node at (2,-0.2) {5};
		\draw (1.9, 0) to (2.1, 0);
		\draw[dashed] (2.1,0) to (2.7,0);
		\draw[-] (2.7,0) to (2.9, 0);
		\node at (3,0) {$\bullet$};
		\draw[implies-, double equal sign distance]  (3.1, 0) to (3.7, 0);
		\node at (3.8,0) {$\bullet$};
		\node at (3.8,-0.2) {$n$};
\node at (2.15, -0.55) {CII$_n,n\geq 5$};
	\end{tikzpicture}
&
\begin{tikzpicture}[baseline=0]
\node at (0, -0.15) {$\tTa{4}(B_2)=\big[ [\tT_{5\cdots n\cdots 5}(B_4 ), \tT_3(B_4) ]_{q_2},B_2\big]_{q_2}$};
\node at (0.8, -0.75) {$ - q_2\tT_3^{-2}(B_2) \tT_{5\cdots n\cdots 5}(\ck_4)$};
\end{tikzpicture}
\\
\hline

\begin{tikzpicture}[baseline=6,scale=1.5]
        \node at (-1,0) {$\bullet$};
        \node at (-1,-0.2) {1};
		\draw[-] (-0.95,0) to (-0.55, 0);
        \node at (-0.5,0) {$\circ$};
        \node at (-0.5,-0.2) {2};
		\draw[-] (-.45,0) to (-0.05, 0);
		\node at (0,0) {$\bullet$};
		\node at (0,-0.2) {3};
		\draw[implies-, double equal sign distance]  (0.05, 0) to (0.75, 0);
		\node at (0.8,0) {$\circ$};
		\node at (0.8,-0.2) {4};
\node at (-1.5, 0) {CII$_4$};
	\end{tikzpicture}
&

\begin{tikzpicture}[baseline=0]
\node at (0,0.45) {$\tTa{4}(B_2)=\big[[B_4,F_3]_{q_4},B_2 \big]_{q_3}$};
\node at (0, -0.15) {$\tTa{2}(B_4)=\big[ \tT_3(B_2),[\tT_3(B_2),B_4]_{q_3^2}\big]$};
\node at (1.6,-0.85) {$-(q_3-q_3^{-1})[F_3,B_4]_{q_3^2}E_1 \tT_3(\ck_2) K_1'^{-1}$};
\end{tikzpicture}
\\
\hline
\begin{tikzpicture}[baseline = 0, scale =1.5]
		\node at (-1,0) {$\circ$};
        \node at (-1,-0.2) {1};
		\draw (-0.95,0) to (-0.55,0);
		\node at (-0.5,0) {$\bullet$};
        \node at (-0.5,-0.2) {2};
		\draw (-0.45,0) to (-0.05,0);
		\node at (0,0) {$\bullet$};
        \node at (0.1,-0.2) {3};
		\draw (0.05,0) to (0.45,0);
		\node at (0.5,0) {$\bullet$};
		\node at (0.5,-0.2) {4};
		\draw (0.55,0) to (0.95,0);
		\node at (1,0) {$\circ$};
		\node at (1,-0.2) {5};
		\draw (0,-0.05) to (0,-0.45);
		\node at (0,-0.4) {$\bullet$};
		\node at (-.15,-0.35) {6};
\node at (-1.5, -0.05) {EIV};
	\end{tikzpicture}
&
\begin{tikzpicture}[baseline=0]
\node at (0, -0.15) {$\tTa{1}(B_5)=\big[ \tT_4\tT_3\tT_2(B_1) ,B_5\big]_q$};
\end{tikzpicture}
\\
\hline
\begin{tikzpicture}[baseline=0,scale=1.5]
		\node  at (-0.65,0) {$\circ$};
		\node  at (0,0) {$\circ$};
		\node  at (0.65,0) {$\circ$};
		\draw[-] (-0.6,0) to (-0.05, 0);
		\draw[-] (0.05, 0) to (0.6,0);
		\node at (-0.65,-0.15) {1};
		\node at (0,-0.15) {2};
		\node at (0.65,-0.15) {3};
        \draw[bend left,<->,red] (-0.65,0.1) to (0.65,0.1);
        \node at (0,0.2) {$\textcolor{red}{\tau}$};
        \node at (-1.5,-0.05) {AIII$_3$};
	\end{tikzpicture}
&

\begin{tikzpicture}[baseline=0]
\node at (0, -0.15) {$\tTa{1}(B_2)=\big[B_3, [B_1, B_2]_q\big]_q -q B_2 \ck_3$};
\end{tikzpicture}
\\
\hline
 \begin{tikzpicture}[baseline=0,scale=1.2]
		\node  at (-2.1,0) {$\circ$};
		\node  at (-1.3,0) {$\circ$};
		\node  at (-0.5,0) {$\bullet$};
		\node  at (0.5,0) {$\bullet$};
		\node  at (1.3,0) {$\circ$};
		\node  at (2.1,0) {$\circ$};
		\draw[-] (-2.05,0) to (-1.35, 0);
		\draw[-] (-1.25,0) to (-0.55, 0);
		\draw[-] (0.55,0) to (1.25, 0);
		\draw[-] (1.35, 0) to (2.05,0);
		\node at (-2.1,-0.2) {1};
		\node at (-1.3,-0.2) {2};
		\node at (-0.5,-0.2) {3};
		\node at (1.3,-0.2) {\small$n-1$};
		\node at (2.1,-0.2) {$n$ };
        \draw[dashed] (-0.5,0) to (0.5,0);
        \draw[bend left,<->,red] (-1.3,0.1) to (1.3,0.1);
        \draw[bend left,<->,red] (-2.1,0.1) to (2.1,0.1);
        \node at (0,0.6) {$\textcolor{red}{\tau} $};
\node at (0, -0.55) {AIII$_n,n\geq 4$};
	\end{tikzpicture}
&

\begin{tikzpicture}[baseline=0]
\node at (0, 0.55) {$\tTa{1}(B_2)=[B_1 ,B_2]_q$};
\node at (0, -0.15) { $\tTa{2}(B_1)=\big[ \tT_{\bw}(B_{n-1} ),[B_2 ,B_1]_q\big]_q- B_1  \tT_{\bw}(\ck_{n-1})$};
\end{tikzpicture}
\\
\hline

\begin{tikzpicture}[baseline=0,scale=1.5]
		\node at (-0.5,0) {$\bullet$};
		\node at (0,0) {$\circ$};
		\node at (0.5,0) {$\bullet$};
        \node at (1,0.3) {$\circ$};
        \node at (1,-0.3) {$\circ$};
		\draw[-] (-0.45,0) to (-0.05, 0);
		\draw[-] (0.05, 0) to (0.45,0);
		\draw[-] (0.5,0) to (0.965, 0.285);
		\draw[-] (0.5, 0) to (0.965,-0.285);
		\node at (-0.5,-0.2) {1};
		\node at (0,-0.2) {2};
		\node at (0.5,-0.2) {3};
        \node at (1,0.15) {4};
        \node at (1,-0.45) {5};
        \draw[bend left,<->,red] (1.1,0.3) to (1.1,-0.3);
        \node at (1.4,0) {$\textcolor{red}{\tau} $};
\node at (-1, 0) {DIII$_5$};
	\end{tikzpicture}
&
\begin{tikzpicture}[baseline=0]
\node at (0,0.25) {$\tTa{2}(B_4)=[\tT_3(B_2 ), B_4]_q$};
\node at (0, -0.55) {$\tTa{4}(B_2)=\big[B_4, [\tT_3(B_5) ,B_2]_q\big]_q - \tT_3^{-2}(B_2)\ck_4$};
\end{tikzpicture}
\\
\hline
\begin{tikzpicture}[baseline = 0, scale =1.5]
		\node at (-1,0) {$\circ$};
        \node at (-1,-0.2) {1};
		\draw (-0.95,0) to (-0.55,0);
		\node at (-0.5,0) {$\bullet$};
        \node at (-0.5,-0.2) {2};
		\draw (-0.45,0) to (-0.05,0);
		\node at (0,0) {$\bullet$};
        \node at (0.1,-0.2) {3};
		\draw (0.05,0) to (0.45,0);
		\node at (0.5,0) {$\bullet$};
		\node at (0.5,-0.2) {4};
		\draw (0.55,0) to (0.95,0);
		\node at (1,0) {$\circ$};
		\node at (1,-0.2) {5};
		\draw (0,-0.05) to (0,-0.35);
		\node at (0,-0.4) {$\circ$};
		\node at (0,-0.55) {6};
        \draw[bend left, <->, red] (-0.9,0.1) to (0.9,0.1);
        \node at (0,0.45) {$\color{red} \tau $};
\node at (-1.5, 0) {EIII};
	\end{tikzpicture}
&
\begin{tikzpicture}[baseline=0]
\node at (0,0.45) {$\tTa{6}(B_1)=[\tT_{23}(B_6 ),B_1]_q$};
\node at (0, -0.35) {$\tTa{1}(B_6)=\big[\tT_4(B_5),[\tT_{32}(B_1),B_6]_q\big]_q$ };
\node at (0.5,-.85) {
$\quad -\tT_{32323}^{-1}(B_6) \tT_4(\ck_5)$};
\end{tikzpicture}
\\
\hline
\end{tabular}
}
\newline
\end{table}

%
%
%
%

%

\newpage

 \begin{table}[H]
\caption{Rank 2 formulas for $\tTb{i}(B_j)$ \, ($i\neq j \in \wItau$)}
     \label{table:rktwoSatake2}
\resizebox{5.5 in}{!}{%
\begin{tabular}{| c | c|}
\hline

\begin{tikzpicture}[baseline=0]
\node at (0, -0.15) {Rank 2 Satake diagrams};
\end{tikzpicture}
&

\begin{tikzpicture}[baseline=0]
\node at (0, -0.15) {Formulas for $\tTb{i}(B_j)$};
\end{tikzpicture}\\
\hline
\begin{tikzpicture}[baseline=0, scale=1.2]
		\node at (-0.5,0) {$\circ$};
		\node at (0.5,0) {$\circ$};
       \draw[-]  (0.45, 0) to (-0.45, 0);
		\node at (-0.5, -.2) {\small 1};
		\node at (0.5,-.2) {\small 2};
\node at (-1.5, -0.05) {AI$_2$};
	\end{tikzpicture}
&
\begin{tikzpicture}[baseline=0]
\node at (0, -0.15) {$\tTb{1}(B_2) = [B_2,B_1]_q$};
\end{tikzpicture}
\\
\hline
\begin{tikzpicture}[baseline=0, scale=1.2]
		\node at (-0.5,0) {$\circ$};
		\node at (0.5,0) {$\circ$};
		\draw[-implies, double equal sign distance]  (0.4, 0) to (-0.4, 0);
		\node at (-0.5, -.2) {\small 1};
		\node at (0.5,-.2) {\small 2};
\node at (-1.5,-0.05) {CI$_2$};
	\end{tikzpicture}
&
\begin{tikzpicture}[baseline=0]
\node at (0, -0.15) {$\tTb{1}(B_2) = \frac{1}{[2]_{q_1}} \big[[B_2 ,B_1]_{q_1^2},B_1\big]-q_1^2 B_2 \ck_1$};
\end{tikzpicture}
\\
\hline
\begin{tikzpicture}[baseline=0, scale=1.5]
		\node at (-0.5,0) {$\circ$};
		\node at (0.5,0) {$\circ$};
		\draw[->]  (0.4, 0) to (-0.4, 0);
		\draw[->]  (0.4, 0.05) to (-0.4, 0.05);
		\draw[->]  (0.4, -0.05) to (-0.4, -0.05);
		\node at (-0.5, -.2) {\small 1};
		\node at (0.5,-.2) {\small 2};
\node at (-1 , -0.05) {G$_2$};
	\end{tikzpicture}
&
\begin{tikzpicture}[baseline=0]
\node at (0,  0.15) {$\tTb{1}(B_2) = \frac{1}{[3]_1! } \Big[\big[ [B_2, B_1]_{q_1^3},B_1 \big]_{q_1}, B_1\Big]_{q_1^{-1}}$};
\node at (0,  -0.55) {$\quad -\frac{1}{[3]_1!} \big( q_1(1+[3]_1)[ B_2, B_1]_{q_1^3} + q_1^3[3]_1 [B_2, B_1]_{q_1^{-1}}\big)\tk_1$};
\end{tikzpicture}
\\
\hline
\begin{tikzpicture}[baseline=0]
\end{tikzpicture}
    \begin{tikzpicture}[baseline=0, scale=1.2]
		\node at (0.5,0) {$\circ$};
		\node at (1.0,0) {$\circ$};
		\node at (1.5,0) {$\bullet$};
		\draw[-] (0.55,0)  to (0.95,0);
		\draw[-] (1.05,0)  to (1.45,0);
		\draw[-] (1.55,0) to (1.8, 0);
		\draw[dashed] (1.8,0) to (2.5,0);
		\draw[-] (2.5,0) to (2.75, 0);
		\node at (2.8,0) {$\bullet$};
		\draw[-implies, double equal sign distance]  (2.85, 0) to (3.45, 0);
		\node at (3.5,0) {$\bullet$};
		\node at (0.5,-.2) {\small 1};
		\node at (1,-.2) {\small 2};
		\node at (1.5,-.2) {\small 3};
		\node at (3.5,-.2) {$n$};
 \node at (2.5, -0.35) {BI$_n,n\geq 3$};
	\end{tikzpicture}
&
\begin{tikzpicture}[baseline=0]
\node at (0, -0.15) {$\tTb{2}(B_1)=\big[[B_1, B_2 ]_{q_2}, \tT_{\bw}^{-1} ( B_2)\big]_{q_2} - q_2  B_1  \ck_2 $};
\end{tikzpicture}\\
\hline
\begin{tikzpicture}[baseline=0]
\end{tikzpicture}
    \begin{tikzpicture}[baseline=0, scale=1.2]
		\node at (0.55,0) {$\circ$};
		\node at (1.05,0) {$\circ$};
		\node at (1.5,0) {$\bullet$};
		\draw[-] (0.6,0)  to (1.0,0);
		\draw[-] (1.1,0)  to (1.4,0);
		\draw[-] (1.4,0) to (1.9, 0);
		\draw[dashed] (1.9,0) to (2.7,0);
		\draw[-] (2.7,0) to (2.9, 0);
		\node at (3,0) {$\bullet$};
		\node at (3.8,0.5) {$\bullet$};
		\node at (3.8,-0.5) {$\bullet$};
        \draw (3,0) to (3.8,0.5);
        \draw (3,0) to (3.8,-0.5);
		\node at (0.5,-.2) {\small 1};
		\node at (1,-.2) {\small 2};
		\node at (1.5,-.2) {\small 3};
 \node at (2.5, -0.45) {DI$_n,n\geq 5$};
	\end{tikzpicture}
&
\begin{tikzpicture}[baseline=0]
\node at (0, -0.15) {$\tTb{2}(B_1)=\big[[B_1, B_2 ]_{q}, \tT_{\bw}^{-1} ( B_2)\big]_{q} - q  B_1  \ck_2 $};
\end{tikzpicture}
\\
\hline
\begin{tikzpicture}[baseline=0]
\end{tikzpicture}
  \begin{tikzpicture}[baseline=0, scale=1.2]
		\node at (0.65,0) {$\circ$};
		\node at (1.5,0) {$\circ$};
		\draw[-] (0.7,0)  to (1.45,0);
		\node at (2.3,0.5) {$\bullet$};
		\node at (2.3,-0.5) {$\bullet$};
        \draw (1.55,0) to (2.3,0.5);
        \draw (1.55,0) to (2.3,-0.5);
		\node at (0.65,-.2) {\small 1};
		\node at (1.5,-.2) {\small 2};
		\node at (2.3,0.3) {\small 3};
		\node at (2.3,-0.3) {\small 4};
 \node at (-.5, 0) {DIII$_4$};
	\end{tikzpicture}
&
\begin{tikzpicture}[baseline=0]
\node at (0, -0.15) {$ \tTb{2}(B_1)=\big[[B_1, B_2 ]_{q}, \tT_{\bw}^{-1} ( B_2)\big]_{q} - q  B_1  \ck_2 $};
\end{tikzpicture}
\\
\hline

\begin{tikzpicture}[baseline=0]
\end{tikzpicture}
 \begin{tikzpicture}[baseline=0,scale=1.2]
		\node at (-0.5,0) {$\bullet$};
		\node  at (0,0) {$\circ$};
		\node at (0.5,0) {$\bullet$};
		\node at (1,0) {$\circ$};
		\node at (1.5,0) {$\bullet$};
		\draw[-] (-0.5,0) to (-0.05, 0);
		\draw[-] (0.05, 0) to (0.5,0);
		\draw[-] (0.5,0) to (0.95,0);
		\draw[-] (1.05,0)  to (1.5,0);
		\node at (-0.5,-0.2) {1};
		\node  at (0,-0.2) {2};
		\node at (0.5,-0.2) {3};
		\node at (1,-0.2) {4};
		\node at (1.5,-0.2) {5};
\node at (-1.5, -0.05) {AII$_5$};
	\end{tikzpicture}
&
\begin{tikzpicture}[baseline=0]
\node at (0, -0.15) {$\tTb{4}(B_2)=[B_2,\tT_3^{-1}(B_4)]_q$};
\end{tikzpicture}
\\
\hline

\begin{tikzpicture}[baseline=6,scale=1.2]
		\node  at (0,0) {$\bullet$};
		\node  at (0,-0.2) {1};
		\draw (0.05, 0) to (0.45, 0);
		\node  at (0.5,0) {$\circ$};
		\node  at (0.5,-0.2) {2};
		\draw (0.55, 0) to (0.95, 0);
		\node at (1,0) {$\bullet$};
		\node at (1,-.2) {3};
		\node at (1.5,0) {$\circ$};
		\node at (1.5,-0.2) {4};
		\draw[-] (1.05,0)  to (1.45,0);
		\draw[-] (1.55,0) to (1.95, 0);
		\node at (2,0) {$\bullet$};
		\node at (2,-0.2) {5};
		\draw (1.9, 0) to (2.1, 0);
		\draw[dashed] (2.1,0) to (2.7,0);
		\draw[-] (2.7,0) to (2.9, 0);
		\node at (3,0) {$\bullet$};
		\draw[implies-, double equal sign distance]  (3.1, 0) to (3.7, 0);
		\node at (3.8,0) {$\bullet$};
		\node at (3.8,-0.2) {$n$};
\node at (2.15, -0.55) {CII$_n,n\geq 5$};
	\end{tikzpicture}
&
\begin{tikzpicture}[baseline=0]
\node at (0, -0.15) {$\tTb{4}(B_2)=\big[B_2,[\tT_3^{-1}(B_4), \tT^{-1}_{5\cdots n\cdots 5}(B_4 )]_{q_2}\big]_{q_2}$};
\node at (0.8, -0.75) {$ - q_2^2 \tT_3^{ 2}(B_2) \tT_{3}(\ck_4)$};
\end{tikzpicture}
\\
\hline

\begin{tikzpicture}[baseline=6,scale=1.5]
        \node at (-1,0) {$\bullet$};
        \node at (-1,-0.2) {1};
		\draw[-] (-0.95,0) to (-0.55, 0);
        \node at (-0.5,0) {$\circ$};
        \node at (-0.5,-0.2) {2};
		\draw[-] (-.45,0) to (-0.05, 0);
		\node at (0,0) {$\bullet$};
		\node at (0,-0.2) {3};
		\draw[implies-, double equal sign distance]  (0.05, 0) to (0.75, 0);
		\node at (0.8,0) {$\circ$};
		\node at (0.8,-0.2) {4};
\node at (-1.5, 0) {CII$_4$};
	\end{tikzpicture}
&

\begin{tikzpicture}[baseline=0]
\node at (0,0.45) {$\tTb{4}(B_2)=\big[B_2,[ F_3,B_4]_{q_4} \big]_{q_3}$};
\node at (0, -0.15) {$\tTb{2}(B_4)=\big[ [B_4,\tT_3^{-1}(B_2)]_{q_3^2},\tT_3^{-1}(B_2)\big]$};
\node at (1.6,-0.85) {$-(q_3-q_3^{-1})[ B_4,F_3]_{q_3^2}E_1 \ck_2 K_1'^{-1}$};
\end{tikzpicture}
\\
\hline
\begin{tikzpicture}[baseline = 0, scale =1.5]
		\node at (-1,0) {$\circ$};
        \node at (-1,-0.2) {1};
		\draw (-0.95,0) to (-0.55,0);
		\node at (-0.5,0) {$\bullet$};
        \node at (-0.5,-0.2) {2};
		\draw (-0.45,0) to (-0.05,0);
		\node at (0,0) {$\bullet$};
        \node at (0.1,-0.2) {3};
		\draw (0.05,0) to (0.45,0);
		\node at (0.5,0) {$\bullet$};
		\node at (0.5,-0.2) {4};
		\draw (0.55,0) to (0.95,0);
		\node at (1,0) {$\circ$};
		\node at (1,-0.2) {5};
		\draw (0,-0.05) to (0,-0.45);
		\node at (0,-0.4) {$\bullet$};
		\node at (-.15,-0.35) {6};
\node at (-1.5, -0.05) {EIV};
	\end{tikzpicture}
&
\begin{tikzpicture}[baseline=0]
\node at (0, -0.15) {$\tTb{1}(B_5)=\big[ B_5, \tT_{4}^{-1}\tT_{3}^{-1}\tT_{2}^{-1} (B_1)\big]_q$};
\end{tikzpicture}
\\
\hline
\begin{tikzpicture}[baseline=0,scale=1.5]
		\node  at (-0.65,0) {$\circ$};
		\node  at (0,0) {$\circ$};
		\node  at (0.65,0) {$\circ$};
		\draw[-] (-0.6,0) to (-0.05, 0);
		\draw[-] (0.05, 0) to (0.6,0);
		\node at (-0.65,-0.15) {1};
		\node at (0,-0.15) {2};
		\node at (0.65,-0.15) {3};
        \draw[bend left,<->,red] (-0.65,0.1) to (0.65,0.1);
        \node at (0,0.2) {$\textcolor{red}{\tau}$};
        \node at (-1.5,-0.05) {AIII$_3$};
	\end{tikzpicture}
&

\begin{tikzpicture}[baseline=0]
\node at (0, -0.15) {$\tTb{1}(B_2)=\big[[ B_2,B_1]_q,B_3\big]_q -q B_2 \ck_1$};
\end{tikzpicture}
\\
\hline
 \begin{tikzpicture}[baseline=0,scale=1.2]
		\node  at (-2.1,0) {$\circ$};
		\node  at (-1.3,0) {$\circ$};
		\node  at (-0.5,0) {$\bullet$};
		\node  at (0.5,0) {$\bullet$};
		\node  at (1.3,0) {$\circ$};
		\node  at (2.1,0) {$\circ$};
		\draw[-] (-2.05,0) to (-1.35, 0);
		\draw[-] (-1.25,0) to (-0.55, 0);
		\draw[-] (0.55,0) to (1.25, 0);
		\draw[-] (1.35, 0) to (2.05,0);
		\node at (-2.1,-0.2) {1};
		\node at (-1.3,-0.2) {2};
		\node at (-0.5,-0.2) {3};
		\node at (1.3,-0.2) {\small$n-1$};
		\node at (2.1,-0.2) {$n$ };
        \draw[dashed] (-0.5,0) to (0.5,0);
        \draw[bend left,<->,red] (-1.3,0.1) to (1.3,0.1);
        \draw[bend left,<->,red] (-2.1,0.1) to (2.1,0.1);
        \node at (0,0.6) {$\textcolor{red}{\tau} $};
\node at (0, -0.55) {AIII$_n,n\geq 4$};
	\end{tikzpicture}
&

\begin{tikzpicture}[baseline=0]
\node at (0, 0.55) {$\tTb{1}(B_2)=[B_2 ,B_1]_q$};
\node at (0, -0.15) { $\tTb{2}(B_1)=\big[ [ B_1,B_2]_q,\tT_{\bw}^{-1}(B_{n-1} )\big]_q- \ck_{2} B_1$};
\end{tikzpicture}
\\
\hline

\begin{tikzpicture}[baseline=0,scale=1.5]
		\node at (-0.5,0) {$\bullet$};
		\node at (0,0) {$\circ$};
		\node at (0.5,0) {$\bullet$};
        \node at (1,0.3) {$\circ$};
        \node at (1,-0.3) {$\circ$};
		\draw[-] (-0.45,0) to (-0.05, 0);
		\draw[-] (0.05, 0) to (0.45,0);
		\draw[-] (0.5,0) to (0.965, 0.285);
		\draw[-] (0.5, 0) to (0.965,-0.285);
		\node at (-0.5,-0.2) {1};
		\node at (0,-0.2) {2};
		\node at (0.5,-0.2) {3};
        \node at (1,0.15) {4};
        \node at (1,-0.45) {5};
        \draw[bend left,<->,red] (1.1,0.3) to (1.1,-0.3);
        \node at (1.4,0) {$\textcolor{red}{\tau} $};
\node at (-1, 0) {DIII$_5$};
	\end{tikzpicture}
&
\begin{tikzpicture}[baseline=0]
\node at (0,0.25) {$\tTb{2}(B_4)=[ B_4,\tT_3^{-1}(B_2 )]_q$};
\node at (0, -0.55) {$\tTb{4}(B_2)=\big[ [B_2,\tT_3^{-1}(B_5)]_q ,B_4\big]_q - q \tT_3^{2}(B_2)\tT_3(\ck_5)$};
\end{tikzpicture}
\\
\hline
\begin{tikzpicture}[baseline = 0, scale =1.5]
		\node at (-1,0) {$\circ$};
        \node at (-1,-0.2) {1};
		\draw (-0.95,0) to (-0.55,0);
		\node at (-0.5,0) {$\bullet$};
        \node at (-0.5,-0.2) {2};
		\draw (-0.45,0) to (-0.05,0);
		\node at (0,0) {$\bullet$};
        \node at (0.1,-0.2) {3};
		\draw (0.05,0) to (0.45,0);
		\node at (0.5,0) {$\bullet$};
		\node at (0.5,-0.2) {4};
		\draw (0.55,0) to (0.95,0);
		\node at (1,0) {$\circ$};
		\node at (1,-0.2) {5};
		\draw (0,-0.05) to (0,-0.35);
		\node at (0,-0.4) {$\circ$};
		\node at (-.15,-0.35) {6};
        \draw[bend left, <->, red] (-0.9,0.1) to (0.9,0.1);
        \node at (0,0.45) {$\color{red} \tau $};
\node at (-1.5, 0) {EIII};
	\end{tikzpicture}
&
\begin{tikzpicture}[baseline=0]
\node at (0,0.45) {$\tTb{6} (B_1)=[B_1,\tT_{2}^{-1}\tT_{3}^{-1}(B_6 )]_q$};
\node at (0, -0.35) {$\tTb{1}(B_6)=\big[[B_6, \tT_{3}^{-1}\tT_{2}^{-1} (B_1)]_q,\tT_4^{-1}(B_5)\big]_q$ };
\node at (0.5,-.85) {
$\quad -q \tT_{32323}(B_6) \tT_{s_4\bw}(\ck_1)$};
\end{tikzpicture}
\\
\hline
\end{tabular}
}%
\newline
\end{table}

\appendix

\section{Proofs of Proposition~\ref{prop:rktwoRij} and Table~\ref{table:rktwoSatake}}
  \label{app1}

In this Appendix, we shall provide constructive proofs for Proposition~\ref{prop:rktwoRij} and verify the rank 2 formulas for $\tTa{i}(B_j)$ in Table~\ref{table:rktwoSatake}. The proofs are based on type-by-type computations in $\tU$ for each rank two Satake diagram. Along the way, we will also specify a reduced expression for $\bs_i$ in $W$.

\subsection{Some preparatory lemmas}
\label{prep:app1}

Denote the $t$-commutator
\[
[C,D]_t =CD -t DC,
\]
for various $q$-powers $t$.
Let $(\I=\bI\cup \wI,\tau)$ be an arbitrary Satake diagram. Recall that $B_i =F_i+ \tT_{\bw}(E_{\tau i})K_i'$ and $B_i^\sigma= F_i+ K_i \tT_{\bw}^{-1}(E_{\tau i})$.

\begin{lemma}
  \label{lem:app1}
Suppose that $i,j\in \wI$ such that $j\not \in \{i,\tau i\}$. Then we have
\begin{align}
[B_i^\sigma,F_j]_{q^{-(\alpha_i,\alpha_j) }}
&=[F_i,F_j]_{q^{-(\alpha_i,\alpha_j) }},
 \label{eq:q-comm} \\
[B_i, \tT_{\bw}( E_{\tau j}) K_j']_{q^{-(\alpha_i,\alpha_j)}}
&=q^{-(\alpha_i,\bw(\alpha_{\tau j}))}\tT_{\bw}\big([E_{\tau i},E_{\tau j}]_{q^{- (\alpha_i,\alpha_j)}}\big) K_i' K_j'.
\label{eq:q-comm2}
\end{align}
\end{lemma}

\begin{proof}
Follows by a simple computation and using the identity $[E_k, F_j]=0$, for $k \neq j$.
\end{proof}

Introduce the following operator (see Lemma~\ref{lem:braid1} for some of the notations)
\begin{align}
  \label{eq:D}
 \cL:=\tT_{w_0}\tT_{\bw}\widehat{\tau}_0 \widehat{\tau}.
\end{align}
We shall formulate several basic properties for $\cL$ below. A systematic use of $\cL$ throughout Appendices~\ref{app1} will allow us to reduce the proofs of many challenging identities to easier ones.

\begin{lemma}
We have \begin{align}
\cL(B_i^\sigma) &= -q^{-(\alpha_i,\alpha_i)}B_{i} \tT_{\bw}(\ck_{\tau i}^{-1}),
  \label{eq:app2} \\
\cL(F_j) &=-q_j^{-2}\tT_{\bw}(E_{\tau j}) K_j' \tT_{\bw}(\ck_{\tau j}^{-1}).
 \label{eq:app3}
\end{align}
\end{lemma}

\begin{proof}
We rewrite the identity \eqref{eq:TBi} as follows:
\begin{align}\notag
 B_{i} \tT_{\bs_i}(\ck_{\tau_{\bullet,i}\tau i})
 &= -q^{-(\alpha_i,\bw\alpha_{\tau i})} \tT_{\bw} \tT_{w_{\bullet,i}}(B_{\tau_{\bullet,i}\tau i}^\sigma)\\
 &=-q^{-(\alpha_i,\bw\alpha_{\tau i})} \tT_{\bw} \tT_{w_0}(B_{\tau_{0}\tau i}^\sigma)= -q^{-(\alpha_i,\bw\alpha_{\tau i})} \cL(B_i^\sigma).\label{eq:app2-1}
\end{align}
Since $\tT_{\bs_i}(\ck_{\tau_{\bullet,i}\tau i})
=\tT_{\bw}\tT_{w_{\bullet,i}}(\ck_{\tau_{\bullet,i}\tau i})
=\vs_{i,\dm}^2 \tT_{\bw}(\ck_{\tau i}^{-1})$, the formula \eqref{eq:app2} follows from \eqref{eq:app2-1}.

By Lemma~\ref{lem:braid1}, we have
$
\cL(F_j)=- K_{\bw(\tau j)}^{-1}\tT_{\bw}(E_{\tau j})=-q_j^{-2}\tT_{\bw}(E_{\tau j}) K_j' \tT_{\bw}(\ck_{\tau j}^{-1}).
$ 
This proves \eqref{eq:app3}.
\end{proof}

\begin{lemma}
\label{lem:cL}
The operator $\cL$ commutes with $\tT_{\bs_i},\tT_j$, for $i\in \wI,j\in \bI$.
\end{lemma}

\begin{proof}
Since $w_0 s_k =s_{\tau_0 k} w_0$, for $k\in \I$, we have
$\tT_{w_0}\tT_{k}^{-1} = \tT_{w_0 s_k} = \tT_{s_{\tau_0 k}w_0} = \tT_{\tau_0 k}^{-1} \tT_{w_0}.$ Hence, $\tT_{w_0}\tT_{k} =\tT_{\tau_0 k}  \tT_{w_0}$ for any $k\in \I$. Therefore, $\tT_{w_0} \widehat{\tau}_0$ commutes with $\tT_k$ ($k\in \I$) and thus commutes with $\tT_{\bs_i},\tT_j$, for $i\in \wI,j\in \bI$.

Similarly, one can show that $\tT_{\bw} \widehat{\tau}$ commutes with $\tT_j$, for $j\in \bI$. Hence, by definition \eqref{eq:D}, the operator $\cL$ commutes with $\tT_j$ for $j\in \bI$.

On the other hand, by definition \eqref{def:bsi}, $\tT_{\bs_i} $, for $i\in \wI $, commutes with both $\tT_{\bw}$ and $ \widehat{\tau}$. Hence, $\cL$ also commutes with $\tT_{\bs_i} $.
\end{proof}
\subsection{Split types of rank 2}

Consider rank 2 split Satake diagrams $(\I=\wI =\{i,j\},\Id)$. In this case, we have
$ 
\bs_i=s_i,\;
B_i^\sigma =F_i +K_i E_i.
$ 

\subsubsection{$c_{ij}=-1$}

In this case, in line with the first line of Table~ \ref{table:rktwoSatake}, Proposition~\ref{prop:rktwoRij} is reformulated and proved as follows.

\begin{lemma}
We have
\begin{align}  \label{eq:app1}
\tT_i^{-1}(F_j)=  [B_i^\sigma,F_j]_{q_i}, \qquad \tT_i^{-1}(E_j K_j')= [B_i,E_j K_j']_{q_i}.
\end{align}
\end{lemma}

\begin{proof}
Follows immediately by Lemma \ref{lem:app1} and the definition of $\tT_i$.
\end{proof}

\subsubsection{$c_{ij}=-2$}

In this case, the rank 2 Satake diagram is given by
 \begin{center}
\begin{tikzpicture}[baseline=0, scale=1.5]
		\node at (-0.5,0) {$\circ$};
		\node at (0.5,0) {$\circ$};
		\draw[-implies, double equal sign distance]  (0.4, 0) to (-0.4, 0);
		\node at (-0.5, -.2) {\small i};
		\node at (0.5,-.2) {\small j};
	\end{tikzpicture}\\
\end{center}
and in line with Table \ref{table:rktwoSatake}, Proposition~\ref{prop:rktwoRij} can be reformulated and proved as follows.

\begin{lemma}
\label{lem:split1}
We have
\begin{align}\label{eq:app4}
\tT_i^{-1}(F_j)&=\frac{1}{[2]_i}\big[B_i^\sigma, [B_i^\sigma,F_j]_{q_i^2} \big]- q_i^2 F_j K_i  K_i',\\
\tT_i^{-1}(E_j K_j')&=\frac{1}{[2]_i}\big[B_i, [B_i, E_j K_j' ]_{q_i^2} \big]- q_i^2 E_j K_j' K_i  K_i'.\label{eq:app5}
\end{align}
\end{lemma}

\begin{proof}
We prove the formula \eqref{eq:app4}. By Lemma~\ref{lem:app1}, we have $[B_i^\sigma,F_j]_{q_i^2}=[F_i,F_j]_{q_i^2}$. By Proposition \ref{prop:braid0}, we have
$
\tT_i^{-1}(F_j)=\frac{1}{[2]_i}\big[F_i, [F_i,F_j]_{q_i^2} \big].
$ 
Now we compute the first term on RHS \eqref{eq:app4} using Lemma~\ref{lem:app1} as follows:
\begin{align*}
\big[B_i^\sigma, [B_i^\sigma,F_j]_{q_i^2} \big]&= \big[B_i^\sigma, [F_i ,F_j]_{q_i^2} \big]\\
&= \big[F_i ,[F_i ,F_j]_{q_i^2} \big]+\big[K_i E_i,  [F_i ,F_j]_{q_i^2} \big]\\
&= [2]_i \tT_i^{-1}(F_j)+K_i  [\frac{K_i-K_i'}{q_i-q_i^{-1}} ,F_j]_{q_i^2} \\
&= [2]_i \tT_i^{-1}(F_j)+ q_i^2[2]_i F_j K_i  K_i'.
\end{align*}
Hence the formula \eqref{eq:app4} holds.

We next prove the formula \eqref{eq:app5}. In this case, we read \eqref{eq:D} as $\cL=\tT_{w_0}$, and note that $\ck_i=\tk_i$. By Lemma~\ref{lem:cL}, $\cL$ commutes with $\tT_i^{-1}$. Applying this operator $\cL$ to the formula \eqref{eq:app4} and then using \eqref{eq:app2}--\eqref{eq:app3}, we obtain
\begin{align} 
\tT_i^{-1}(E_j K_j')\tT_i^{-1}(\tk_{j}^{-1})
&=\frac{q_i^{-4}}{[2]_i}\big[B_i \tk_i^{-1} , [B_i \tk_i^{-1}, E_j K_j'\tk_{j}^{-1} ]_{q_i^2} \big]
- q_i^2 E_j K_j' \tk_{j}^{-1}\tT_{w_0}( \tk_i).
 \label{eq:app6}
\end{align}
Recall our symmetries $\tT_{j}$ are defined in \S \ref{Double} by normalizing a variant of Lusztig's symmetries $\tT''_{j,+1}$. In this case, we have $\tT_{w_0}( \tk_i)=q_i^{-4}\tk_i^{-1}$ and $ \tT_i^{-1}(\tk_{j}^{-1})= q_i^{-4}\tk_{j}^{-1} \tk_i^{-2}$. Hence, since $\tk_i,\tk_j$ are central, \eqref{eq:app6} is simplified as the following formula
\begin{align}
\tT_i^{-1}(E_j K_j')\tk_{j}^{-1} \tk_i^{-2}&=\Big(\frac{1}{[2]_i}\big[B_i , [B_i, E_j K_j']_{q_i^2} \big]
- q_i^2 E_j K_j' K_i  K_i'\Big)\tk_{j}^{-1} \tk_i^{-2},\label{eq:app7}
\end{align}
which clearly implies the formula \eqref{eq:app5}.
\end{proof}

\subsubsection{ $c_{ij}=-3$ }

Consider the Satake diagram of split type $G_2$
 \begin{center}
\begin{tikzpicture}[baseline=0, scale=1.5]
		\node at (-0.5,0) {$\circ$};
		\node at (0.5,0) {$\circ$};
		\draw[->]  (0.4, 0) to (-0.4, 0);
		\draw[->]  (0.4, 0.05) to (-0.4, 0.05);
		\draw[->]  (0.4, -0.05) to (-0.4, -0.05);
		\node at (-0.5, -.2) {\small i};
		\node at (0.5,-.2) {\small j};
	\end{tikzpicture}\\
\end{center}

In this case, we have $q_i=q$ and $q_j=q^3.$
\begin{lemma}
\label{lem:G2-0}
We have
\begin{align}\label{eq:G2-1}
\big[K_i E_i, [F_i ,F_j]_{q^3} \big]_q&=q^3[3] F_j K_i K_i',\\
\Big[K_i E_i ,\big[F_i , [F_i ,F_j]_{q^3} \big]_q \Big]_{q^{-1}}&=q(1+[3])[B_i^\sigma,F_j]_{q^3} K_i K_i'.\label{eq:G2-2}
\end{align}
\end{lemma}

\begin{proof}
The first identity \eqref{eq:G2-1} is derived as follows:
\begin{align*}
\text{LHS}\eqref{eq:G2-1}
&=K_i \big[ E_i, [F_i ,F_j]_{q^3} \big]
=K_i [\frac{K_i-K_i^{-1}}{q-q^{-1}} ,F_j]_{q^3}
=q^3[3] K_i K_i' F_j = \text{RHS}\eqref{eq:G2-1}.
\end{align*}

We next compute
\begin{align*}
\text{LHS}\eqref{eq:G2-2}
&=K_i\Big[ E_i ,\big[F_i , [F_i ,F_j]_{q^3} \big]_q \Big]\\
&=K_i \big[\frac{K_i-K_i^{-1}}{q-q^{-1}} , [F_i ,F_j]_{q^3} \big]_q +K_i \big[F_i , [\frac{K_i-K_i'}{q-q^{-1}} ,F_j]_{q^3} \big]_q  \\
&= q K_i K_i'[F_i ,F_j]_{q^3} + q^{3}[3] K_i \big[F_i , K_i' F_j  \big]_q \\
&= (q + q[3])[F_i ,F_j]_{q^3}K_i K_i'\\
&= (q + q[3])[B_i^\sigma ,F_j]_{q^3}K_i K_i',
\end{align*}
where the last equality follows from Lemma~\ref{lem:app1}.
This proves \eqref{eq:G2-2}.
\end{proof}

In line with Table~\ref{table:rktwoSatake}, Proposition~\ref{prop:rktwoRij} can be reformulated and proved as follows.
\begin{lemma}
\label{lem:G2-1}
We have
\begin{align}\notag
\tT_i^{-1} (F_j) &=\frac{1}{[3]! } \Big[B_i^\sigma,\big[B_i^\sigma, [B_i^\sigma,F_j]_{q^3} \big]_q \Big]_{q^{-1}}\\
&\quad -\frac{1}{[3]! } \Big( q(1+[3])[B_i^\sigma , F_j]_{q^3} + q^3[3] [B_i^\sigma,F_j]_{q^{-1}}\Big)\tk_i.\label{eq:G2-3}
\end{align}
\end{lemma}

\begin{proof}
By Proposition~\ref{prop:braid0}, 
$\tT_i^{-1} (F_j) =\frac{1}{[3]! } \Big[F_i,\big[F_i, [F_i,F_j]_{q^3} \big]_q \Big]_{q^{-1}}.$ By Lemma~\ref{lem:app1}, we have $[B_i^\sigma,F_j]_{q^3}=[F_i,F_j]_{q^3}$. Then we have
\begin{align} \label{eq:G2-4}
\Big[B_i^\sigma,\big[B_i^\sigma, [B_i^\sigma,F_j]_{q^3} \big]_q \Big]_{q^{-1}}
=\Big[B_i^\sigma,\big[F_i , [F_i ,F_j]_{q^3} \big]_q \Big]_{q^{-1}}+ \Big[B_i^\sigma,\big[K_i E_i, [F_i ,F_j]_{q^3} \big]_q \Big]_{q^{-1}}.
\end{align}
Using Lemma~\ref{lem:G2-0}, we rewrite RHS \eqref{eq:G2-4} as
\begin{align} \notag
& \Big[F_i ,\big[F_i , [F_i ,F_j]_{q^3} \big]_q \Big]_{q^{-1}}+\Big[K_i E_i ,\big[F_i , [F_i ,F_j]_{q^3} \big]_q \Big]_{q^{-1}}+ q^3[3] [B_i^\sigma, F_j  ]_{q^{-1}} K_i K_i'\\
&= [3]!\tT_i^{-1} (F_j) +q(1+[3])[B_i^\sigma,F_j]_{q^3} K_i K_i' + q^3[3] [B_i^\sigma, F_j  ]_{q^{-1}} K_i K_i'.
\label{eq:G2-5}
\end{align}
Now the desired formula \eqref{eq:G2-3} follows from \eqref{eq:G2-4}-\eqref{eq:G2-5}.
\end{proof}

\begin{lemma}
\label{lem:G2-2}
We have
\begin{align}\notag
\tT_i^{-1} (E_j K_j') &=\frac{1}{[3]! } \Big[B_i ,\big[B_i , [B_i ,E_j K_j']_{q^3} \big]_q \Big]_{q^{-1}}\\
& \quad -\frac{1}{[3]! } \big( q(1+[3])[B_i  , E_j K_j']_{q^3}   - q^3[3] [B_i ,E_j K_j']_{q^{-1}}\big) \tk_i.\label{eq:G2-6}
\end{align}
\end{lemma}

\begin{proof}
In this case, $\ck_i=\tk_i$ and $\ck_j=\tk_j$  are central. By \eqref{eq:app2}--\eqref{eq:app3}, we have
\begin{align}
\label{eq:G2-7}
\cL(F_j)=-q_j^{-2} E_j K_j' \tk_j^{-1},\qquad \cL(B_i^\sigma)=-q^{-2} B_i \tk_i^{-1}.
\end{align}

Recall from Lemma~\ref{lem:cL} that $\cL$ commutes with $\tT_i$. Applying $\cL$ to \eqref{eq:G2-3} and then using \eqref{eq:G2-7}, we have
\begin{align}\notag
\tT_i^{-1} & (E_j K_j') \tT_i^{-1}(\tk_j^{-1})\\\label{eq:G2-8}
&=-q^{-6}\frac{1}{[3]! } \Big[B_i  ,\big[B_i , [B_i ,E_j K_j']_{q^3} \big]_q \Big]_{q^{-1}}\tk_j^{-1} \tk_i^{-3}\\\notag
&+q^{-2}\frac{1}{[3]! } \big( q(1+[3])[B_i^\sigma , E_j K_j']_{q^3} - q^3[3] [B_i^\sigma,E_j K_j']_{q^{-1}} \big) \cL(\tk_i)\tk_j^{-1}\tk_i^{-1}.
\end{align}

Since $s_i(\alpha_j)=\alpha_j+3\alpha_i$, by Proposition~\ref{prop:braid0}, we have $\tT_i^{-1}(\tk_j^{-1})=-q^{-6}\tk_j^{-1} \tk_i^{-3}$. Note also that $\cL(\tk_i)=q^{-4} \tk_i^{-1}$. Hence, \eqref{eq:G2-8} implies the desired formula \eqref{eq:G2-6}.
\end{proof}

\subsection{Type AII}\label{AII-1}
Consider the rank 2 Satake diagram of type AII$_5$
\begin{center}
   \begin{tikzpicture}[baseline=0,scale=1.5]
		\node at (-0.5,0) {$\bullet$};
		\node  at (0,0) {$\circ$};
		\node at (0.5,0) {$\bullet$};
		\node at (1,0) {$\circ$};
		\node at (1.5,0) {$\bullet$};
		\draw[-] (-0.5,0) to (-0.05, 0);
		\draw[-] (0.05, 0) to (0.5,0);
		\draw[-] (0.5,0) to (0.95,0);
		\draw[-] (1.05,0)  to (1.5,0);
		\node at (-0.5,-0.2) {1};
		\node  at (0,-0.2) {2};
		\node at (0.5,-0.2) {3};
		\node at (1,-0.2) {4};
		\node at (1.5,-0.2) {5};
	\end{tikzpicture}\\
$\bs_4=s_4 s_3 s_5 s_4.$
\end{center}
In this case, Proposition~\ref{prop:rktwoRij} is reformulated and proved as follows.

\begin{lemma}  \label{lem:AII}
We have
\begin{align*}
\tT^{-1}_{\bs_4}(F_2)&= [ \tT_3(B_4^\sigma) ,F_2 ]_q,\\
\tT^{-1}_{\bs_4}\big(\tT_{\bw}(E_2)K_2'\big)&=[\tT_3(B_4),\tT_{\bw}(E_2)K_2']_q.
\end{align*}
\end{lemma}

\begin{proof}
The first formula follows by $\tT^{-1}_{\bs_4} =\tT^{-1}_{4354}$, Proposition~\ref{prop:braid0}, and the formula \eqref{eq:q-comm}.

We prove the second formula. By \eqref{eq:app2}--\eqref{eq:app3}, we have
\begin{align}
\label{eq:AII-1}
\cL(F_2)=-q^{-2} \tT_{\bw}(E_2) K_2' \tT_{\bw}(\ck_2^{-1}),\qquad \cL(B_4^\sigma)=-q^{-2} B_4 \tT_{\bw}(\ck_4^{-1}).
\end{align}
Recall from Lemma~\ref{lem:cL} that the operator $\cL$ in \eqref{eq:D} commutes with $\tT_3, \tT_{\bs_4}$. Applying the operator $\cL$ to both sides of the first formula and then using \eqref{eq:AII-1}, we have
\begin{align}
\label{eq:AII-2}
\tT^{-1}_{\bs_4}\big(\tT_{\bw}(E_2)K_2'\big)\tT_{w_{\bullet,4}}(\ck_2^{-1})&=-q^{-2}[\tT_3(B_4)\tT_{5}(\ck_4^{-1}),\tT_{\bw}(E_2)K_2'\tT_{\bw}(\ck_2^{-1})]_q.
\end{align}
For a weight reason, we have
\begin{align*}
\tT_{5}(\ck_4^{-1})\tT_{\bw}(E_2)K_2'&=q\tT_{\bw}(E_2)K_2'\tT_{5}(\ck_4^{-1}), \\
\tT_{\bw}(\ck_2^{-1})\tT_3(B_4)&=q\tT_3(B_4)\tT_{\bw}(\ck_2^{-1}).
\end{align*}
Using these two identities, we simplify \eqref{eq:AII-2} as
\begin{align}
\label{eq:AII-3}
\tT^{-1}_{\bs_4}\big(\tT_{\bw}(E_2)K_2'\big)\tT_{w_{\bullet,4}}(\ck_2^{-1})&=-q^{-1}[\tT_3(B_4),\tT_{\bw}(E_2)K_2']_q\tT_{5}(\ck_4^{-1})\tT_{\bw}(\ck_2^{-1}).
\end{align}
Finally, by Proposition~\ref{prop:braid0}, we have $\tT_{w_{\bullet,4}}(\ck_2)=-q \tT_{5}(\ck_4 )\tT_{\bw}(\ck_2 )$. Hence, \eqref{eq:AII-3} implies the second desired formula.
\end{proof}


%
%
\subsection{Type CII$_n$, $n\geq 5$}
 \label{CIIn-1} 

 Consider the rank 2 Satake diagram of type CII$_n$, for $n\geq 5$:
\begin{center}
\begin{tikzpicture}[baseline=6,scale=1.5]
		\node  at (0,0) {$\bullet$};
		\node  at (0,-0.2) {1};
		\draw (0.05, 0) to (0.45, 0);
		\node  at (0.5,0) {$\circ$};
		\node  at (0.5,-0.2) {2};
		\draw (0.55, 0) to (0.95, 0);
		\node at (1,0) {$\bullet$};
		\node at (1,-.2) {3};
		\node at (1.5,0) {$\circ$};
		\node at (1.5,-0.2) {4};
		\draw[-] (1.05,0)  to (1.45,0);
		\draw[-] (1.55,0) to (1.95, 0);
		\node at (2,0) {$\bullet$};
		\node at (2,-0.2) {5};
		\draw (1.9, 0) to (2.1, 0);
		\draw[dashed] (2.1,0) to (2.7,0);
		\draw[-] (2.7,0) to (2.9, 0);
		\node at (3,0) {$\bullet$};
		\node at (3,-0.2) {n-1};
		\draw[implies-, double equal sign distance]  (3.1, 0) to (3.7, 0);
		\node at (3.8,0) {$\bullet$};
		\node at (3.8,-0.2) {n};
	\end{tikzpicture}\\
$\vs_{2,\dm}=-q_2^{-1},\qquad \vs_{4,\dm}=-q_4^{-1/2}$
\\
$\bs_4=s_{4\cdots n \cdots 4} s_3 s_{4\cdots n \cdots 4}.$
\end{center}
Note that $q_2=q_4=q$. The notation $4\cdots n \cdots 4$ (with the local minima/maxima indicated) denotes a sequence $4 \ 5 \cdots n-1 \ n\ n-1 \cdots 5\ 4$, and we denote $s_{4\cdots n \cdots 4} =s_4\cdots s_n \cdots s_4$.


In this case, Proposition~\ref{prop:rktwoRij} is reformulated and proved as Lemmas~\ref{lem:CIIn-1}--\ref{lem:CIIn-2} below.

\begin{lemma}  \label{lem:CIIn-1}
We have
\begin{align}\label{eq:app10}
\tT_{\bs_4}^{-1}(F_2)&=\big[[\tT_{5\cdots n\cdots 5}(B_4^\sigma), \tT_3 (B_4^\sigma)]_q,F_2\big]_q-q\tT_3^{-2}(F_2)\tT_3(K_4')\tT_{5\cdots n\cdots 5}(K_4).
\end{align}
\end{lemma}

\begin{proof}
Since $s_{5\cdots n \cdots 5} s_4 s_{5\cdots n \cdots 5}(\alpha_4)=\alpha_4$, we have $\tT_{4}^{-1}\tT_{5\cdots n\cdots 5 }^{-1}(F_4)=\tT_{ 5\cdots n\cdots 5 }(F_4).$ Then
\begin{align*}
\tT_{\bs_4}^{-1}(F_2) =\tT^{-1}_{4\cdots n\cdots 4} [F_3,F_2]_q
&= \big[\tT_4^{-1} \tT_{5\cdots n\cdots 5}^{-1} ([F_4,F_3]_q),F_2\big]_q\\
&= \Big[\big[\tT_{5\cdots n\cdots 5}(F_4), [F_4,F_3]_q\big]_q,F_2\Big]_q.
\end{align*}
On the other hand, we compute RHS \eqref{eq:app10} as follows. First, note that \[
[K_4 \tT^{-1}_{w_\bullet}(E_4),F_3]_q=q^{-1} \tT_{5\cdots n\cdots 5}^{-1}(E_4) K_3 K_4,
\]
and hence
\begin{align*}
\big[[K_4 \tT_{w_\bullet}(E_4),F_3]_q, F_2\big]_q=[\tT_{5\cdots n\cdots 5}^{-1}(E_4),F_2]K_3K_4=0.
\end{align*}
Thus, we have
\begin{align*}
&\big[ [\tT_{5\cdots n\cdots 5}(B_4^\sigma), \tT_3(B_4^\sigma)]_q,F_2\big]_q\\
&=\Big[\big[\tT_{5\cdots n\cdots 5}(B_4^\sigma), [B_4^\sigma,F_3]_q\big]_q,F_2\Big]_q\\
&=\Big[\big[\tT_{5\cdots n\cdots 5}(B_4^\sigma), [F_4,F_3]_q\big]_q,F_2\Big]_q\\
&=\Big[\big[\tT_{5\cdots n\cdots 5}(F_4), [F_4 ,F_3]_q\big]_q,F_2\Big]_q+\Big[\big[\tT_{5\cdots n\cdots 5}(K_4) \tT_3^{-1}(E_4), [F_4 ,F_3]_q\big]_q,F_2\Big]_q\\
&=\tT_{\bs_4}^{-1}(F_2)+q\Big[\big[ \tT_3^{-1}(E_4), [F_4 ,F_3]_q\big] ,F_2\Big]_q\tT_{5\cdots n\cdots 5}(K_4)\\
&=\tT_{\bs_4}^{-1}(F_2)+q\Big[ [\tT_3^{-1}(F_3),F_3] ,F_2\Big]_{q^2}\tT_3(K_4')\tT_{5\cdots n\cdots 5}(K_4)\\
&=\tT_{\bs_4}^{-1}(F_2)+q\tT_3^{-2}(F_2)\tT_3(K_4')\tT_{5\cdots n\cdots 5}(K_4),
\end{align*}
as desired. This proves the formula \eqref{eq:app10}.
\end{proof}

\begin{lemma}
  \label{lem:CIIn-2}
We have
\begin{align}\notag
\tT_{\bs_4}^{-1}(\tT_{\bw}(E_2)K_2')
&=\big[[\tT_{ 5\cdots n\cdots 5}(B_4), \tT_3(B_4)]_q,\tT_{\bw}(E_2)K_2'\big]_q\\
&\quad -q\tT_3^{-2}\big(\tT_{\bw}(E_2)K_2'\big)\tT_3(K_4')\tT_{5\cdots n\cdots 5}(K_4).\label{eq:app11}
\end{align}
\end{lemma}

\begin{proof}
By Lemma~\ref{lem:cL}, the operator $\cL$ in \eqref{eq:D} commutes with $\tT_3,\tT_{5\cdots n \cdots 5},\tT_{\bs_4}$. Applying $\cL$ to \eqref{eq:app10} and then using \eqref{eq:app2}-\eqref{eq:app3}, we obtain
\begin{align}\notag
&\tT_{\bs_4}^{-1}(\tT_{\bw}(E_2)K_2')\tT_{w_{\bullet,4}}(\ck_2^{-1})\\\notag
&=q^{-4}\big[[\tT_{ 5\cdots n\cdots 5}(B_4)\tT_{3}(\ck_4^{-1}), \tT_3(B_4)\tT_{5\cdots n\cdots 5}(\ck_4^{-1})]_q,\tT_{\bw}(E_2)K_2'\tT_{\bw}(\ck_2^{-1})\big]_q\\
&-q\tT_3^{-2}\big(\tT_{\bw}(E_2)K_2'\big)\tT_{\bw}(\ck_2^{-1})\cL(\tT_3(K_4')\tT_{5\cdots n\cdots 5}(K_4)).\label{eq:app12}
\end{align}
Recalling $ck_i$ from \eqref{def:Ki}, we have
  \begin{align*}
 \ck_4 B_4&= q^{-3} B_4 \ck_4 ,\\
\tT_{\bw}(\ck_2 )\tT_{ 5\cdots n\cdots 5}(B_4 )\tT_3(B_4 )&=\tT_{ 5\cdots n\cdots 5}(B_4 )\tT_3(B_4 )\tT_{\bw}(\ck_2 ),\\
\cL(\tT_3(K_4')\tT_{5\cdots n\cdots 5}(K_4))&=q^{-1} \tT_{5\cdots n\cdots 5}(\ck_4^{-1})\tT_3(\ck_4^{-1})\tT_3(K_4')\tT_{5\cdots n\cdots 5}(K_4).
\end{align*}
Using these formulas, we simplify \eqref{eq:app12} as
 \begin{align}\notag
&\tT_{\bs_4}^{-1}(\tT_{\bw}(E_2)K_2')\tT_{w_{\bullet,4}}(\ck_2^{-1})\\\notag
&=q^{-1}\big[[\tT_{ 5\cdots n\cdots 5}(B_4), \tT_3(B_4)]_q,\tT_{\bw}(E_2)K_2'\big]_q\tT_{3}(\ck_4^{-1})\tT_{5\cdots n\cdots 5}(\ck_4^{-1})\tT_{\bw}(\ck_2^{-1})\\
&-\tT_3^{-2}\big(\tT_{\bw}(E_2)K_2'\big)\tT_3(K_4')\tT_{5\cdots n\cdots 5}(K_4)\tT_{5\cdots n\cdots 5}(\ck_4^{-1})\tT_3(\ck_4^{-1})\tT_{\bw}(\ck_2^{-1}).\label{eq:app13}
\end{align}
Finally, by \eqref{def:Ki} we have
$\tT_{w_{\bullet,4}}(\ck_2)=q \tT_{3}(\ck_4)\tT_{5\cdots n\cdots 5}(\ck_4)\tT_{\bw}(\ck_2).$ Therefore, the formula \eqref{eq:app11} follows from \eqref{eq:app13}.
\end{proof}
\vspace{1em}

\subsection{Type CII$_4$} \label{CII4-1}

Consider the rank 2 Satake diagram of type CII$_4$:
\begin{center}
\begin{tikzpicture}[baseline=6,scale=1.5]
        \node at (2,0) {$\bullet$};
        \node at (2,-0.2) {1};
		\draw[-] (2.05,0) to (2.45, 0);
        \node at (2.5,0) {$\circ$};
        \node at (2.5,-0.2) {2};
		\draw[-] (2.55,0) to (2.95, 0);
		\node at (3,0) {$\bullet$};
		\node at (3,-0.2) {3};
		\draw[implies-, double equal sign distance]  (3.05, 0) to (3.75, 0);
		\node at (3.8,0) {$\circ$};
		\node at (3.8,-0.2) {4};
	\end{tikzpicture}\\
$\bs_4=s_4 s_3 s_4,\qquad \bs_2 = s_2 s_1 s_3 s_2$.
\end{center}
In this case, Proposition~\ref{prop:rktwoRij} is reformulated and proved as Lemmas~\ref{lem:CII4-1}--\ref{lem:CII4-2} below.

\begin{lemma} \label{lem:CII4-1}
We have
\begin{align}\label{eq:CII4-1}
\tT^{-1}_{\bs_4}(F_2)&=\big[[B_4^\sigma,F_3]_{q_4},F_2\big]_{q_3},\\
\tT^{-1}_{\bs_2}(F_4)&=\big[\tT_3(B_2^\sigma),[\tT_3(B_2^\sigma),F_4]_{q_3^2}\big]-(q_3-q_3^{-1})[F_3,F_4]_{q_3^2}E_1 K_2 K_2' K_3.\label{eq:CII4-2}
\end{align}
\end{lemma}
\begin{proof}
The first formula \eqref{eq:CII4-1} follows by a direct computation.

We prove \eqref{eq:CII4-2}. We have
\begin{align*}
\tT_{\bs_2}^{-1}(F_4)&=\big[\tT_2^{-1}(F_3),[\tT_2^{-1}(F_3),F_4]_{q_3^2}\big]= \big[\tT_3(F_2),[\tT_3(F_2),F_4]_{q_3^2}\big].
\end{align*}
Hence, recalling that $B_2^\sigma=F_2 + K_2\tT_{13}^{-1}(E_2)$, we have
\begin{align*}
\big[ \tT_3(B_2^\sigma),[\tT_3(B_2^\sigma),F_4]_{q_3^2}\big]&=\big[ \tT_3(B_2^\sigma),[\tT_3(F_2),F_4]_{q_3^2}\big]\\
&= \big[ \tT_3(F_2),[\tT_3(F_2),F_4]_{q_3^2}\big]+ \big[ K_2K_3\tT_1^{-1}(E_2),[\tT_3(F_2),F_4]_{q_3^2}\big]\\
&= \tT_{\bs_2}^{-1}(F_4) + \big[[\tT_1^{-1}(E_2),\tT_3(F_2)],F_4\big]_{q_3^2} K_2K_3\\
&= \tT_{\bs_2}^{-1}(F_4) + (q_3-q_3^{-1})[F_3,F_4]_{q_3^2}E_1 K_2 K_2' K_3.
\end{align*}
Thus, \eqref{eq:CII4-2} is proved.
\end{proof}
\begin{lemma}\label{lem:CII4-2}
We have
\begin{align}\label{eq:CII4-3}
\tT^{-1}_{\bs_4}\big(\tT_{\bw}(E_2)K_2'\big)&=\big[[B_4,F_3]_{q_4},\tT_{\bw}(E_2)K_2'\big]_{q_3},\\\notag
\tT^{-1}_{\bs_2}\big(\tT_{\bw}(E_4)K_4'\big)&=\big[ \tT_3(B_2),[\tT_3(B_2),\tT_{\bw}(E_4)K_4']_{q_3^2}\big]\\
&-(q_3-q_3^{-1})[F_3,\tT_{\bw}(E_4)K_4']_{q_3^2}E_1 K_2 K_2' K_3.\label{eq:CII4-4}
\end{align}
\end{lemma}
\begin{proof}
We shall prove the formula \eqref{eq:CII4-4} only, and skip a similar proof for \eqref{eq:CII4-3}.

By Lemma~\ref{lem:cL}, the operator $\cL$ defined in \eqref{eq:D} commutes with $\tT_{3},\tT_{\bs_2}$. Applying $\cL$ to the identity \eqref{eq:CII4-2} and then using \eqref{eq:app2}-\eqref{eq:app3}, we have
\begin{align}\notag
&\tT^{-1}_{\bs_2}\big(\tT_{\bw}(E_4)K_4'\big)\tT_{w_{\bullet,2}}(\ck_4^{-1})\\\notag
&=q_2^{-4}\big[ \tT_3(B_2)\tT_{1}(\ck_2^{-1}),[\tT_3(B_2)\tT_{1}(\ck_2^{-1}),\tT_{\bw}(E_4)K_4'\tT_{\bw}(\ck_4^{-1})]_{q_3^2}\big]\\
&-(q_3-q_3^{-1})q_3^{-4}[F_3 K_3 K_3'^{-1},\tT_{\bw}(E_4)K_4'\tT_{\bw}(\ck_4^{-1})]_{q_3^2} E_1 K_1^{-1} K_1' \cL(K_2 K_2' K_3).\label{eq:app14}
\end{align}
For a weight reason, we have
\begin{align*}
\tT_{1}(\ck_2^{-1})\tT_{\bw}(E_4)&=q_3^2 \tT_{\bw}(E_4)\tT_{1}(\ck_2^{-1}),\\
\tT_{\bw}(\ck_4^{-1}) \tT_3(B_2)& =q_3^2 \tT_3(B_2) \tT_{\bw}(\ck_4^{-1}),\\
\tT_{1}(\ck_2^{-1}) \tT_3(B_2)\tT_{\bw}(E_4)&=\tT_3(B_2)\tT_{\bw}(E_4)\tT_{1}(\ck_2^{-1}),\\
\tT_{1}(\ck_2^{-1})\tT_{\bw}(\ck_4^{-1})\tT_3(B_2)&=\tT_3(B_2)\tT_{1}(\ck_2^{-1})\tT_{\bw}(\ck_4^{-1}).
\end{align*}
We also have $\cL(K_2 K_2' K_3)=q_2^{-2}\tT_{\bw}(K_2 K_2')^{-1} K_3$. Hence, \eqref{eq:app14} is simplified as
\begin{align}\notag
\tT^{-1}_{\bs_2} & \big(\tT_{\bw}(E_4)K_4'\big)\tT_{w_{\bullet,2}}(\ck_4^{-1})\\\notag
&=q_2^{-2}\big[ \tT_3(B_2) ,[\tT_3(B_2) ,\tT_{\bw}(E_4)K_4']_{q_3^2}\big]\tT_{\bw}(\ck_4^{-1})\tT_{1}(\ck_2^{-1})^2\\
&-(q_3-q_3^{-1})q_2^{-2} [F_3  ,\tT_{\bw}(E_4)K_4']_{q_3^2} E_1 K_2 K_2' K_3\tT_{\bw}(\ck_4^{-1}) \tT_{1}(\ck_2^{-1})^2.\label{eq:app15}
\end{align}
By the definition of $\ck_i$ in \eqref{def:Ki}, we have $\tT_{w_{\bullet,2}}(\ck_4^{-1})=q_2^{-2}\tT_{\bw}(\ck_4^{-1}) \tT_{1}(\ck_2^{-1})^2$. Thus, \eqref{eq:app15} implies the desired formula \eqref{eq:CII4-4}.
\end{proof}

\subsection{Type EIV}\label{EIV-1}

Consider the rank 2 Satake diagram of type EIV:
\begin{center}
\begin{tikzpicture}[baseline = 0, scale =1.5]
		\node at (-1,0) {$\circ$};
        \node at (-1,-0.2) {1};
		\draw (-0.95,0) to (-0.55,0);
		\node at (-0.5,0) {$\bullet$};
        \node at (-0.5,-0.2) {2};
		\draw (-0.45,0) to (-0.05,0);
		\node at (0,0) {$\bullet$};
        \node at (0.1,-0.2) {3};
		\draw (0.05,0) to (0.45,0);
		\node at (0.5,0) {$\bullet$};
		\node at (0.5,-0.2) {4};
		\draw (0.55,0) to (0.95,0);
		\node at (1,0) {$\circ$};
		\node at (1,-0.2) {5};
		\draw (0,-0.05) to (0,-0.45);
		\node at (0,-0.5) {$\bullet$};
		\node at (0,-0.7) {6};
	\end{tikzpicture}\\
$\bs_1=s_1 s_2 s_3 s_4 s_6 s_3 s_2 s_1.$
\end{center}
In this case, Proposition~\ref{prop:rktwoRij} is reformulated and proved as Lemma~\ref{lem:EIV} below.
\begin{lemma}  \label{lem:EIV}
\begin{align}\label{eq:EIV1}
\tT^{-1}_{\bs_1}(F_5)&=\big[ \tT_4\tT_3\tT_2(B_1^\sigma) ,F_5\big]_q,\\
\tT^{-1}_{\bs_1}\big(\tT_{\bw}(E_5) K_5'\big)&=\big[ \tT_4\tT_3\tT_2(B_1) ,\tT_{\bw}(E_5) K_5'\big]_q.\label{eq:EIV2}
\end{align}
\end{lemma}

\begin{proof}
We prove the formula \eqref{eq:EIV1}. Indeed, we have
\begin{align*}
\tT_{\bs_1}^{-1}(F_5)=\tT_1^{-1}\tT_2^{-1}\tT_3^{-1}[F_4,F_5]_q
&=  [\tT_1^{-1}\tT_2^{-1}\tT_3^{-1}(F_4),F_5]_q
\\
&= [\tT_4\tT_3\tT_2(F_1),F_5]_q
= [\tT_4\tT_3\tT_2(B_1^\sigma),F_5]_q.
\end{align*}

We next prove the formula \eqref{eq:EIV2}. Recall from Lemma~\ref{lem:cL} that $\tT_j$, for $j\in \bI$, commutes with $\cL$ in \eqref{eq:D}. Applying $\cL$ to the formula \eqref{eq:EIV1} and then using \eqref{eq:app2}-\eqref{eq:app3}, we have
\begin{align}\notag
&\tT^{-1}_{\bs_1}\big(\tT_{\bw}(E_5) K_5'\big)\tT_{w_{\bullet,1}}(\ck_5^{-1})\\
&=-q^{-2}\big[ \tT_{432}(B_1) \tT_{632}(\ck_1^{-1}),\tT_{\bw}(E_5) K_5'\tT_{\bw}(\ck_5^{-1})\big]_q.\label{eq:app17}
\end{align}
By a weight consideration, we have
\begin{align*}
\tT_{\bw}(\ck_5^{-1})\tT_{432}(B_1)&= q\tT_{432}(B_1)\tT_{\bw}(\ck_5^{-1}),\\
\tT_{632}(\ck_1^{-1})\tT_{\bw}(E_5)&=q\tT_{\bw}(E_5)\tT_{632}(\ck_1^{-1}).
\end{align*}
Hence, using these two identities, \eqref{eq:app17} is simplified as
\begin{align}\notag
&\tT^{-1}_{\bs_1}\big(\tT_{\bw}(E_5) K_5'\big)\tT_{w_{\bullet,1}}(\ck_5^{-1})\\
&=-q^{-1}\big[ \tT_{432}(B_1),\tT_{\bw}(E_5) K_5'\big]_q\tT_{632}(\ck_1^{-1})\tT_{\bw}(\ck_5^{-1}).\label{eq:app18}
\end{align}
Finally, by the definition \eqref{def:Ki} of $\ck_i$,  $\tT_{w_{\bullet,1}}(\ck_5^{-1})=-q^{-1}\tT_{632}(\ck_1^{-1})\tT_{\bw}(\ck_5^{-1})$. Then \eqref{eq:app18} implies the desired formula \eqref{eq:EIV2}.
\end{proof}
\subsection{Type AIII$_3$}

Consider the rank 2 Satake diagram of type AIII$_3$:
\begin{center}
   \begin{tikzpicture}[baseline=0,scale=1.5]
		\node  at (-0.5,0) {$\circ$};
		\node  at (0,0) {$\circ$};
		\node  at (0.5,0) {$\circ$};
		\draw[-] (-0.45,0) to (-0.05, 0);
		\draw[-] (0.05, 0) to (0.45,0);
		\node at (-0.5,-0.2) {1};
		\node at (0,-0.2) {2};
		\node at (0.5,-0.2) {3};
        \draw[bend left,<->,red] (-0.5,0.1) to (0.5,0.1);
        \node at (0,0.4) {$\textcolor{red}{\tau} $};
	\end{tikzpicture}
\\
$\vs_{1,\diamond}=\vs_{3,\diamond}=-q^{-1},\qquad \vs_{2,\diamond}=-q^{-2}$
\\
$\quad \bs_1=s_1 s_3,\qquad\quad \bs_2=s_2.$
\end{center}
In this case, Proposition~\ref{prop:rktwoRij} is reformulated and proved as the following lemma.

\begin{lemma}\label{lem:AIII}
We have
\begin{align}\label{eq:AIII1}
\tT_{\bs_1}^{-1}(F_2)&=\big[B_3^\sigma, [B_1^\sigma, F_2]_q\big]_q -q F_2 K_3 K_1',\\
\tT_{\bs_1}^{-1}(E_2K_2')&=\big[B_3, [B_1, E_2 K_2' ]_q\big]_q -q E_2 K_2' K_3 K_1'.\label{eq:AIII2}
\end{align}
\end{lemma}

\begin{proof}
By Lemma~\ref{lem:app1}, we have $[B_1^\sigma,F_2]_q=[F_1,F_2]_q$.
Then the first term on the RHS of \eqref{eq:AIII1} is computed as follows:
\begin{align*}
\big[B_3^\sigma, [B_1^\sigma, F_2]_q\big]_q &= \big[ K_3 E_1, [F_1 , F_2]_q\big]_q+\big[F_3, [F_1, F_2]_q\big]_q\\
&=q\big[ [E_1,F_1 ], F_2\big]_q K_3+\big[F_3, [F_1, F_2]_q\big]_q\\
&=q[\frac{K_1-K_1'}{q-q^{-1}},F_2]_q K_3 +\big[F_3, [F_1, F_2]_q\big]_q\\
&=q F_2 K_3 K_1'+ \big[F_3, [F_1, F_2]_q\big]_q\\
&=  \tT_{13}^{-1}(F_2)+ q F_2 K_3 K_1'.
\end{align*}
This proves the formula \eqref{eq:AIII1}.

We next prove \eqref{eq:AIII2}. In this case, $\tau_0=\tau\neq \Id$, $\tau_{\bullet,1}=\Id$, and we simplify $\cL$ in \eqref{eq:D} as $\cL=\tT_{w_0}$. We also have $\ck_i=\tk_i$ for $i=1,2,3$. Applying the operator $\cL=\tT_{w_0}$ to the identity \eqref{eq:AIII1} and then using \eqref{eq:app2}-\eqref{eq:app3}, we have
\begin{align}
\tT_{\bs_1}^{-1}(E_2K_2')\tT_{\bs_1}(\tk_2^{-1})&=q^{-4}\big[B_3 \tk_1^{-1}, [B_1\tk_3^{-1}, E_2 K_2'\tk_2^{-1} ]_q\big]_q
 -q E_2 K_2'\tk_2^{-1} \cL(K_3 K_1').\label{eq:AIII3}
\end{align}
We have $\cL(K_3 K_1')=q^{-2}\tk_1^{-1}\tk_3^{-1}K_3 K_1'$. Note also that $\tk_2$ is central and $\tk_3,\tk_1$ commute with $E_2$. Hence, \eqref{eq:AIII3} can be rewritten as
\begin{align}\notag
\tT_{\bs_1}^{-1}(E_2K_2')\tT_{\bs_1}(\tk_2^{-1})&=q^{-2}\big[B_3, [B_1, E_2 K_2' ]_q\big]_q\tk_1^{-1}\tk_3^{-1}\tk_2^{-1} \\
&-q^{-1} E_2 K_2'\tk_2^{-1}\tk_1^{-1}\tk_3^{-1}  K_3 K_1' .\label{eq:AIII4}
\end{align}
Finally, since $\bs_1(\alpha_2)=\alpha_2+\alpha_1+\alpha_3 $, we have $\tT_{\bs_1}(\tk_2^{-1})= q^{-2}\tk_2^{-1}\tk_1^{-1}\tk_3^{-1}$. Therefore the desired formula \eqref{eq:AIII2} follows from \eqref{eq:AIII4}.
\end{proof}


%
%
\subsection{Type AIII$_n,n\geq 4$}

Consider the rank 2 Satake diagram of type AIII$_n,n\geq 4$:
\begin{center}
   \begin{tikzpicture}[baseline=0,scale=1.5]
		\node  at (-2.1,0) {$\circ$};
		\node  at (-1.3,0) {$\circ$};
		\node  at (-0.5,0) {$\bullet$};
		\node  at (0.5,0) {$\bullet$};
		\node  at (1.3,0) {$\circ$};
		\node  at (2.1,0) {$\circ$};
		\draw[-] (-2.05,0) to (-1.35, 0);
		\draw[-] (-1.25,0) to (-0.55, 0);
		\draw[-] (0.55,0) to (1.25, 0);
		\draw[-] (1.35, 0) to (2.05,0);
		\node at (-2.1,-0.2) {1};
		\node at (-1.3,-0.2) {2};
		\node at (-0.5,-0.2) {3};
		\node at (0.5,-0.2) {\small$n-2$};
		\node at (1.3,-0.2) {\small$n-1$};
		\node at (2.1,-0.2) {$n$ };
        \draw[dashed] (-0.5,0) to (0.5,0);
        \draw[bend left,<->,red] (-1.3,0.1) to (1.3,0.1);
        \draw[bend left,<->,red] (-2.1,0.1) to (2.1,0.1);
        \node at (0,0.6) {$\textcolor{red}{\tau} $};
	\end{tikzpicture}
\\
$\vs_{1,\diamond}=\vs_{n,\diamond}=-q^{-1},\qquad \vs_{2,\diamond}=\vs_{n-1,\diamond}=-q^{-1/2}$,\\
$\quad \bs_1=s_1 s_n, \qquad\quad \bs_2= s_2 \cdots s_{n-1} \cdots s_2.$
\end{center}

We first have a simple observation.

\begin{lemma} \label{lem:AIIIn2}
For any $3\leq s \leq n-2$, $\tT_{2\cdots n-2} (F_{n-1})$ is fixed by $\tT_s $.
\end{lemma}
\begin{proof}
Recall from Proposition~\ref{prop:braid0} that $\tT_s$ satisfies the braid relation. Then we have
\begin{align*}
\tT_{s}\tT_{2\cdots n-2} (F_{n-1})&=\tT_{2\cdots s-2}\tT_s\tT_{s-1}\tT_s \tT_{s+1\cdots n-2}(F_{n-1})\\
&=\tT_{2\cdots s-2} \tT_{s-1}\tT_s \tT_{s-1}\tT_{s+1\cdots n-2}(F_{n-1})\\
&=\tT_{2\cdots s-2} \tT_{s-1}\tT_s\tT_{s+1\cdots n-2}\tT_{s-1}(F_{n-1})\\
&=\tT_{2\cdots s-2} \tT_{s-1}\tT_s \tT_{s+1\cdots n-2}(F_{n-1})=\tT_{2\cdots n-2} (F_{n-1}).
\end{align*}
Hence, $\tT_{2\cdots n-2} (F_{n-1})$ is fixed by $\tT_s $ for $3\leq s \leq n-2$.
\end{proof}

In this case, Proposition~\ref{prop:rktwoRij} is reformulated and proved as Lemmas \ref{lem:AIIIn1}--\ref{lem:AIIIn3} below.

\begin{lemma}  \label{lem:AIIIn1}
We have
\begin{align}\label{eq:AIIIn1}
\tT_{\bs_1}^{-1}(F_2)&=[B_1^\sigma,F_2]_q,\\
\tT_{\bs_2}^{-1}(F_1)&= \big[ \tT_{\bw}(B_{n-1}^\sigma),[B_2^\sigma,F_1]_q\big]_q - F_1 K_2'  K_{\bw(\alpha_{n-1})} .\label{eq:AIIIn2}
\end{align}
\end{lemma}

\begin{proof}
The formula \eqref{eq:AIIIn1} follows from Lemma~\ref{lem:app1}.

We prove \eqref{eq:AIIIn2}. By a direct computation, we have
\begin{align*}
\tT_{\bs_2}^{-1}(F_1)
&= \big[\tT_{2\cdots n-1 \cdots 3}^{-1}(F_2),[F_2,F_1]_q\big]_q\\
&= \big[\tT_{2\cdots n-2}^{-1}\tT_{2\cdots n-2}(F_{n-1}),[F_2,F_1]_q\big]_q
\\
&=\big[ \tT_{3\cdots n-2}(F_{n-1}),[F_2,F_1]_q\big]_q
\\
&= \big[ \tT_{\bw}(F_{n-1}),[F_2,F_1]_q\big]_q,
\end{align*}
where the last equality follows by applying Lemma \ref{lem:AIIIn2} and noting that  $\bw(\alpha_{n-1})=s_{3\cdots n-2}(\alpha_{n-1})$.
Recalling that $B_{n-1}^\sigma=F_{n-1}+K_{n-1}\tT_{\bw}^{-1}(E_2)$, we compute the RHS of \eqref{eq:AIIIn2} as follows:
\begin{align*}
\big[ \tT_{\bw}(B_{n-1}^\sigma),[B_2^\sigma,F_1]_q\big]_q
&=\big[ \tT_{\bw}(B_{n-1}^\sigma),[F_2 ,F_1]_q\big]_q\\
&=\big[ \tT_{\bw}(F_{n-1} ),[F_2 ,F_1]_q\big]_q+\big[ \tT_{\bw}(K_{n-1}) E_2 ,[F_2 ,F_1]_q\big]_q\\
&=\tT_{\bs_2}^{-1}(F_1) + \big[ E_2 ,[F_2 ,F_1]_q\big] \tT_{\bw}(K_{n-1})\\
&=\tT_{\bs_2}^{-1}(F_1) + F_1 K_2'K_{\bw(\alpha_{n-1})}.
\end{align*}
This proves the formula \eqref{eq:AIIIn2}.
\end{proof}

\begin{lemma}\label{lem:AIIIn3}
We have
\begin{align}\label{eq:AIIIn3}
\tT_{\bs_1}^{-1}\big(\tT_{\bw} (E_{n-1})K_2'\big)&=[B_1,\tT_{\bw}(E_{n-1})K_2']_q,\\\label{eq:AIIIn4}
\tT_{\bs_2}^{-1}(E_n K_1')&= \big[ \tT_{\bw}(B_{n-1}),[B_2,E_n K_1']_q\big]_q - E_n K_1' K_2' K_{\bw(\alpha_{n-1})}.
\end{align}
\end{lemma}

\begin{proof}
Note that $\tT_{\bs_1}=\tT_1 \tT_n$ commutes with $\tT_{\bw}$. Hence, we have
\begin{align*}
\tT_{\bs_1}^{-1}\big(\tT_{\bw} (E_{n-1})K_2'\big)&=\vs_{1,\dm}^{-1/2}\tT_{\bw} \tT_{\bs_1}^{-1} (E_{n-1})K_2'K_1'\\
&=\vs_{1,\dm}^{-1}\tT_{\bw}  \big([E_{n-1} E_n]_{q^{-1}}\big)K_2'K_1'\\
&= \tT_{\bw}  [E_n, E_{n-1} ]_{q}K_2'K_1'\\
&=[B_1,\tT_{\bw}(E_{n-1})K_2']_q
\end{align*}
where the last step follows from Lemma~\ref{lem:app1}. Hence, we have proved \eqref{eq:AIIIn3}.

We next prove \eqref{eq:AIIIn4}. In this case,  $\tau_0=\tau$, $\tT_{\bw}(\ck_n)=\ck_n=\tk_n$, and we simplify $\cL$ in \eqref{eq:D} as $\cL=\tT_{\bw}\tT_{w_0}$. Applying $\cL$ to \eqref{eq:AIIIn2} and then using \eqref{eq:app2}-\eqref{eq:app3}, we have
\begin{align}\notag
\tT_{\bs_2}^{-1} &(E_n K_1')\tT_{w_{\bullet,2}}(\tk_n^{-1})\\\notag
&= q^{-4}\big[ \tT_{\bw}(B_{n-1})\ck_2^{-1},[B_2\tT_{\bw}(\ck_{n-1}^{-1}),E_n K_1' \tk_n^{-1} ]_q\big]_q\\
 &\quad - E_n K_1' \tk_n^{-1} \cL(K_2'  K_{\bw(\alpha_{n-1})}).\label{eq:AIIIn5}
\end{align}
For a weight reason, we have
\begin{align*}
\tT_{\bw}(\ck_{n-1}^{-1}) E_n &=q E_n\tT_{\bw}(\tk_{n-1}^{-1}), \\
\tk_n^{-1} B_2&=q B_2 \tk_n^{-1},\\
\ck_2^{-1} B_2 E_n&=q^2 B_2 E_n\ck_2^{-1},\\
\tk_n^{-1}\tT_{\bw}(\ck_{n-1}^{-1}) \tT_{\bw}(B_{n-1})&=q^2\tT_{\bw}(B_{n-1})\tk_n^{-1}\tT_{\bw}(\ck_{n-1}^{-1}).
\end{align*}
In addition, by \eqref{def:Ki}, we have $\cL(K_2'  K_{\bw\alpha_{n-1}})=q^{-1}\tT_{\bw}(\ck_{n-1}^{-1})\ck_2^{-1}  K_2' K_{\bw\alpha_{n-1}}$. Using these formulas, we rewrite \eqref{eq:AIIIn5} as
\begin{align}\notag
\tT_{\bs_2}^{-1}  (E_n K_1')\tT_{w_{\bullet,2}}(\tk_n^{-1})
&= q^{-1}\big[ \tT_{\bw}(B_{n-1}),[B_2,E_n K_1' ]_q\big]_q\tk_n^{-1}\tT_{\bw}(\ck_{n-1}^{-1})\ck_2^{-1}\\
 & \quad -q^{-1} E_n K_1' \tk_n^{-1} \tT_{\bw}(\ck_{n-1}^{-1})\ck_2^{-1}  K_2'  K_{\bw\alpha_{n-1}}.
 \label{eq:AIIIn6}
\end{align}
Finally, we have $\tT_{w_{\bullet,2}}(\tk_n^{-1}) = q^{-1}\tk_n^{-1} \tT_{\bw}(\ck_{n-1}^{-1})\ck_2^{-1}$. Then the formula \eqref{eq:AIIIn4} follows from \eqref{eq:AIIIn6}.
\end{proof}


%
%
\subsection{Type DIII$_5$}

Consider the rank 2 Satake diagram of type DIII$_5$:
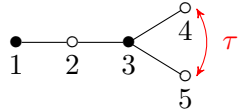
\begin{center}
   \begin{tikzpicture}[baseline=0,scale=1.5]
		\node at (-0.5,0) {$\bullet$};
		\node at (0,0) {$\circ$};
		\node at (0.5,0) {$\bullet$};
        \node at (1,0.3) {$\circ$};
        \node at (1,-0.3) {$\circ$};
		\draw[-] (-0.45,0) to (-0.05, 0);
		\draw[-] (0.05, 0) to (0.45,0);
		\draw[-] (0.5,0) to (0.965, 0.285);
		\draw[-] (0.5, 0) to (0.965,-0.285);
		\node at (-0.5,-0.2) {1};
		\node at (0,-0.2) {2};
		\node at (0.5,-0.2) {3};
        \node at (1,0.1) {4};
        \node at (1,-0.5) {5};
        \draw[bend left,<->,red] (1.1,0.3) to (1.1,-0.3);
        \node at (1.4,0) {$\textcolor{red}{\tau} $};
	\end{tikzpicture}\\
$\quad \vs_{2,\diamond}=-q^{-1},\qquad\quad \vs_{4,\diamond}=\vs_{5,\diamond}=-q^{-1/2}$, \\
$\bs_2=s_2 s_1 s_3 s_2, \qquad \bs_4= s_4 s_5 s_3 s_4 s_5.$
\end{center}
In this case, Proposition~\ref{prop:rktwoRij} is reformulated and proved as Lemmas \ref{lem:DIII1}--\ref{lem:DIII2} below.

\begin{lemma}  \label{lem:DIII1}
We have
\begin{align}\label{eq:DIII1}
\tT_{\bs_2}^{-1}(F_4)&=[\tT_3( B_2^\sigma) ,F_4 ]_q,\\
\tT_{\bs_4}^{-1}(F_2)&=\big[B_4^\sigma ,  [\tT_3(B_5^\sigma) ,F_2]_q\big]_q -\tT_3^{-2}(F_2)  K_4 K_5'K_3'.\label{eq:DIII2}
\end{align}
\end{lemma}

\begin{proof}
The proof for \eqref{eq:DIII1} is similar to that of Lemma \ref{lem:AII}, and thus omitted.

We prove \eqref{eq:DIII2}. By a direct computation, we have
\begin{align}\notag
\tT_{\bs_4}^{-1}(F_2)&=\Big[\big[F_4 , [F_5 , F_3]_q\big]_q,F_2 \Big]_q=\big[F_4 , [\tT_3(F_5) , F_2]_q\big]_q.
\end{align}
Note that $B_5^\sigma=F_5+ K_5 \tT_3^{-1}(E_4)$. Since
$[\tT_3(K_5) E_4 ,F_2 ]_q =  q [E_4,F_2] K_3 K_5=0,$
we have $[\tT_3(B_5^\sigma) , F_2]_q=[\tT_3(F_5) , F_2]_q$. We now compute the first term of RHS~\eqref{eq:DIII2} as
\begin{align*}
\big[B_4^\sigma ,  [\tT_3(B_5^\sigma) ,F_2]_q\big]_q
 &=\big[B_4^\sigma , [\tT_3(F_5),F_2 ]_q\big]_q\\
 &=\big[F_4  ,[\tT_3(F_5) ,F_2 ]_q\big]_q+ \big[K_4 \tT_3^{-1}(E_4),[\tT_3(F_5)  ,F_2]_q\big]_q\\
 &=\tT_{\bs_4}^{-1}(F_2)+ K_4 \big[ \tT_3^{-1}(E_4),[\tT_3(F_5)  ,F_2]_q\big] \\
 &=\tT_{\bs_4}^{-1}(F_2)- q^{-1}\big[[E_3 ,F_3]_{q^2}, F_2\big]_q  K_4 K_5'\\
 &=\tT_{\bs_4}^{-1}(F_2)+ \tT_3^{-2}(F_2) K_4 K_5' K_3'.
\end{align*}
This proves \eqref{eq:DIII2}.
\end{proof}

\begin{lemma}\label{lem:DIII2}
We have
\begin{align}\label{eq:DIII3}
\tT_{\bs_2}^{-1}(\tT_{\bw}(E_5)K_4')&=[ \tT_3(B_2) ,\tT_{\bw}(E_5)K_4']_q,\\\label{eq:DIII4}
\tT_{\bs_4}^{-1}(\tT_{\bw}(E_2)K_2')&=\big[B_4 ,  [\tT_3(B_5) ,\tT_{\bw}(E_2)K_2']_q\big]_q -\tT_3^{-2}(\tT_{\bw}(E_2)K_2')  K_4 K_5'K_3'.
\end{align}
\end{lemma}

\begin{proof}
We prove \eqref{eq:DIII4}. The proof for  \eqref{eq:DIII3} is easier and hence omitted.

By Lemma~\ref{lem:cL}, the operator $\cL$ defined in \eqref{eq:D} commutes with $\tT_{3},\tT_{\bs_4}$. Applying $\cL$ to \eqref{eq:DIII2} and using \eqref{eq:app2}-\eqref{eq:app3}, we have
\begin{align}\notag
\tT_{\bs_4}^{-1} & (\tT_{\bw}(E_2)K_2')\tT_{w_{\bullet,4}}(\ck_2^{-1})\\\notag
&=q^{-4}\big[B_4 \tT_{3}(\ck_5^{-1}),  [\tT_3(B_5) \ck_4^{-1},\tT_{\bw}(E_2)K_2'\tT_{\bw }(\ck_2^{-1})]_q\big]_q\\
& \quad -\tT_3^{-2}(\tT_{\bw}(E_2)K_2')\tT_{ \bw }(\ck_2^{-1}) \cL( K_4 K_5'K_3').
 \label{eq:DIII5}
\end{align}
For a weight reason, we have
\begin{align*}
\ck_4^{-1} \tT_{\bw}(E_2)&=q\tT_{\bw}(E_2)\ck_4^{-1},\\
\tT_{\bw }(\ck_2^{-1})\tT_3(B_5) &= q\tT_3(B_5)\tT_{\bw }(\ck_2^{-1}),\\
\tT_{3}(\ck_5^{-1}) \tT_3(B_5)\tT_{\bw}(E_2)&=q^2\tT_3(B_5)\tT_{\bw}(E_2)\tT_{3}(\ck_5^{-1}) ,\\
\ck_4^{-1}\tT_{\bw }(\ck_2^{-1})B_4 &= q^2B_4\ck_4^{-1}\tT_{\bw }(\ck_2^{-1}).
\end{align*}
We also have $\cL( K_4 K_5'K_3')=q^{-1}\tT_{3}(\ck_5^{-1})\ck_4^{-1}K_4 K_5'K_3'$. Hence \eqref{eq:DIII5} is written as
\begin{align}\notag
\tT_{\bs_4}^{-1} & (\tT_{\bw}(E_2)K_2')\tT_{w_{\bullet,4}}(\ck_2^{-1})\\\notag
&=q^{-1}\big[B_4,  [\tT_3(B_5),\tT_{\bw}(E_2)K_2']_q\big]_q \tT_{\bw }(\ck_2^{-1})\tT_{3}(\ck_5^{-1})\ck_4^{-1}\\
& -q^{-1} \tT_3^{-2}(\tT_{\bw}(E_2)K_2')K_4 K_5'K_3'\tT_{ \bw }(\ck_2^{-1}) \tT_{3}(\ck_5^{-1})\ck_4^{-1}.\label{eq:DIII6}
\end{align}
Finally, by definition of $\ck_i$ \eqref{def:Ki}, we have $\tT_{w_{\bullet,4}}(\ck_2^{-1})=q^{-1}\tT_{ \bw }(\ck_2^{-1}) \tT_{3}(\ck_5^{-1})\ck_4^{-1}$. Thus, \eqref{eq:DIII6} implies \eqref{eq:DIII4}.
\end{proof}


%
%
\subsection{Type EIII}

Consider the rank 2 Satake diagram of type EIII:
\begin{center}
\begin{tikzpicture}[baseline = 0, scale =1.5]
		\node at (-1,0) {$\circ$};
        \node at (-1,-0.2) {1};
		\draw (-0.95,0) to (-0.55,0);
		\node at (-0.5,0) {$\bullet$};
        \node at (-0.5,-0.2) {2};
		\draw (-0.45,0) to (-0.05,0);
		\node at (0,0) {$\bullet$};
        \node at (0.1,-0.2) {3};
		\draw (0.05,0) to (0.45,0);
		\node at (0.5,0) {$\bullet$};
		\node at (0.5,-0.2) {4};
		\draw (0.55,0) to (0.95,0);
		\node at (1,0) {$\circ$};
		\node at (1,-0.2) {5};
		\draw (0,-0.05) to (0,-0.45);
		\node at (0,-0.5) {$\circ$};
		\node at (0,-0.7) {6};
        \draw[bend left, <->, red] (-0.9,0.1) to (0.9,0.1);
        \node at (0,0.45) {$\color{red} \tau $};
	\end{tikzpicture}\\
$\vs_{1,\diamond}=\vs_{5,\diamond}=-q^{-1/2},\qquad \vs_{6,\diamond}=-q^{-1}$\\
$\qquad \bs_1=s_1 \cdots s_5 \cdots s_1,\qquad \bs_6=s_6 s_3 s_2 s_4 s_3 s_6$\\
$\bw=s_3s_2 s_4 s_3 s_2 s_4 =s_2 s_4 s_3 s_2 s_4  s_3.$
\end{center}
In this case, Proposition~\ref{prop:rktwoRij} is reformulated and proved as Lemmas ~\ref{lem:EIII1}--\ref{lem:EIII2} below.

\begin{lemma}  \label{lem:EIII1}
We have
\begin{align}\label{eq:EIII1}
\tT_{\bs_6}^{-1}(F_1)&=[\tT_{23}(B_6^\sigma),F_1]_q,\\
\tT_{\bs_1}^{-1}(F_6)&=\big[\tT_4(B_5^\sigma),[\tT_{32}(B_1^\sigma),F_6]_q\big]_q -\tT_{32323}^{-1}(F_6) K_1' K_2' K_3'K_4 K_5\label{eq:EIII2}
\end{align}
\end{lemma}

\begin{proof}
We have
\begin{align*}
\tT_{\bs_6}^{-1}(F_1)&=\tT_{632}^{-1}(F_1)=[\tT_{63}^{-1}(F_2),F_1]_q =[\tT_{23}(F_6),F_1]_q=[\tT_{23}(B_6^\sigma),F_1]_q.
\end{align*}
Hence, \eqref{eq:EIII1} follows.

We next prove \eqref{eq:EIII2}. We have
\begin{align}\notag
\tT_{\bs_1}^{-1}(F_6)
&=\tT_{1\cdots 5 \cdots 3}^{-1}(F_6)
=\tT_{123}^{-1}[\tT_{454}^{-1}(F_3),F_6]_q
=\tT_{123}^{-1}[\tT_{34}(F_5),F_6]_q
\\
&=\big[\tT_4(F_5),[\tT_{12}^{-1}(F_3),F_6]_q\big]_q
=\big[\tT_4(F_5),[\tT_{32}(F_1),F_6]_q\big]_q.
  \label{eq:EIII3}
\end{align}
Recall that $B_1^\sigma=F_1+K_1\tT_{\bw}^{-1}(E_5)$. Hence,
\begin{align}\notag
[\tT_{32}(B_1^\sigma),F_6]_q&=[\tT_{32}(F_1),F_6]_q + [K_{123} \tT_3\tT_{434}^{-1}(E_5),F_6]_q\\
&=[\tT_{32}(F_1),F_6]_q + K_{123}[ \tT_{4}^{-1}(E_5),F_6]= [\tT_{32}(F_1),F_6]_q .\label{eq:EIII4}
\end{align}
On the other hand, we have
\begin{align}\notag
\tT_{32323}^{-1}(F_6) &= \tT_{323}^{-1}[\tT_2^{-1}(F_3),F_6]_q=\tT_{323}^{-1}[\tT_3(F_2),F_6]_q\\\notag
&= - [\tT_3^{-1}(E_2 K_2'^{-1}),\tT_{232}^{-1}(F_6)]_q\\\notag
&= - q^{-1} [\tT_3^{-1}(E_2),\tT_{23}^{-1}(F_6)]_{q^2}  K_2'^{-1} K_3'^{-1}\\\notag
&= - q^{-1} \big[\tT_3^{-1}(E_2),[\tT_{2}^{-1}(F_3),F_6]_q \big]_{q^2}  K_2'^{-1} K_3'^{-1}\\\label{eq:EIII5}
&= - q^{-1} \big[[\tT_3^{-1}(E_2), \tT_{2}^{-1}(F_3)]_{q^2} ,F_6 \big]_q K_2'^{-1} K_3'^{-1},
\end{align}
We now rewrite RHS \eqref{eq:EIII2} as follows:
\begin{align*}
 \big[\tT_4(B_5^\sigma), & [\tT_{32}(B_1^\sigma),F_6]_q\big]_q
\\
\overset{\eqref{eq:EIII4}}{=}&\big[\tT_4(B_5^\sigma),[\tT_{32}(F_1 ),F_6]_q\big]_q\\
\overset{\qquad}{=}&\big[\tT_4(F_5 ),[\tT_{32}(F_1 ),F_6]_q\big]_q+ \big[K_4K_5\tT_{232}^{-1}(E_1),[\tT_{32}(F_1 ),F_6]_q\big]_q\\
\overset{\eqref{eq:EIII3}}{=}&\tT_{\bs_1}^{-1}(F_6) + K_4K_5 \big[\tT_{32}^{-1} (E_1),[\tT_{32}(F_1 ),F_6]_q\big]\\
\overset{\qquad}{=}&\tT_{\bs_1}^{-1}(F_6) - q^{-1} \big[[\tT_3^{-1}(E_2),\tT_3(F_2)]_{q^2},F_6\big]_q K_1' K_4 K_5\\
\overset{\eqref{eq:EIII5}}{=}&\tT_{\bs_1}^{-1}(F_6) + \tT_{32323}^{-1}(F_6) K_1' K_2' K_3' K_4 K_5.
\end{align*}
Therefore, the formula \eqref{eq:EIII2} follows.
\end{proof}

\begin{lemma}\label{lem:EIII2}
We have
\begin{align}\label{eq:EIII6}
\tT_{\bs_6}^{-1}\big(\tT_{\bw}(E_5)K_1'\big)&=[\tT_{23}(B_6),\tT_{\bw}(E_5)K_1']_q,\\\notag
\tT_{\bs_1}^{-1}\big(\tT_{\bw}(E_6)K_6'\big)&=\big[\tT_4(B_5),[\tT_{32}(B_1),\tT_{\bw}(E_6)K_6']_q\big]_q \\
&-\tT_{32323}^{-1}\big(\tT_{\bw}(E_6)K_6'\big) K_1' K_2' K_3'K_4 K_5.\label{eq:EIII7}
\end{align}
\end{lemma}

\begin{proof}
 Recall from Lemma~\ref{lem:cL} that the operator $\cL$ defined in \eqref{eq:D} commutes with each of the automorphisms $\tT_4,\tT_{32},\tT_{23},\tT_{\bs_1},\tT_{\bs_6}$.

We first prove the formula \eqref{eq:EIII6}.
 Applying $\cL$ to \eqref{eq:EIII1} and then using \eqref{eq:app2}-\eqref{eq:app3}, we obtain
\begin{align}
\tT_{\bs_6}^{-1} & \big(\tT_{\bw}(E_5)K_1'\big) \tT_{w_{\bullet,6}}(\ck_5^{-1})
\notag \\
&= -q^{-2}[\tT_{23}(B_6)\tT_{432}(\ck_6^{-1}),\tT_{\bw}(E_5)K_1'\tT_{ \bw}(\ck_5^{-1})]_q
 \notag \\
&= -q^{-1}[\tT_{23}(B_6),\tT_{\bw}(E_5)K_1']_q\tT_{432}(\ck_6^{-1})\tT_{ \bw}(\ck_5^{-1}),
 \label{eq:EIII9}
\end{align}
where the last equality follows by a weight consideration.
On the other hand, we have
$\tT_{w_{\bullet,6}}(\ck_5^{-1})=-q^{-1}\tT_{432}(\ck_6^{-1})\tT_{ \bw}(\ck_5^{-1}).$
Thus the formula \eqref{eq:EIII6} follows from \eqref{eq:EIII9}.

We next prove the formula \eqref{eq:EIII7}.
Applying $\cL$ in the identity \eqref{eq:D} to \eqref{eq:EIII2} and using \eqref{eq:app2}-\eqref{eq:app3}, we obtain
\begin{align}\notag
&\tT_{\bs_1}^{-1}\big(\tT_{\bw}(E_6)K_6'\big)\tT_{w_{\bullet,1}}(\ck_6^{-1})\\\notag
&=q^{-4}\big[\tT_4(B_5)\tT_4\tT_{\bw}(\ck_1^{-1}),[\tT_{32}(B_1)\tT_{32}\tT_{\bw}(\ck_5^{-1}),\tT_{\bw}(E_6)K_6'\tT_{\bw}(\ck_6^{-1})]_q\big]_q \\
&-\tT_{32323}^{-1}\big(\tT_{\bw}(E_6)K_6'\big)\tT_{\bw}(\ck_6^{-1}) \cL(K_1' K_2' K_3'K_4 K_5).\label{eq:EIII10}
\end{align}
Note that $\tT_4\tT_{\bw}(\ck_1^{-1})= \tT_{32}(K_1^{-1})\tT_4(K_5')^{-1}$ and $\tT_{32}\tT_{\bw}(\ck_5^{-1})= \tT_4(K_5^{-1})\tT_{32}(K_1')^{-1}$. We also note that $K_1' K_2' K_3'K_4 K_5=\tT_{32}(K_1')\tT_4(K_5)$ and then $\cL(K_1' K_2' K_3'K_4 K_5)=q^{-1}\tT_{4}(K_5')^{-1}\tT_{32}(K_1^{-1})$. Hence, \eqref{eq:EIII10} can be rewritten as
\begin{align}\notag
&\tT_{\bs_1}^{-1}\big(\tT_{\bw}(E_6)K_6'\big)\tT_{w_{\bullet,1}}(\ck_6^{-1})\\\notag
&=q^{-4}\big[\tT_4(B_5)\tT_{32}(K_1^{-1})\tT_4(K_5')^{-1},[\tT_{32}(B_1 K_1'^{-1}) \tT_4(K_5^{-1}),\tT_{\bw}(E_6)K_6'\tT_{\bw}(\ck_6^{-1})]_q\big]_q \\
&\quad-q^{-1}\tT_{32323}^{-1}\big(\tT_{\bw}(E_6)K_6'\big)\tT_{\bw}(\ck_6^{-1}) \tT_{4}(K_5')^{-1}\tT_{32}(K_1^{-1}).\label{eq:EIII11}
\end{align}
For a weight reason, we have
\begin{align*}
\tT_4(K_5^{-1})\tT_{32}(K_1')^{-1}\tT_{\bw}(E_6)&=q\tT_{\bw}(E_6)\tT_4(K_5^{-1})\tT_{32}(K_1')^{-1},
\\
\tT_{\bw}(\ck_6^{-1})\tT_{32}(B_1)&=q\tT_{32}(B_1)\tT_{\bw}(\ck_6^{-1}),
\\
\tT_{32}(K_1)\tT_4(K_5')[\tT_{32}(B_1),\tT_{\bw}(E_6)K_6']_q&=q^{-2}[\tT_{32}(B_1),\tT_{\bw}(E_6)K_6']_q\tT_{32}(K_1)\tT_4(K_5'),
\\
\tT_4(K_5^{-1})\tT_{32}(K_1')^{-1}\tT_{\bw}(\ck_6^{-1})\tT_4(B_5)&=q^2\tT_4(B_5)\tT_4(K_5^{-1})\tT_{32}(K_1')^{-1}\tT_{\bw}(\ck_6^{-1}).
\end{align*}
Using the above four identities, we rewrite \eqref{eq:EIII11} as
\begin{align}\notag
&\tT_{\bs_1}^{-1}\big(\tT_{\bw}(E_6)K_6'\big)\tT_{w_{\bullet,1}}(\ck_6^{-1})\\\notag
&=q^{-1}\big[\tT_4(B_5),[\tT_{32}(B_1),\tT_{\bw}(E_6)K_6']_q\big]_q \tT_{\bw}(\ck_6^{-1})\tT_{32}(K_1 K_1')^{-1}\tT_4(K_5 K_5')^{-1} \\
&-q^{-1}\tT_{32323}^{-1}\big(\tT_{\bw}(E_6)K_6'\big) K_1' K_2' K_3'K_4 K_5\tT_{\bw}(\ck_6^{-1})\tT_{32}(K_1 K_1')^{-1}\tT_4(K_5 K_5')^{-1}.\label{eq:EIII12}
\end{align}
Moreover, we have
$\tT_{w_{\bullet,1}}(\ck_6^{-1})=q^{-1} \tT_{\bw}(\ck_6^{-1})\tT_{32}(K_1 K_1')^{-1}\tT_4(K_5 K_5')^{-1}.$
Thus, \eqref{eq:EIII12} implies the desired formula \eqref{eq:EIII7}.
\end{proof}



\begin{thebibliography}{AWW18}

\bibitem[Ar62]{Ar62} S.~Araki,
{\em On root systems and an infinitesimal classification of irreducible symmetric spaces}, J. Math. Osaka City Univ. {\bf 13} (1962), 1--34.

\bibitem[AV22]{AV22} A.~Appel and B.~Vlaar,
\textit{Universal $k$-matrices for quantum Kac-Moody algebras}, Represent. Theory {\bf 26} (2022), 764--824. 

\bibitem[Br13]{Br13} T. Bridgeland,
{\em Quantum groups via Hall algebras of complexes}, Ann. Math. {\bf 177} (2013), 739--759.

\bibitem[BBMR]{BBMR}
V. Back-Valente, N. Bardy-Panse, H. Ben Messaoud, G. Rousseau,
{\em Formes presque-d\'eploy\'ees des alg\`ebres de Kac-Moody: classification et racines relatives}, J. Algebra {\bf 171} (1995), 43--96.

\bibitem[BK19]{BK19} M.~Balagovic and S.~Kolb,
\textit{Universal $K$-matrix for quantum symmetric pairs}, J. Reine Angew. Math. {\bf 747} (2019), 299--353.



\bibitem[BV21]{BV21}
R. Bezrukavnikov, Roman and K. Vilonen,
{\em Koszul duality for quasi-split real groups}, Invent. Math. {\bf 226} (2021), 139--193.

\bibitem[BW18a]{BW18a} H. Bao and W. Wang,
\textit{ A new approach to Kazhdan-Lusztig theory of type B via quantum symmetric pairs}, Asterisque {\bf 402} (2018), vii+134pp. \arxiv{1310.0103}

\bibitem[BW18b]{BW18b} H. Bao and W. Wang,
\textit{Canonical bases arising from quantum symmetric pairs}, Inventiones Math. {\bf 213} (2018), 1099--1177.

\bibitem[BW21]{BW21} H. Bao and W. Wang,
\textit{Canonical bases arising from quantum symmetric pairs of Kac-Moody type}, Compositio Math. {\bf 157} (2021), 1507--1537. 

\bibitem[Ch07]{Ch07} L. Chekhov,
{\em Teichm\"uller theory of bordered surfaces}, SIGMA Symmetry Integrability Geom. Methods Appl. {\bf 3} (2007),
Paper 066, 37 pp.

\bibitem[CLW21]{CLW21a} X.~Chen, M.~Lu and W.~Wang,
{\em Serre-Lusztig relations for $\imath$quantum groups}, Commun. Math. Phys. {\bf 382} (2021), 1015--1059.

\bibitem[CLW23]{CLW23} X.~Chen, M.~Lu and W.~Wang,
{\em Serre-Lusztig relations for $\imath$quantum groups III}, J. Pure Appl. Algebra {\bf 227} (2023), no. 4, Paper No. 107253.

\bibitem[Dob19]{Dob19} L.~Dobson,
\textit{Braid group actions and quasi $K$-matrices for quantum symmetric pairs}, Ph.D. thesis, School of Mathematics, Statistics and Physics, Newcastle University, 2019.

\bibitem[Dob20]{Dob20} L.~Dobson,
\textit{ Braid group actions for quantum symmetric pairs of type AIII/AIV}, J. Algebra {\bf 564} (2020), 151--198.

\bibitem[DK19]{DK19} L.~Dobson and S.~Kolb, \textit{Factorisation of quasi $K$-matrices for quantum symmetric pairs}, Selecta Math. (N.S.) {\bf 25} (2019), 63.

\bibitem[Ja95]{Ja95} J.C. Jantzen,
{\em Lectures on quantum groups}, Grad. Studies in Math. {\bf 6}, Amer. Math. Soc., Providence (1996).

\bibitem[KR90]{KR90} A.N. Kirillov and N. Reshetikhin,
{\em $q$-Weyl group and a multiplicative formula for universal $R$-matrices}, Commun. Math. Phys. {\bf 134} (1990), 421--431.

\bibitem[Ko14]{Ko14} S. Kolb,
{\em Quantum symmetric Kac-Moody pairs}, Adv. Math. {\bf 267} (2014), 395--469.

\bibitem[Ko21]{Ko21} S. Kolb,
\textit{The bar involution for quantum symmetric pairs – hidden in plain sight},
Hypergeometry, Integrability and Lie theory, 69–77,
Contemp. Math., 780, Amer. Math. Soc., Providence (2022).

\bibitem[KP11]{KP11} S. Kolb and J. Pellegrini,
\textit{Braid group actions on coideal subalgebras of quantized enveloping algebras}, J. Algebra {\bf 336} (2011), 395--416.

\bibitem[KY20]{KY20} S. Kolb and M. Yakimov,
\textit{Symmetric pairs for Nichols algebras of diagonal type via star products},  Adv. Math. {\bf365} (2020), 107042, 69 pp.

\bibitem[Let99]{Let99} G.~Letzter,
{\em Symmetric pairs for quantized enveloping algebras}, J. Algebra {\bf 220} (1999), 729--767.

\bibitem[Let02]{Let02} G.~Letzter,
{\em Coideal subalgebras and quantum symmetric pairs},
New directions in Hopf algebras (Cambridge), MSRI publications, {\bf 43}  (2002), Cambridge University Press, 117--166.

\bibitem[LS90]{LS90} S.~Levendorskii and Ya.~Soibelman,
\textit{Some applications of the quantum Weyl groups}, J.~ Geom. Phys. {\bf 7} (1990), 241--254.

\bibitem[Lus76]{Lus76} G.~Lusztig,
\textit{Coxeter orbits and eigenspaces of Frobenius}, Invent. Math. {\bf 28} (1976), 101--159.

\bibitem[Lus90a]{Lus90a} G.~Lusztig,
\textit{Finite-dimensional Hopf algebras arising from quantized universal enveloping algebra}, J. Amer. Math. Soc. {\bf 3} (1990), 257--296.


\bibitem[Lus90b]{Lus90b} G.~Lusztig,
\textit{Quantum groups at roots of 1}, Geom. Dedicata {\bf 35} (1990), 89--114.

\bibitem[Lus93]{Lus93} G.~Lusztig, {\em Introduction to quantum groups},
Modern Birkh\"auser Classics, Reprint of the 1993 Edition, Birkh\"auser, Boston, 2010.

\bibitem[Lus03]{Lus03}
G. Lusztig,
{\em Hecke algebras with unequal parameters},
CRM Monograph Series {\bf 18}, Amer. Math. Soc., Providence, RI, 2003; for an enhanced version, see \href{https://arxiv.org/abs/math/0208154}{arXiv:0208154v2}



\bibitem[LW21a]{LW21a} M.~Lu and W.~Wang,
{\em Hall algebras and quantum symmetric pairs II: reflection functors}, Commun. Math. Phys. {\bf 381} (2021), 799--855.

\bibitem[LW21b]{LW21b} M.~ Lu and W.~Wang,
{\em Braid group symmetries on quasi-split $\imath$quantum groups via $\imath$Hall algebras}, Selecta Math. {\bf 28}, 84 (2022).

\bibitem[LW22]{LW22} M.~Lu and W.~Wang,
{\em Hall algebras and quantum symmetric pairs I: foundations}, Proc. Lond. Math. Soc. (3) {\bf 124} (2022), 1--82. 

\bibitem[MR08]{MR08} A. Molev and E. Ragoucy,
{\em Symmetries and invariants of twisted quantum algebras and associated Poisson algebras}, Rev. Math. Phys. {\bf 20} (2008), 173--198.

\bibitem[OV90]{OV90} A. Onishchik and E. Vinberg,
{\em Lie groups and algebraic groups}, Springer Series in Soviet Mathematics. Springer, Berlin, 1990.

\bibitem[Rin96]{Rin96} C.M. Ringel,
{\em PBW-bases of quantum groups}, J. Reine Angrew. Math. {\bf 470} (1996), 51--88.

\bibitem[W21a]{W21a} H.~Watanabe,
{\em Classical weight modules over $\imath$quantum groups}, J. Algebra {\bf 578} (2021), 241--302.

\bibitem[W21b]{W21b} H.~Watanabe,
{\em Crystal bases of modified ıquantum groups of certain quasi-split types}, \arxiv{2110.07177}

 \end{thebibliography}
\end{document}